\documentclass{amsart}

\usepackage[T1]{fontenc}
\usepackage{enumerate, amsmath, amsfonts, amssymb, amsthm, thmtools, mathrsfs, wasysym, graphics, graphicx, xcolor, url, hyperref, hypcap, xargs, multicol, pdflscape, multirow, hvfloat, array, ae, aecompl, pifont, mathtools, a4wide, float, blkarray, overpic, nicefrac}
\usepackage[shortlabels, inline]{enumitem}
\usepackage[noabbrev,capitalise]{cleveref}
\usepackage[normalem]{ulem}
\usepackage{marginnote}
\hypersetup{colorlinks=true, citecolor=darkblue, linkcolor=darkblue, hypertexnames=false}
\usepackage[all]{xy}
\usepackage{tikz}
\usepackage{tikz-cd}
\usetikzlibrary{trees, decorations, decorations.pathmorphing, decorations.markings, decorations.shapes, shapes, arrows, matrix, calc, fit, intersections, patterns, angles}
\graphicspath{{figures/}{figures/diagonals/}{figures/walks/}{figures/tubes/}{figures/blocks/}}
\makeatletter\def\input@path{{figures/}}\makeatother
\usepackage{caption}
\usepackage[export]{adjustbox}

\newtheorem{theorem}{Theorem}[section]
\newtheorem{corollary}[theorem]{Corollary}
\newtheorem{proposition}[theorem]{Proposition}
\newtheorem{lemma}[theorem]{Lemma}

\newtheorem*{theorem*}{Theorem}

\theoremstyle{definition}
\newtheorem{definition}[theorem]{Definition}
\newtheorem{example}[theorem]{Example}
\newtheorem{remark}[theorem]{Remark}

\newtheorem{notation}[theorem]{Notation}

\crefname{equation}{Equation}{Equations}

\newcommand{\R}{\mathbb{R}} 
\newcommand{\N}{\mathbb{N}} 
\newcommand{\Z}{\mathbb{Z}} 
\renewcommand{\c}[1]{{\mathcal{#1}}} 
\renewcommand{\b}[1]{{\boldsymbol{#1}}} 

\newcommand{\set}[2]{\left\{ #1 \;\middle|\; #2 \right\}} 
\newcommand{\bigset}[2]{\big\{ #1 \;\big|\; #2 \big\}} 
\newcommand{\ssm}{\smallsetminus} 
\newcommand{\dotprod}[2]{\langle \, #1 \; | \; #2 \, \rangle} 
\newcommand{\bigdotprod}[2]{\big\langle \, #1 \; \big| \; #2 \, \big\rangle} 
\newcommand{\one}{{\boldsymbol{1}}} 
\newcommand{\zero}{{\boldsymbol{0}}} 
\newcommand{\eqdef}{\mbox{\,\raisebox{0.2ex}{\scriptsize\ensuremath{\mathrm:}}\ensuremath{=}\,}} 
\newcommand{\simplex}{\triangle} 

\DeclareMathOperator{\conv}{conv} 
\DeclareMathOperator{\cone}{cone} 
\DeclareMathOperator{\suppmin}{\mu} 

\newcommand{\ie}{\textit{i.e.}~} 
\definecolor{darkblue}{rgb}{0,0,0.7} 
\definecolor{yellow}{RGB}{252,209,33} 
\definecolor{green}{RGB}{57,181,74} 
\definecolor{violet}{RGB}{147,39,143} 
\newcommand{\darkblue}{\color{darkblue}} 
\newcommand{\defn}[1]{\textsl{\darkblue #1}} 
\newcommand{\para}[1]{\medskip\noindent\textsc{#1.}} 

\usepackage{todonotes}

\newcommandx{\polytope}[1][1 = P]{\mathsf{#1}} 
\newcommandx{\Asso}[1][1=n]{\mathsf{Asso}(#1)} 
\newcommandx{\Zono}[1][1=\digraph]{\mathsf{Zono}(#1)} 
\newcommand{\Nest}{\mathsf{Nest}} 
\newcommand{\Acycl}{\mathsf{Acyc}} 
\newcommand{\AcyclProj}{\mathsf{A\overline{cyc}}} 
\newcommand{\OrderPolytope}{\mathsf{OrdPol}} 

\newcommand{\ground}{S} 
\newcommand{\connectedComponents}{\kappa} 
\newcommand{\building}{\mathcal{B}} 
\newcommandx{\nested}[1][1=N]{\mathcal{#1}} 
\newcommandx{\nestedComplex}[2][1=\lattice,2=\building]{\mathfrak{N}_{#1}(#2)} 
\newcommand{\seminestedComplex}{\mathfrak{SN}} 
\newcommandx{\fbuilding}[1][1=F]{\mathcal{#1}} 
\newcommand{\acyclicNestedComplex}{\mathfrak{A}} 
\newcommandx{\restrGround}[2][1=B, 2=\nested]{{\ground_{#1 \in #2}}} 
\newcommandx{\complRestrGround}[3][1=R, 2=B, 3=\nested]{{#1_{#2 \in #3}}} 
\newcommandx{\restrBuilding}[3][1=\building, 2=B, 3=\nested]{{#1_{#2 \in #3}}} 
\newcommandx{\restrOM}[2][1=B, 2=\nested]{{\OM_{#1 \in #2}}} 
\newcommand{\block}[2]{B(#1,#2)} 
\newcommandx{\vertexNest}[2][1=\nested, 2=\b{\lambda}]{\b{v}({#1}, {#2})} 
\newcommandx{\con}[1][1=B]{\mathsf{C}_{{#1}}}
\newcommandx{\nestedFan}[2][1=\lattice,2=\building]{\Sigma_{#1}(#2)}

\newcommandx{\graphG}[1][1=G]{#1} 
\newcommandx{\digraph}[1][1=D]{#1} 
\newcommandx{\hypergraph}[1][1=H]{\graphG[#1]} 
\newcommandx{\tube}[1][1=T]{#1} 
\newcommandx{\tubes}[1][1=\graphG]{\building(#1)} 
\newcommandx{\tubing}[1][1=T]{\mathcal{#1}} 
\newcommandx{\pipe}[1][1=Q]{#1} 
\newcommandx{\piping}[1][1=Q]{\mathcal{#1}} 
\newcommand{\pipingComplex}{\mathfrak{P}} 
\newcommand{\affinePipingComplex}{\mathfrak{AP}} 
\newcommand{\precdot}{\mathrel{{\prec \!\! \cdot}}} 

\newcommandx{\OM}[1][1 = M]{\c{#1}} 
\newcommandx{\dependences}[1][1 = \b{A}]{\mathcal{D}(#1)} 
\newcommandx{\evaluations}[1][1 = \b{A}]{\mathcal{D}^*(#1)} 
\newcommandx{\circuits}[1][1 = \b{A}]{\mathcal{C}(#1)} 
\newcommandx{\cocircuits}[1][1 = \b{A}]{\mathcal{C}^*(#1)} 
\newcommandx{\vectors}[1][1 = \b{A}]{\mathcal{V}(#1)} 
\newcommandx{\covectors}[1][1 = \b{A}]{\mathcal{V}^*(#1)} 
\newcommandx{\rank}[1][1 = \b{A}]{\mathrm{rk}(#1)} 
\newcommandx{\corank}[1][1 = \b{A}]{\mathrm{rk}^*(#1)} 
\newcommand{\signature}{\sigma} 
\newcommandx{\UOM}[1][1=\OM]{\underline{#1}} 
\newcommandx{\Ucircuits}[1][1 = \b{A}]{\underline{\mathcal{C}}(#1)} 
\newcommandx{\Ucocircuits}[1][1 = \b{A}]{\underline{\mathcal{C}}^*(#1)} 
\newcommand{\FL}{\FP} 
\newcommandx{\sdOM}[2][1=\OM,2=\fbuilding]{\sd({#2},{#1})} 
\newcommandx{\circuitSupports}[1][1 = \b{A}]{\underline{\mathcal{C}}(#1)} 
\newcommand{\Flats}{\mathcal{F}l} 
\newcommandx{\asd}[1][1=F]{a_{#1}} 

\newcommandx{\poset}[1][1 = P]{#1} 
\newcommand{\lattice}{\mathcal{L}} 
\newcommand{\semi}{\mathcal{S}} 
\newcommand{\topone}{\hat{1}} 
\newcommand{\botzero}{\hat{0}} 
\DeclareMathOperator{\Blow}{Bl} 
\newcommandx{\Bl}[2][1 = \semi, 2=X]{\Blow_{#2}({#1})} 

\renewcommand{\t}[1]{{\smash{\tilde{#1}}}} 
\newcommand{\tb}[1]{\t{\b{#1}}} 
\newcommandx{\affinePoset}[1][1 = \poset]{\t{#1}} 
\renewcommand{\mod}[2]{\text{mod}(#1)} 
\newcommand{\quo}[2]{\mathrm{quo}(#1)} 
\newcommand{\cl}[2]{\text{cl}(#1)} 
\newcommandx{\elev}[1][1 = \affinePoset]{E(#1)} 
\newcommandx{\quotientGraph}[1][1 = \affinePoset]{#1/\Z}

\newcommandx{\sph}[1][1 = n]{\mathbb{S}^{#1}} 
\newcommandx{\ball}[1][1 = n]{\mathbb{B}^{#1}} 
\newcommand{\FP}{\mathcal{L}} 
\newcommand{\CW}{\Delta} 
\newcommand{\inter}[1]{#1^\circ} 
\newcommandx{\spa}[1][1 = \CW]{\left\lVert{#1}\right\rVert} 
\newcommand{\join}{\ast} 
\DeclareMathOperator{\sd}{sd} 
\DeclareMathOperator{\ostar}{star} 
\newcommand{\cstar}{\overline \ostar} 

\DeclareMathOperator{\ini}{in} 
\newcommandx{\Bergfan}[1][1 = \OM]{\c{B}({#1})} 
\newcommandx{\posBergfan}[1][1 = \OM]{\c{B}^+({#1})} 

\makeatletter
\def\part{\@startsection{part}{1}%
\z@{.7\linespacing\@plus\linespacing}{.8\linespacing}%
{\LARGE\sffamily\centering}}
\makeatother

\makeatletter
\def\l@section{\@tocline{1}{1pt}{0pc}{}{}}
\def\l@part{\@tocline{1}{7pt}{0pc}{}{}}
\makeatother
\let\oldtocpart=\tocpart
\renewcommand{\tocpart}[2]{\sc\large\oldtocpart{#1}{#2}}
\let\oldtocsection=\tocsection
\renewcommand{\tocsection}[2]{\bf\oldtocsection{#1}{#2}}
\let\oldtocsubsubsection=\tocsubsubsection
\renewcommand{\tocsubsubsection}[2]{\quad\oldtocsubsubsection{#1}{#2}}

\title{Facial nested complexes and acyclonestohedra}

\thanks{Partially supported by the Spanish project PID2022-137283NB-C21 of MCIN/AEI/10.13039/501100011033 / FEDER, UE, by the Spanish--German project COMPOTE (AEI PCI2024-155081-2 \& DFG 541393733), by the Severo Ochoa and María de Maeztu Program for Centers and Units of Excellence in R\&D (CEX2020-001084-M), by the Departament de Recerca i Universitats de la Generalitat de Catalunya (2021 SGR 00697), and by the French--Austrian project PAGCAP (ANR-21-CE48-0020 \& FWF I 5788).}

\author{Chiara Mantovani}
\address[Chiara Mantovani]{LIX, \'Ecole Polytechnique, Palaiseau}
\email{chiara.mantovani@lix.polytechnique.fr}
\urladdr{\url{http://www.lix.polytechnique.fr/~mantovani/}}

\author{Arnau Padrol}
\address[Arnau Padrol]{Universitat de Barcelona \& Centre de Recerca Matemàtica, Barcelona}
\email{arnau.padrol@ub.edu}
\urladdr{\url{https://www.ub.edu/comb/arnaupadrol/}}

\author{Vincent Pilaud}
\address[Vincent Pilaud]{Universitat de Barcelona \& Centre de Recerca Matemàtica, Barcelona}
\email{vincent.pilaud@ub.edu}
\urladdr{\url{https://www.ub.edu/comb/vincentpilaud/}}

\begin{document}

\begin{abstract}
Nested complexes of building sets on semilattices were studied by E.~M.~Feichtner and her coauthors as combinatorial abstractions of C.~De Concini and C.~Procesi's wonderful compactifications of complements of subspace arrangements.

We study nested complexes of building sets on the Las Vergnas face lattices of oriented matroids.
Such a nested complex is the face lattice of an oriented matroid, obtained by iterated stellar subdivisions of the positive tope.
It follows that if the oriented matroid is realizable, the nested complex is isomorphic to the boundary complex of a polytope.

We turn this into an explicit and combinatorially meaningful polytopal realization.
For this, we first prove that the facial nested complex can be embedded as the acyclic (with respect to the oriented matroid) subcomplex of the nested complex of a well-chosen boolean building set.
We then show that, in the realizable case, this acyclic subcomplex can be geometrically selected as the section of a nestohedron by the evaluation space of the vector configuration, which we call acyclonestohedron.
Moreover, by results of G.~Gaiffi, the interior of any polytope admits a stratified compactification that is diffeormorphic as a stratified space to the corresponding acyclonestohedron.

Our framework generalizes the poset associahedra and affine poset cyclohedra recently introduced by P.~Galashin, from order and affine order polytopes to any (oriented matroid) polytope.
Poset associahedra are the graphical acyclonestohedra, and our approach recovers as particular cases the main results of P.~Galashin (polytopality via stellar subdivisions and the connection to the compactification of the order polytope) and answers some of his open questions (explicit polytopal realizations and algebraic interpretations).
Besides poset associahedra, our framework unifies various other existing families of nested-like polytopes, such as the simple polytope nestohedra, the hyperoctahedral nestohedra, the design graph associahedra and the permutopermutohedra, to which our palette of results can be directly applied.

The core of our construction is the embedding of the facial nested complex inside a boolean nested complex.
More generally, we provide conditions that guarantee an embedding between nested complexes over two lattices.
For instance, any atomic nested complex has a canonical embedding inside a boolean nested complex.
As another application, we embed nested complexes over lattices of faces into nested complexes over lattices of flats, recovering as a particular case, the embedding of the positive Bergman complex into the Bergman complex.
\end{abstract}

\vspace*{-.3cm}
\maketitle

\vspace*{-.4cm}
\centerline{\includegraphics[scale=.45]{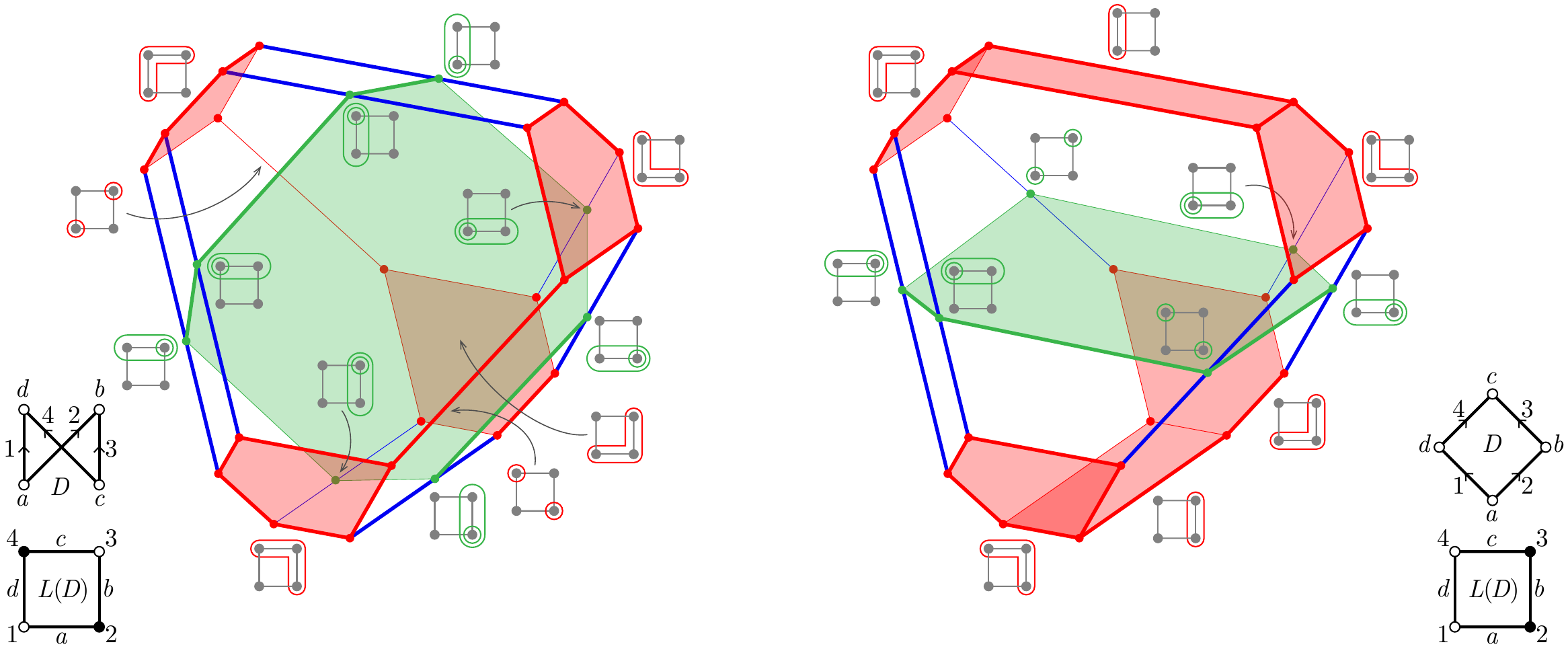}}

\clearpage
\enlargethispage{.4cm}
\tableofcontents


\section{Introduction}

\para{Lattice nested complexes}
In their influential paper~\cite{DeConciniProcesi1995}, C.~De Concini and C.~Procesi introduced \defn{wonderful compactifications} of the complement of a subspace arrangement. These are smooth compactifications in which the subspaces are replaced by a divisor with normal crossings via a sequence of blow-ups, as in W.~Fulton and R.~MacPherson's compactification~\cite{FultonMacPherson1994}. The boundary of the compactification is stratified, and the combinatorics of this stratification is governed by \defn{building sets} and \defn{nested complexes}: the building set determines which elements of the intersection lattice are blown up, and the nested complex describes the intersection pattern of the resulting strata.

The combinatorial structures underlying wonderful compactifications were further explored by E.~M.~Feichtner and her coauthors in a series of foundational papers~\cite{FeichtnerKozlov2004,FeichtnerYuzvinsky2004,FeichtnerMuller2005,FeichtnerSturmfels}, which developed into a combinatorial theory of building sets and nested complexes with rich connections to algebra, topology, and combinatorics.
We refer to~\cite{Feichtner} for a gentle survey.
In particular, E.~M.~Feichtner and D.~Kozlov extended the notions of building sets and nested complexes from intersection lattices to arbitrary meet-semilattices~\cite{FeichtnerKozlov2004}.
So far, two principal types of lattices have played a central role in applications within the geometric and algebraic combinatorics community. 

On the one hand, particular attention has been given to the \defn{lattice of flats of a matroid}. The nested complex of this lattice triangulates the \defn{Bergman fan} of the matroid~\cite{FeichtnerSturmfels, ArdilaKlivans2006}, providing a geometric realization of the nested complex as a fan. The associated toric variety gives rise to the \defn{Chow ring} of the matroid~\cite{FeichtnerYuzvinsky2004}, which in the realizable case specializes to the cohomology ring of the wonderful compactification.   
A key ingredient in the celebrated proof of the Heron--Rota--Welsh conjecture by K.~Adiprasito, J.~Huh, and E.~Katz~\cite{AdiprasitoHuhKatz} consists in proving the Hodge--Riemann relations for these Chow rings.

On the other hand, building sets and nested complexes on the \defn{boolean lattice} have been extensively studied, notably by A.~Postnikov~\cite{Postnikov} and~E.~M.~Feichtner and B.~Sturmfels~\cite{FeichtnerSturmfels}, to the point that the terms \emph{building set} and \emph{nested complex} often implicitly refer to this boolean case.
A remarkable property of boolean nested complexes is that they admit explicit geometric realizations as the boundary complexes of convex polytopes called \defn{nestohedra}~\cite{Postnikov, FeichtnerSturmfels, Zelevinsky}.
The nestohedron associated to a boolean building set is constructed as the Minkowski sum of (any positive dilation of) the faces of the standard simplex corresponding to the blocks of the building set.
This can also be seen as an explicit recipe for truncating the corresponding faces of the simplex.
A boolean building set can additionally be interpreted as the collection of (vertex sets of) connected subhypergraphs of a given hypergraph.
When the underlying hypergraph is a graph, the building set is said to be \defn{graphical}, and the nestohedron is the \defn{graph associahedron} of M.~Carr and S.~Devadoss~\cite{CarrDevadoss, Devadoss}.
Many important families of polytopes arise as special cases, including \defn{permutahedra} and \defn{associahedra}.

\para{Facial nested complexes}
In this paper, we focus on the \defn{Las Vergnas face lattices of oriented matroids}, which generalize face lattices of polytopes.
While a matroid abstracts the notion of dependence found in structures such as vector spaces or graphs~\cite{Oxley}, an oriented matroid additionally captures the sign information of dependencies, akin to distinguishing base orientations in geometry or to considering directed graphs~\cite{BjornerLasVergnasSturmfelsWhiteZiegler}.
The Las Vergnas face lattice of an oriented matroid encodes its acyclic quotients and, in the realizable case, coincides with the face lattice of the corresponding polytope.

We systematically develop the theory of nested complexes of building sets over Las Vergnas face lattices, which we call \defn{facial nested complexes}.
Beyond boolean nested complexes, which are the facial nested complexes of the simplex, we show that several families of simplicial complexes and polytopes previously studied in the literature are, in fact, facial nested complexes.
These include poset associahedra~\cite{Galashin}, simple polytope nestohedra~\cite{Almeter,Petric}, hyperoctahedral nestohedra~\cite{Almeter}, design graph associahedra~\cite{DevadossHeathVipismakul}, permutopermutahedra~\cite{Gaiffi-permutonestohedra,CastilloLiu-permutoAssociahedron}, and simple permutoassociahedra~\cite{BaralicIvanovicPetric,Ivanovic}.

Our framework thus unifies a range of previously disparate constructions under a single combinatorial and geometric perspective.
Many of the structural and geometric results in these works follow directly from their reinterpretation as facial nested complexes, which brings with it a rich package of combinatorial, topological, geometric, and algebraic consequences.
In fact, we believe that most `nested-like' polytopes are captured by our framework of facial nestohedra (although some clearly fall outside: for instance the permutoassociahedron of~\cite{Kapranov,ReinerZiegler,CastilloLiu-permutoAssociahedron} is not even simple).

\begin{figure}
	\capstart
	\centerline{
		\begin{tikzpicture}

			\node[violet] at (0,9.5) {LATTICE NESTED COMPLEXES \hypersetup{citecolor=violet}\cite{FeichtnerKozlov2004}};
			\draw[rounded corners=18pt, ultra thick, violet] (-8.2, -1.6) rectangle (8.2, 9.9);

			\node[text width=5cm, cyan] at (-5.3,8.7) {WONDERFUL MODELS};
			\node[text width=5cm, cyan] at (-5.3,8.3) {(Flat lattices of matroids)};
			\node[text width=5cm, cyan] at (-5.3,7.9) {\hypersetup{citecolor=cyan}\cite{DeConciniProcesi1995}};
			\draw[rounded corners=16pt, ultra thick, cyan] (-8, -1) -- (-8, 9.1) -- (-.6, 9.1) -- (-.6, 5.1) -- (4.1, 5.1) -- (4.1, -1) -- cycle;

			\node[text width=3cm, blue] at (-5.9,6.9) {NESTOHEDRA};
			\node[text width=3cm, blue] at (-5.9,6.5) {(Boolean lattices)};
			\node[text width=3cm, blue] at (-5.9,6.1) {\hypersetup{citecolor=blue}\cite{FeichtnerSturmfels,Postnikov}};
			\node at (-2.5,6) {\includegraphics[scale=.2]{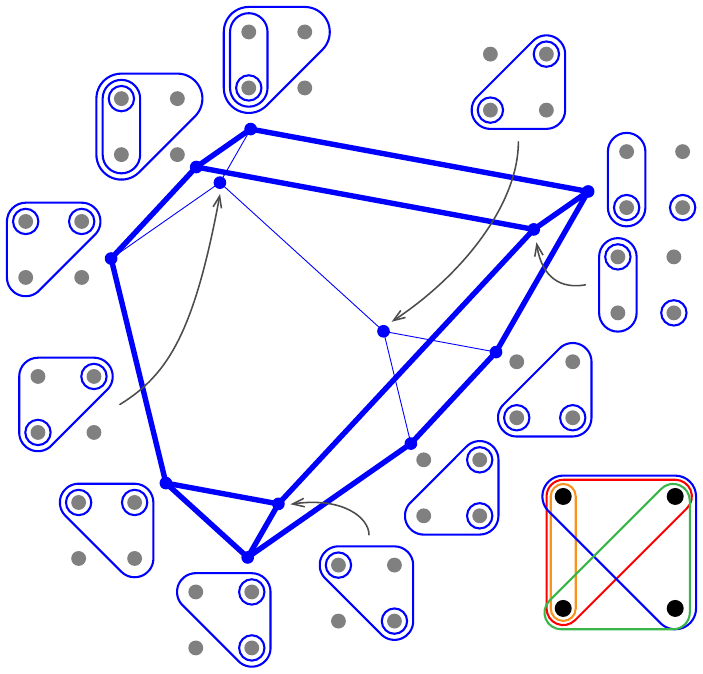}};
			\draw[rounded corners=14pt, ultra thick, blue] (-7.6, -.8) -- (-7.6, 7.3) -- (-.8, 7.3) -- (-.8, 4.9) -- (3.9, 4.9) -- (3.9, -.8) -- cycle;
			
			\node[text width=6cm, green] at (-4.2,4.1) {GRAPH};
			\node[text width=6cm, green] at (-4.2,3.7) {ASSOCIAHEDRA};
			\node[text width=6cm, green] at (-4.2,3.3) {\hypersetup{citecolor=green}\cite{CarrDevadoss}};
			\node at (-5.6,1.4) {\includegraphics[scale=.2]{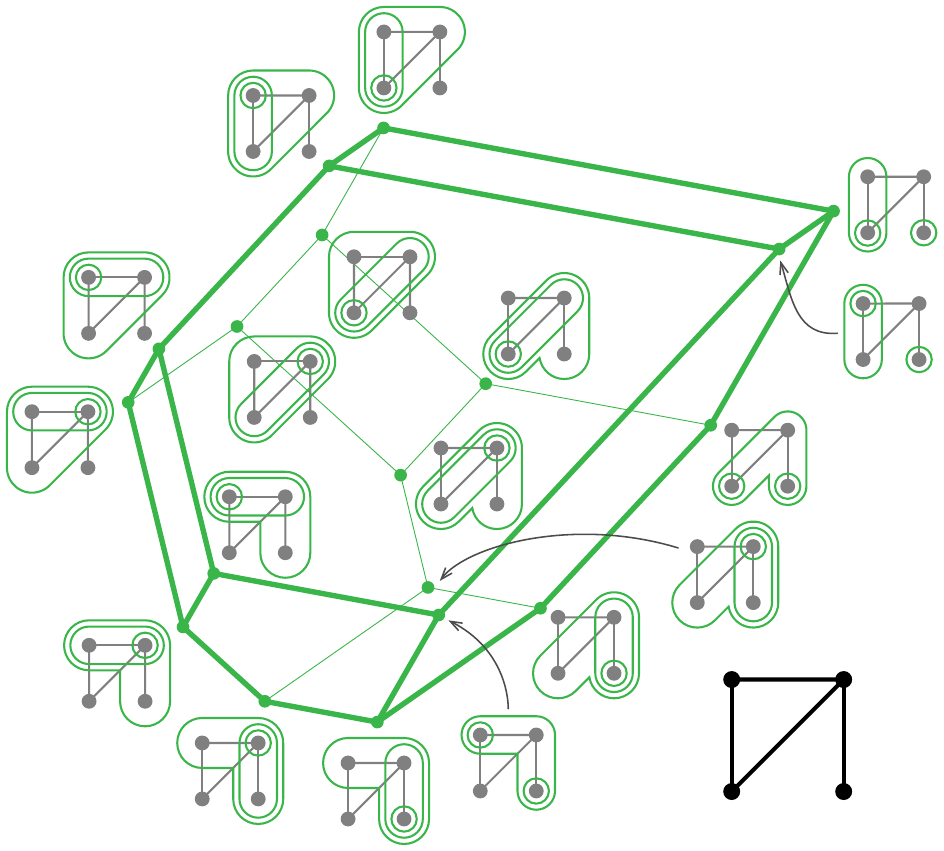}};
			\draw[rounded corners=12pt, ultra thick, green] (-7.4, -.6) rectangle (3.7, 4.5);

			\node[text width=6cm, align = right, red] at (4.8,8.7) {ACYCLONESTOHEDRA};
			\node[text width=6cm, align = right, red] at (4.8,8.3) {(Face lattices of oriented matroids)};
			\node[text width=6cm, align = right, red] at (2.8,7.8) {\tiny \rotatebox{90}{$\subseteq$}};
			\node[text width=6cm, align = right, red] at (4.8,7.4) {\tiny $\polytope$-nested complexes~\hypersetup{citecolor=red}\cite{Almeter,Petric}};
			\node[text width=6cm, align = right, red] at (4.8,7.1) {\tiny hyperoctahedral nested complexes};
			\node[text width=6cm, align = right, red] at (4.8,6.8) {\tiny type~$B$ permutahedron};
			\node[text width=6cm, align = right, red] at (4.8,6.5) {\tiny design graph associahedra~\hypersetup{citecolor=red}\cite{DevadossHeathVipismakul}};
			\node[text width=6cm, align = right, red] at (4.8,6.2) {\tiny biassociahedron \hypersetup{citecolor=red}\cite{BarnardReading}};
			\node[text width=6cm, align = right, red] at (4.8,5.9) {\tiny permutopermutohedron \hypersetup{citecolor=red}\cite{Gaiffi-permutonestohedra,CastilloLiu-permutoAssociahedron}};
			\node[text width=6cm, align = right, red] at (4.8,5.6) {\tiny simple permutoassociahedron \hypersetup{citecolor=red}\cite{BaralicIvanovicPetric,Ivanovic}};
			\node[text width=6cm, align = right, red] at (2.8,5.3) {\tiny \rotatebox{90}{...}};
			\node at (1.2,6.8) {\includegraphics[scale=.2]{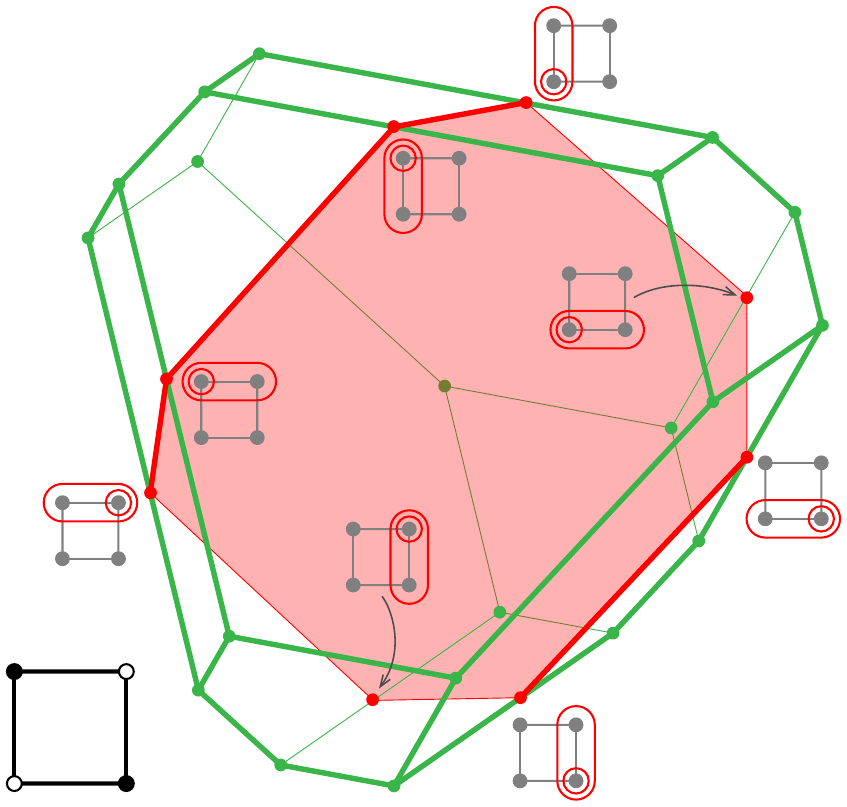}};
			\draw[rounded corners=16pt, ultra thick, red] (-7.8, -1.4) -- (-7.8, 7.5) -- (-.4, 7.5) -- (-.4, 9.1) -- (8, 9.1) -- (8, -1.4) -- cycle;

			\node[text width=4cm, align = right, orange] at (5.6,4.3) {POSET};
			\node[text width=4cm, align = right, orange] at (5.6,3.9) {ASSOCIAHEDRA};
			\node[text width=4cm, align = right, orange] at (5.6,3.5) {\hypersetup{citecolor=orange}\cite{Galashin}};
			\node at (5.9,1.2) {\includegraphics[scale=.2]{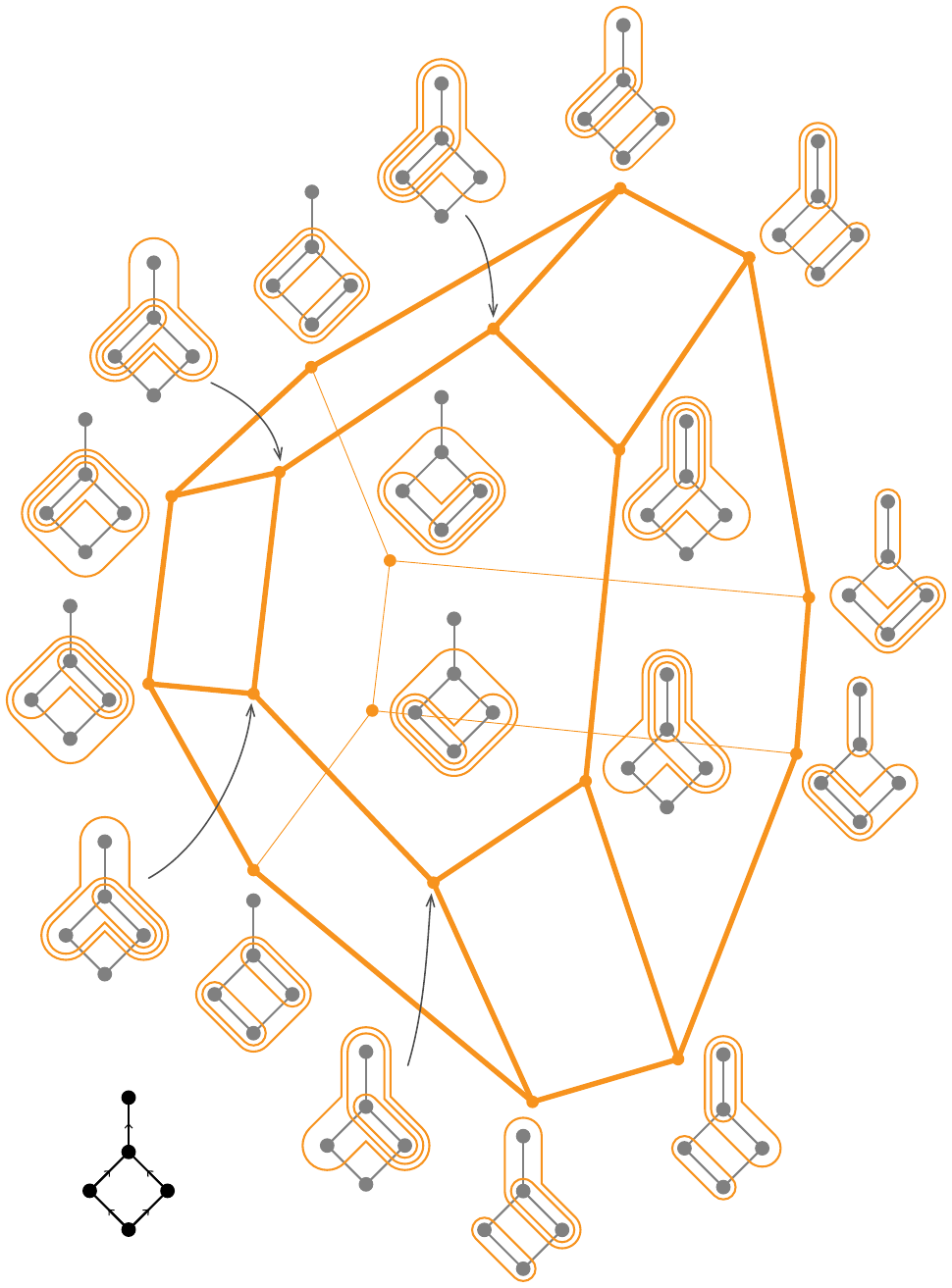}};
			\draw[rounded corners=12pt, ultra thick, orange] (-3.7, -1.2) rectangle (7.8, 4.7);

			\node[text width=6cm, yellow] at (-.3,3.9) {BLOCK};
			\node[text width=6cm, yellow] at (-.3,3.5) {GRAPH};
			\node[text width=6cm, yellow] at (-.3,3.1) {ASSOCIAHEDRA};
			\node[text width=6cm, yellow] at (4.5,3.7) {\tiny permutahedron};
			\node[text width=6cm, yellow] at (-.3,1.9) {\tiny associahedron};
			\node at (0,1.7) {\includegraphics[scale=.2]{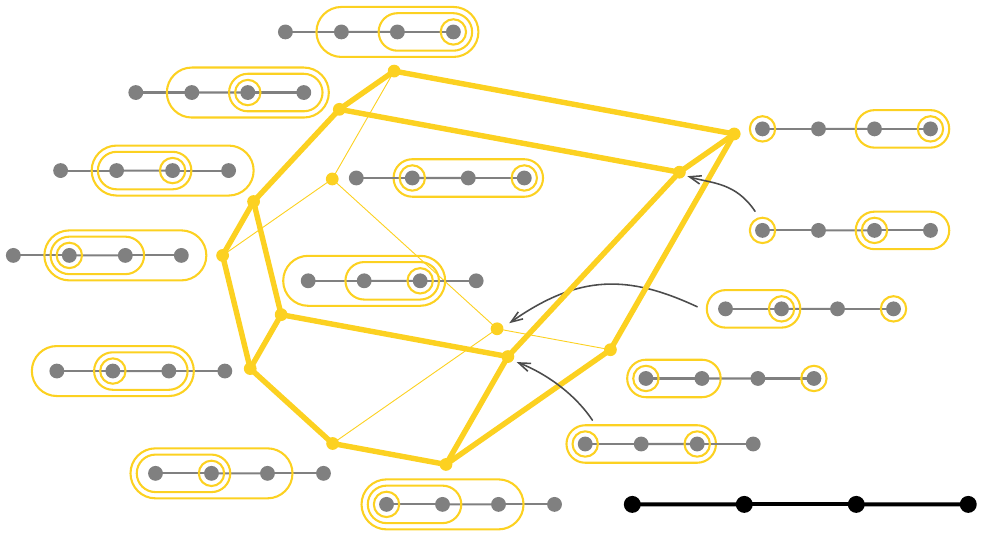} \hspace{-.6cm} \includegraphics[scale=.2]{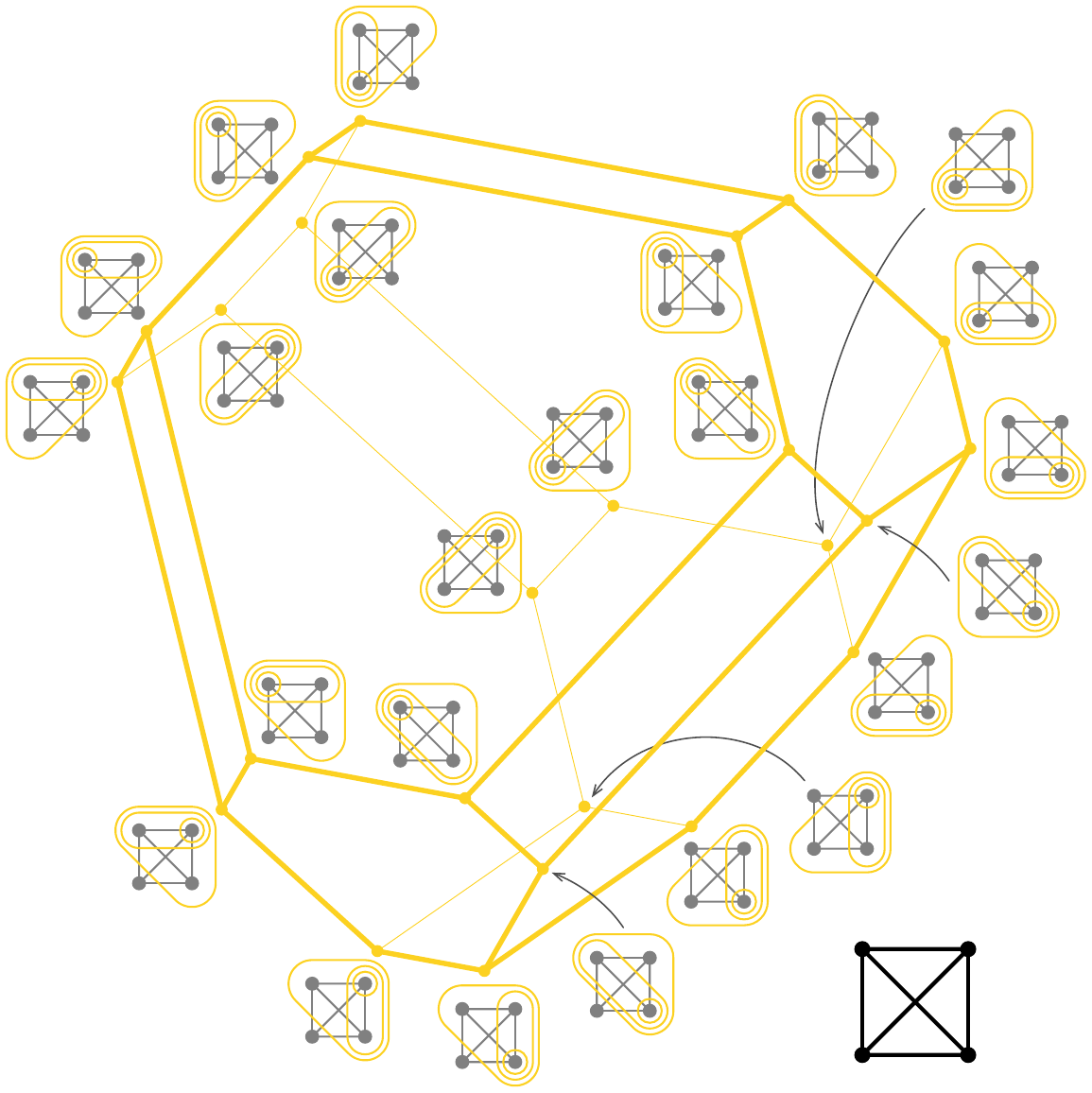}};
			\draw[rounded corners=10pt, ultra thick, yellow] (-3.5, -.4) rectangle (3.5, 4.3);
		\end{tikzpicture}
	}
	\caption{Inclusions and connections between different types of nested complexes.}
	\label{fig:intro}
\end{figure}
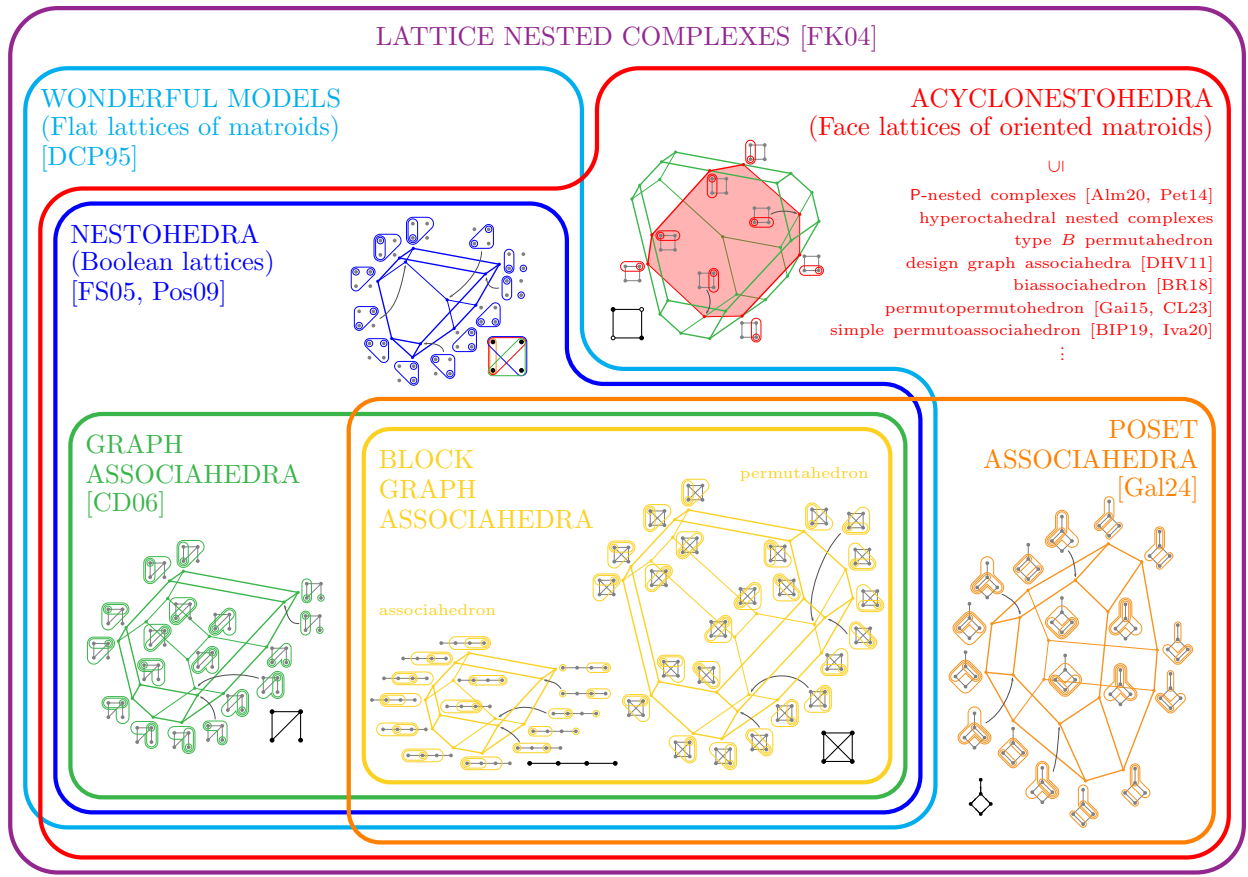

\para{Acyclonestohedra}
A central goal in many of the aforementioned works is to demonstrate that these nested-like complexes arise as boundary complexes of polytopes.
As it turns out, all facial nested complexes of an oriented matroid can be obtained through iterated stellar subdivisions of the positive tope.
It follows that, if the oriented matroid is realizable, the nested complex is dually realized by iterated face truncations of the polar of the underlying polytope.
This realization is not entirely satisfactory, however, as it does not provide explicit coordinates.
Namely, stellar subdivisions and face truncations rely on successive choices of sufficiently small parameters, which are challenging to control.

We provide effective and combinatorially meaningful polytopal realizations of facial nested complexes arising from realizable oriented matroids.
Specifically, using the fact that every polytope can be presented as a section of a simplex, we show that every facial nested complex on a polytope can be dually realized as a section of a boolean nestohedron.
This yields explicit coordinates for polytopes that we call \defn{acyclonestohedra}, as they geometrically embed facial nested complexes as the acyclic part (with respect to the underlying oriented matroid) of boolean nested complexes.
We also note that acyclonestohedra are removahedra of the omnitruncation of the polytope, for which our construction provides coordinates via an explicit non-iterative formula.

\cref{fig:intro} illustrates some examples of acyclonestohedra in the panorama of lattice nested complexes.

\para{Poset associahedra and compactifications}
Our original motivation for studying facial nested complexes of oriented matroids comes from P.~Gala\-shin's recent work on \defn{poset associahedra} and \defn{affine poset cyclohedra}~\cite{Galashin}.
In particular, acyclonestohedra specialize to poset associahedra when the oriented matroid is graphical.

Beyond the construction via stellar subdivisions, our approach allows us to extend most of the results in~\cite{Galashin} from order and affine order polytopes to arbitrary (oriented matroid) polytopes.
One of P.~Galashin's main results is that poset associahedra model compactifications of the space of order-preserving maps $P\to \R$, which can be identified with the interior of an order polytope.
We observe that results from~\cite{Gaiffi2003} imply that the interior of any polytope admits a stratified compactification as a $C^\infty$-manifold with corners, obtained via real blow-ups, that is diffeomorphic as a stratified space to the corresponding acyclonestohedron.
Facial nested complexes also appear to model the boundary structure of \defn{wondertopes} in a recent alternative compactification via projective blow-ups~\cite{BraunerEurPrattVlad}.

Furthermore, we are able to address some questions left open in~\cite{Galashin}.
First of all, acyclonestohedra provide the desired explicit realizations of (affine) poset associahedra.
An alternative explicit construction of poset associahedra was independently proposed by A.~Sack~\cite{Sack}, which we show can be recovered as a projection of our acyclonestohedron.
Secondly, the toric variety of the facial nested fan~\cite{FeichtnerYuzvinsky2004} yields an algebraic variety reflecting the combinatorics of the poset associahedron.

\para{Nested complex embeddings}
Our realization of acyclonestohedra is the geometric incarnation of the embedding of the facial nested complex as the acyclic subcomplex of a boolean nested complex, valid for every oriented matroid, realizable or not.
An analogous embedding of \defn{flatial nested complexes} (those defined over the lattice of flats of the matroid) inside boolean nested complexes for minimal and maximal flatial and boolean building sets was implicitly used to define the Bergman fan in~\cite{ArdilaKlivans2006}, and later extended to arbitrary building sets in~\cite{FeichtnerSturmfels}.
Similarly, embeddings of facial nested complexes into flatial nested complexes for the minimal and maximal facial and flatial building sets were exploited in~\cite{ArdilaKlivansWilliams2006,ArdilaReinerWilliams2006} to embed the positive Bergman complex into the Bergman complex.
Here, we extend this by showing that any flatial building set restricts to a facial building set, for which the associated facial nested complex embeds into the associated flatial nested complex.
If moreover every element of the matroid is a vertex of the oriented matroid polytope, then all facial building sets arise this way.
(Otherwise, elements that are not visible in the face lattice may break the building set structure when considered inside the flat lattice).

To prove these results, we more generally consider arbitrary order embeddings between lattices, and we study conditions on these embeddings and building sets that guarantee embeddings of the corresponding nested complexes.
As an application of their results on the convex geometry of building sets, this topic was also independently investigated by S.~Backman and R.~Danner in~\cite{BackmanDanner}, and we refer to the extensive list of examples in~\cite[Sect.~4]{BackmanDanner} for further illustrations of the ubiquity of these embeddings.
We identify some families of \defn{tame order embeddings} that behave well with respect to nested complexes, and we study when they allow for pushing or pulling building sets and nested complexes.
As a particular case, we prove that any atomic nested complex has a canonical embedding inside a boolean nested complex, generalizing the aforementioned embeddings of the facial and flatial nested complexes.
Our results are broader than these boolean embeddings, and cover in particular the aforementioned embeddings of the facial inside the flatial nested complex.
Finally, we note that the consistency criterion of~\cite[Def.~4.4 \& Thm.~4.5]{BackmanDanner} can a posteriori be derived from our study.

\para{Organization}
The paper is organized as follows.
\cref{part:preliminaries} is a gentle presentation of building sets and nested complexes (\cref{sec:buildingSetsNestedComplexes}) and of oriented matroids (\cref{sec:orientedMatroids}).
To build intuition, we start with boolean nested complexes and oriented matroids from vector configurations, and we further explore the graphical settings.
Our focus lies especially on restrictions and contractions to describe links (\cref{subsec:restrictionContractionNestedComplex,subsec:restrictionContractionOM}), and on specific operations (\cref{subsec:operationsNestedComplex,subsec:operationsOM}) to restrict to connected structures.
Since we could not find these results stated in full lattice generality in the literature at the time of writing, we give a unified presentation that includes detailed proofs for the construction of the lattice building closure (extending~\cite[Lem.~3.10]{FeichtnerSturmfels} and independently proven in~\cite{BackmanDanner}), the description of links in lattice nested complexes (extending~\cite[Prop.~3.2]{Zelevinsky}, \cite[Sect.~4.3]{DeConciniProcesi1995}, and~\cite[Thm~1.6]{BraunerEurPrattVlad}), and the treatment of direct sum and free sum operations.

In \cref{part:facialAcyclicNestedComplexes}, we consider the facial nested complex of a facial building set (\cref{sec:facialBuildingSetsFacialNestedComplexes}) and the acyclic nested complex of an oriented building set (\cref{sec:orientedBuildingSetsAcyclicNestedComplexes}).
We prove that these two settings essentially coincide (\cref{sec:facialNestedComplexesVsAcyclicNestedComplexes}).
Namely, the facial building sets are precisely the facial parts of the oriented building sets, and their nested complexes coincide.

\cref{part:realizations} addresses geometric realizations of facial nested complexes (or equivalently acyclic nested complexes).
First, the facial nested complexes of an oriented matroid~$\OM$ are shown to be face lattices of oriented matroids, which are realizable whenever~$\OM$ is (\cref{sec:orientedMatroidRealizations}). This follows from an oriented matroid interpretation of the combinatorial blow-up operation from \cite{FeichtnerKozlov2004}.
Second, and most important, we provide an alternative polytopal realization of facial nested complexes of realizable oriented matroids as sections of nestohedra, with explicit and combinatorially meaningful coordinates (\cref{sec:polytopalRealizations}).
Finally, we connect facial nested complexes to the work of G.~Gaiffi~\cite{Gaiffi2003} on stratified compactifications of polyhedral cones and to the work of S.~Brauner, C.~Eur, E.~Pratt, and R.~Vlad~\cite{BraunerEurPrattVlad} on wondertopes (\cref{sec:compactifications}).

\cref{part:posetAssociahedra} is devoted to the graphical case.
Namely, we prove that the poset associahedra (\cref{sec:posetAssociahedra}) and affine poset cyclohedra (\cref{sec:affinePosetCyclohedra}) defined by P.~Galashin in~\cite{Galashin} are facial nested complexes of graphical oriented building sets.
Applying \cref{sec:orientedMatroidRealizations,sec:polytopalRealizations,sec:compactifications}, we thus recover the polytopal realizations of~\cite[Thm.~2.1]{Galashin} by stellar subdivisions of order polytopes, we obtain nice explicit realizations as graphical acyclonestohedra, and we recover stratified compactifications of the order polytopes akin to~\cite[Thm.~1.9\,\&\,1.12]{Galashin}.

Lastly, \cref{part:embeddings} is devoted to more general nested complex embeddings and the resulting realizations.
We first investigate conditions on poset embeddings between two lattices that guarantee nested complex embeddings (\cref{sec:latticeEmbedding}).
We then apply these conditions to embed atomic nested complexes as subcomplexes of boolean nested complexes (\cref{sec:atomic}) and derive realizations of atomic nested complexes as subfans of boolean nested fans and as subcomplexes of boundary complexes of nestohedra, recovering the fan realizations of E.~M.~Feichtner and S.~Yuzvinsky~\cite{FeichtnerYuzvinsky2004}.
We finally apply our conditions to explain the embeddings of nested complexes over the face lattice inside nested complexes over the flat lattice of an oriented matroid that appears in the positive Bergman fan construction (\cref{sec:Bergman}).

\para{Preliminary versions}
The geometric realization of poset associahedra appeared as the Master thesis of C.~Mantovani~\cite{Mantovani}, co-supervised by the other authors.
Preliminary versions of our results were presented in the Simons Center Workshop on Combinatorics and Geometry of Convex Polyhedra (Mar.~2023), the Oberwolfach workshop on Geometric, Algebraic, and Topological Combinatorics (Dec.~2023), and of the 36th International Conference on Formal Power Series and Algebraic Combinatorics (FPSAC 24, Bochum, Jul.~2024).
Extended abstracts can be found in~\cite{Oberwolfach, FPSAC}.


\clearpage
\part{Nested complexes and oriented matroids}
\label{part:preliminaries}

In this part, we remind the reader of classical notions on building sets and nested complexes (\cref{sec:buildingSetsNestedComplexes}) and on oriented matroids (\cref{sec:orientedMatroids}).
We focus in particular on restrictions and contractions (\cref{subsec:restrictionContractionNestedComplex,subsec:restrictionContractionOM}) that enables us to describe links, and on some operations (\cref{subsec:operationsNestedComplex,subsec:operationsOM}) that enable us to consider only connected objects.


\section{Building sets, nested sets, and nested complexes}
\label{sec:buildingSetsNestedComplexes}

We now recall some aspects of building sets and nested sets.
In the original definition of \mbox{E.~M.~Feichtner} and D.~Kozlov~\cite{FeichtnerKozlov2004}, building sets and nested sets depend upon an underlying meet semilattice (generalizing the work on lattices of flats of subspace arrangements from~\cite{DeConciniProcesi1995}).
In this paper, we will work with (not-necessarily boolean) face lattices of polytopes and oriented matroids.
Our presentation starts with the classical setting on the boolean lattice (\cref{subsec:boolean}), and its generalization to the level of general lattices (\cref{subsec:latticeBuildingSets,subsec:latticeNestedSets}).
We then define restrictions and contractions to describe links in nested complexes (\cref{subsec:restrictionContractionNestedComplex}), recall the operation of combinatorial blow-up (\cref{subsec:combinatorialblowup}), and conclude with two operations on lattices and building sets that enable us to restrict to connected building sets (\cref{subsec:operationsNestedComplex}).


\subsection{Building sets and nested sets}
\label{subsec:boolean}

We first recall the classical definitions of (boolean) building sets, nested sets, nested complexes, nested fans and nestohedra from~\cite{Postnikov, FeichtnerSturmfels, Zelevinsky, Pilaud-removahedra} and their specializations to the graphical case~\cite{CarrDevadoss}.
The presentation and some illustrations are borrowed from~\cite{PadrolPilaudPoullot-deformedNestohedra}.

\subsubsection{Building sets}

We start with the notion of building sets.

\begin{definition}[\cite{FeichtnerKozlov2004, FeichtnerSturmfels, Postnikov}]
\label{def:booleanBuildingSet}
A \defn{building set} on a ground set~$\ground$ is a set~$\building$ of non-empty subsets of~$\ground$ such that
\begin{itemize}
\item $\building$ contains all singletons~$\{s\}$ for~$s \in \ground$,
\item if~$B,B' \in \building$ and~$B \cap B' \ne \varnothing$, then~$B \cup B' \in \building$.
\end{itemize} 
We denote by~$\connectedComponents(\building)$ the set of \defn{$\building$-connected components} (\ie maximal elements of~$\building$), and we say that $\building$ is \defn{connected} if $\connectedComponents(\building)=\{\ground\}$.
\end{definition}

\begin{example}
\label{exm:minMaxBooleanBuildingSet}
The set~$2^\ground \ssm \{\varnothing\}$ of all non-empty subsets of~$\ground$ (resp.~the set~$\set{\{s\}}{s \in \ground}$ of singletons of~$\ground$) is a building set on~$\ground$, which contains (resp.~is contained in) any building set on~$\ground$.
\end{example}

\begin{example}[\cite{CarrDevadoss}]
\label{exm:graphicalBuildingSet}
Consider a graph~$\graphG$ on~$\ground$.
A \defn{tube} of~$\graphG$ is a non-empty subset of~$\ground$ which induces a connected subgraph of~$\graphG$.
The set~$\tubes$ of all tubes of~$\graphG$ is a building set, called the \defn{graphical building set} of~$\graphG$.
Moreover, the blocks of $\connectedComponents(\tubes)$ are the vertex sets of the connected components~$\connectedComponents(\graphG)$~of~$\graphG$.
\end{example}

\begin{remark}[\cite{DosenPetric}]
\label{rem:hypergraphicalBuildingSet}
More generally, a hypergraph~$\hypergraph$ on~$\ground$ defines a building set~$\building(\hypergraph)$ on~$\ground$ given by all non-empty subsets of~$\ground$ which induce connected subhypergraphs of~$\hypergraph$ (a path in~$\hypergraph$ is a sequence of vertices where any two consecutive ones belong to a common hyperedge of~$\hypergraph$).
Any building set~$\building$ on~$\ground$ is the building set of a hypergraph, but not always of a graph.
\end{remark}

\begin{example}
\label{exm:booleanBuildingSet}
For a more specific instance, consider the building set~$\building_\circ$ on~$\{1, \dots, 5\}$ defined by
\[
\building_\circ \eqdef \{1, 2, 3, 4, 5, 6, 12, 14, 25, 123, 124, 125, 456, 1234, 1235, 1245, 1456, 2456, 12345, 12456, 123456\}
\]
(since all labels have a single digit, we can abuse notation and write $123$ for $\{1,2,3\}$).
It has a single $\building_\circ$-connected component $\connectedComponents(\building_\circ) = \{123456\}$. 
It is the building set of the hypergraph~$\{12, 14, 25, 123, 456\}$ represented in \cref{fig:exmNested}\,(left).
\begin{figure}
	\capstart
	\centerline{\includegraphics[scale=.9]{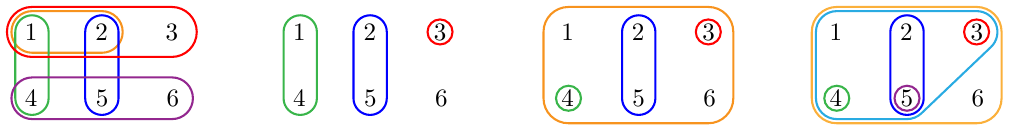}}
	\caption{A hypergraph generating the building set~$\building_\circ$ of \cref{exm:booleanBuildingSet} (left), a subset of~$\building_\circ$ which is not a nested set on~$\building_\circ$ (middle left), a nested set on~$\building_\circ$ (middle right), and a maximal nested set on~$\building_\circ$ (right).}
	\label{fig:exmNested}
\end{figure}
\end{example}

We immediately observe that any set of non-empty subsets of~$\ground$ can be completed to a minimal building set.

\begin{definition}[{\cite[Lem.~3.10]{FeichtnerSturmfels}}]
\label{def:booleanBuildingSetClosure}
The \defn{building closure} of $\c{X} \subseteq 2^\ground$ is the smallest building set $\building$ containing $\c{X}$.
It is composed of the singletons $\{s\}$ for $s\in \ground$ and the sets $\bigcup_{Y \in \c{Y}} Y$ for each $\c{Y} \subseteq \c{X}$ with a connected intersection graph.
\end{definition}

\subsubsection{Nested sets}

Once a building set is chosen, we consider its nested sets and nested complex.

\begin{definition}[\cite{FeichtnerKozlov2004, FeichtnerSturmfels, Postnikov}]
\label{def:booleanNestedSet}
Let~$\building$ be a building set on~$\ground$.
A \defn{nested set}~$\nested$ on~$\building$ is a subset~$\nested$ of~$\building$ containing~$\connectedComponents(\building)$ such that
\begin{itemize}
\item for any~$B,B' \in \nested$, either~$B \subseteq B'$ or~$B' \subseteq B$ or~$B \cap B' = \varnothing$,
\item for any~$k \ge 2$ pairwise disjoint~$B_1,\dots,B_k \in \nested$, the union~$B_1 \cup \dots \cup B_k$ is not in~$\building$.
\end{itemize}
The \defn{nested complex} of~$\building$ is the simplicial complex~$\nestedComplex[][\building]$ whose faces are $\nested \ssm \connectedComponents(\building)$ for all nested sets~$\nested$ on~$\building$.
\end{definition}

\begin{remark}
\label{rem:nestedComplex}
Note that it is convenient to include~$\connectedComponents(\building)$ in all nested sets (as in~\cite{Postnikov}), but to remove $\connectedComponents(\building)$ from all nested sets when defining the $\building$-nested complex (as in~\cite{Zelevinsky}).
In particular, the nested complex is a simplicial sphere of dimension~$|\ground| - |\connectedComponents(\building)|$.
\end{remark}

\begin{example}
\label{exm:orderComplexBoolean}
For the maximal building set~$2^\ground \ssm \{\varnothing\}$ (resp.~minimal building set~$\set{\{s\}}{s \in \ground}$) considered in \cref{exm:minMaxBooleanBuildingSet}, the nested sets are the chains in the boolean lattice on~$\ground$ (resp.~all subsets of~$\ground$), hence the nested complex is isomorphic to the refinement complex on proper ordered partitions of~$\ground$ (resp.~to the inclusion complex on proper subsets of~$\ground$).
\end{example}

\begin{example}[\cite{CarrDevadoss}]
\label{exm:graphicalNested}
Following on \cref{exm:graphicalBuildingSet}, consider a graph~$\graphG$ on~$\ground$.
Two tubes~$\tube, \tube'$ of~$\graphG$ are \defn{compatible} if they are either nested (\ie $\tube \subseteq \tube'$ or~$\tube' \subseteq \tube$), or disjoint and non-adjacent~(\ie~${\tube \cup \tube' \notin \tubes}$).
Note that any connected component of~$\graphG$ is compatible with any other tube of~$\graphG$.
A \defn{tubing} on~$\graphG$ is a set~$\tubing$ of pairwise compatible tubes of~$\graphG$ containing all connected components~$\connectedComponents(\graphG)$.
Examples are illustrated in \cref{fig:exmGraphicalNested}.
Tubings are precisely the nested sets of the graphical building set~$\tubes$.
The nested complex~$\nestedComplex[][\tubes]$ is a \defn{graphical nested complex}.
\begin{figure}
	\capstart
	\centerline{\includegraphics[scale=.5]{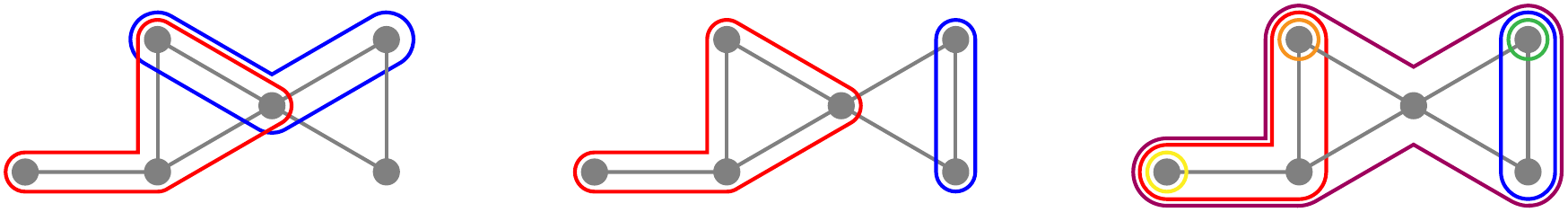}}
	\caption{Some incompatible tubes (left and middle), and a maximal tubing (right).}
	\label{fig:exmGraphicalNested}
\end{figure}
\end{example}

\begin{example}
\label{exm:nested}
Consider the building set~$\building_\circ$ of \cref{exm:booleanBuildingSet}.
The subset~$\{14, 25, 3\}$ is not a nested set on~$\building_\circ$ since its union is~$12345 \in \building_\circ$ (and it does not contain~$\connectedComponents(\building)$).
In contrast, the subsets~$\{3, 4, 25, 123456\}$ and~$\{3, 4, 5, 25, 12345, 123456\}$ are nested sets on~$\building_\circ$ and the latter is maximal.
These subsets of~$\building_\circ$ are illustrated in \cref{fig:exmNested}.
\end{example}

\subsubsection{Contractions, restrictions and links}

We now introduce restrictions and contractions of building sets to describe the links of nested complexes, following~\cite[Prop.~3.2]{Zelevinsky}.
We invite the reader to check with \cref{def:booleanBuildingSet} that the following define building sets.

\begin{definition}
\label{def:restrictionContractionBuildingSet}
For any building set~$\building$ on~$\ground$ and any~$R \subseteq \ground$, define
\begin{itemize}
\item the \defn{restriction} of~$\building$ to~$R$ as the building set~$\building_{|R} \eqdef \set{B \in \building}{B \subseteq R}$ on~$R$,
\item the \defn{contraction} of~$R$ in~$\building$ as the building set~$\building_{\!/R} \eqdef \set{B \ssm R}{B \in \building \text{ and } B \not\subseteq R}$ on~$\ground \ssm R$.
\end{itemize}
\end{definition}

\pagebreak

\begin{example}
\label{exm:restrictionContractionGraphicalBuildingSet}
Following on \cref{exm:graphicalBuildingSet,exm:graphicalNested}, consider a graph~$\graphG$ on~$\ground$ and $R \subseteq \ground$.
Denote~by
\begin{itemize}
\item $\graphG_{|R}$ the subgraph of~$\graphG$ induced by~$R$,
\item $\graphG_{\!/R}$ the \defn{reconnected complement} of~$R$ in~$\graphG$, \ie the graph on~$\ground \ssm R$ with an edge~$\{r,s\}$ if there is a path between~$r$ and~$s$ in~$\graphG$ with vertices in~$R \cup \{r,s\}$, see~\cite{CarrDevadoss}.
\end{itemize}
Then~$\tubes_{|R} = \tubes[\graphG_{|R}]$ and~$\tubes_{\!/R} = \tubes[\graphG_{\!/R}]$.
\end{example}

\begin{example}
For instance, consider the building set~$\building_\circ$ of \cref{exm:booleanBuildingSet}, illustrated in \cref{fig:exmNested}.
For~$R \eqdef 123$, we have~${\building_\circ}_{|R} = \{1, 2, 3, 12, 123\}$ and~${\building_\circ}_{\!/R} = \{4, 5, 6, 456, 45\}$.
\end{example}

Note that our definition of contraction in \cref{def:restrictionContractionBuildingSet} slightly differs from that of~\cite[Def.~3.1]{Zelevinsky}, in order to fit better with the contraction of oriented matroids, see \cref{rem:contraction}.
However, it coincides for~$R \in \building$, which enables us to borrow directly the following statement.
Recall that the \defn{link} of a face~$s$ in a simplicial complex~$\simplex$ is defined by $\set{t \in \simplex}{s \cap t = \varnothing \text{ and } s \cup t \in \simplex}$ and the \defn{join} of two simplicial complexes~$\simplex$ and~$\simplex'$ is defined by~$\simplex \join \simplex' \eqdef \set{s \sqcup s'}{s \in \simplex, \; s' \in \simplex'}$.

\begin{proposition}[{\cite[Prop.~3.2]{Zelevinsky}}]
\label{prop:linksNestedComplex}
For any~$R \in \building \ssm \connectedComponents(\building)$, the link~$\set{\nested \ssm \{B\}}{R \in \nested \in \nestedComplex[][\building]}$ is the join~$\nestedComplex[][\building_{|R}] \join \nestedComplex[][\building_{\!/R}]$.
\end{proposition}

\subsubsection{Polyhedral realizations}

Finally, we recall the known polyhedral realizations of nested complexes.
Namely, we show that the nested complex~$\nestedComplex[][\building]$ of a building set~$\building$ is realized by the nested fan~$\nestedFan[]$, or dually by the nestohedron~$\Nest(\building)$.
Examples of these fans and polytopes are illustrated in \cref{fig:nestedFans,fig:nestohedra,fig:graphicalNestedFans,fig:graphAssociahedra}.
We denote by~$(\b{e}_s)_{s \in \ground}$ the standard basis of~$\R^\ground$, and by~$\b{e}_X \eqdef \sum_{x \in X} \b{e}_x$ the characteristic vector of a subset~$X \subseteq \ground$.
We refer to~\cite{Ziegler,HenkRichterGebertZiegler1997} for basic notions on polyhedral geometry.

\begin{definition}[\cite{Postnikov, FeichtnerSturmfels, Zelevinsky}]
\label{def:nestedFan}
For a building set~$\building$, the \defn{nested fan}~$\nestedFan[]$ is the collection of cones
\[
\nestedFan[] \eqdef \set{\con[\nested]}{\nested \in \nestedComplex[][\building]},
\]
where~$\con[\nested] \eqdef \cone \bigl( \set{\b{e}_B}{B \in \nested}\bigr)$ is the cone generated by the characteristic vectors of the blocks in~$\nested$.
\end{definition}

\begin{theorem}[\cite{Postnikov, FeichtnerSturmfels, Zelevinsky}]
\label{thm:nestedFan}
For a building set~$\building$, the nested fan~$\nestedFan[]$ forms a simplicial fan of~$\R^\ground$, and thus provides a polyhedral realization of the nested complex~$\nestedComplex[][\building]$.
\end{theorem}

\begin{remark}
The nested fan~$\nestedFan[]$ is not complete nor flat (in the sense that the dimension of its affine span is larger than its intrinsic dimension).
Its projection along~$\R^{\connectedComponents(\building)} \eqdef \bigoplus_{B \in \connectedComponents(\building)} \R\b{e}_B$ is flat,
and the Minkowski sum~$\nestedFan[] + \R^{\connectedComponents(\building)} \eqdef \set{\con + \R^{\connectedComponents(\building)}}{\con \in \nestedFan[]}$ is a complete fan with the same combinatorial structure.
\end{remark}

For a building set~$\building$, we denote by~$\R^\building_+ \eqdef \set{\b{\lambda} \in \R^\building}{\lambda_B > 0 \text{ for all } B \in \building \text{ with } |B| \ge 2}$.

\begin{definition}[\cite{Postnikov, FeichtnerSturmfels}]
\label{def:nestohedron}
For a building set~$\building$ and any~$\b{\lambda} = (\lambda_B)_{B \in \building} \in \R^\building_+$, the \defn{nestohedron}~$\Nest(\building, \b{\lambda})$ is the Minkowski sum~$\sum_{B \in \building} \lambda_B \triangle_B$, where~$\triangle_B \eqdef \conv\set{\b{e}_b}{b \in B}$ denotes the face of the standard simplex~$\triangle_\ground$ corresponding to~$B$.
\begin{figure}
	\capstart
	\centerline{\includegraphics[scale=.45]{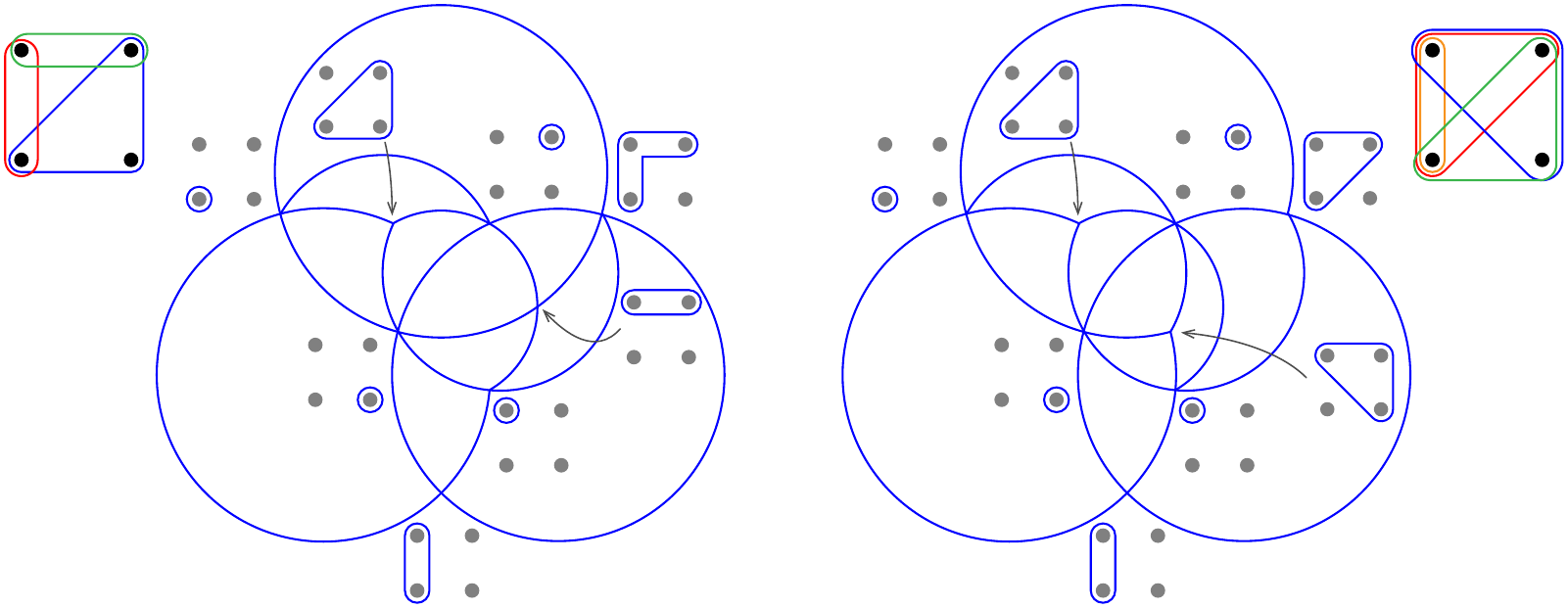}}
	\caption{Two nested fans. The rays are labeled by the corresponding blocks. As the fans are $3$-dimensional, we intersect them with the sphere and stereographically project them from the direction~$(-1,-1,-1)$. Illustration from~\cite{PadrolPilaudPoullot-deformedNestohedra}.}
	\label{fig:nestedFans}
\end{figure}
\begin{figure}
	\capstart
	\centerline{\includegraphics[scale=.45]{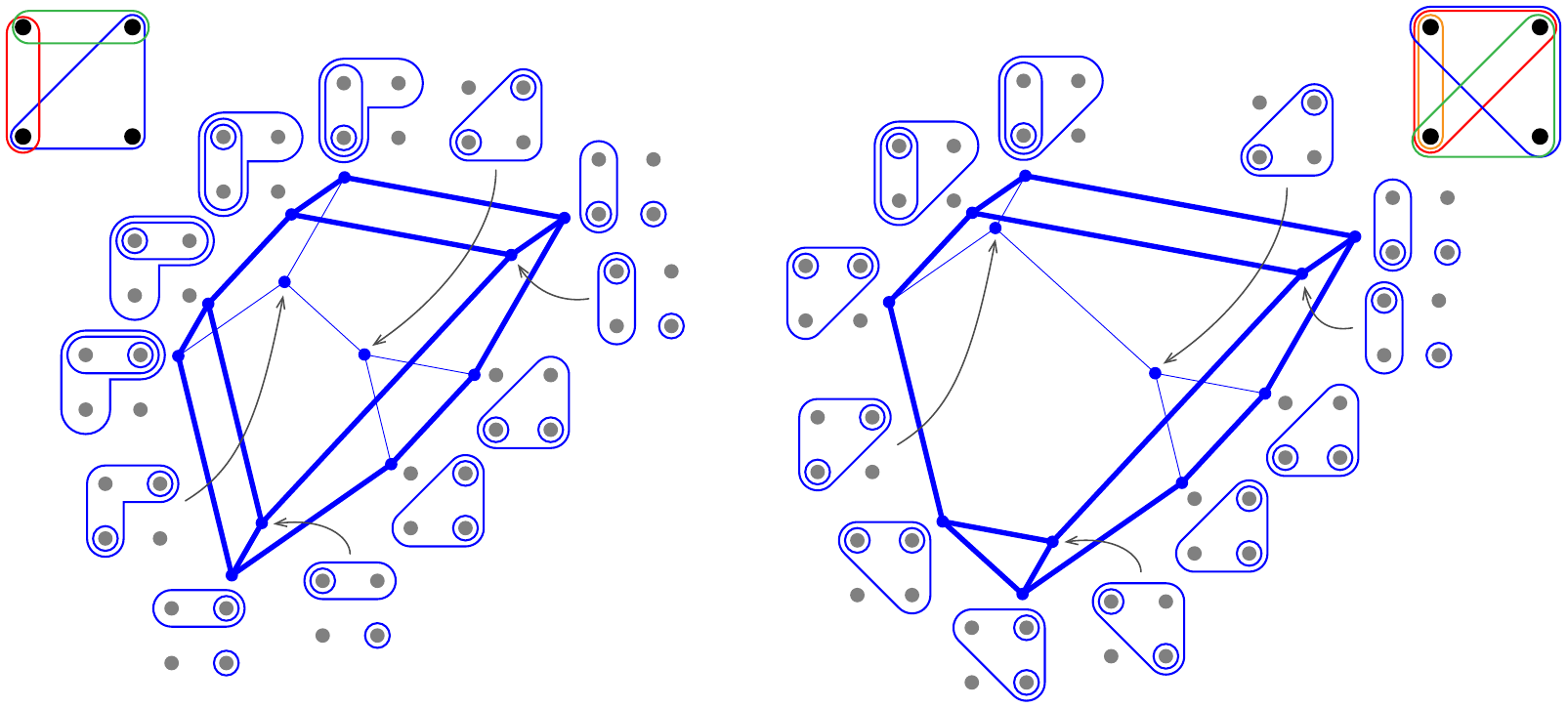}}
	\caption{Two nestohedra whose vertices are labeled by the corresponding maximal nested sets. The maximal block is omitted in all nested sets. Illustration from~\cite{PadrolPilaudPoullot-deformedNestohedra}.}
	\label{fig:nestohedra}
\end{figure}
\end{definition}

\begin{theorem}[\cite{Postnikov, FeichtnerSturmfels, Zelevinsky}]
\label{thm:nestohedron}
For a building set~$\building$ and any~$\b{\lambda} \in \R^\building_+$, the normal fan of the nestohedron~$\Nest(\building, \b{\lambda})$ is the nested fan~$\nestedFan[] + \R^{\connectedComponents(\building)}$, hence the nested complex~$\nestedComplex$ is isomorphic to the boundary complex of the polar of the nestohedron~$\Nest(\building, \b{\lambda})$.
\end{theorem}

Note that for this statement, it is important that~$\lambda_B > 0$ for~$|B| \ge 2$, while~$\lambda_B$ is irrelevant for~$|B| = 1$ (as it corresponds to translations in the Minkowski sum), hence the definition of~$\R^\building_+$.
We will also need the explicit vertex and facet descriptions of the nestohedron~$\Nest(\building, \b{\lambda})$.

\begin{proposition}
\label{prop:vertexDescriptionNestohedron}
For any~$\b{\lambda} \in \R^\building_+$, the vertex of the nestohedron~$\Nest(\building, \b{\lambda})$ corresponding to a maximal nested set~$\nested$ is
\[
\vertexNest = \sum_{s \in \ground} \sum_{\substack{B \in \building \\ s \in B \subseteq \block{s}{\nested}}} \lambda_B \b{e}_s,
\]
where~$\block{s}{\nested}$ denotes the inclusion minimal block of~$\nested$ containing~$s$.
\end{proposition}

\begin{proposition}
\label{prop:facetDescriptionNestohedron}
For any~$\b{\lambda} \in \R^\building_+$, the nestohedron~$\Nest(\building, \b{\lambda})$ is given by the equalities~$g_B(\b{x}) = 0$ for all~$B \in \connectedComponents(\building)$ and the inequalities~$g_B(\b{x}) \ge 0$ for all~$B \in \building$, where
\[
g_B(\b{x}) \eqdef \bigdotprod{\sum_{b \in B} \b{e}_b}{\b{x}} - \sum_{\substack{B' \in \building \\ B' \subseteq B}} \lambda_{B'}.
\]
\end{proposition}

\begin{remark}
\label{rem:truncations}
Note that adding the inequalities of \cref{prop:facetDescriptionNestohedron} in reverse inclusion order of~$\building$ constructs the nestohedron~$\Nest(\building_\circ, \b{\lambda})$ by a sequence of face truncations (see \cref{exm:stellarSubdivisionTruncation}) of a simplex.
We will extend this perspective to all oriented matroids in \cref{sec:orientedMatroidRealizations}.
\end{remark}

\begin{example}
For the maximal building set~$2^\ground \ssm \{\varnothing\}$ (resp.~minimal building set~$\set{\{s\}}{s \in \ground}$) considered in \cref{exm:minMaxBooleanBuildingSet,exm:orderComplexBoolean},  the (projection of the) nested fan is the \defn{braid fan} defined the hyperplanes~$\set{\b{x} \in \R^\ground}{x_s = x_t}$ for all~$s \ne t \in \ground$ (resp.~the simplex fan), and the nestohedron is a permutahedron (resp.~a simplex).
\end{example}

\begin{example}
\label{exm:graphAssociahedron}
\begin{figure}
	\capstart
	\centerline{\includegraphics[scale=.45]{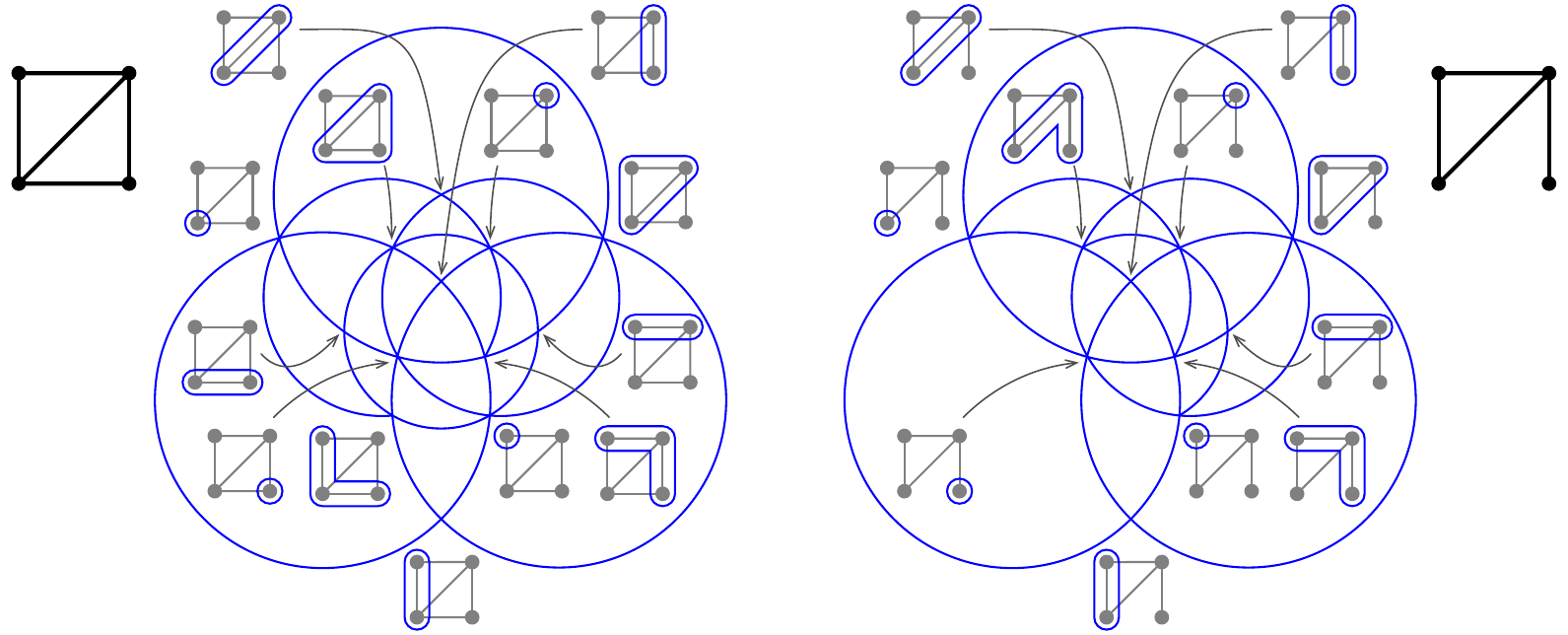}}
	\caption{Two graphical nested fans. The rays are labeled by the corresponding tubes. As the fans are $3$-dimensional, we intersect them with the sphere and stereographically project them from the direction~$(-1,-1,-1)$. Illustration from~\cite{PadrolPilaudPoullot-deformedNestohedra}.}
	\label{fig:graphicalNestedFans}
\end{figure}
\begin{figure}
	\capstart
	\centerline{\includegraphics[scale=.45]{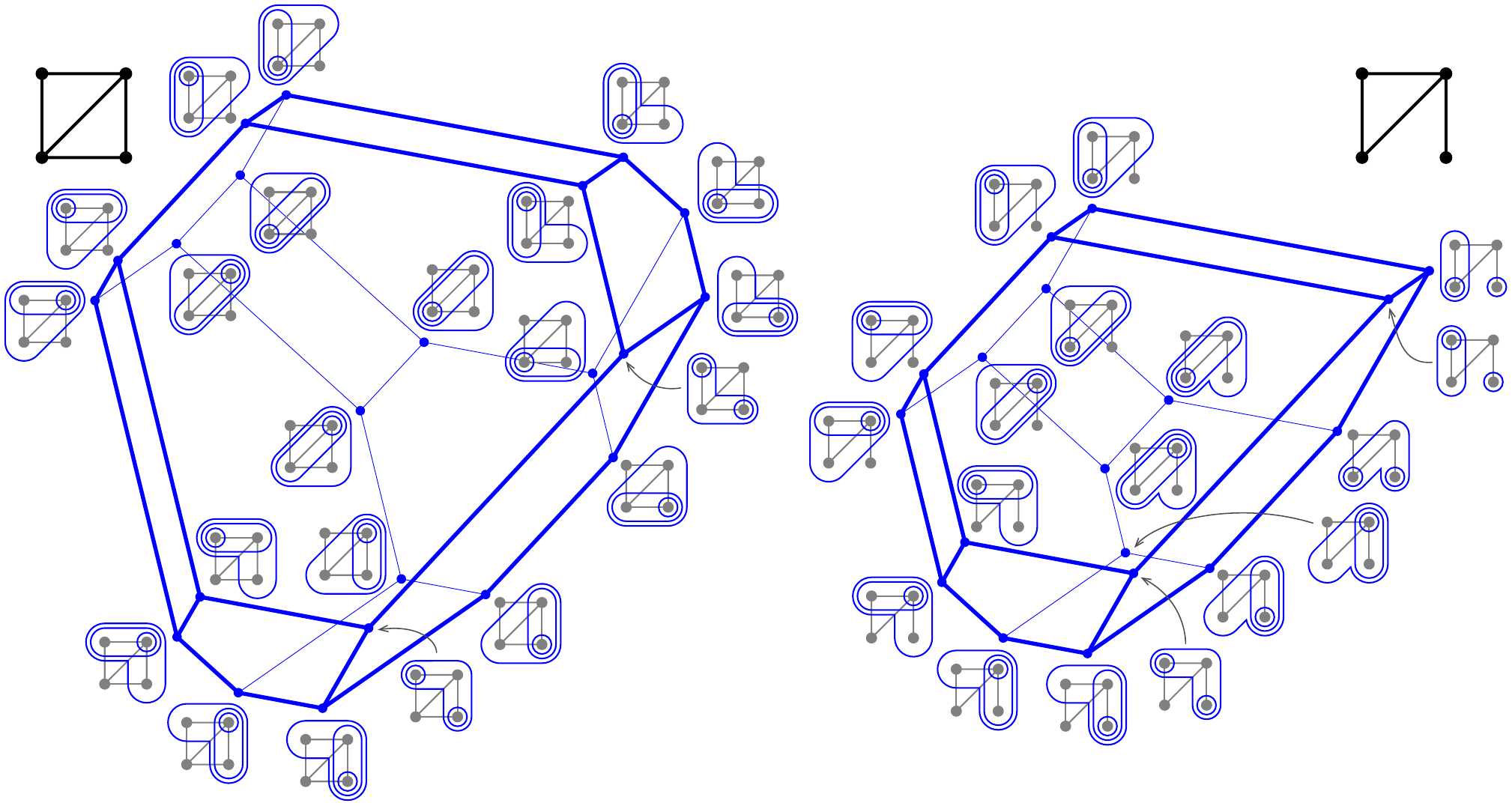}}
	\caption{Two graph associahedra, whose vertices are labeled by the corresponding maximal tubings. The maximal tube is omitted in all tubings. Illustration from~\cite{PadrolPilaudPoullot-deformedNestohedra}.}
	\label{fig:graphAssociahedra}
\end{figure}
\enlargethispage{.3cm}
Following on \cref{exm:graphicalBuildingSet,exm:graphicalNested,exm:restrictionContractionGraphicalBuildingSet}, for a graph~$\graphG$ on~$\ground$, the nested fan of~$\tubes$ is called \defn{graphical nested fan}, and the nestohedron of~$\tubes$ is the \defn{graph associahedron} of~$\graphG$, introduced in~\cite{CarrDevadoss,Devadoss}.
For instance, the associahedron of the complete graph is a permutahedron, the associahedron of a path graph is an associahedron, and the associahedron of a cycle graph is a cyclohedron.
See \cref{fig:permutahedronAssociahedronCyclohedron}.
Two other examples are illustrated \mbox{in \cref{fig:graphicalNestedFans,fig:graphAssociahedra}.}
\begin{figure}
	\capstart
	\centerline{\includegraphics[scale=.45]{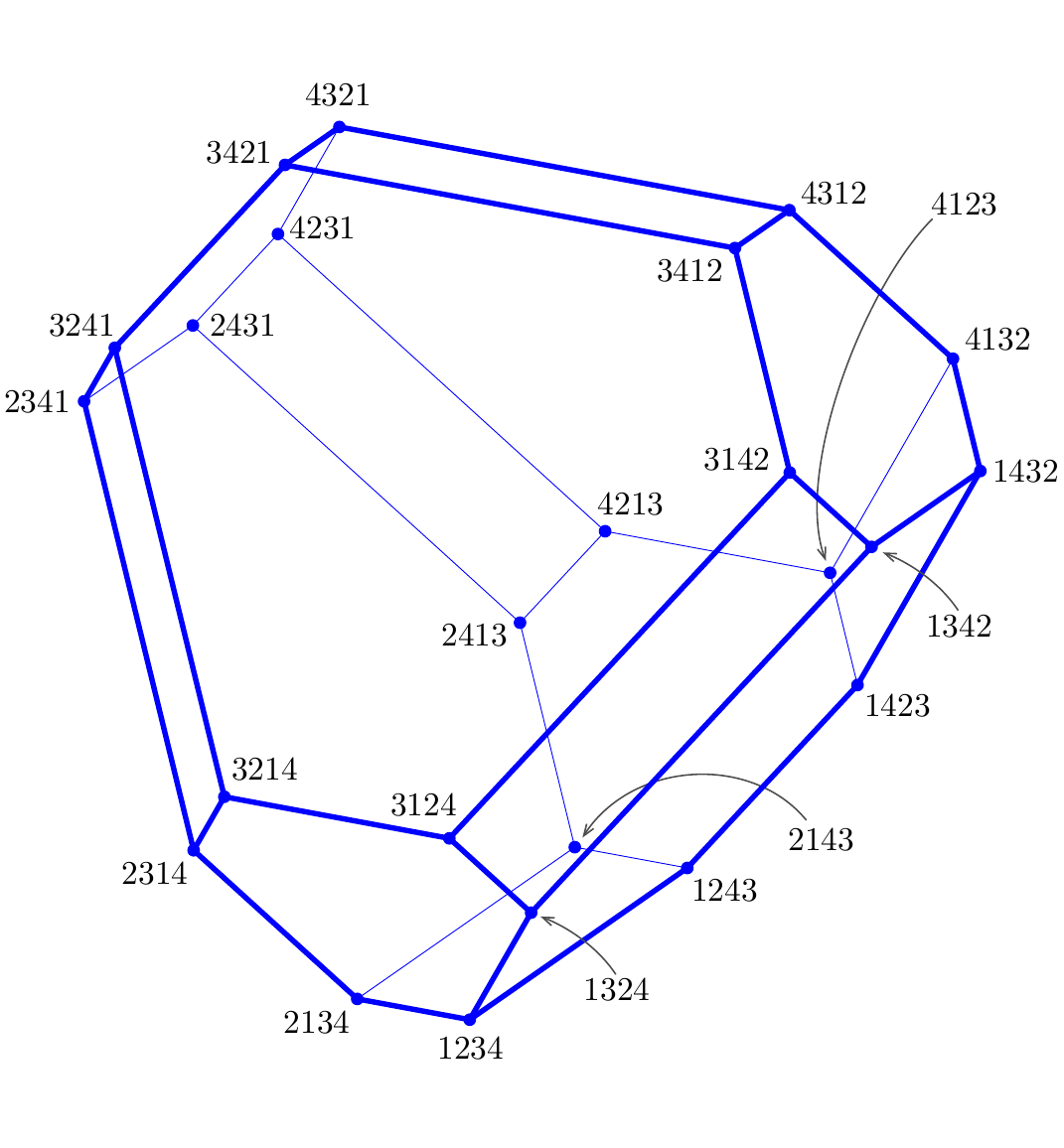} \quad \includegraphics[scale=.45]{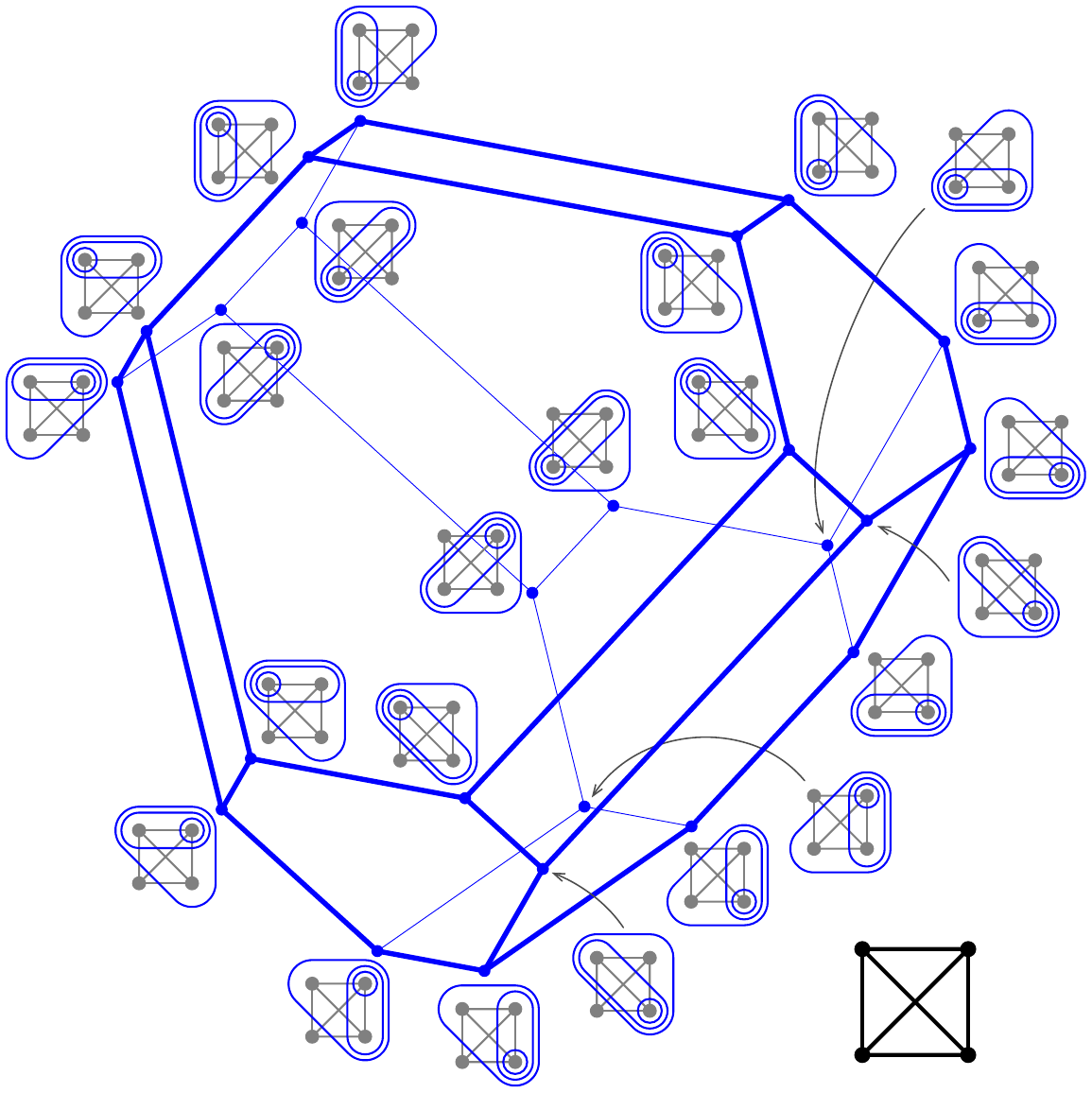}}
	\bigskip
	\centerline{\includegraphics[scale=.45]{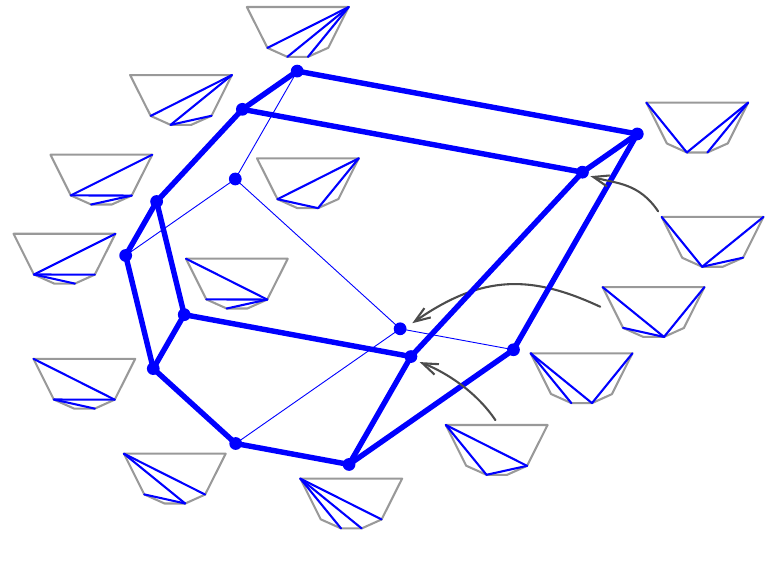} \quad \includegraphics[scale=.45]{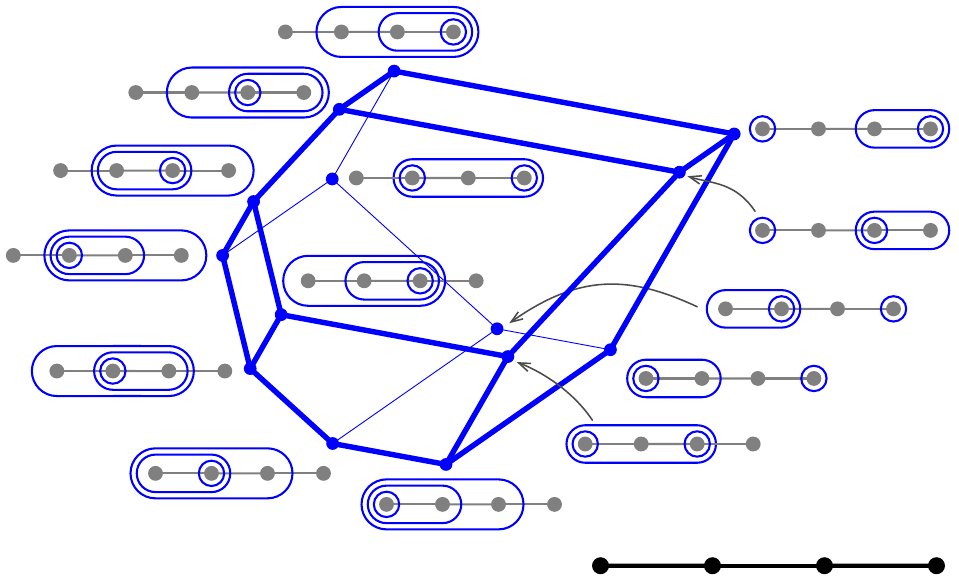}}
	\bigskip
	\centerline{\includegraphics[scale=.45]{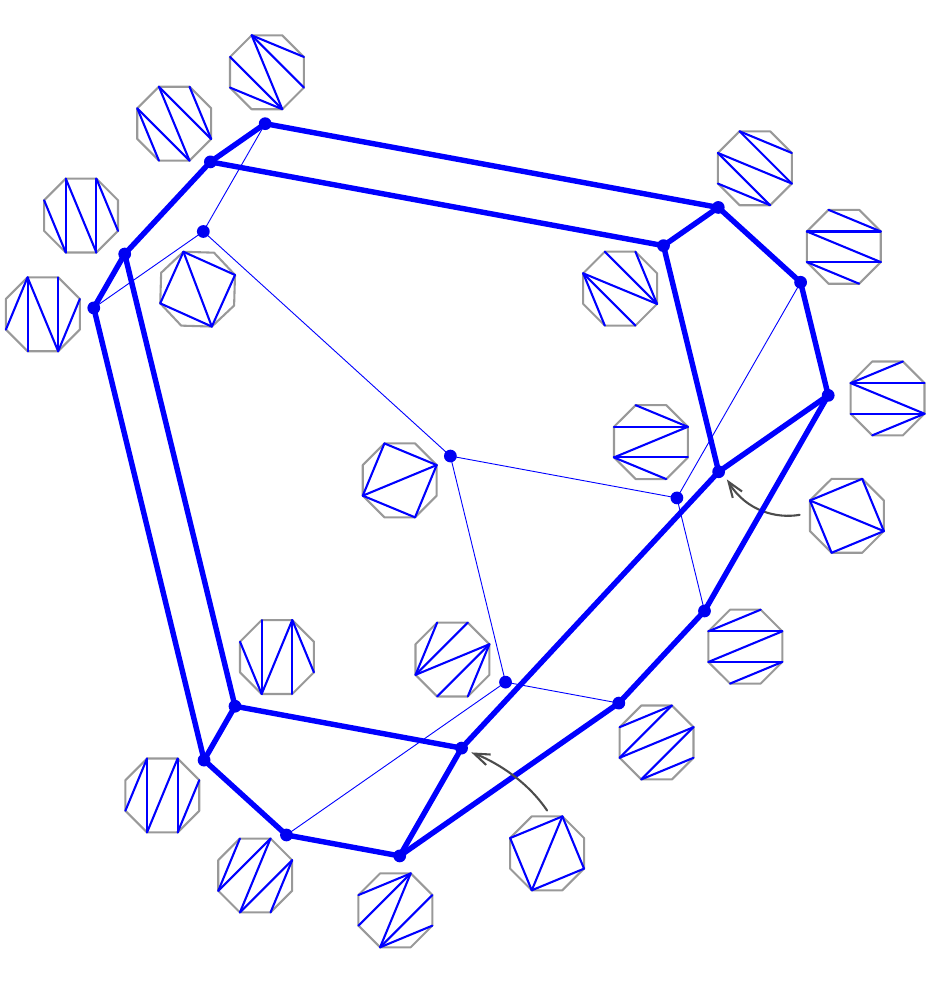} \quad \includegraphics[scale=.45]{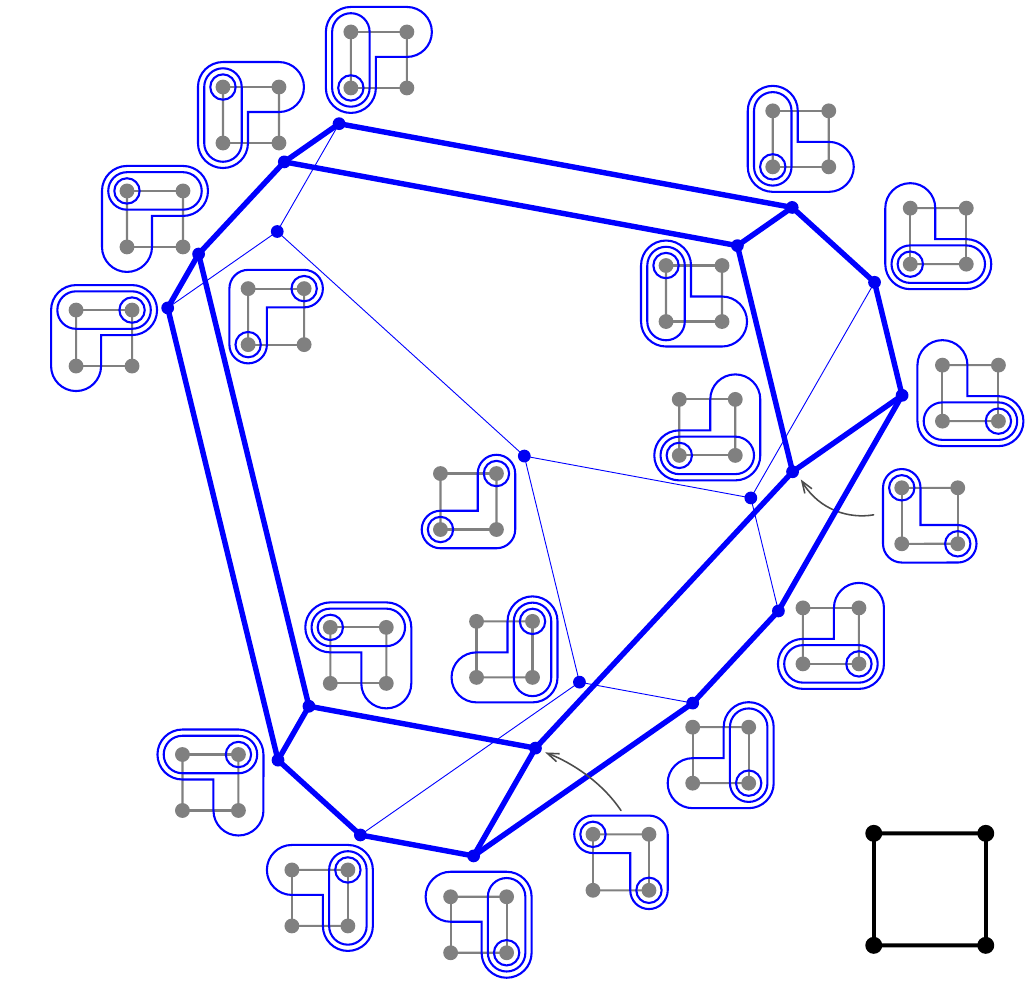}}
	\caption{The permutahedron (top), the associahedron (middle), and the cyclohedron (bottom), with vertices labeled by standard combinatorial objects (left) and by maximal tubings of special graphs (right). The maximal tube is omitted in all tubings.}
	\label{fig:permutahedronAssociahedronCyclohedron}
\end{figure}
\end{example}

\begin{example}
\enlargethispage{.3cm}
Consider the building set~$\building_\circ$ of \cref{exm:booleanBuildingSet}.
The vertex~$\vertexNest[\nested_\circ]$ of the nestohedron~$\Nest(\building_\circ, \b{\lambda})$ corresponding to the maximal nested set~$\nested_\circ \eqdef \{3, 4, 5, 25, 12345, 123456\}$ on~$\building_\circ$ illustrated in \cref{fig:exmNested}\,(right) is given by
\begin{multline*}
\vertexNest[\nested_\circ] = (\lambda_1 + \lambda_{12} + \lambda_{14} + \lambda_{123} + \lambda_{124} + \lambda_{125} + \lambda_{1234} + \lambda_{1235} + \lambda_{1245} + \lambda_{12345}) \b{e}_1 \\ + (\lambda_2 + \lambda_{25}) \b{e}_2 + \lambda_3 \b{e}_3 + \lambda_4 \b{e}_4 + \lambda_5 \b{e}_5 + (\lambda_6 + \lambda_{456} + \lambda_{1456} + \lambda_{2456} + \lambda_{12456} + \lambda_{123456}) \b{e}_6.
\end{multline*}
The inequalities corresponding to the blocks~$25$, $125$, and~$1234$ are given by
\begin{gather*}
x_2 + x_5 \ge \lambda_2 + \lambda_5 + \lambda_{25}
\qquad
x_1 + x_2 + x_5 \ge \lambda_1 + \lambda_2 + \lambda_5 + \lambda_{12} + \lambda_{25} + \lambda_{125}
\\
x_1 + x_2 + x_3 + x_4 \ge \lambda_1 + \lambda_2 + \lambda_3 + \lambda_4 + \lambda_{12} + \lambda_{14} + \lambda_{123} + \lambda_{124} + \lambda_{1234}.
\end{gather*}
Note that since the nested set~$\nested_\circ$ contains the block~$25$ (but neither~$125$ nor~$1234$), the first inequality (but not the last two) is tight for the vertex~$\vertexNest[\nested_\circ]$.
\end{example}


\pagebreak
\subsection{Lattice building sets}
\label{subsec:latticeBuildingSets}

Following~\cite{FeichtnerKozlov2004, FeichtnerSturmfels}, we now extend \cref{subsec:boolean} from the boolean lattice to any finite lattice.
Given a finite lattice~$\lattice$, we denote its bottom and top elements by~$\botzero$ and~$\topone$ respectively.
For~${X \le Y}$ in~$\lattice$, we denote by~$[X,Y] \eqdef \set{Z \in \lattice}{X \le Z \le Y}$ the corresponding interval.
If~$\lattice$ is not clear from the context, we add subscripts like~$\le_\lattice$, $\vee_\lattice$, $\botzero_\lattice$, etc.
For~${\c{X} \subseteq \lattice}$, we denote by~$\max(\c{X})$ the maximal elements of~$\c{X}$, and for~${Y \in \lattice}$, we set~$\c{X}_{< Y}\eqdef \set{X \in \c{X}}{X <Y}$, and we define analogously $\c{X}_{\le Y}$, $\c{X}_{> Y}$, and $\c{X}_{\ge Y}$. 
Recall that the \defn{Cartesian product}~$\poset \times \poset$ of two posets~$\poset$ and~$\poset'$ is defined by~$(X,X') \le_{\poset \times \poset'} (Y,Y')$ if and only if~$X \le_{\poset} Y$ and~$X' \le_{\poset'} Y'$.
Finally, for~$n \in \N$, we use the classical notation~$[n] \eqdef \{1, 2, \dots, n\}$.
The following extends \cref{def:booleanBuildingSet} from the boolean lattice to any finite~lattice.

\begin{definition}[{\cite[Def.~2.2]{FeichtnerKozlov2004}, \cite[Def.~3.1]{FeichtnerSturmfels}}]
\label{def:latticeBuildingSet}
Let $\lattice$ be a finite lattice.
A subset $\building \subseteq \lattice_{>\botzero}$ is an \defn{$\lattice$-building set} if, for any $Y \in  \lattice_{>\botzero}$, there is an isomorphism of partially ordered sets
\[
\phi_Y : \prod_{j \in [k]} \lattice_{\le B_j} \xrightarrow{\quad\cong\quad} \lattice_{\le Y},
\]
where $\{B_1, \dots, B_k\} \eqdef \max(\building_{\le Y})$ and the $j$-th component of the map $\phi_Y$ is the inclusion of intervals $\lattice_{\le B_j} \subseteq \lattice_{\le Y}$ in $\lattice$.
The elements in $\building$ are called \defn{blocks}.
\end{definition}

\begin{remark}
Said differently, the lower interval below any element~$Y$ is (isomorphic to) the Cartesian product of the lower intervals of the maximal elements of $\building$ below~$Y$.
Note that $\building$ must contain all \defn{irreducibles} of~$\lattice$ (\ie all elements~$Y \in \lattice$ such that~$\lattice_{\le Y}$ does not decompose as a Cartesian product), in particular the \defn{atoms} of~$\lattice$ (\ie all elements of~$\lattice$ that cover~$\botzero$) and all join-irreducible elements of~$\lattice$ (\ie the elements of~$\lattice$ that cover a single element).
\end{remark}

\begin{example}
\label{exm:minMaxBuildingSet}
As in \cref{exm:minMaxBooleanBuildingSet}, for any finite lattice~$\lattice$, the set~$\lattice_{>\botzero}$ (resp.~the set of irreducibles of~$\lattice$) is always an $\lattice$-building set, which contains (resp.~is contained in) in any $\lattice$-building set.
\end{example}

\begin{remark}
\label{rem:reccursiveDefinitionLatticeBuildingSet}
Note that \cref{def:latticeBuildingSet} is equivalent to the following recursive definition.
A subset~$\building$ of~$\lattice_{>\botzero}$ forms an $\lattice$-building set if and only if
\begin{itemize}
\item for any~$X \in \max(\lattice_{<\topone})$, the set~$\building_{\le X}$ forms an $\lattice_{\le X}$-building set,
\item $\lattice \cong \prod_{B \in \max(\building)} \lattice_{\le B}$.
\end{itemize}
\end{remark}

\begin{definition}
We denote by~$\connectedComponents(\building) \eqdef \max(\building)$ the set of \defn{$\building$-connected components} (\ie maximal elements of~$\building$), and we say that $\building$ is \defn{connected} if $|\connectedComponents(\building)| = 1$.
\end{definition}

We immediately observe that any subset of~$\lattice$ can be completed to a minimal $\lattice$-building set, generalizing \cref{def:booleanBuildingSetClosure} from \cite[Lem.~3.10]{FeichtnerSturmfels}.
We note that this result was independently obtained in~\cite{BackmanDanner}, as a consequence of their result that the $\lattice$-building sets form an intersection lattice (in fact, a supersolvable convex geometry).

\begin{proposition}
\label{prop:buildingClosure}
For any~$\c{X} \subseteq \lattice$, the set of $\lattice$-building sets containing~$\c{X}$ admits a unique inclusion minimal element, that we call the \defn{$\lattice$-building closure} of~$\c{X}$ and denote by~$\tilde{\c{X}}$.
Moreover, $\tilde{\c{X}} \ssm \c{X}$ is contained in the set of elements~$Y \in \lattice$ such that~$\lattice_{\le Y} \not\cong \prod_{X \in \max(\c{X}_{\le Y})} \lattice_{\le X}$.
\end{proposition}

\begin{proof}
The proof works by induction on~$|\lattice_{>\botzero} \ssm \c{X}|$.
If~$\c{X} = \lattice_{>\botzero}$, it is an $\lattice$-building set.
If~$\c{X}$ is not an $\lattice$-building set, then the set~$\c{Y}$ of elements~$Y \in \lattice$ such that~$\lattice_{\le Y} \not\cong \prod_{X \in \max(\c{X}_{\le Y})} \lattice_{\le X}$ is non-empty.
Let~$Y$ be any minimal element of~$\c{Y}$, let~$\c{X}' \eqdef \c{X} \cup \{Y\}$, and let~$\c{Y}'$ be the set of elements~$Z \in \lattice$ such that~$\lattice_{\le Z} \not\cong \prod_{X \in \max(\c{X}'_{\le Z})} \lattice_{\le X}$.
We claim that 
\begin{enumerate}[(i)]
\item any $\lattice$-building set containing~$\c{X}$ actually contains~$\c{X}'$, 
\item $\c{Y}' \subset \c{Y}$.
\end{enumerate}
This proves the statement by induction.
Indeed, $\c{X}'$ admits a building closure~$\tilde{\c{X}}'$ and~$\tilde{\c{X}}' \ssm \c{X}' \subseteq \c{Y}'$.
We thus conclude from the claim that~$\c{X}$ admits a building closure~$\tilde{\c{X}} = \tilde{\c{X}}'$ and that~$\tilde{\c{X}} \ssm \c{X} = (\tilde{\c{X}}' \ssm \c{X}') \cup \{Y\} \subseteq \c{Y}' \cup \{Y\} \subseteq \c{Y}$.

We prove (i) by contradiction.
Assume that there is an $\lattice$-building set~$\building$ containing~$\c{X}$ but not~$Y$.
Let~$\{B_1, \dots, B_k\} \eqdef \max(\building_{\le Y})$ so that~$\lattice_{\le Y} \cong \prod_{j \in [k]} \lattice_{\le B_j}$.
Hence, we have~$\max(\c{X}_{\le Y}) = \bigsqcup_{j \in [k]} \max(\c{X}_{\le B_j})$.
Moreover, by minimality of~$Y$, we have~$\lattice_{\le B_j} \cong \prod_{X \in \max(\c{X}_{\le B_j})} \lattice_{\le X}$.
We thus obtain that
\[
\lattice_{\le Y} \cong \prod_{j \in [k]} \lattice_{\le B_j} \cong \prod_{j \in [k]}\prod_{X \in \max(\c{X}_{\le B_j})} \lattice_{\le X} = \prod_{X \in \max(\c{X}_{\le Y})} \lattice_{\le X},
\]
contradicting that~$Y \in \c{Y}$.

We prove (ii) by contraposition. 
Let~$Z \notin \c{Y}$, so that~$\lattice_{\le Z} \cong \lattice_{\le Z_1} \times \dots \times \lattice_{\le Z_\ell}$ where $\{Z_1, \dots, Z_\ell\} \eqdef \max(\c{X}_{\le Z})$.
We claim that~$Y \notin \max(\c{X}'_{\le Z})$, so that~$\max(\c{X}'_{\le Z}) = \max(\c{X}_{\le Z})$ and thus $\lattice_{\le Z} \cong \prod_{X \in \max(\c{X}'_{\le Z})} \lattice_{\le X}$, which implies that~$Z \notin \c{Y}'$.
To prove the claim, assume that~$Y \le Z$, and decompose~$Y = (Y_1, \dots, Y_\ell)$ in the product~$\lattice_{\le Z_1} \times \dots \times \lattice_{\le Z_\ell}$.
If there is~$j \in [\ell]$ such that~$Y = Y_j$, then~$Y \le Z_j$ so that~$Y \notin \max(\c{X}'_{\le Z})$.
Otherwise, by minimality of~$Y$, we have~$\lattice_{\le Y_j} \cong \prod_{X \in \max(\c{X}_{\le Y_j})} \lattice_{\le X}$.
We thus obtain that
\[
\lattice_{\le Y} \cong \prod_{j \in [k]} \lattice_{\le Y_j} \cong \prod_{j \in [k]}\prod_{X \in \max(\c{X}_{\le Y_j})} \lattice_{\le X} = \prod_{X \in \max(\c{X}_{\le Y})} \lattice_{\le X},
\]
contradicting that~$Y \in \c{Y}$.
\end{proof}


\subsection{Lattice nested sets}
\label{subsec:latticeNestedSets}

The following extends \cref{def:booleanNestedSet} from the boolean lattice to any finite lattice.

\begin{definition}[{\cite[Def.~2.2]{FeichtnerKozlov2004}, \cite[Def.~3.2]{FeichtnerSturmfels}}]
\label{def:latticeNestedSet}
Let $\lattice$ be a finite lattice and~$\building$ be an $\lattice$-building set.
An \defn{$\lattice$-nested set} $\nested$ on~$\building$ is a subset of~$\building$ containing~$\connectedComponents(\building)$ and such that for any $k\geq 2$  pairwise incomparable elements $B_1, \dots, B_k \in \nested \ssm \connectedComponents(\building)$, the join $B_1\vee \cdots \vee B_k$ does not belong to $\building$.
The \defn{$\lattice$-nested complex} of~$\building$ is the simplicial complex~$\nestedComplex$ whose faces are $\nested \ssm \connectedComponents(\building)$ for all $\lattice$-nested sets~$\nested$ on~$\building$.
\end{definition}

\begin{remark}
Note that as in the boolean case (\cref{rem:nestedComplex}), we include~$\connectedComponents(\building)$ in all $\lattice$-nested sets on~$\building$ (as in~\cite{Postnikov}), but to remove~$\connectedComponents(\building)$ from all $\lattice$-nested sets on~$\building$ when defining the $\lattice$-nested complex of~$\building$ (as in~\cite{Zelevinsky,FeichtnerSturmfels}).
\end{remark}

\begin{example}
\label{exm:orderComplex}
As in \cref{exm:orderComplexBoolean}, for the maximal $\lattice$-building set~$\lattice_{>\botzero}$ of \cref{exm:minMaxBuildingSet}, the $\lattice$-nested sets are the chains in~$\lattice_{>\botzero}$, hence the $\lattice$-nested complex is the \defn{order complex} of~$\lattice_{>\botzero}$.
\end{example}

We will use the following characterization of nested antichains of an $\lattice$-building set~$\building$.

\begin{proposition}[{\cite[Prop.~2.8\,(2)]{FeichtnerKozlov2004}}]
\label{prop:nestedFactorization}
If~$B_1, \dots, B_k$ are pairwise incomparable elements \linebreak of an \mbox{$\lattice$-buil}\-ding set~$\building$, then the set~$\{B_1, \dots, B_k\} \cup \connectedComponents(\building)$ is $\lattice$-nested if and only \linebreak if~${\max(\building_{\le B_1 \vee \dots \vee B_k}) = \{B_1, \dots, B_k\}}$.
\end{proposition}

\cref{prop:nestedFactorization} actually extends to the following recursive definition of $\lattice$-nested sets on~$\building$.

\begin{remark}
\label{rem:reccursiveDefinitionLatticeNestedSet}
A subset~$\nested$ of~$\building$ containing~$\connectedComponents(\building)$ forms an $\lattice$-nested set on~$\building$ if and only if, denoting~$\{B_1, \dots, B_k\} \eqdef \max(\nested \ssm \connectedComponents(\building))$,
\begin{itemize}
\item  for all~$i \in [k]$, the set~$\nested_{\le B_i}$ forms an $\lattice_{\le B_i}$-nested set on~$\building_{\le B_i}$,
\item $\max(\building_{\le B_1 \vee \dots \vee B_k}) = \{B_1, \dots, B_k\}$.
\end{itemize}
\end{remark}

\begin{remark}\label{rmk:semilattice}
\cref{def:latticeBuildingSet,def:latticeNestedSet} slightly differ from the original definitions of E.~M.~Feichtner and D.~Kozlov~\cite[Defs.~2.2~\&~2.7]{FeichtnerKozlov2004} which were stated in the more general setting of meet semilattices.
One can however recover their definitions from ours as follows.
For a meet semilattice~$\semi$, denote by~$\semi^+$ the lattice obtained by adding a top element~$\topone$.
For~$\building \subseteq \semi$, denote by~$\building^+ \eqdef \building \sqcup \{\topone\} \subseteq \semi^+$.
Then~$\building$ is an \defn{$\semi$-semibuilding set} if and only if~$\building^+$ is an $\semi^+$-building set, and the \defn{$\semi$-seminested complex}~$\seminestedComplex_{\semi}(\building)$ of~$\building$ is the $\semi^+$-nested complex~$\nestedComplex[\lattice][\building^+]$ of~$\building^+$.
Note that for a lattice~$\lattice$ and~$\building \subseteq \lattice$, the $\lattice$-nested complex~$\nestedComplex$ slightly differs from the $\lattice$-seminested complex~$\seminestedComplex_\lattice(\building)$.
Namely, $\seminestedComplex_\lattice(\building)$ is the join of~$\nestedComplex$ with the simplex with vertices~$\connectedComponents(\building)$.
We thus decided to make explicit the distinction between nested complexes and seminested complexes (which are called nested complexes in \cite[Def.~2.7]{FeichtnerKozlov2004}).
\end{remark}


\subsection{Restriction, contraction, and links}
\label{subsec:restrictionContractionNestedComplex}

We now extend \cref{def:restrictionContractionBuildingSet,prop:linksNestedComplex} to describe the links in $\lattice$-nested complexes of $\lattice$-building sets for any lattice~$\lattice$.
We provide proofs here as we have not found explicit descriptions of these links in this generality in the literature at the time of writing (the boolean case is described in~\cite[Prop.~3.2]{Zelevinsky}, the case of the lattice of flats of a realizable matroid is described in~\cite[Sect.~4.3]{DeConciniProcesi1995}, and the case of the lattice of flats of a matroid was concomitantly and independently established in~\cite[Thm.~1.6]{BraunerEurPrattVlad}).

\begin{definition}
\label{def:restrictionContractionLatticeBuildingSet}
For any $\lattice$-building set~$\building$ and element~$X \in \lattice$, define
\begin{itemize}
\item the \defn{restriction} of~$\building$ to~$X$ as~$\building_{|X} \eqdef \set{B \in \building}{B \le X}$,
\item the \defn{contraction} of~$X$ in~$\building$ as~$\building_{\!/X} \eqdef \set{B \vee X}{B \in \building \text{ and } B \not\le X}$.
\end{itemize}
\end{definition}

\begin{proposition}
\label{prop:restrictionContractionLatticeBuildingSet}
For any $\lattice$-building set~$\building$ and any~$X \in \building$, the restriction~$\building_{|X}$ is an $\lattice_{\le X}$-building set and the contraction~$\building_{\!/X}$ is an $\lattice_{\ge X}$-building set.
\end{proposition}

\begin{proof}
It is immediate for the restriction, we just prove it for the contraction.
Observe first that $\building_{\!/X} \subseteq \lattice_{> X}$.
Consider~$Y \in \lattice_{\ge X}$, and let~$\{B_1, \dots, B_k\} \eqdef \max(\building_{\le Y})$ so that~${\lattice_{\le Y} \cong \prod_{i \in [k]} \lattice_{\le B_i}}$.
Observe that~$\max\bigl( (\building_{\!/X})_{\le Y} \bigr) = \{C_1, \dots, C_k\} \ssm \{X\}$ where~${C_i \eqdef B_i \vee X}$.
Indeed, any~$C \in (\building_{\!/X})_{\le Y}$ is of the form~$B \vee X$ for some~$B \in \building_{\le Y}$.
By maximality of~$\{B_1, \dots, B_k\}$, there is~$i \in [k]$ such that~${B \le B_i}$, so that~$C = B \vee X \le B_i \vee X = C_i$ for some~$i \in [k]$.
Moreover, the~$C_i = B_i \vee X$ are pairwise incomparable by the product structure~$\lattice_{\le Y} \cong \prod_{i \in [k]} \lattice_{\le B_i}$.
Note that, as~$X \in \building$ and~$X \le Y$, there is~$\ell \in [k]$ with~$X \le B_\ell$.
If~$X = B_\ell$, then~$C_\ell = B_\ell \vee X = X$ is the minimal element of~$\lattice_{\ge X}$ hence does not belong to~$\building_{\!/X}$, and we have to remove it from~$\{C_1, \dots, C_k\}$ to get $\max\bigl( (\building_{\!/X})_{\le Y} \bigr)$.
Finally, again by the product structure~$\lattice_{\le Y} \cong \prod_{i \in [k]} \lattice_{\le B_i}$, we have
\[
[X ,Y] \cong [X, B_\ell] \times \prod_{i \ne \ell} \lattice_{\le B_i} \cong [X, B_\ell] \times \prod_{i \ne \ell} [X, B_i \vee X] \cong \prod_{i \in [k]} [X,C_i] \cong \prod_{C \in \max( (\building_{\!/X})_{\le Y})} [X,C],
\]
where the last isomorphism also holds when~$X = B_\ell$ since~$[X,B_\ell] = \{X\}$ is a singleton.
\end{proof}

\begin{proposition}
\label{prop:linksLatticeNestedComplex}
For any $X \in \building \ssm \connectedComponents(\building)$, the link of the $\lattice$-nested set~$\{X\}$ in~$\nestedComplex$ is isomorphic to the join~$\nestedComplex[\lattice_{\le X}][\building_{|X}] \join \nestedComplex[\lattice_{\ge X}][\building_{\!/X}]$.
\end{proposition}

\begin{proof}
For~$X \in \nested \in \nestedComplex$, define~$\phi_{|X}(\nested) \eqdef \nested_{< X}$ and~$\phi_{\!/X}(\nested) \eqdef \set{B \vee X}{B \in \nested \text{ and } B \not\le X}$.
It is immediate that~$\phi_{|X} \in \nestedComplex[\lattice_{\le X}][\building_{|X}]$, and we prove that~$\phi_{\!/X}(\nested) \in \nestedComplex[\lattice_{\ge X}][\building_{\!/X}]$.
Consider~$k \ge 2$ pairwise incomparable elements~$C_1, \dots, C_k$ of~$\phi_{\!/X}(\nested)$, and let~$C = C_1 \vee \dots \vee C_k$.
Let~$B_i \in \nested$ such that~$C_i = B_i \vee X$.
Note that~$C = B_1 \vee \dots \vee B_k \vee X$ and that~$B_1, \dots, B_k$ are pairwise incomparable (as~$C_1, \dots, C_k$ are).
Hence, as~$\{B_1, \dots, B_k, X\} \subseteq \nested$, we obtain by applying \cref{prop:nestedFactorization} that
\begin{itemize}
\item either $X$ is incomparable with all~$B_i$, then~$\max(\building_{\le C}) = \{B_1, \dots, B_k, X\}$, 
\item or~$X \le B_i$ for some~$i \in [k]$, then~$C = B_1 \vee \dots \vee B_k$ and~$\max(\building_{\le C}) = \{B_1, \dots, B_k\}$.
\end{itemize}
Assume now by means of contradiction that~$C = C_1 \vee \dots \vee C_k \in \building_{\!/X}$.
Let~$B \in \building$ be such that~$C = B \vee X$. Hence, again by \cref{prop:nestedFactorization},
\begin{itemize}
\item either~$C \in \building$ so that~$\max(\building_{\le C}) = \{C\}$, 
\item or $\{B,X\}$ are incomparable and $\lattice$-nested, so that~$\max(\building_{\le C}) = \{B, X\}$.
\end{itemize}
As~$B_i \ne X$ for all~$i \in [k]$, neither~$\{C\}$ nor~$\{B,X\}$ is of the form~$\{B_1, \dots, B_k, X\}$ or~$\{B_1, \dots, B_k\}$.
This contradicts our two previous descriptions of~$\max(\building_{\le C})$.
We thus conclude that~$\nested_{\!/X}$ is indeed $\lattice_{\ge X}$-nested.
Hence, the map~$\nested \mapsto \phi_{|X}(\nested) \join \phi_{\!/X}(\nested)$ sends the link of~$\{X\}$ in~$\nestedComplex$ to the join~$\nestedComplex[\lattice_{\le X}][\building_{|X}] \join \nestedComplex[\lattice_{\ge X}][\building_{\!/X}]$.
To prove that it is an isomorphism, we construct the reverse map.

For~$C \in \building_{\!/X}$, define~$\psi(C) \eqdef C$ if~$C \in \building$, and $\psi(C) \eqdef B$ if~$C \notin \building$ and~$C = B \vee X$ with~$B \in \building$ (this~$B$ is then unique by \cref{prop:nestedFactorization}).
Given~$\nested_{|X}$ in~$\nestedComplex[\lattice_{\le X}][\building_{|X}]$ and~$\nested_{\!/X} \in \nestedComplex[\lattice_{\ge X}][\building_{\!/X}]$,~define
\(
\nested \eqdef \nested_{|X} \cup  \{X\} \cup \psi(\nested_{\!/X}).
\)
Note that~$X \in \nested \subseteq \building$ and that~$\phi_{|X}(\nested) = \nested_{|X}$ while~${\phi_{\!/X}(\nested) = \nested_{\!/X}}$.
We thus just need to prove that~$\nested$ is $\lattice$-nested.
Let~$B_1, \dots, B_k \in \nested$ pairwise incomparable, and let~$B = B_1 \vee \dots \vee B_k$.
We distinguish three cases, depending on whether there is~$\ell \in [k]$ such that~$B_\ell > X$, or~$B_\ell = X$, or not.

Assume that there is~$\ell \in [k]$ such that~$B_\ell > X$.
As~$B_1, \dots, B_k$ are pairwise incomparable, we have~$B_i \not\le X$ for all~$i \in [k]$.
Hence,~$C_i \eqdef B_i \vee X$ is in~$\nested_{\!/X}$ and~$B_i = \psi(C_i)$.
Moreover, assume that there is~$C_i \le C_j$.
If~$C_j \in \building$, then~$B_j = \psi(C_j) = C_j$ so that~${B_i \le C_i \le C_j = B_j}$.
If~$C_j \notin \building$, then~$\max(\building_{\le C_j}) = \{B_j, X\}$ by \cref{prop:nestedFactorization}, and~$B_i \in \building$ with~$B_i \le C_j$ and~${B_i \not\le X}$ implies that~$B_i \le B_j$.
As we assumed that~$B_1, \dots, B_k$ are pairwise incomparable, this implies that~$C_1, \dots, C_k$ are pairwise incomparable.
Since~$C_1, \dots, C_k \in \nested_{\!/X}$ we thus obtain \linebreak that~${C_1 \vee \dots \vee C_k}$ is not in~$\building_{\!/X}$, hence not in~$\building$.
As~$B = B_1 \vee \dots \vee B_k \ge X$ since~$B_\ell \ge X$, we obtain that~${B = B_1 \vee \dots \vee B_k = C_1 \vee \dots \vee C_k}$ is not in~$\building$.

The case where there is~$\ell \in [k]$ with~$B_\ell = X$ is similar.
If~$k > 2$, then the same proof holds, except that we remove~$C_\ell = B_\ell = X$ from~$C_1, \dots, C_k$.
If~$k = 2$, then~$B = B_i \vee X$ with~$B_i$ and~$X$ incomparable, thus~$B_i = \psi(B_i \vee X) \ne B_i \vee X$, so that~$B = B_i \vee X \notin \building$.

Assume now that~$B_i \not\ge X$ for all~$i \in [k]$.
Assume without loss of generality that~$B_i \not\le X$ if and only if~$i \in [\ell]$.
If~$\ell = 0$, then~$B_1, \dots, B_k \in \nested_{|X}$ so that~$B \notin \building$.
Otherwise, let~$Y \eqdef B_1 \vee \dots \vee B_\ell \vee X$.
By the previous case, we have~$Y \notin \building$, hence~$\max(\building_{\le Y}) = \{B_1, \dots, B_\ell, X\}$ by \cref{prop:nestedFactorization}.
As~$B_i < B \eqdef B_1 \vee \dots \vee B_k \le Y$, we obtain that~$B \notin \building$.

We conclude that the map~$\nested \mapsto \phi_{|X}(\nested) \join \phi_{\!/X}(\nested)$ is an isomorphism from the link of~$\{X\}$ in~$\nestedComplex$ to the join~$\nestedComplex[\lattice_{\le X}][\building_{|X}] \join \nestedComplex[\lattice_{\ge X}][\building_{\!/X}]$.
\end{proof}

We finally deduce from \cref{prop:linksLatticeNestedComplex} the links of arbitrary nested sets.
We first need to extend the contraction to joins of blocks.

\begin{lemma}
For~$X,Y \in \lattice$, we have
\(
\building_{\!/X \vee Y} = (\building_{\!/X})_{\!/X \vee Y} = (\building_{\!/Y})_{\!/X \vee Y}. 
\)
Hence, $\building_{\!/X}$ is an \mbox{$\lattice_{\ge X}$-building} set for any~$X \eqdef X_1 \vee \dots \vee X_k$ with~$X_1, \dots, X_k \in \building$.
\end{lemma}

\begin{proof}
Observe first that~$X \vee Y \in \lattice_{\ge X}$ so that~$(\building_{\!/X})_{\!/X \vee Y}$ is well-defined.
Moreover,
\begin{align*}
(\building_{\!/X})_{\!/X \vee Y} 
& = \set{B' \vee X \vee Y}{B' \in \building_{\!/X} \text{ and } B' \not\le X \vee Y} \\
& = \set{B \vee X \vee Y}{B \in \building \text{ and } B \not\le X \text{ and } B \vee X \not\le X \vee Y} \\
& = \set{B \vee X \vee Y}{B \in \building \text{ and } B \not\le X \vee Y} = \building_{\!/X \vee Y}.
\end{align*}
The last equality holds since~$B \le X \vee Y$ implies~$B \vee X \le X \vee Y$ and conversely both $B \le X$ and~$B \vee X \le X \vee Y$ imply~$B \le X \vee Y$.
\end{proof}

\begin{definition}
For an $\lattice$-building set~$\building$, an $\lattice$-nested set~$\nested$ on~$\building$ and~$X \in \nested$, we define
\[
\lattice_{X \in \nested} \eqdef [Y, X]
\qquad\text{and}\qquad
\building_{X \in \nested} \eqdef (\building_{|X})_{\!/Y}
\]
where~$Y = \bigvee_{Z \in \nested, \; Z < X} Z$.
\end{definition}

\begin{proposition}
The link of a nested set~$\nested$ in the nested complex~$\nestedComplex$ is isomorphic to the join of the nested complexes~$\nestedComplex[\lattice_{X \in \nested}][\building_{X \in \nested}]$ for all~$X \in \nested$.
\end{proposition}

\begin{proof}
Apply \cref{prop:linksLatticeNestedComplex} iteratively.
\end{proof}

\begin{remark}
\label{rem:contraction}
Note that the contractions of \cref{def:restrictionContractionBuildingSet,def:restrictionContractionLatticeBuildingSet} slightly differ.
Namely, for boolean building sets we delete~$R$ from all~$B \in \building$ with~$B \not\subseteq R$, while for lattice building sets we join~$X$ to all~$B \in \building$ with~$B \not\le X$.
It is equivalent for boolean building sets since the interval above~$R$ in the boolean lattice on~$[n]$ is isomorphic to the boolean lattice on~$[n] \ssm R$.
We will use the same identification for contractions of building sets over atomic lattices, in order to keep compatibility with the contraction on oriented matroids of \cref{def:restrictionContractionOM}.
In contrast, this identification makes no sense for contractions of building sets on arbitrary lattices.
\end{remark}


\subsection{Combinatorial blow-ups}
\label{subsec:combinatorialblowup}

The operation of combinatorial blow-ups of meet semilattices was introduced in \cite[Def.~3.1]{FeichtnerKozlov2004} to give a recursive description of nested complexes. 
Following \cite[Sect.~4.2]{FeichtnerKozlov2004}, we will see in \cref{exm:stellarSubdivisionTruncation} that combinatorial blow-ups on face lattices of polytopes correspond to stellar subdivisions, or dually truncations.

\begin{definition}[{\cite[Def.~3.1]{FeichtnerKozlov2004}}]
\label{def:combinatorialBlowUps}
	Let $\semi$ be a meet semilattice, and let $X\in\semi$. We define the \defn{combinatorial blow-up} of $\semi$ at $X$ as the poset $\Bl$ whose elements are
	\begin{itemize}
		\item $Y\in\semi$ such that $Y\not \geq_\semi X$,
		\item $(X,Y)$ for $Y\in \semi$  such that $Y\not \geq_\semi X$ and $Y \vee_\semi X$ exists,
	\end{itemize}
and whose order relations are 
	\begin{itemize}
	\item $Y >_{\Bl} Z$ if $Y >_\semi Z$,
	\item $(X,Y) >_{\Bl} (X,Z)$ if $Y >_\semi Z$,
	\item $(X,Y) >_{\Bl} Z$ if $Y \ge_\semi Z$,
\end{itemize}	
where in all three cases $Y, Z \not\ge_\semi X$.
\end{definition}

We state the following result in terms of semilattices as in \cref{rmk:semilattice}, which is the original formulation from \cite{FeichtnerKozlov2004}, as this is the form we will use latter.

\begin{theorem}[{\cite[Thm.~3.4]{FeichtnerKozlov2004}}]\label{thm:blowup}
	Let $\semi$ be a meet semilattice, and let $\building=\{B_1,\dots,B_k\} \subseteq \semi$ be an $\semi$-semibuilding set ordered so that if $B_i\geq_\semi B_j$ then $i<j$. Consider the sequence of combinatorial blow-ups
	\[\Bl[\semi][\building] \eqdef \Bl[{\Bl[\cdots {\Bl[\semi][B_{1}]}][B_{k-1}]}][B_k].\]
	Then $\Bl[\semi][\building]$ is isomorphic to the face poset of the $\semi$-seminested complex~$\seminestedComplex_{\semi}(\building)$.
\end{theorem}


\subsection{Operations}
\label{subsec:operationsNestedComplex}

We now briefly discuss two operations on lattices, building sets, and nested complexes, which enable us to restrict our proofs to connected building sets (we will see later in \cref{sec:facialBuildingSetsFacialNestedComplexes} that these operations can also be performed at the level of face lattices of oriented matroids).
Again, these operations are not original, but we could not find the results below explicitly in the lattice generality in the existing literature.

\begin{definition}
The \defn{Cartesian product}~$\lattice \times \lattice'$ of two lattices~$\lattice$ and~$\lattice'$ is the lattice where for all~$X, Y \in \lattice$ and~$X', Y' \in \lattice'$, we have
\begin{itemize}
\item $(X,X') \le_{\lattice \times \lattice'} (Y,Y')$ if and only if~$X \le_\lattice Y$ and~$X' \le_{\lattice'} Y'$,
\item the join~$(X,X') \vee_{\lattice \times \lattice'} (Y,Y') = (X \vee_\lattice Y, X' \vee_{\lattice'} Y')$, and 
\item the meet~$(X,X') \wedge_{\lattice \times \lattice'} (Y,Y') = (X \wedge_\lattice Y, X' \wedge_{\lattice'} Y')$.
\end{itemize}
\end{definition}

\begin{proposition}
\label{prop:directSumLatticeBuildingSets}
If~$\building$ is an $\lattice$-building set and~$\building'$ is an $\lattice'$-building set, then
\[
\building \oplus \building' \eqdef \bigl( \building \times \{\botzero_{\lattice'}\} \bigr) \sqcup \bigl( \{\botzero_\lattice\} \times \building' \bigr)
\]
is an $(\lattice \times \lattice')$-building set, whose $(\lattice \times \lattice')$-nested complex is isomorphic to~$\nestedComplex \join \nestedComplex[\lattice'][\building']$.
\end{proposition}

\begin{proof}
Let~$Y \in \lattice$ with~$\max(\building_{\le Y}) = \{B_1, \dots, B_k\}$ and~$Y' \in \lattice'$ with~$\max(\building'_{\le Y'}) = \{B'_1, \dots, B'_{k'}\}$.
Then
\[
\max \bigl( (\building \oplus \building')_{\le (Y,Y')} \bigr) = \{(B_1, \botzero_{\lattice'}), \dots, (B_k, \botzero_{\lattice'})\} \sqcup \{(\botzero_\lattice, B'_1), \dots, (\botzero_\lattice, B'_{k'})\} \qquad \text{and}
\]
\[
(\lattice \times \lattice')_{\le (Y,Y')} \cong \lattice_{\le Y} \times \lattice'_{\le Y'} \cong \prod_{i \in [k]} \lattice_{\le B_i} \times \prod_{i \in [k']} \! \lattice_{\le B'_i} \cong \prod_{i \in [k]} (\lattice \times \lattice')_{\le (B_i, \botzero_{\lattice'})} \times \prod_{i \in [k']} \! (\lattice \times \lattice')_{\le (\botzero_\lattice, B'_i)}.
\]
Hence, $\building \oplus \building'$ is indeed an $(\lattice \times \lattice')$-building set.
Consider now~$\nested \subseteq \building \times \{\botzero_{\lattice'}\}$ and~$\nested' \subseteq \{\botzero_\lattice\} \times \building'$.
As~$\bigvee_{\lattice \times \lattice'} (\nested \sqcup \nested') = (\bigvee_\lattice \nested, \bigvee_{\lattice'} \nested')$, we indeed obtain that $\nested \sqcup \nested' \in \nestedComplex[\lattice \times \lattice'][\building \oplus \building']$ if and only if~$\nested \in \nestedComplex$ and~$\nested' \in \nestedComplex[\lattice'][\building']$.
\end{proof}

\begin{definition}
\label{def:freeProductLattices}
The \defn{free product}~$\lattice \boxtimes \lattice'$ of two finite lattices~$\lattice$ and~$\lattice'$ is the lattice obtained as the Cartesian product~$\bigl( \lattice \ssm \{\topone_\lattice\} \bigr) \times \bigl( \lattice' \ssm \{\topone_{\lattice'}\} \bigr)$ to which we add a top element~$\topone_{\lattice \boxtimes \lattice'}$.
\end{definition}

As~$\lattice \ssm \{\topone_\lattice\}$ and~$\lattice' \ssm \{\topone_{\lattice'}\}$ are both meet semilattices, $\bigl( \lattice \ssm \{\topone_\lattice\} \bigr) \times \bigl( \lattice' \ssm \{\topone_{\lattice'}\} \bigr)$ is a meet semilattice, hence~$\lattice \boxtimes \lattice'$ is a meet semilattice with a maximal element, hence indeed a lattice.
Note that, for~$x_1, \dots, x_k \in \lattice$ with~$x_1 \vee_\lattice \dots \vee_\lattice x_k <_\lattice \topone_\lattice$ and~$x_1', \dots, x_k' \in \lattice$ with~$x_1' \vee_{\lattice'} \dots \vee_{\lattice'} x_k' <_{\lattice'} \topone_{\lattice'}$, we have 
\begin{align*}
(x_1, x_1') \vee_{\lattice \boxtimes \lattice'} \dots \vee_{\lattice \boxtimes \lattice'} (x_k, x_k') 
& = (x_1, x_1') \vee_{\lattice \times \lattice'} \dots \vee_{\lattice \times \lattice'} (x_k, x_k') \\
& = (x_1 \vee_\lattice \dots \vee_\lattice x_k, x_1' \vee_{\lattice'} \dots \vee_{\lattice'} x_k').
\end{align*}

\begin{proposition}
\label{prop:freeSumLatticeBuildingSets}
If~$\building$ be a connected $\lattice$-building set and~$\building'$ is a connected $\lattice'$-building set, then
\[
\building \boxplus \building' \eqdef \bigl( (\building \ssm \{\topone_\lattice\}) \times \{\botzero_{\lattice'}\} \bigr) \sqcup \bigl( \{\botzero_\lattice\} \times (\building' \ssm \{\topone_{\lattice'}\} ) \bigr) \sqcup \{\topone_{\lattice \boxtimes \lattice'}\}
\]
is a connected $(\lattice \boxtimes \lattice')$-building set, whose $(\lattice \boxtimes \lattice')$-nested complex is isomorphic to~$\nestedComplex \join \nestedComplex[\lattice'][\building']$.
\end{proposition}

\begin{proof}
Consider~${Z \in \building \boxplus \building'}$.
If~$Z = \topone_{\lattice \boxtimes \lattice'}$, then~$Z \in \building \boxplus \building'$ and there is nothing to prove.
Otherwise, $Z = (Y,Y')$ where~${Y \in \lattice \ssm \{\topone_\lattice\}}$ and~$Y' \in \lattice' \ssm \{\topone_{\lattice'}\}$.
If~$\max(\building_{\le Y}) = \{B_1, \dots, B_k\}$ and~$\max(\building'_{\le Y'}) = \{B'_1, \dots, B'_{k'}\}$, then
\[
\max \bigl( (\building \boxplus \building')_{\le Z} \bigr) = \{(B_1, \botzero_{\lattice'}), \dots, (B_k, \botzero_{\lattice'})\} \sqcup \{(\botzero_\lattice, B'_1), \dots, (\botzero_\lattice, B'_{k'})\} \qquad \text{and}
\]
\[
(\lattice \boxtimes \lattice')_{\le Z} \cong (\lattice \boxtimes \lattice')_{\le (Y,Y')} \cong \prod_{i \in [k]} (\lattice \boxtimes \lattice')_{\le (B_i, \botzero_{\lattice'})} \times \prod_{i \in [k']} \! (\lattice \boxtimes \lattice')_{\le (\botzero_\lattice, B'_i)},
\]
by restricting the isomorphism above for the Cartesian product~$\lattice \times \lattice'$.
Hence, $\building \boxplus \building'$ is indeed an $(\lattice \boxtimes \lattice')$-building set.
Consider now~${\nested \subseteq (\building \ssm \{\topone_\lattice\}) \times \{\botzero_{\lattice'}\}}$ and~${\nested' \subseteq \{\botzero_\lattice\} \times (\building' \ssm \{\topone_{\lattice'}\})}$.
If~$\bigvee_\lattice \nested = \topone_\lattice$, then~$\nested \notin \nestedComplex$ (because ${\topone_\lattice \in \building}$ by our assumption that~$\building$ is connected) and $\bigvee_{\lattice \boxtimes \lattice'} \nested \sqcup \nested' \ge \bigvee_{\lattice \boxtimes \lattice'} \nested' = \topone_{\lattice \boxtimes \lattice'}$ so that~${\nested \sqcup \nested' \notin \nestedComplex[\lattice \boxtimes \lattice'][\building \boxplus \building']}$.
Similarly, if~$\bigvee_{\lattice'} \nested' = \topone_{\lattice'}$, then~$\nested' \notin \nestedComplex[\lattice'][\building']$ and~$\nested \sqcup \nested' \notin \nestedComplex[\lattice \boxtimes \lattice'][\building \boxplus \building']$.
Now if~$\bigvee_\lattice \nested < \topone_\lattice$ and~$\bigvee_{\lattice'} \nested' < \topone_{\lattice'}$, then we have~$\bigvee_{\lattice \boxtimes \lattice'} \nested \sqcup \nested' = \bigvee_{\lattice \times \lattice'} \nested \sqcup \nested' = (\bigvee_\lattice \nested, \bigvee_{\lattice'} \nested')$.
We conclude that~$\nested \sqcup \nested' \in \nestedComplex[\lattice \boxtimes \lattice'][\building \boxplus \building']$ if and only if~$\nested \in \nestedComplex$ and~$\nested' \in \nestedComplex[\lattice'][\building']$ as desired.
\end{proof}

\begin{corollary}
\label{coro:connectedLatticeBuildingSet}
For any $\lattice$-building set~$\building$ with~$\connectedComponents(\building) = \{B_1, \dots, B_k\}$, the $\lattice$-nested complex of~$\building$ is isomorphic to the $(\lattice_{\le B_1} \boxtimes \dots \boxtimes \lattice_{\le B_k})$-nested complex of the connected $(\lattice_{\le B_1} \boxtimes \dots \boxtimes \lattice_{\le B_k})$-building set~$\building_{\le B_1} \boxplus \dots \boxplus \building_{\le B_k}$.
\end{corollary}

\begin{proof}
Let~$\lattice_i \eqdef \lattice_{\le B_i}$ and~$\building_i \eqdef \building_{\le B_i}$ for~$i \in [k]$.
Then
\[
\nestedComplex[\lattice][\building] \cong \nestedComplex[\lattice_1][\building_1] \join \dots \join \nestedComplex[\lattice_k][\building_k] \cong \nestedComplex[\lattice_1 \boxtimes \dots \boxtimes \lattice_k][\building_{\le B_1} \boxplus \dots \boxplus \building_{\le B_k}].
\]
The first isomorphism holds by \cref{prop:directSumLatticeBuildingSets} since~$\lattice \cong \lattice_1 \times \dots \times \lattice_k$ and~${\building \cong \building_1 \oplus \dots \oplus \building_k}$ as~$\building$ is an $\lattice$-building set and~$\connectedComponents(\building) = \{B_1, \dots, B_k\}$.
The second isomorphism holds by \cref{prop:freeSumLatticeBuildingSets} since~$\building_i$ is a connected $\lattice_i$-building set for all~$i \in [k]$.
Composing these two isomorphisms shows the statement.
\end{proof}


\section{Oriented matroids}
\label{sec:orientedMatroids}

We now recall some aspects of oriented matroids.
See~\cite{BjornerLasVergnasSturmfelsWhiteZiegler} for a detailed reference on the subject.
We start from the intuitive definition of the oriented matroid of a vector configuration~(\cref{subsec:vectorConfigurations}) before providing the general abstract definition of oriented matroid (\cref{subsec:orientedMatroids}).
We then consider the Las Vergnas face lattice of an oriented matroid (\cref{subsec:LasVergnasLattice}), define restrictions and contractions (\cref{subsec:restrictionContractionOM}) and stellar subdivisions (\cref{subsec:stellarSubdivisions}), study the two natural operations of direct and free sum (\cref{subsec:operationsOM}), and conclude with decompositions into connected components~(\cref{subsec:connectedComponents}).


\subsection{Vector configurations}
\label{subsec:vectorConfigurations}

\enlargethispage{.4cm}
We first consider a finite vector configuration, and present some combinatorial encodings of its linear dependences and evaluations.

\begin{definition}
\label{def:dependencesEvaluationsVectorConfiguration}
For a finite vector configuration~$\b{A} \eqdef (\b{a}_s)_{s \in \ground} \in (\R^d)^\ground$, we denote by
\begin{itemize}
\item $\dependences \eqdef \set{\b{\delta} \in \R^\ground}{\sum_{s \in \ground} \delta_s \b{a}_s = \zero}$ the vector space of linear \defn{dependences} on~$\b{A}$,
\item $\evaluations \eqdef\! \set{(f(\b{a}_s))_{s \in \ground} \!\in\! \R^\ground\!}{\!f \!\in\! \smash{(\R^d)}^*\!}$ the vector space of \defn{evaluations} of linear forms~on~$\b{A}$.
\end{itemize}
Note that the dimension of the evaluation space~$\evaluations$ is the \defn{rank}~$\rank$ of~$\b{A}$, while the dimension of the dependence space~$\dependences$ is the \defn{corank}~$\corank \eqdef |\ground| - \rank$ of~$\b{A}$.
Moreover, $\dependences$ and~$\evaluations$ are orthogonal spaces.
\end{definition}

\begin{definition}
A \defn{signed subset} of~$\ground$ is a pair~$x = (x_+,x_-)$ with~$x_+, x_- \subseteq \ground \text{ and } x_+ \cap x_- = \varnothing$.
The \defn{support} of~$x$ is $\underline{x} \eqdef x_+ \cup x_-$, and the \defn{opposite} of~$x$ is~$-x \eqdef (x_-, x_+)$.
The \defn{signature} of~${\b{\delta} \in \R^\ground}$ is the signed subset~$\signature(\b{\delta}) \eqdef (\set{s \in \ground}{\delta_s > 0}, \set{s \in \ground}{\delta_s < 0})$.
We denote by~$\signature(\ground)$ the set of signed subsets~of~$\ground$.
\end{definition}

\begin{definition}
\label{def:orientedMatroidVectorConfiguration}
The \defn{oriented matroid}~$\OM(\b{A})$ of a finite vector configuration~$\b{A} \subset \R^d$ is the combinatorial data given equivalently by
\begin{itemize}
\item the \defn{vectors}~$\vectors$ of~$\b{A}$, \ie the signatures of linear dependences of~$\b{A}$,
\item the \defn{covectors}~$\covectors$ of~$\b{A}$, \ie the signatures of linear evaluations on~$\b{A}$,
\item the \defn{circuits}~$\circuits$ of~$\b{A}$, \ie the support minimal signatures of linear dependences of~$\b{A}$,
\item the \defn{cocircuits}~$\cocircuits$ of~$\b{A}$, \ie the support minimal signatures of linear evaluations on~$\b{A}$.
\end{itemize}
\end{definition}

\begin{example}
\label{exm:vectorConfiguration}
Consider the vector configuration
\[
\b{A}_\circ \eqdef \left\{ 
	\begin{bmatrix} 0 \\ 0 \\ 0 \\ 1 \end{bmatrix},
	\begin{bmatrix} 0 \\ 0 \\ 0 \\ 1 \end{bmatrix},
	\begin{bmatrix} 0 \\ 0 \\ 1 \\ 1 \end{bmatrix},
	\begin{bmatrix} 1 \\ 0 \\ 0 \\ 1 \end{bmatrix},
	\begin{bmatrix} 0 \\ 1 \\ 0 \\ 1 \end{bmatrix},
	\begin{bmatrix} 1 \\ 1 \\ 0 \\ 1 \end{bmatrix}
\right\} \subset \R^4.
\]
It has $13$ vectors, $153$ covectors, $6$ circuits, and $14$ cocircuits.
Its circuits are $(1,2)$, $(16,45)$, $(26,45)$ and their opposites, and its cocircuits are $(12,6)$, $(124,\varnothing)$, $(125,\varnothing)$, $(3,\varnothing)$, $(46,\varnothing)$, $(4,5)$, $(56,\varnothing)$ and their opposites.
(Here again, we abuse notation and write $123$ for $\{1,2,3\}$ since all labels have a single digit).
\end{example}

\begin{remark}
We say that~$x = (x_+, x_-)$ and~$y = (y_+, y_-)$ are \defn{sign orthogonal}, and we write~$x \perp y$, if~$(x_+ \cap y_+) \cup (x_- \cap y_-) \ne \varnothing \iff (x_+ \cap y_-) \cup (x_- \cap y_+) \ne \varnothing$.
Observe that
\begin{itemize}
\item $v \in \{-, 0, +\}^\ground$ is a vector of~$\vectors$ if and only if~$v \perp c^*$ for all cocircuits~$c^* \in \cocircuits$,
\item $v^* \in \{-, 0, +\}^\ground$ is a covector of~$\covectors$ if and only if~$c \perp v^*$ for all circuits~$c \in \circuits$.
\end{itemize}
Note in particular that $c \perp c^*$ for any circuit~$c \in \circuits$ and cocircuit~$c^* \in \cocircuits$.
\end{remark}

\begin{remark}
Note that any of the four sets of vectors~$\vectors$, covectors~$\covectors$, circuits~$\circuits$ and cocircuits~$\cocircuits$ completely determines the other three.
To summarize, we reproduce the diagram from~\cite[Coro.~6.9]{Ziegler}:

\centerline{
\begin{tikzpicture}[xscale=3, yscale=2]
	\node (dep) at (0,1) {$\dependences$};
	\node (eval) at (0,0) {$\evaluations$};
	\node (vec) at (1,1) {$\vectors$};
	\node (covec) at (1,0) {$\covectors$};
	\node (circ) at (2,1) {$\circuits$};
	\node (cocirc) at (2,0) {$\cocircuits$};
	\draw[->] (dep) -- (vec) node[midway, above] {$\signature$};
	\draw[->] (vec) -- (circ) node[midway, above] {$\suppmin$};
	\draw[->] (eval) -- (covec) node[midway, below] {$\signature$};
	\draw[->] (covec) -- (cocirc) node[midway, below] {$\suppmin$};
	\draw[<->] (dep) -- (eval) node[midway, left] {orthogonal spaces};
	\draw[<->] (vec) -- (covec) node[midway, left] {$\perp$};
	\draw[->] (circ) -- (covec) node[near end, below] {$\perp$};
	\draw[->] (cocirc) -- (vec) node[near end, above] {$\perp$};
\end{tikzpicture}
}

\noindent
where 
\begin{itemize}
\item $\signature(X)$ denotes the set of signatures of the elements in~$X \subseteq \R^\ground$,
\item $\suppmin(X)$ denotes the set of support minimal elements of~$X \subseteq \sigma(\ground)$, and
\item ${\perp}(X) \eqdef \set{y \in \sigma(\ground)}{x \perp y \text{ for all } x \in X}$ denotes the sign vectors which are sign orthogonal to all sign vectors in~$X \subseteq \{-, 0, +\}^\ground$.
\end{itemize}
\end{remark}

\begin{example}
\label{exm:graphicalOM}
Consider a directed graph~$\digraph$ with vertex set~$V$ and arc set~$\ground$ (loops and multiple arcs are allowed).
Let~$(\b{b}_v)_{v \in V}$ denote the standard basis of~$\R^V$.
The \defn{incidence configuration}~$\b{A}_{\digraph}$ of~$\digraph$ has a vector~$\b{a}_{(u,v)} \eqdef \b{b}_u - \b{b}_v \in \R^V$ for each arc~$(u,v)$ of~$\digraph$.
Its oriented matroid is called the \defn{graphical oriented matroid}~$\OM(\digraph)$ of~$\digraph$.
Its ground set is the set~$\ground$ of arcs of~$\digraph$, and it has
\begin{itemize}
\item a vector~$v$ for each collection of cycles, with clockwise arcs~$v_+$ and counter-clockwise~arcs~$v_-$,
\item a covector~$v^*$ for each oriented edge cut (meaning a subset~$X$ of edges, together with an acyclic orientation of the graph obtained by contracting the connected components of~$\digraph \ssm X$), with forward arcs~$v^*_+$ and backward arcs~$v^*_-$,
\item a circuit~$c$ for each simple cycle, with clockwise arcs~$c_+$ and counter-clockwise arcs~$c_-$,
\item a cocircuit~$c^*$ for each support minimal edge cut, with forward arcs~$c^*_+$ and backward arcs~$c^*_-$.
\end{itemize}
See \cite[Prop.~1.1.7 \& Chap.~5]{Oxley} and \cite[Sect.~1.1]{BjornerLasVergnasSturmfelsWhiteZiegler}.
\end{example}

\begin{example}
\label{exm:graphicalVectorConfiguration}
Consider the directed graph~$\digraph_\circ$ of \cref{fig:exmDigraph}\,(left).
Its oriented matroid coincides with that of the vector configuration~$\b{A}_\circ$ \cref{exm:vectorConfiguration}.
In fact, the vector configuration~$\b{A}_\circ$ is mapped to the incidence configuration of~$\digraph_\circ$ by the linear map with matrix
\[
\begin{bmatrix}
	1 & 0 & 1 & -1 \\
	-1 & 0 & -1 & 0 \\
	0 & -1 & -1 & 1 \\
	0 & 1 & 0 & 0 \\
	0 & 0 & 1 & 0
\end{bmatrix}.
\]
\begin{figure}[h]
	\capstart
	\centerline{\includegraphics[scale=.9]{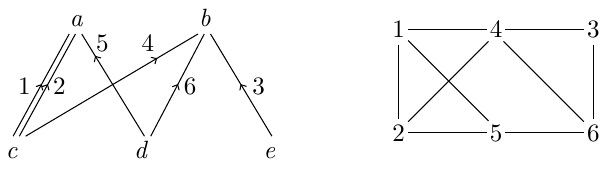}}
	\caption{A directed graph~$\digraph_\circ$ (left) and its line graph~$L(\digraph_\circ)$ (right).}
	\label{fig:exmDigraph}
\end{figure}
\end{example}


\subsection{Oriented matroids}
\label{subsec:orientedMatroids}

We now define abstract oriented matroids, which are combinatorial abstractions for the dependences and evaluations of vector configurations considered in \cref{def:dependencesEvaluationsVectorConfiguration,def:orientedMatroidVectorConfiguration}.
There are various cryptomorphic axiomatizations of oriented matroids, see~\cite[Sect.~3]{BjornerLasVergnasSturmfelsWhiteZiegler}.
For completeness, we define here oriented matroids via the circuit axioms.
Note that we will not use this precise definition, only some of its consequences recalled in this section.

\begin{definition}
\label{def:OM}
An \defn{oriented matroid} with ground set~$\ground$ is the combinatorial data~$\OM$ given by four sets~$\vectors[\OM]$, $\covectors[\OM]$, $\circuits[\OM]$ and~$\cocircuits[\OM]$ of signed subsets of~$\ground$, respectively called the \defn{vectors}, \defn{covectors}, \defn{circuits} and \defn{cocircuits} of~$\OM$, and connected by
\[
\vectors[\OM] = {\perp} \bigl( \cocircuits[\OM] \bigr),
\quad
\covectors[\OM] = {\perp} \bigl( \circuits[\OM] \bigr),
\quad
\circuits[\OM] = \suppmin \bigl( \vectors[\OM] \bigr)
\quad\text{and}\quad
\cocircuits[\OM] = \suppmin \bigl( \covectors[\OM] \bigr),
\]
(where~$\suppmin(X)$ denotes the set of support minimal elements of~$X$, and~${\perp}(X)$ denotes the set of sign vectors sign orthogonal to all elements of~$X$), and such that~$\circuits[\OM]$ (or equivalently~$\cocircuits[\OM]$) satisfy the following circuit axioms
\begin{itemize}
\item $\varnothing \notin \circuits[\OM]$,
\item if $c \in \circuits[\OM]$ then $-c \in \circuits[\OM]$,
\item if~$c, c' \in \circuits[\OM]$ with~$\underline{c} \subseteq \underline{c}'$, then $c = c'$ or $c = -c'$,
\item for any~$c,c' \in \circuits[\OM]$ with~$c \ne -c'$ and~$s \in c_+ \cap c'_-$, there exists~$c'' \in \circuits[\OM]$ such that~$c''_+ \subseteq (c_+ \cup c'_+) \ssm \{s\}$ and~$c''_- \subseteq (c_- \cup c'_-) \ssm \{s\}$.
\end{itemize}
\end{definition}

\begin{example}
\label{exm:realizableOM}
For a vector configuration~$\b{A} \eqdef (\b{a}_s)_{s \in \ground} \in (\R^d)^\ground$, the oriented matroid~$\OM(\b{A})$ of \cref{def:orientedMatroidVectorConfiguration} is an oriented matroid in the sense of \cref{def:OM}.
\end{example}

\begin{definition}
\label{def:realizableOM}
An oriented matroid~$\OM$ is \defn{realizable} if it is the oriented matroid~$\OM(\b{A})$ of a vector configuration~$\b{A} \eqdef (\b{a}_s)_{s \in \ground} \in (\R^d)^\ground$, \ie~$\vectors[\OM] = \vectors$, $\covectors[\OM] = \covectors$,
$\circuits[\OM] = \circuits$, and $\cocircuits[\OM] = \cocircuits$ (these four conditions are actually equivalent).
\end{definition}

\begin{remark}
\label{rem:matroids}
Note that the supports of the vectors (resp.~covectors, resp.~circuits, resp.~cocircuits) of an oriented matroid form the dependences (resp.~codependences, resp.~circuits, resp.~cocircuits) of a classical (unoriented) matroid~\cite{Oxley}.
In other words, any oriented matroid~$\OM$ has an underlying classical matroid~$\UOM$.
However, not all classical matroids can be oriented to an oriented matroid.
We have decided to focus here on oriented matroids, even if some of the notions we will cover are actually classical matroid notions (in particular connected components in \cref{subsec:connectedComponents}).
\end{remark}

\begin{definition}
\label{def:acyclicOM}
An oriented matroid~$\OM$ on~$\ground$ is \defn{acyclic} if the following equivalent conditions (see~\cite[Prop.~3.4.8]{BjornerLasVergnasSturmfelsWhiteZiegler}) are satisfied:
\begin{enumerate}[(i)]
\item it has no positive circuit,
\item it has no positive vector,
\item $(\ground, \varnothing)$ is a covector,
\item for any~$s \in \ground$, there is a cocircuit~$c^* = (c^*_+, \varnothing)$ with~$s \in c^*_+$.
\end{enumerate}
\end{definition}

\begin{example}
\label{exm:acyclicRealizableOM}
Following on \cref{exm:realizableOM}, a realizable oriented matroid~$\OM(\b{A})$ is acyclic if and only if~$\b{A}$ has no positive linear dependence, \ie if and only if~$\b{A}$ is contained in a positive linear half-space of~$\R^d$.
For instance, the oriented matroid~$\OM(\b{A}_\circ)$ of the vector configuration~$\b{A}_\circ$ of \cref{exm:vectorConfiguration} is acyclic.
\end{example}

\begin{example}
\label{exm:acyclicGraphicalOM}
Following on \cref{exm:graphicalOM}, a graphical oriented matroid~$\OM(\digraph)$ is acyclic if and only if~$\digraph$ is acyclic (\ie has no directed cycle).
For instance, the graphical oriented matroid~$\OM(\digraph_\circ)$ of the directed graph~$\digraph_\circ$ of \cref{exm:graphicalVectorConfiguration} is acyclic.
\end{example}


\subsection{Las Vergnas face lattices}
\label{subsec:LasVergnasLattice}

The Las Vergnas face lattices of oriented matroids are the analogs of the face lattices of convex polyhedral cones.
Note that a face of a polyhedral cone is witnessed by a supporting linear evaluation that is zero on the vectors of the face and positive on the remaining vectors.
These induce non-negative covectors, \ie covectors of the form $(c^*_+,\varnothing)$.

\begin{definition}
\label{def:faceOM}
Let $\OM$ be an acyclic oriented matroid. A set $F\subseteq \ground$ is a \defn{face} of $\OM$ if it is the complement of a non-negative covector, \ie if~$(\ground\ssm F,\varnothing)\in\covectors[\OM]$.
\end{definition}

\begin{definition}
The \defn{Las Vergnas face lattice} $\FL(\OM)$ of an acyclic oriented matroid $\OM$ is the poset of faces of $\OM$ ordered by~inclusion.
It is always a lattice~\cite[Prop.~4.3.5]{BjornerLasVergnasSturmfelsWhiteZiegler}, with bottom element~$\varnothing$ and top element~$\ground$.
The \defn{face semilattice}~$\FL(\OM)_{<\topone}$ is the Las Vergnas face lattice without its top element
\end{definition}

\begin{example}
\label{exm:coneAcyclicRealizableOM}
Following on \cref{exm:realizableOM,exm:acyclicRealizableOM}, for an acyclic vector configuration $\b{A}$, the Las Vergnas face lattice~$\FL(\OM(\b{A}))$ is isomorphic to the face lattice of the cone~$\R_{\geq 0}\b{A}$ generated by~$\b{A}$, hence to the face lattice of a convex polytope (obtained by any section of~$\R_{\geq 0}\b{A}$ by a suitable hyperplane~$\b{H}$)\footnote{In order to have this isomorphism, we use slightly different conventions for face lattices of polytopes and of polyhedral cones. For polytopes, we consider the empty set as a face of dimension~$-1$, which is the minimal element of the face lattice. In contrast, for polyhedral cones we consider the origin to be the minimal face of dimension~$0$. This way, we obtain an order-preserving bijection between $i$-dimensional faces of the polytope $\b{H}\cap\R_{\geq 0}\b{A}$ and $(i+1)$-faces of the polyhedral cone $\R_{\geq 0}\b{A}$.}.
Conversely, the face lattice~$\FL(\polytope)$ of any polytope~$\polytope$ is isomorphic to the Las Vergnas face lattice~$\FL(\OM(\b{A}_{\polytope}))$ where~$\b{A}_{\polytope}$ is the vector configuration obtained by homogeneization of the vertices of~$\polytope$, \ie
\[
\b{A}_{\polytope} \eqdef \set{\begin{bmatrix} \b{v} \\ 1 \end{bmatrix}}{\b{v} \text{ vertex of } \polytope}.
\]
\end{example}

\begin{example}
\label{exm:orderPolytope}
Following on \cref{exm:graphicalOM,exm:acyclicGraphicalOM}, for a directed acyclic graph~$\digraph$ with vertex set~$V$, the Las Vergnas face lattice~$\FL(\OM(\digraph))$ of the graphical oriented matroid~$\OM(\digraph)$ is isomorphic to the refinement lattice on partitions $\pi \eqdef \{\pi_1, \dots, \pi_k\}$ of~$V$ into  connected subsets of~$\digraph$ such that the contraction~$\digraph_{\!/\pi} \eqdef \digraph_{\!/\pi_1/\pi_2/ \dots /\pi_k}$ is acyclic.
It is also isomorphic to the face lattice of the order polytope of~$\digraph$, defined as
\[
\OrderPolytope(\digraph) \eqdef \bigset{\b{x} \in \R^V}{\sum_{(u,v) \in \digraph} \b{x}_v - \b{x}_u = 1 \text{ and } x_u \le x_v \text{ for each arc } (u,v) \in \digraph }.
\]
See~\cite[Thm.~1.2]{Stanley-posetPolytopes} and \cite[Sect.~2.1]{Galashin} (we use the version of the latter as we work with directed graphs which do not necessarily admit a global source and a global sink).
\end{example}

\begin{remark}
 There are other natural ways to define the face lattice of an oriented matroid, see \cite[Chap.~4]{BjornerLasVergnasSturmfelsWhiteZiegler}. For the constructions in the sequel, it is very important not to mistake the Las Vergnas face lattice with its opposite, called the Edmonds--Mantel face lattice.
As we will only use this notion of face lattice of oriented matroids, for brevity from now on we will simply refer to the face lattice of an oriented matroid, meaning its Las Vergnas face lattice.
\end{remark}

The face lattice of an oriented matroid can be interpreted topologically as the lattice of cells of a regular cell decomposition of a sphere. We refer to \cite[Sect.~4.7]{BjornerLasVergnasSturmfelsWhiteZiegler} for background.

\begin{theorem}[{\cite[Thm.~4.3.5]{BjornerLasVergnasSturmfelsWhiteZiegler}}]
\label{thm:faceLatticeOrientedMatroidSphere}
The face lattice~$\FL(\OM)$ of an acyclic oriented matroid~$\OM$ of rank~$r$ is isomorphic to the face lattice of a regular cell decomposition~$\Delta(\OM)$ of the $(r-2)$-sphere.
\end{theorem}


\pagebreak
\subsection{Restriction and contraction}
\label{subsec:restrictionContractionOM}

We now recall the definitions of restrictions and contractions in oriented matroids.

\begin{definition}
\label{def:restrictionContractionOM}
For any~$R \subseteq \ground$, define
\begin{itemize}
\item the \defn{restriction} of~$\OM$ to~$R$ as the oriented matroid~$\OM_{|R}$ on~$R$ with
\[
\vectors[\OM_{|R}] = \set{v \in \vectors[\OM]}{\underline{v} \subseteq R}
\qquad\text{and}\qquad
\covectors[\OM_{|R}] = \set{v^* \cap R}{v^* \in \covectors[\OM]},
\]
\item the \defn{contraction} of~$R$ in~$\OM$ as the oriented matroid~$\OM_{\!/R}$ on~$\ground \ssm R$ with
\[
\vectors[\OM_{\!/R}] = \set{v \ssm R}{v \in \vectors[\OM]}
\qquad\text{and}\qquad
\covectors[\OM_{\!/R}] = \set{v^* \in \covectors[\OM]}{\underline{v}^* \cap R = \varnothing},
\]
\end{itemize}
where~$v \cap R \eqdef (v_+ \cap R, v_- \cap R)$ and~$v \ssm R \eqdef (v_+ \ssm R, v_- \ssm R)$.
Note that acyclicity is preserved by restriction but not by contraction in general (see \cref{lem:acyclicContraction}).
\end{definition}

\begin{remark}
\label{rem:restrictionContractionOM}
It immediately follows from \cref{def:restrictionContractionOM} that
\[
\circuits[\OM_{|R}] = \set{c \in \circuits[\OM]}{\underline{c} \subseteq R}
\qquad\text{and}\qquad
\cocircuits[\OM_{\!/R}] = \set{c^* \in \cocircuits[\OM]}{\underline{c}^* \cap R = \varnothing}.
\]
In contrast, we can only describe $\cocircuits[\OM_{|R}]$ (resp.~$\circuits[\OM_{\!/R}]$) as support minimal elements of~$\covectors[\OM_{|R}]$ (resp.~$\vectors[\OM_{\!/R}]$).
\end{remark}

\begin{example}
\label{exm:resctrictionContractionRealizableOM}
Following on \cref{exm:realizableOM,exm:acyclicRealizableOM,exm:coneAcyclicRealizableOM}, consider a vector configuration \linebreak ${\b{A} \eqdef (\b{a}_s)_{s \in \ground}}$ and~$R \subseteq \ground$. Denote by
\begin{itemize}
\item $\b{A}_{|R}$ the vector subconfiguration~$(\b{a}_r)_{r \in R}$,
\item $\b{A}_{\!/R}$ the vector configuration obtained by projecting the vectors~$\b{a}_s$ with~$s \notin R$ on the space orthogonal to all vectors~$\b{a}_r$ with~$r \in R$.
\end{itemize}
Then~$\OM(\b{A})_{|R} = \OM(\b{A}_{|R})$ and~$\OM(\b{A})_{\!/R} = \OM(\b{A}_{\!/R})$.
\end{example}

\begin{example}
\label{exm:resctrictionContractionGraphicalOM}
Following on \cref{exm:graphicalOM,exm:acyclicGraphicalOM,exm:orderPolytope}, consider a directed graph~$\digraph$ and a subset~$R$ of arcs of~$\digraph$. Denote by
\begin{itemize}
\item $\digraph_{|R}$ the subgraph of~$\digraph$ formed by the arcs in~$R$,
\item $\digraph_{\!/R}$ the contraction of the arcs of~$R$ in~$\digraph$.
\end{itemize}
Then~$\OM(\digraph)_{|R} = \OM(\digraph_{|R})$ and~$\OM(\digraph)_{\!/R} = \OM(\digraph_{\!/R})$.
\end{example}

\begin{example}
For instance, consider the oriented matroid~$\OM(\b{A}_\circ)$ of \cref{exm:vectorConfiguration}.
For~$R \eqdef 123$, the restriction~$\OM(\b{A}_\circ)_{|R}$ and the contraction~$\OM(\b{A}_\circ)_{\!/R}$ are the oriented matroids of the vector configurations
\[
{\b{A}_\circ}_{|R} = \left\{ 
	\begin{bmatrix} 0 \\ 0 \\ 0 \\ 1 \end{bmatrix},
	\begin{bmatrix} 0 \\ 0 \\ 0 \\ 1 \end{bmatrix},
	\begin{bmatrix} 0 \\ 0 \\ 1 \\ 1 \end{bmatrix}
\right\}
\qquad\text{and}\qquad
{\b{A}_\circ}_{\!/R} =
\left\{ 
	\begin{bmatrix} 1 \\ 0 \\ 0 \\ 0 \end{bmatrix},
	\begin{bmatrix} 0 \\ 1 \\ 0 \\ 0 \end{bmatrix},
	\begin{bmatrix} 1 \\ 1 \\ 0 \\ 0 \end{bmatrix}
\right\}
\]
respectively.
They indeed coincide with the graphical oriented matroids of the restriction to and of the contraction of the arcs labeled~$1,2,3$ in the directed graph of \cref{exm:graphicalVectorConfiguration,fig:exmDigraph}.
\end{example}

The following lemma explains the connection between acyclic contractions, non-negative covectors and faces of an acyclic oriented matroid.

\begin{lemma}
\label{lem:acyclicContraction}
Let $\OM$ be an acyclic oriented matroid on the ground set~$\ground$ and~$R \subseteq \ground$.
Then
\[
\text{$\OM_{\!/R}$~is acyclic} \iff (\ground \ssm R, \varnothing) \in \covectors[\OM_{\!/R}] \iff (\ground \ssm R, \varnothing) \in \covectors[\OM] \iff R \in \FL(\OM).
\]
\end{lemma}

\begin{proof}
The first equivalence is \cref{def:acyclicOM}, the last one is \cref{def:faceOM}.
For the middle one, observe that~$(\ground \ssm R) \cap R = \varnothing$.
Hence, \cref{rem:restrictionContractionOM} ensures that~$\ground \ssm R$ is a covector of~$\OM_{\!/R}$ if and only if it is a covector of~$\OM$.
\end{proof}

Finally, we describe the face lattices of restrictions and contractions of oriented matroids.

\begin{proposition}
\label{prop:restrictionContractionFL}
The face lattices of restrictions and contractions on faces are intervals of the face lattice.
Namely, for an acyclic oriented matroid~$\OM$ and~$X \in \FL(\OM)$, we have
\[
\FL(\OM_{|X}) = \FL(\OM)_{\le X}
\qquad\text{and}\qquad
\FL(\OM_{\!/X}) \cong \FL(\OM)_{\ge X},
\]
where the last isomophism is given by~$F \mapsto F \cup X$.
\end{proposition}

\begin{proof}
We first prove that~$\FL(\OM_{|X}) = \FL(\OM)_{\le X}$.
For~$F \subseteq X$, we have directly
\[
F \in \FL(\OM) \iff (\ground \ssm F, \varnothing) \in  \covectors[\OM] \Longrightarrow (X \ssm F, \varnothing) \in \covectors[\OM_{|X}] \iff F \in \FL(\OM_{|X})
\]
For the converse, if~$(X \ssm F, \varnothing) \in \covectors[\OM_{|X}]$, then there are~$R_+, R_- \subseteq \ground \ssm X$ such \linebreak that~$((X \ssm F) \cup R_+, R_-) \in \covectors[\OM]$.
As~$X \in \FL(\OM)$, we also have~$(\ground \ssm X, \varnothing) \in \covectors[\OM]$.
Hence,~$(\ground \ssm F, \varnothing) = (\ground \ssm X, \varnothing) \circ ((X \ssm F) \cup R_+, R_-) \in \covectors[\OM]$, where~$\circ$ is the composition of signed sets (see~\cite[Sect.~3.1]{BjornerLasVergnasSturmfelsWhiteZiegler}).

We now prove that~$\FL(\OM_{\!/X}) \cong \FL(\OM)_{\ge X}$.
For~$F \subseteq \ground \ssm X$, we have
\begin{align*}
F\cup X \in \FL(\OM) 
& \iff (\ground \ssm (F \cup X), \varnothing) \in  \covectors[\OM] \\
& \iff ((\ground \ssm X) \ssm F, \varnothing) \in \covectors[\OM_{\!/X}] \\
& \iff F \in \FL(\OM_{\!/X}).
\qedhere
\end{align*}
\end{proof}


\subsection{Stellar subdivisions}
\label{subsec:stellarSubdivisions}

In this section, we describe the operation of stellar subdivision for oriented matroids. They are the oriented matroid incarnation of the combinatorial blow-ups from \cref{subsec:combinatorialblowup}.
Even though stellar subdivisions are usually presented topologically at the level of cell complexes, when performed on the face lattice of an oriented matroid they can always be realized by oriented matroids (even though these realizations are not unique). This is why we decided to simplify the exposition and present them directly at the level of oriented matroids. 
In particular, they can be performed via lexicographic extensions, see \cite[Sect.~7.2]{BjornerLasVergnasSturmfelsWhiteZiegler}. Lexicographic extensions preserve realizability, and hence stellar subdivisions of realizable oriented matroids are realizable. In fact, if~$\polytope$ is a convex polytope and~$\OM$ is the oriented matroid of its vertices, then the stellar subdivisions of~$\polytope$ in the sense of \cite[Sect.~14.1.2]{HenkRichterGebertZiegler1997} realize the stellar subdivisions of~$\FL(\OM)$, see \cref{exm:stellarSubdivisionTruncation}.

\begin{proposition}[{\cite[Prop.~9.2.3 \& Sect.~7.2]{BjornerLasVergnasSturmfelsWhiteZiegler}}]
\label{prop:stellarSubdivisionOrientedMatroid}

Let $\OM$ be an acyclic oriented matroid with ground set~$\ground$, and $F\neq \ground$ be one of its proper faces.
Then there is an oriented matroid on~$\ground \cup \{\asd[F]\}$, that we call the \defn{stellar subdivision} of $\OM$ at $F$ and denote by \defn{$\sd(F,\OM)$}, whose face semilattice~$\FL(\sd(F,\OM))_{<\topone}$ is isomorphic to the combinatorial blow-up at~$F$ of the face semilattice~$\FL(\OM)_{<\topone}$ of~$\OM$. Here, the isomorphism is the identity on the common faces and given by~$(F,X)\mapsto \asd[F]\cup X$ for the other.
	
Moreover, $\sd(F, \FL(\OM))$ can be chosen to be realizable when~$\OM$ is realizable.
\end{proposition}

\begin{remark}
The definition above contains some abuse of notation, as the oriented \linebreak matroid~$\sd(F,\OM)$ is not uniquely defined (only its face lattice is). Therefore, we will only use this operation when we are considering face lattices.
\end{remark}

\begin{remark}
Alternatively we can describe the \defn{stellar subdivision} $\sd(F, \OM)$ as the operation on $\FL(\OM)$ that replaces the (open) star of $F$ by the cone over its boundary with apex $\asd[F]$:
\[
\FL(\sd(F, \OM))_{<\topone} = \FL(\OM)_{<\topone} \ssm(\ostar(F,\FL)) \cup \{\asd[F]\} \times \left(\cstar(F,\FL) \ssm \ostar(F,\FL) \right),
\]
where~$\ostar(F,\FL) \eqdef \FL_{F \le \cdot <\topone}$ is the \defn{(open) star} of $F$ and $\cstar(F,\FL) \eqdef \bigcup_{G\in \ostar(F,\FL)} \FL_{\le G}$ is the \defn{closed star} of~$F$.
\end{remark}

\begin{example}
\label{exm:stellarSubdivisionTruncation}
Following on \cref{exm:realizableOM,exm:acyclicRealizableOM,exm:coneAcyclicRealizableOM,exm:resctrictionContractionRealizableOM},  if~$\b{A} \eqdef (\b{a}_s)_{s \in \ground} \in (\R^d)^{\ground}$ is a vector configuration, and $F\subset \ground$ is a proper face of $\FL(\OM(\b A))$, then $\sd(F, \OM(\b A))$ is realized by the vector configuration $\b{A}\cup \b{a}_F$, where 
\[
\b{a}_F\eqdef \sum_{s\in F}\b a_s -\varepsilon \sum_{s\in \ground}\b a_s
\]
for some $\varepsilon>0$ sufficiently small.

If~$\polytope$ is a polytope with vertex set $(\b{v}_s)_{s \in \ground}$, then the stellar subdivision at $F$ can be realized by $\conv(\polytope\cup \b v_F)$, where 
\[
\b{v}_F\eqdef \frac{1+\varepsilon}{|F|}\sum_{s\in F}\b v_s -\frac{\varepsilon}{|\ground \ssm F|} \sum_{s\in \ground \ssm F}\b v_s
\]
for some $\varepsilon>0$ sufficiently small.
This operation is commonly referred to as \defn{stellar subdivision} in polyhedral combinatorics \cite[Sect.~14.1.2]{HenkRichterGebertZiegler1997}.
The given choice of $\b v_F$ is not canonical; any other point beyond the facets containing~$F$ and beneath the remaining facets would also work.
See \cref{fig:stellarSubdivisionsFaceTruncations}\,(left).

Stellar subdivision of polytopes is the polar operation to truncation. If $\polytope$ contains the origin in its interior, and $\polytope^\triangle$ is its polar, then the stellar subdivision of $\polytope$ at $F$ is polar to the truncation of~$F^\diamond$ on $\polytope^\triangle$, where~$F^\diamond$ is the face of $\polytope^\triangle$  associated to~$F$ (see~\cite[Sect.~2.3]{Ziegler}), and the \defn{truncation} of~$F^\diamond$ consists in intersecting~$\polytope^\triangle$ with a halfspace that does not intersect~$F^\diamond$ but contains all the vertices of~$\polytope^\triangle$ not in~$F^\diamond$. This can be achieved by considering a supporting hyperplane for~$F^\diamond$ and pushing it inside by slightly perturbing the right-hand side.
See \cref{fig:stellarSubdivisionsFaceTruncations}\,(right).

\begin{figure}
	\capstart
	\centerline{\includegraphics[scale=.5]{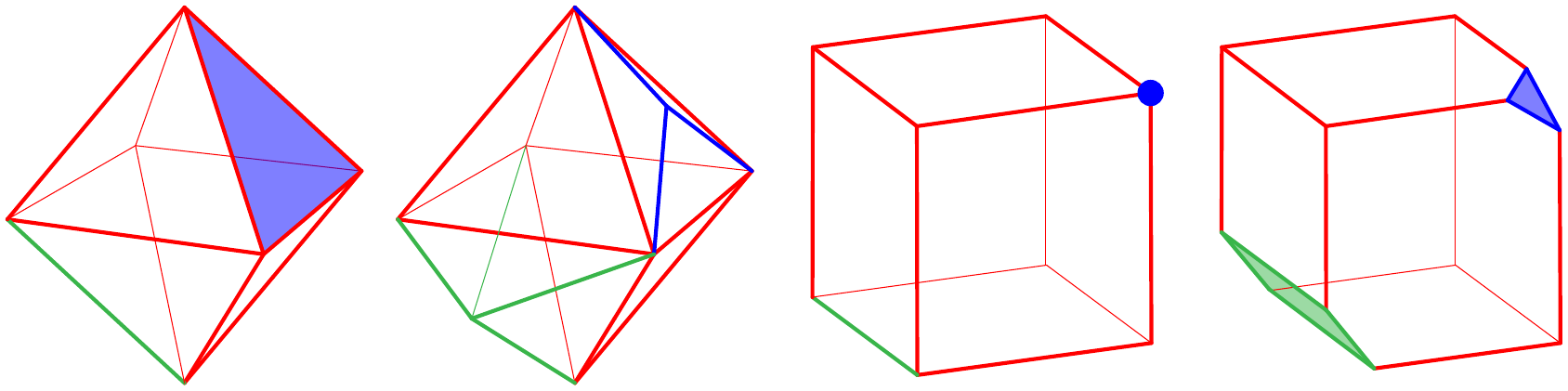}}
	\caption{Duality between two stellar subdivisions on an octahedron (left) and two face truncations on a cube (right).}
	\label{fig:stellarSubdivisionsFaceTruncations}
\end{figure}
\end{example}


\subsection{Operations}
\label{subsec:operationsOM}

We conclude by two operations on oriented matroids that will realize the Cartesian and free products of their face lattices, defined in \cref{subsec:operationsNestedComplex}.

\begin{definition}[{\cite[Prop.~7.6.1]{BjornerLasVergnasSturmfelsWhiteZiegler}}]
\label{def:directSumOM}
The \defn{direct sum} of two oriented matroids~$\OM$ and~$\OM'$ on disjoint ground sets~$\ground$ and~$\ground'$ is the oriented matroid $\OM \oplus \OM'$ on $\ground \cup \ground'$ with 
\begin{itemize}
\item vectors~$\vectors[\OM \oplus \OM']=\set{(v_+ \cup v'_+, v_- \cup v'_-)}{v \in \vectors[\OM], v' \in \vectors[\OM]}$,
\item covectors~$\covectors[\OM \oplus \OM']=\set{(v_+^*\cup v'^*_+, v^*_- \cup v'^*_-)}{v^* \in \covectors[\OM], v'^* \in \covectors[\OM']}$,
\item circuits~$\circuits[\OM \oplus \OM'] = \circuits[\OM] \cup \circuits[\OM']$,
\item cocircuits~$\cocircuits[\OM \oplus \OM'] = \cocircuits[\OM] \cup \cocircuits[\OM']$.
\end{itemize}
\end{definition}

\begin{example}
\label{exm:directSumRealizableOM}
Following on \cref{exm:realizableOM,exm:acyclicRealizableOM,exm:coneAcyclicRealizableOM,exm:resctrictionContractionRealizableOM,exm:stellarSubdivisionTruncation}, for two vector configurations~$\b{A} \in (\R^d)^{\ground}$ and~$\b{A}' \in (\R^{d'})^{\ground'}$, we have
\[
\OM(\b{A}) \oplus \OM(\b{A}') = \OM(\b{A} \oplus \b{A}'),
\]
where~$\b{A} \oplus \b{A}' \eqdef \set{(\b{a}_{s}, \zero_{d'})}{s \in \ground} \cup \set{(\zero_{d}, \b{a}'_{s'})}{s' \in \ground'}$.
If~$\polytope, \polytope'$ are two polytopes, \linebreak then ${\b{A}_{\polytope} \oplus \b{A}_{\polytope'} \simeq \b{A}_{\polytope \join \polytope'}}$, where~$\polytope \join \polytope'$ is known as the join of~$\polytope$ and~$\polytope'$ (see \cite[Sect.~14.1.3]{HenkRichterGebertZiegler1997}) and~$\simeq$ is a linear isomorphism.
\end{example}

\begin{example}
\label{exm:directSumGraphicalOM}
Following on \cref{exm:graphicalOM,exm:acyclicGraphicalOM,exm:orderPolytope,exm:resctrictionContractionGraphicalOM}, for two directed graphs~$\digraph$ and~$\digraph'$, we have~$\OM(\digraph) \oplus \OM(\digraph') = \OM(\digraph \oplus \digraph')$, where~$\digraph \oplus \digraph'$ is the disjoint union of~$\digraph$ and~$\digraph'$.
\end{example}

\begin{definition}
\label{def:freeSumOM}
A \defn{free sum} of two acyclic oriented matroids~$\OM$ and~$\OM'$ on disjoint ground sets~$\ground$ and~$\ground'$ is an oriented matroid~$\OM \boxplus \OM'$ on $\ground \cup \ground'$ constructed as~$\sd(\ground, \OM \oplus \OM')_{\!/\asd[\ground]}$.  
\end{definition}

Since the stellar subdivision is not uniquely defined, neither is the free sum~$\OM \boxplus \OM'$.
However, we will see that its face lattice~$\FL(\OM \boxplus \OM')$ is uniquely defined.

\begin{example}
\label{exm:freeSumRealizableOM}
Following on \cref{exm:realizableOM,exm:acyclicRealizableOM,exm:coneAcyclicRealizableOM,exm:resctrictionContractionRealizableOM,exm:stellarSubdivisionTruncation,exm:directSumRealizableOM}, for two vector configurations~$\b{A} \in (\R^d)^{\ground}$ and~$\b{A}' \in (\R^{d'})^{\ground'}$, we have
\[
\OM(\b{A}) \boxplus \OM(\b{A}') = \OM(\b{A} \boxplus \b{A}'),
\]
where~$\b{A} \boxplus \b{A}'$ is obtained by placing~$\b{A}$ in~$\R^d \times \{\zero_{d'-1}\}$ and~$\b{A}'$ in~$\{\zero_{d-1}\} \times \R^{d'}$ such that the half-line~$\zero^{d-1} \times \R_{>0} \times \zero^{d'-1}$ belongs to the interior of the two generated cones.
If~$\polytope, \polytope'$ are two polytopes, then $\b{A}_{\polytope} \oplus \b{A}_{\polytope'}$ is combinatorially equivalent to~$\b{A}_{\polytope \oplus \polytope'}$, where~$\polytope \oplus \polytope'$ is known as the direct sum of~$\polytope$ and~$\polytope'$ (see \cite[Sect.~14.1.3]{HenkRichterGebertZiegler1997}).
\end{example}

Finally, we describe the face lattices of direct sums and free sums of oriented matroids.

\begin{proposition}
\label{prop:directSumCartesianProductFL}
The face lattice of the direct sum is the Cartesian product of face lattices.
Namely, if~$\OM$ and~$\OM'$ are acyclic oriented matroids on disjoint ground sets $\ground$ and $\ground'$, then
\[
\FL(\OM \oplus \OM') \cong \FL(\OM) \times \FL(\OM'),
\]
where the isomorphism is given by $F \mapsto (F \cap \ground, F \cap \ground')$.
Conversely, for an acyclic oriented matroid $\OM$ with ground set $\ground \sqcup \ground'$, if~$\FL(\OM) \cong \FL(\OM_{|\ground}) \times \FL(\OM_{|\ground'})$, then $\OM = \OM_{|\ground} \oplus \OM_{|\ground'}$.
\end{proposition}

\begin{proof}
The first statement follows directly from the description of the covectors of $\OM \oplus \OM'$.

The converse statement follows from the following stronger result. If $\OM$ is an oriented matroid on~$\ground \sqcup \ground'$, and $\rank[\OM] = \rank[\OM_{|\ground}]+\rank[\OM_{|\ground'}]$, then $\OM = \OM_{|\ground} \oplus \OM_{|\ground'}$. Indeed, every circuit of~${\OM_{|\ground} \oplus \OM_{|\ground'}}$ is automatically a circuit of~$\OM$. This means that there is a strong map from~${\OM_{|\ground} \oplus \OM_{|\ground'}}$ to $\OM$, see \cite[Prop.~7.7.1]{BjornerLasVergnasSturmfelsWhiteZiegler}. If both have the same rank, then they must coincide by \cite[Exm.~7.25]{BjornerLasVergnasSturmfelsWhiteZiegler}.
\end{proof}

Note that there is a nuance in the formulation of the converse statement in \cref{prop:directSumCartesianProductFL}. It is not enough that~$\FL(\OM) \cong \FL(\OM_{|X}) \times \FL(\OM_{|Y})$ to conclude that~$\OM = \OM_{|X} \oplus \OM_{|Y}$, because there could be additional elements not visible in the face lattice and have $X\cup Y\neq \ground$.

\begin{proposition}
\label{prop:freeSumFreeProductFL}
The face lattice of the free sum is the free product of face lattices.
Namely, if~$\OM$ and~$\OM'$ are acyclic oriented matroids on disjoint ground sets $\ground$ and $\ground'$, then
\[
\FL(\OM \boxplus \OM') \cong \FL(\OM) \boxtimes \FL(\OM'),
\]
where the isomorphism is given by $F \mapsto (F \cap \ground, F \cap \ground)$.
\end{proposition}

\begin{proof}
By \cref{prop:restrictionContractionFL}, the face lattice of $\sd(\ground, \OM \oplus \OM')_{\!/\asd[\ground]}$ is the link of $\asd[\ground]$ in $\sd(\ground, \OM \oplus \OM')$, that is, the set of faces of $\sd(\ground,\OM \oplus \OM')$ containing~$\ground$.
By construction of the stellar subdivision, these are in bijection with the faces~$G$ of $\OM \oplus \OM$ that do not contain~$\ground$ but that are contained in a facet of $\OM \oplus \OM'$ containing~$\ground$ (plus the top element).
Faces of $\OM \oplus \OM'$ are of the form $F \sqcup F'$ with $F \in \FL(\OM)$ and~$F' \in \FL(\OM')$.
Note first that $F \sqcup F'$ is a facet containing $\ground$ if and only if $F' \neq \ground'$ (otherwise it would not be a facet).
And its faces $G \sqcup G' \subseteq F \sqcup F'$ not containing~$\ground$ give precisely the description of~$\FL(\OM) \boxtimes \FL(\OM')$ (except for the top element).
\end{proof}


\subsection{Connected components}
\label{subsec:connectedComponents}

We conclude by the following notion of connected components which will be useful along the paper.
Note that the connected components of an oriented matroid actually only depend on the underlying unoriented matroid~$\UOM$.
We have however decided to focus our presentation on oriented matroids, and to skip classical matroids (see \cref{rem:matroids}).

\begin{definition}
\label{def:connectedComponentsOM}
The \defn{connected components} of an oriented matroid~$\OM$ are equivalence classes of the equivalence relation on the ground set~$\ground$ given by $s \sim t$ if either $s = t$ or there is $c \in \circuits[\OM]$ with $s, t \in \underline{c}$.
\end{definition}

\begin{example}
\label{exm:connectedComponentsGraphicalOM}
Following on \cref{exm:graphicalOM,exm:acyclicGraphicalOM,exm:orderPolytope,exm:resctrictionContractionGraphicalOM,exm:directSumGraphicalOM}, the connected components of a graphical oriented matroid~$\OM(\digraph)$ are the edge sets of the $2$-connected components of the unoriented graph underlying the directed graph~$\digraph$.
Note that, in contrast, the (oriented) matroids do not capture $1$-connectivity (the graphical matroid of a tree on $n$ vertices coincides with that of~$n-1$ disjoint edges).
See \cite[Chap.~4]{Oxley}.
\end{example}

\cref{def:connectedComponentsOM} directly gives the following decomposition, which is well-known at the level of unoriented matroids~\cite[Prop.~4.2.8]{Oxley}, and carries over straightforwardly to oriented matroids.

\begin{proposition}
\label{prop:connectedComponentsDecomposition}
Let $\OM$ be an oriented matroid on $\ground$. If $\ground_1,\dots,\ground_k$ are the connected components of~$\OM$, then $\OM=\OM_{|\ground_1}\oplus\cdots\oplus\OM_{|\ground_k}$.
\end{proposition}

Combining this proposition with \cref{prop:directSumCartesianProductFL} we directly get the following result.

\begin{corollary}
\label{cor:connectedComponentsOfFacesAreFaces}
If $\OM$ is acyclic, then its connected components are faces.
\end{corollary}


\clearpage
\part{Facial nested complexes and acyclic nested complexes}
\label{part:facialAcyclicNestedComplexes}

In this part, we consider two families of simplicial complexes associated to an oriented matroid~$\OM$: the facial nested complexes of facial building sets (\cref{sec:facialBuildingSetsFacialNestedComplexes}) and the acyclic nested complexes of oriented building sets (\cref{sec:orientedBuildingSetsAcyclicNestedComplexes}).
We then prove that these two families essentially coincide (\cref{sec:facialNestedComplexesVsAcyclicNestedComplexes}).
We provide these two perspectives as they naturally generalize different former constructions, and are both useful for geometric realizations.
Namely, the facial nested complexes are adapted to the realizations by stellar subdivisions (\cref{sec:orientedMatroidRealizations}), while the acyclic nested complexes are useful for our realizations by sections of nestohedra~(\cref{sec:polytopalRealizations}).


\section{Facial building sets and facial nested complexes}
\label{sec:facialBuildingSetsFacialNestedComplexes}

We start with building sets and nested complexes over the Las Vergnas face lattice of an oriented matroid~$\OM$.
We define facial building sets (\cref{subsec:facialBuildingSets}) and facial nested complexes (\cref{subsec:facialNestedComplexes}), discuss their restrictions and contractions to describe their links (\cref{subsec:restrictionContractionFacialNestedComplex}), and describe two operations to restrict to connected facial building sets (\cref{subsec:operationsFacialNestedComplex}).


\subsection{Facial building sets}
\label{subsec:facialBuildingSets}

We first give a name to the building sets over Las Vergnas face lattices.

\begin{definition}
\label{def:facialBuildingSet}
A \defn{facial building set} is a pair~$(\fbuilding,\OM)$, where~$\OM$ is an oriented matroid on a ground set~$\ground$, and~$\fbuilding$ is a $\FL(\OM)$-building set over the Las Vergnas face lattice~$\FL(\OM)$.
We also often say that~$\fbuilding$ is a facial building set for~$\OM$.
We say that~$(\fbuilding, \OM)$ is \defn{realizable} if~$\OM$ is realizable.
\end{definition}

In the remaining of this section, we just want to explore the facial building sets of some natural examples of oriented matroids.

\begin{example}
\label{exm:facialBuildingSetSimplex}
Consider the independent oriented matroid~$\OM_\triangle$ on~$[n]$ (\ie with no circuit).
It is the oriented matroid of the vector configuration obtained by homogeneization of the vertices of an $(n-1)$-dimensional simplex.
Its Las Vergnas face lattice~$\FL(\OM_\triangle)$ is the boolean lattice.
Hence, its facial building sets are just the boolean building sets of \cref{def:booleanBuildingSet}.
\end{example}

\begin{example}
\label{exm:facialBuildingSetSimplicialPolytope}
For a simplicial polytope~$\polytope$, \cref{rem:reccursiveDefinitionLatticeBuildingSet} implies that the connected $\FL(\polytope)$-building sets are precisely the subsets of faces of~$\polytope$ containing~$\polytope$ and whose restriction to any facet of~$\polytope$ forms a boolean building set.
(Note that the connected restriction is empty when~$\polytope$ is irreducible.)
Up to polarity, this recovers the description of the $\polytope^\triangle$-building sets of J.~Almeter~\cite{Almeter} (see also \cite[Sect.~3]{Petric}).
\end{example}

\begin{example}
\label{exm:facialBuildingSetCrossPolytope1}
Let~$[\underline{n}] \sqcup [\overline{n}] \eqdef \set{\underline{i}}{i \in [n]} \sqcup \set{\overline{i}}{i \in [n]}$.
Consider the vectors~$\b{v}_{\underline{i}} \eqdef \b{e}_i+\b{e}_{n+1}$ and~$\b{v}_{\overline{i}} \eqdef -\b{e}_i+\b{e}_{n+1}$ for~$i \in [n]$ obtained by homogeneization of the vertices of an $n$-dimensional cross-polytope.
Consider the oriented matroid~$\OM_\Diamond$ on~$[\underline{n}] \sqcup [\overline{n}]$ defined by the vector configuration~$\set{\b{v}_{\underline{i}}}{i \in [n]} \sqcup \set{\b{v}_{\overline{i}}}{i \in [n]}$.
Its Las Vergnas face lattice~$\FL(\OM_\Diamond)$ is isomorphic to the face lattice of the cross-polytope.
Namely, the faces of~$\OM_\Diamond$ are the set~$[\underline{n}] \sqcup [\overline{n}]$ and all its subsets which do not contain~$\{\underline{i}, \overline{i}\}$ for each~$i \in [n]$.
We call \defn{hyperoctahedral building sets} the facial building sets for~$\OM_\Diamond$ (an analogous definition holds for any Dynkin type).
Since the cross-polytope is not a join, a direct application of \cref{exm:facialBuildingSetSimplicialPolytope} gives the following description. 
\end{example}

\begin{proposition}
\label{prop:facialBuildingSetCrossPolytope}
A collection~$\fbuilding$ of subsets of~$[\underline{n}] \sqcup [\overline{n}]$ is a hyperoctahedral building set if and only~if
\begin{itemize}
\item $\{\underline{i}, \overline{i}\} \not\subseteq F$ for all~$F \in \fbuilding$ distinct from~$[\underline{n}] \sqcup [\overline{n}]$ and all~$i \in [n]$,
\item $\fbuilding$ contains the set~$[\underline{n}] \sqcup [\overline{n}]$ and the singletons~$\{\underline{i}\}$ and~$\{\overline{i}\}$ for all~$i \in [n]$,
\item for any~$F, G \in \fbuilding$ such that~$\{\underline{i}, \overline{i}\} \not\subseteq F \cup G$ for all~$i \in [n]$, if~$F \cap G \ne \varnothing$ then~$F \cup G \in \fbuilding$.
\end{itemize}
\end{proposition}

\begin{example}
\label{exm:facialBuildingSetCrossPolytope2}
There are two natural simple ways to construct hyperoctahedral building sets:
\begin{enumerate}
\item For a graph~$\graphG$ on~$[\underline{n}] \sqcup [\overline{n}]$, the set of all non-empty subsets of~$[\underline{n}] \sqcup [\overline{n}]$ which do not contain~$\{\underline{i}, \overline{i}\}$ for each~$i \in [n]$ and induce a connected subgraph of~$\graphG$, together with the set~$[\underline{n}] \sqcup [\overline{n}]$ itself, form a \defn{graphical} hyperoctahedral building set.
\item Given two (boolean) building sets~$\underline{\building}$ on~$[\underline{n}]$ and~$\overline{\building}$ on~$[\overline{n}]$, the set~$\underline{\building} \sqcup \overline{\building} \sqcup \{[\underline{n}] \sqcup [\overline{n}]\}$ is a hyperoctahedral building set that we call \defn{bi-building set}.
\end{enumerate}
\end{example}

\begin{example}
\label{exm:designBuildingSet}
For instance, for a (boolean) building set~$\building$, the \defn{design building set} of~\cite{DevadossHeathVipismakul} is the bi-building set~$\building^\square \eqdef \building \sqcup \set{\{i^\square\}}{i \in [n]} \sqcup \{[n] \sqcup [n^\square]\}$ on~$[n] \sqcup [n^\square]$.
See also \cref{exm:designNestedComplex}.
\end{example}


\subsection{Facial nested complexes}
\label{subsec:facialNestedComplexes}

We now consider nested complexes on facial building sets.

\begin{definition}
\label{def:facialNestedComplex}
The \defn{facial nested complex}~$\nestedComplex[][\fbuilding,\OM]$ of a facial building set~$(\fbuilding,\OM)$ is the \mbox{$\FL(\OM)$-nes}\-ted complex~$\nestedComplex[\FL(\OM)][\building]$ of~$\fbuilding$.
\end{definition}

\begin{example}
\label{exm:barycentricSubdivision1}
Following \cref{exm:orderComplexBoolean,exm:orderComplex}, for the maximal building set~$\FL(\OM)_{>\botzero}$, the facial nested complex~$\nestedComplex[][\FL(\OM)_{>\botzero},\OM]$ is isomorphic to the \defn{barycentric subdivision} of the positive tope of~$\OM$.
By \cref{exm:orderComplex}, it has a face for each flag of faces of~$\OM$, in particular a vertex for each face of~$\OM$ and a facet for each complete flag of faces of~$\OM$.
\end{example}

\begin{example}
\label{exm:facialNestedComplexSimplex}
Consider the independent oriented matroid~$\OM_\triangle$ on~$[n]$ (\ie with no circuit).
According to~\cref{exm:facialBuildingSetSimplex}, its facial nested complexes are the boolean nested complexes of \cref{def:booleanNestedSet}.
\end{example}

\begin{example}
\label{exm:facialNestedComplexSimplicialPolytope}
Following \cref{exm:facialBuildingSetSimplicialPolytope}, for a simplicial polytope~$\polytope$ and a connected $\FL(\polytope)$-building set~$\fbuilding$, the $\FL(\polytope)$-nested complex~$\nestedComplex[\FL(\polytope)][\fbuilding]$ is the union, over all facets~$F$ of~$\polytope$, of the joins of~$\nestedComplex[][\fbuilding_{\le F}]$ with the simplex~$\simplex_{\connectedComponents(\fbuilding_{<F})}$ with vertices the connected components of~$\fbuilding_{<F}$.
Up to polarity, this recovers the description of the $\polytope^\triangle$-nested sets of J.~Almeter~\cite{Almeter} (see also \cite[Sect.~3]{Petric}).
\end{example}

\begin{example}
\label{exm:facialNestedComplexCrossPolytope1}
Consider the oriented matroid~$\OM_\Diamond$ of \cref{exm:facialBuildingSetCrossPolytope1}.
We call \defn{hyperoctahedral nested sets} the facial nested sets for~$\OM_\Diamond$ (an analogous definition holds for any Dynkin type).
A direct application of \cref{exm:facialNestedComplexSimplicialPolytope} gives the following description. 
\end{example}

\begin{proposition}
\label{prop:facialNestedComplexCrossPolytope}
A subset~$\nested$ of a hyperoctahedral building set~$\fbuilding$ containing~$[\underline{n}] \sqcup [\overline{n}]$  is a hyperoctahedral nested set if and only if
\begin{enumerate}
\item for any~$F,G \in \nested$, either~$F \subseteq G$ or~$G \subseteq F$ or~$F \cap G = \varnothing$,
\item for any~$k \ge 2$ pairwise disjoint~$F_1,\dots,F_k \in \nested$, the union~$F_1 \cup \dots \cup F_k$ is in~$\FL(\OM_\Diamond)$ but not in~$\fbuilding$.
\end{enumerate}
\end{proposition}

\begin{example}
\label{exm:facialNestedComplexCrossPolytope2}
As already observed in~\cite[Prop.~4.1]{Almeter}, the hyperoctahedral nested complex of the full hyperoctahedral building set is isomorphic to the boundary complex of the polar of the type~$B$ permutahedron. See \cref{fig:facialNestedComplexesOctahedron}\,(left).
\end{example}

\begin{example}
\label{exm:facialNestedComplexCrossPolytope3}
As already observed in~\cite[Prop.~4.1]{Almeter}, the hyperoctahedral nested complex of the bi-building set~$\underline{\building} \sqcup \overline{\building} \sqcup \{[\underline{n}] \sqcup [\overline{n}]\}$ where~$\underline{\building} = \overline{\building} = 2^{[n]} \ssm \{\varnothing\}$ (resp.~$\underline{\building} = \overline{\building} = \set{[i,j]}{1 \le i \le j \le n}$) is isomorphic to the boundary complex of the polar of the type~$A$ permutahedron (resp.~biassociahedron of~\cite{BarnardReading}).
\end{example}

\begin{example}
\label{exm:designNestedComplex}
For the design building set~$\building^\square$ of a boolean building set~$\building$ defined in \cref{exm:designBuildingSet}, the nested sets are of the form~$\nested \sqcup \nested^\square \sqcup \{[n] \sqcup [n^\square]\}$, where~$\nested$ is a $\building$-nested set, $\nested^\square$ is a subset of~$\set{\{i^\square\}}{i \in [n]}$, such that if~$\{i^\square\} \in \nested^\square$, then~$i \notin N$ for any~$N \in \nested$.
Hence, for a graphical building set~$\tubes$ of a graph~$\graphG$, the nested complex of the design building set~$\tubes^\square$ is the design nested complex of~\cite{DevadossHeathVipismakul}.
See \cref{fig:facialNestedComplexesOctahedron}\,(right) for some examples, and \cite[Sect.~5.4]{MannevillePilaud-compatibilityFans} for an extended discussion of the combinatorial isomorphisms between nested complexes and design nested complexes.
\begin{figure}
	\capstart
	\centerline{\includegraphics[scale=.5]{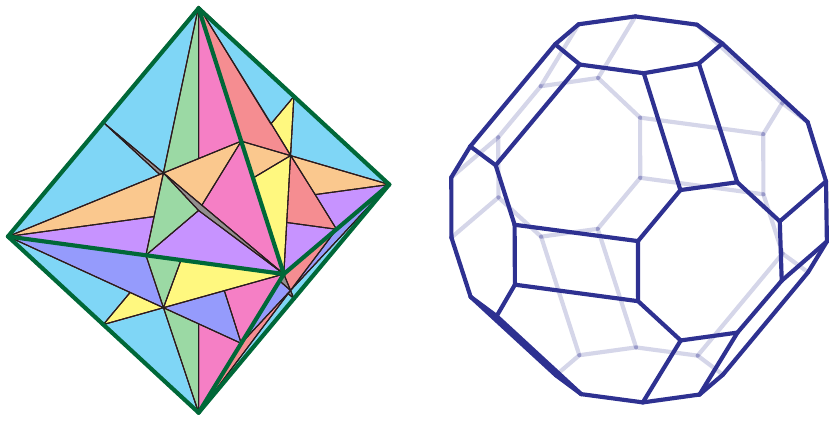}\qquad\includegraphics[scale=.7]{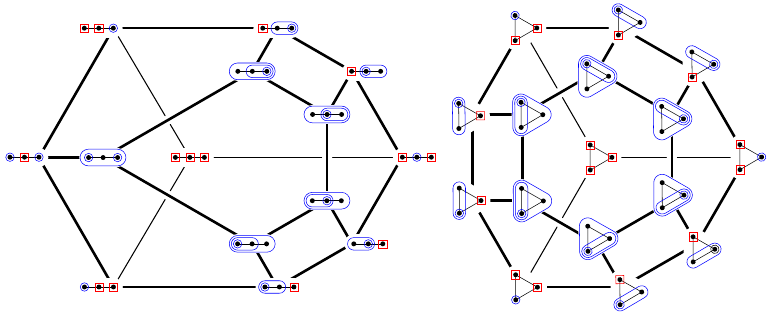}}
	\caption{The type~$B$ permutahedron (left) and some design graph associahedra (right).}
	\label{fig:facialNestedComplexesOctahedron}
\end{figure}
\end{example}


\subsection{Restriction, contraction, and links}
\label{subsec:restrictionContractionFacialNestedComplex}

We now briefly discuss the properties of facial building sets and facial nested sets with respect to restriction and contraction operations defined in \cref{def:restrictionContractionLatticeBuildingSet} for building sets and in \cref{def:restrictionContractionOM} for oriented matroids.
We do a slight abuse of notation in the building set contraction~$\fbuilding_{\!/X}$, identifying a face~$F$ containing~$X$ with~$F \ssm X$, in order to be compatible with the oriented matroid contraction~$\OM_{\!/X}$.
This makes sense since~$\FL(\OM)$ is atomic (see \cref{rem:contraction}).

\begin{proposition}
\label{prop:restrictionContractionFacialBuildingSet}
If~$(\fbuilding, \OM)$ is a facial building set and~$X \in \fbuilding$, then~$(\fbuilding_{|X}, \OM_{|X})$ and~$(\fbuilding_{\!/X}, \OM_{\!/X})$ are facial building sets.
\end{proposition}

\begin{proof}
Follows from \cref{prop:restrictionContractionLatticeBuildingSet,prop:restrictionContractionFL}.
\end{proof}

From \cref{prop:linksLatticeNestedComplex,prop:restrictionContractionFacialBuildingSet}, we obtain that the links in facial nested complexes are joins of facial nested complexes.

\begin{proposition}
\label{prop:singleLinkFacialBuildingSet}
If~$(\fbuilding, \OM)$ is a facial building set and~$X \in \fbuilding \ssm \connectedComponents(\fbuilding)$, then the link of~$X$ \linebreak in the facial nested complex~$\nestedComplex[][\fbuilding, \OM]$ is isomorphic to the join of facial nested \linebreak complexes~${\nestedComplex[][\fbuilding_{|X}, \OM_{|X}] \join \nestedComplex[][\fbuilding_{\!/X}, \OM_{\!/X}]}$.
\end{proposition}

\begin{proof}
Immediate consequence of \cref{prop:linksLatticeNestedComplex,prop:restrictionContractionFacialBuildingSet}.
\end{proof}

To describe links of arbitrary faces, we need the following definitions.

\begin{definition}
\label{def:restrictionContractionFacialBuildingSet}
For a facial building set~$(\fbuilding, \OM)$, a facial nested set~$\nested$ and~$X \in \nested$, we define
\[
\restrGround[X] \eqdef X \ssm Y,
\qquad
\restrBuilding[\fbuilding][X] \eqdef (\fbuilding_{|X})_{\!/Y}
\qquad\text{and}\qquad
\restrOM[X] \eqdef (\OM_{|X})_{\!/Y}
\]
where~$Y = \bigvee_{Z \in \nested, \; Z < X} Z$.
\end{definition}

\begin{proposition}
\label{prop:restrictionContractionFacialBuildingSet2}
For any~$X \in \nested \in \nestedComplex[][\fbuilding, \OM]$, the pair~$(\restrBuilding[\fbuilding][X], \restrOM[X])$ is a facial building set on~$\restrGround[X]$.
\end{proposition}

\begin{proof}
Immediate consequence of \cref{prop:restrictionContractionFacialBuildingSet}.
\end{proof}

\begin{proposition}
\label{prop:linksFacialNestedComplex}
For a facial nested set~$\nested$, the link of~$\nested \ssm \connectedComponents(\fbuilding)$ in the facial nested complex~$\nestedComplex[][\fbuilding, \OM]$ is the join of the facial nested complexes~$\nestedComplex[][{\restrBuilding[\fbuilding][X]},{ \restrOM[X]}]$ for all~$X \in \nested$.
\end{proposition}

\begin{proof}
Immediate consequence of \cref{prop:singleLinkFacialBuildingSet}.
\end{proof}


\subsection{Operations}
\label{subsec:operationsFacialNestedComplex}

We now combine \cref{subsec:operationsNestedComplex,subsec:operationsOM} to define two operations on facial building sets and facial nested complexes.

\begin{proposition}
\label{prop:directSumFacialBuildingSets}
If~$(\fbuilding, \OM)$ and~$(\fbuilding', \OM')$ are two facial building sets, then
\[
(\fbuilding, \OM) \oplus (\fbuilding', \OM') \eqdef (\fbuilding \oplus \fbuilding', \OM \oplus \OM')
\]
is a facial building set, whose facial nested complex~$\nestedComplex[][(\fbuilding, \OM) \oplus (\fbuilding', \OM')]$ is isomorphic to the join of facial nested complexes~$\nestedComplex[][\fbuilding, \OM] \join \nestedComplex[][\fbuilding', \OM']$.
\end{proposition}

\begin{proof}
Combine~\cref{prop:directSumLatticeBuildingSets,prop:directSumCartesianProductFL}.
\end{proof}

\begin{proposition}
\label{prop:freeSumFacialBuildingSets}
If~$(\fbuilding, \OM)$ and~$(\fbuilding', \OM')$ are two connected facial building sets, then
\[
(\fbuilding, \OM) \boxplus (\fbuilding', \OM') \eqdef (\fbuilding \boxplus \fbuilding', \OM \boxplus \OM')
\]
is a connected facial building set, whose facial nested complex~$\nestedComplex[][(\fbuilding, \OM) \boxplus (\fbuilding', \OM')]$ is isomorphic to the join of facial nested complexes~$\nestedComplex[][\fbuilding, \OM] \join \nestedComplex[][\fbuilding', \OM']$.
\end{proposition}

\begin{proof}
Combine~\cref{prop:freeSumLatticeBuildingSets,prop:freeSumFreeProductFL}.
\end{proof}

\begin{corollary}
\label{coro:connectedFacialBuildingSet}
For any facial building set~$(\fbuilding, \OM)$ with~$\connectedComponents(\fbuilding) = \{F_1, \dots, F_k\}$, the facial nested complex~$\nestedComplex[][\fbuilding, \OM]$ is isomorphic to the facial nested complex of the connected facial building set~${(\fbuilding_{\le F_1} \boxplus \dots \boxplus \fbuilding_{\le F_k}, \OM_{|F_1} \boxplus \dots \boxplus \OM_{|F_k})}$.
\end{corollary}

\begin{proof}
Immediate consequence of \cref{coro:connectedLatticeBuildingSet,prop:directSumFacialBuildingSets,prop:freeSumFacialBuildingSets}.
\end{proof}

\begin{example}
Let~$\b{m} \eqdef (m_1, \dots, m_\ell)$ with~$m_i \ge 1$ for all~$i \in [\ell]$.
Define~$[\b{m}] \eqdef [m_1^1] \sqcup \dots \sqcup [m_\ell^\ell]$ where~$[m_i^i] \eqdef \{1^i, 2^i, \dots, m_i^i\}$ denote the first~$m_i$ integers colored by~$i$ for all~$i \in [\ell]$.
Consider the free sum of simplices~$\Diamond_\b{m} \eqdef \triangle_{[m_1]} \oplus \dots \oplus \triangle_{[m_\ell]}$.
The vertices of~$\Diamond_\b{m}$ correspond to the elements of~$[\b{m}]$ and the faces of~$\Diamond_\b{m}$ correspond to the set~$[\b{m}]$ itself together with all its subsets which do not contain~$[m_i^i]$ for each~$i \in [\ell]$.
Consider the oriented matroid~$\OM_{\b{m}}$ on~$[\b{m}]$ defined by the vector configuration obtained by homogeneization of the vertices of~$\Diamond_\b{m}$.
A collection~$\fbuilding$ of subsets of~$[\b{m}]$ is a facial building set for~$\OM_\b{m}$ if and only if
\begin{itemize}
\item $[m_i^i] \not\subseteq F$ for all~$F \in \fbuilding$ distinct from~$[\b{m}]$ and all~$i \in [\ell]$,
\item $\fbuilding$ contains the set~$[\b{m}]$ and the singletons~$\{n^i\}$ for all~$i \in [\ell]$ and~$n \in [m_i]$,
\item for any~$F, G \in \fbuilding$ such that~$[m_i^i] \not\subseteq F \cup G$ for all~$i \in [\ell]$, if~$F \cap G \ne \varnothing$ then~$F \cup G \in \fbuilding$.
\end{itemize}
Moreover, a subset~$\nested$ of a facial building set~$\fbuilding$ for~$\OM_\b{m}$ containing~$[\b{m}]$ is a facial nested set if and~only~if
\begin{enumerate}
\item for any~$F, G \in \nested$, either~$F \subseteq G$ or~$G \subseteq F$ or~$F \cap G = \varnothing$,
\item for any~$k \ge 2$ pairwise disjoint~$F_1,\dots,F_k \in \nested$, the union~$F_1 \cup \dots \cup F_k$ is in~$\FL(\OM_\b{m})$ but not in~$\fbuilding$.
\end{enumerate}
This specializes to \cref{prop:facialBuildingSetCrossPolytope,prop:facialNestedComplexCrossPolytope} when~$\b{m} = (2, \dots, 2)$.
\end{example}


\section{Oriented building sets and acyclic nested complexes}
\label{sec:orientedBuildingSetsAcyclicNestedComplexes}

We now switch to a different family of simplicial complexes associated to an oriented matroid~$\OM$.
We first define oriented building sets (\cref{subsec:orientedBuildingSets}), and then their acyclic nested complexes (\cref{subsec:acyclicNestedComplex}).


\subsection{Oriented building sets}
\label{subsec:orientedBuildingSets}

We first combine building sets and oriented matroids to obtain the following combinatorial data.

\begin{definition}
\label{def:orientedBuildingSet}
An \defn{oriented building set} is a pair~$(\building, \OM)$ where~$\building$ is a building set and~$\OM$ is an oriented matroid on the same ground set~$\ground$ such that~$\underline{c} \in \building$ for any~$c \in \circuits[\OM]$.
We also often say that~$\building$ is an oriented building set for~$\OM$.
We say that~$(\building, \OM)$ is \defn{realizable} if~$\OM$ is realizable.
\end{definition}

\begin{example}
Consider the building set~$\building_\circ$ of \cref{exm:booleanBuildingSet} and the vector configuration~$\b{A}_\circ$ of \cref{exm:vectorConfiguration}. As the supports~$12$, $1456$, and~$2456$ of the circuits~$\circuits[\b{A}_\circ]$ are all blocks of~$\building_\circ$, the pair~$(\building_\circ, \OM(\b{A}_\circ))$ is an oriented building set.
\end{example}

\begin{example}
\label{exm:graphicalOrientedBuildingSet}
Consider a directed graph~$\digraph$ with vertex set~$V$ and arc set~$\ground$.
The \defn{line graph} of~$\digraph$ is the graph~$L(\digraph)$ on~$\ground$ with an edge between two arcs of~$\digraph$ if and only if they share an endpoint.
See \cref{fig:exmDigraph} for an illustration.
The \defn{graphical oriented building set} of~$\digraph$ is the pair~$(\building(L(\digraph)), \OM(\digraph))$, where~$\building(L(\digraph))$ is the graphical building set of~$L(\digraph)$ defined in \cref{exm:graphicalBuildingSet} and~$\OM(\digraph)$ is the graphical oriented matroid of~$\digraph$ defined in \cref{exm:graphicalOM}.
Note that it is indeed an oriented building set: $\ground$ is the ground set of both~$\building(L(\digraph))$ and~$\OM(\digraph)$, and the circuits in~$\OM(\digraph)$ are cycles in~$\digraph$, hence of~$L(\digraph)$, thus belong to~$\building(L(\digraph))$.
\end{example}

\begin{remark}
Given an oriented matroid~$\OM$, the building sets~$\building$ such that~$(\building, \OM)$ is an oriented building set form an upper ideal in the inclusion poset of building sets.
This upper ideal is generated by a unique building set~$\building(\OM)$, which is the minimal building set containing~$\underline{c}$ for all~$c \in \circuits[\OM]$, that is, the building closure of~$\set{\underline{c}}{c \in \circuits[\OM]}$.
In other words, the blocks of~$\building(\OM)$ are the connected components of~$\OM$ and all its restrictions.
For instance, for the vector configuration~$\b{A}_\circ$ of \cref{exm:vectorConfiguration}, the minimal building set is~$\building(\OM(\b{A}_\circ)) = \{1, 2, 3, 4, 5, 6, 12, 1456, 2456, 123456\}$.
\end{remark}

We now briefly discuss the properties of oriented building sets with respect to the operations of restriction and contraction defined in \cref{def:restrictionContractionBuildingSet} for building sets and in \cref{def:restrictionContractionOM} for oriented matroids.

\begin{proposition}
\label{prop:restrictionContractionOrientedBuildingSet}
If~$(\building, \OM)$ is an oriented building set on~$\ground$ and~$R \subseteq \ground$, then the restriction~$(\building_{|R}, \OM_{|R})$ and contraction~$(\building_{\!/R}, \OM_{\!/R})$ are oriented building sets on~$R$ and~$\ground \ssm R$ respectively.
\end{proposition}

\begin{proof}
If~$c \in \circuits[\OM_{|R}]$, then~$c \in \circuits[\OM]$ so that~$\underline{c} \in \building$, and~$\underline{c} \subseteq R$ so that~$\underline{c} \in \building_{|R}$.
If~$c' \in \circuits[\OM_{\!/R}]$, then~$c' = c' \ssm R$ for some~$c \in \circuits[\OM]$, so that~$\underline{c} \in \building$ and thus~$\underline{c'} = \underline{c} \ssm R \in \building_{\!/R}$.
\end{proof}

Our next definition is the analogue of \cref{def:restrictionContractionFacialBuildingSet} for oriented building sets.

\begin{definition}
\label{def:restrictionContractionOrientedBuildingSet}
For an oriented building set~$(\building, \OM)$, a nested set~$\nested$ on~$\building$ and~$B \in \nested$, we define
\[
\restrGround \eqdef B \ssm R,
\qquad
\restrBuilding \eqdef (\building_{|B})_{\!/R}
\qquad\text{and}\qquad
\restrOM \eqdef (\OM_{|B})_{\!/R}
\]
where~$R = \complRestrGround \eqdef \bigcup_{C \in \nested, \; C \subsetneq B} C$.
\end{definition}

\begin{proposition}
\label{prop:restrictionContractionOrientedBuildingSet2}
For any~$B \in \nested \in \nestedComplex[][\building]$, the pair~$(\restrBuilding, \restrOM)$ is an oriented building set on~$\restrGround$.
\end{proposition}

\begin{proof}
Immediate consequence of \cref{prop:restrictionContractionOrientedBuildingSet}.
\end{proof}

\begin{example}
\label{exm:graphicalRestrictionContractionOrientedBuildingSet}
Consider the graphical oriented building set of a directed graph~$\digraph$ from~\cref{exm:graphicalOrientedBuildingSet}, and a tube~$\tube$ in a tubing~$\tubing$ of~$L(\digraph)$.
The oriented building set~$(\building(L(\digraph)), \OM(\digraph))_{\tube \in \tubing}$ is the graphical oriented building set of the directed graph obtained as the contraction in the restriction~$\digraph_{|\tube}$ of all arcs contained in some tube~$\tube' \in \tubing$ with~$\tube' \subsetneq \tube$.
\end{example}


\subsection{Acyclic nested complex}
\label{subsec:acyclicNestedComplex}

Using \cref{def:restrictionContractionOrientedBuildingSet}, we are now ready to define acyclic nested sets and the acyclic nested complex of an oriented building set.
Here and throughout, we define~$\bigcup \nested \eqdef N_1 \cup \dots \cup N_k$ for~$\nested \eqdef \{N_1, \dots, N_k\}$.

\begin{lemma}
\label{lem:characterizationAcyclicNestedSets}
The following conditions are equivalent for a nested set~$\nested$~on~$\building$:
\begin{enumerate}[(i)]
\item $\restrOM$ is acyclic for every~$B \in \nested$,
\item there is no~$c \in \circuits[\OM]$ and~$B \in \nested$ with~$\underline{c} \subseteq B$ such that~$\restrGround$ intersects~$c_+$~but~not~$c_-$,
\item there is no~$c \in \circuits[\OM]$ and~$\nested' \subseteq \nested$ with~$c_+ \not\subseteq \bigcup \nested'$ while~$c_- \subseteq \bigcup \nested'$,
\item $\OM_{\!/\bigcup \nested'}$ is acyclic for every~$\nested' \subseteq \nested$,
\item $\ground \ssm \bigcup \nested' \in \covectors[\OM]$ for every~$\nested' \subseteq \nested$,
\item $\bigcup \nested'$ is a face of $\OM$ for every~$\nested' \subseteq \nested$.
\end{enumerate}
We then say that~$\nested$ is \defn{acyclic} for~$\OM$.
\end{lemma}

\begin{proof}
By \cref{def:restrictionContractionOM},
\[
\vectors[\restrOM] = \set{v \ssm R}{v \in \vectors[\OM_{|B}]} = \set{v \ssm R}{v \in \vectors[\OM] \text{ and } \underline{v} \subseteq B},
\]
where~$R \eqdef \bigcup_{C \in \nested, \; C \subsetneq B} C$ and~$v \ssm R \eqdef (v_+ \ssm R, v_- \ssm R)$.
The equivalence (i) $\Leftrightarrow$ (ii) thus follows from \cref{def:acyclicOM}.

The equivalent (ii) $\Leftrightarrow$ (iii) is immediate: if~$c$ and~$B$ fail~(ii), then~$c$ and~$\nested' \eqdef \set{C \in \nested}{C \subsetneq B}$ fail~(iii), and conversely if $c$ and~$\nested'$ fail~(iii), then~$c$ and~$B \eqdef \min\set{C \in \nested}{\underline{c} \subseteq C}$ fail~(ii).

The equivalence (iii) $\Leftrightarrow$ (iv) follows from \cref{def:acyclicOM} again.

Finally, we have~(iv) $\Leftrightarrow$ (v) $\Leftrightarrow$ (vi) by \cref{lem:acyclicContraction}.
\end{proof}

\begin{definition}
\label{def:acyclicNestedComplex}
The \defn{acyclic nested complex} of an oriented building set~$(\building, \OM)$ is the simplicial complex~$\acyclicNestedComplex(\building, \OM)$ whose faces are~$\nested \ssm \connectedComponents(\building)$ for all nested sets~$\nested$ of~$\building$ which are acyclic for~$\OM$.
\end{definition}

Note that the conditions~(iii), (iv), (v) or (vi)  of \cref{lem:characterizationAcyclicNestedSets} immediately show that the acyclic nested complex~$\acyclicNestedComplex(\building, \OM)$ is indeed a simplicial complex.
We will prove in \cref{sec:orientedMatroidRealizations} that it is actually a $(\rank[\OM]-|\connectedComponents(\building)|-1)$-dimensional simplicial sphere, and in \cref{sec:polytopalRealizations} that it is the boundary complex of a simplicial $(\rank[\OM]-|\connectedComponents(\building)|)$-dimensional polytope as soon as the oriented matroid~$\OM$ is realizable.

\begin{example}
\label{exm:simplexAcyclicNestedComplex}
For any building set~$\building$ on~$\ground$, the nested complex~$\nestedComplex[][\building]$ is the acyclic nested complex~$\acyclicNestedComplex(\building, \OM_\triangle)$, where~$\OM_\triangle$ is the independent (\ie with no circuit) oriented matroid~on~$\ground$.
\end{example}

\begin{remark}
\label{rem:trivialCases}
\begin{itemize}
\item If~$\OM$ is not acyclic, then the acyclic nested complex~$\acyclicNestedComplex(\building, \OM)$ is empty.
\item If~$\OM$ contains a circuit~$c = (c_+, c_-)$ with~$|c_-| = 1$, then~$\acyclicNestedComplex(\building, \OM)$ is actually isomorphic to~$\acyclicNestedComplex(\building_{|\ground \ssm \{s\}}, \OM_{|\ground \ssm \{s\}})$.
\end{itemize}
\end{remark}

\begin{example}
\label{exm:graphicalAcyclicNestedComplex}
From \cref{exm:graphicalOrientedBuildingSet}, consider a directed graph~$\digraph$ and its graphical oriented building set~$(\building(L(\digraph)), \OM(\digraph))$.
\begin{figure}
	\capstart
	\centerline{\includegraphics[scale=.45]{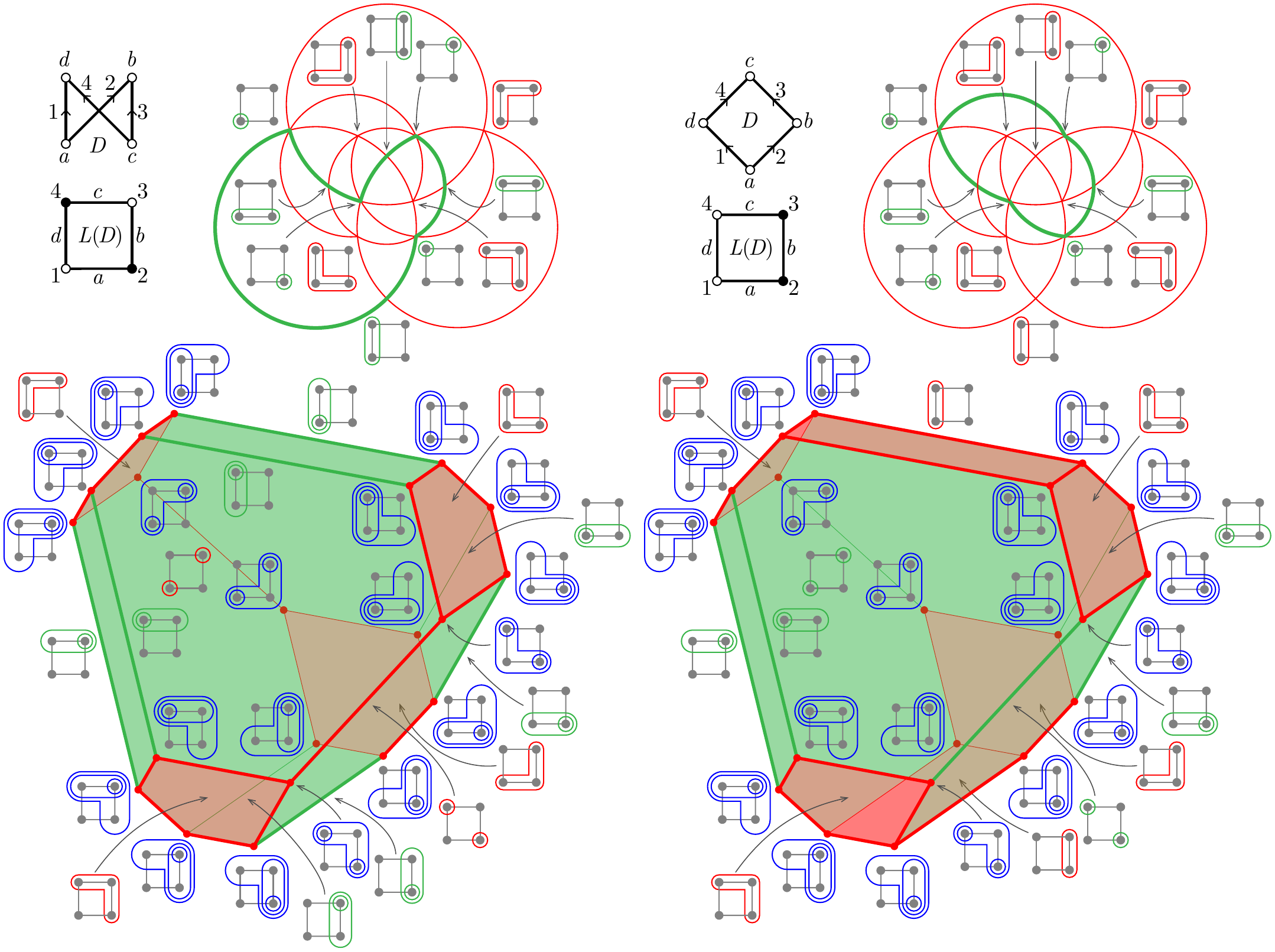}}
	\caption{Two graphical acyclic nested complexes. For each one, we have drawn the directed graph~$\digraph$, and its line graph~$L(\digraph)$ with vertices colored black and white according to the sign of the corresponding arcs in the only circuit of~$\digraph$. On the graphical nested fan of~$L(\digraph)$ (top), we have drawn the acyclic tubings of~$L(\digraph)$ in green. On the graph associahedron of~$L(\digraph)$ (bottom), we have drawn all maximal tubings of~$L(\digraph)$ in blue, the minimal cyclic tubings of~$L(\digraph)$ in red, and the maximal acyclic tubings of~$L(\digraph)$ in green.}
	\label{fig:acyclicNestedComplexes}
\end{figure}
The \defn{graphical acyclic nested complex}~$\acyclicNestedComplex(\building(L(\digraph)), \OM(\digraph))$ is then given by all tubings~$\tubing$ on~$L(\digraph)$ such that for each tube~$\tube \in \tubing$, the contraction in the restriction~$\digraph_{|\tube}$ of all arcs contained in some tube~$\tube' \in \tubing$ with~$\tube' \subsetneq \tube$ yields an acyclic directed graph.
\cref{fig:acyclicNestedComplexes} illustrates two graphical acyclic nested complexes.
Note that these two directed graphs have the same line graph, but distinct graphical oriented matroids, and thus distinct graphical acyclic nested complexes.
\end{example}

\begin{example}
\label{exm:simplexAcyclicNestedComplexGraphical}
It follows from \cref{exm:simplexAcyclicNestedComplex} that the graphical acyclic nested complex of~$\digraph$ is isomorphic to the classical nested complex of the line graph~$L(\digraph)$ when~$\digraph$ is an oriented forest (for instance, it is isomorphic to the simplicial permutahedron if~$\digraph$ is a star, and to the simplicial associahedron if~$\digraph$ is a path).
\end{example}

\begin{remark}
\label{rem:posetAssociahedra}
It follows from \cref{rem:trivialCases} that the graphical acyclic nested complex of~$\digraph$~is
\begin{itemize}
\item empty if~$\digraph$ is cyclic (\ie has a directed cycle),
\item isomorphic to the graphical acyclic nested complex of the Hasse diagram of the transitive closure of~$\digraph$ if~$\digraph$ is acyclic.
\end{itemize}
Hence, graphical acyclic nested complexes are in fact intrinsically associated to posets.
The graphical case of \cref{exm:graphicalOrientedBuildingSet,exm:graphicalAcyclicNestedComplex} actually motivated \cref{def:orientedBuildingSet,def:acyclicNestedComplex}, and was inspired from the poset associahedra defined in~\cite{Galashin}.
See \cref{sec:posetAssociahedra}.
\end{remark}

We now describe the links in the acyclic nested complex~$\acyclicNestedComplex(\building, \OM)$.
The reader is invited to write a direct proof of the following statement.
It also follows from the analogue \cref{prop:linksFacialNestedComplex} on facial nested complexes, and the connection of \cref{sec:facialNestedComplexesVsAcyclicNestedComplexes} between facial nested complexes and acyclic nested complexes.

\begin{proposition}
\label{prop:linksAcyclicNestedComplex}
For an acyclic nested set~$\nested$, the link of~$\nested \ssm \connectedComponents(\building)$ in the acyclic nested complex~$\acyclicNestedComplex(\building, \OM)$ is the join of the acyclic nested complexes~$\acyclicNestedComplex(\restrBuilding, \restrOM)$ for all~$B \in \nested$.
\end{proposition}

\begin{example}
\label{exm:graphicalLinksAcyclicNestedComplex}
In view of \cref{exm:graphicalRestrictionContractionOrientedBuildingSet}, all links in a graphical acyclic nested complexes are graphical acyclic nested complexes.
\end{example}

We now consider with a little more care the maximal acyclic nested sets, which will be important in our proof of \cref{thm:acyclonestohedron}.

\begin{proposition}
\label{prop:rank1}
If~$\nested$ is a maximal acyclic nested set of~$(\building, \OM)$, then~$\restrOM$ has rank~$1$ for any~$B \in \nested$.
\end{proposition}

\begin{proof}
Assume by contradiction that~$\restrOM$ has rank larger than~$1$ for some~$B \in \nested$.
Hence, it admits a proper positive cocircuit~$c^*$ with~$\varnothing \ne c^*_+ \ne \restrGround$ (consider any facet of~$\restrOM$).
Define
\[
\nested' \eqdef \nested \cup \connectedComponents(\set{B' \in \building}{B' \subseteq B \ssm c^*_+}).
\]
We claim that
\begin{enumerate}[(i)]
\item $\nested \subsetneq \nested'$, since~$\varnothing \ne c^*_+ \ne \restrGround$.
\item $\nested'$ is a nested set of~$\building$. Indeed,
	\begin{itemize}
	\item Any~$B', B'' \in \nested$ (resp.~$B', B'' \in \nested' \ssm \nested$) are clearly nested or disjoint since they belong to a nested set (resp.~since they are connected components of a building set). Consider now~$B' \in \nested$ and~$B'' \in \nested' \ssm \nested$. Since~$B$ and~$B'$ belong to the nested set~$\nested$, they are nested or disjoint. If~$B \subseteq B'$, then~$B'' \subset B'$. If~$B' \subsetneq B$, then~$B' \subseteq B \ssm c^*_+$, so that~$B' \subseteq B''$ or~$B' \cap B'' = \varnothing$. Otherwise, $B' \cap B = \varnothing$, thus~$B' \cap B'' = \varnothing$.
	\item Consider~$k \ge 2$ pairwise disjoint~$B_1, \dots, B_k$ in~$\nested'$. If they all belong to~$\nested$, then ${B_1 \cup \dots \cup B_k} \notin \building$. Otherwise, let~$B_j \in \nested' \ssm \nested$. If~$I = \set{i \in [k]}{B_i \not\subseteq B} \ne \varnothing$, then~$B \cup B_1 \cup \dots \cup B_k = B \cup \bigcup_{i \in I} B_i \notin \building$ (since~$\nested$ is a nested set) so that~$B_1 \cup \dots \cup B_k \notin \building$ (since~$\building$ is a building set and~$B \cap B_j \ne \varnothing$). If~$I = \varnothing$, then $B_1 \cup \dots \cup B_k \notin \building$ by maximality of~$B_j$.
	\item $\nested'$ contains~$\nested$ which contains~$\connectedComponents(\building)$.
	\end{itemize}
\item $\nested'$ is still acyclic for~$\OM$. Indeed,
	\begin{itemize}
	\item For~$B' \in \nested \ssm \{B\}$, we have~$\complRestrGround[R][B'][\nested'] = \complRestrGround[R][B'][\nested]$ so that~$\restrOM[B'][\nested'] = \restrOM[B'][\nested]$ is acyclic.
	\item For~$B' \in \connectedComponents(\set{B' \in \building}{B' \subseteq B \ssm c^*_+})$, we have~$\complRestrGround[R][B'][\nested'] = \complRestrGround[R][B][\nested] \cap B'$ so that \linebreak $\restrOM[B'][\nested'] = (\restrOM[B][\nested])_{|B'}$ is acyclic since~$ \restrOM[B][\nested]$ is.
	\item As circuits and cocircuits are sign orthogonal and~$c^*$ is a positive cocircuit of~$\restrOM$, $c^*_+ \cap c_+ \ne \varnothing$ implies~$c^*_+ \cap c_- \ne \varnothing$ for any circuit~$c$ of~$\restrOM$, hence for any circuit~$c$ of~$\OM$ with~$\underline{c} \subseteq B$. Since~$\restrGround[B][\nested'] = c^*_+$, we obtain that~$\restrOM[B][\nested]$ is acyclic by \cref{lem:characterizationAcyclicNestedSets}.
	\end{itemize}
\end{enumerate}
This contradicts the maximality of~$\nested$.
\end{proof}

\begin{corollary}
\label{coro:pureAcyclicNestedComplex}
The acyclic nested complex~$\acyclicNestedComplex(\building, \OM)$ is pure of dimension~$\rank[\OM]-|\connectedComponents(\building)|-1$.
\end{corollary}

\begin{proof}
For a maximal nested set~$\nested$, we have~$\connectedComponents(\building) \subseteq \nested$ and~${\rank[\OM] = \sum_{B \in \nested} \rank[\restrOM] = |\nested|}$ by \cref{prop:rank1}.
Hence, the corresponding face has cardinality~$|\nested \ssm \connectedComponents(\building)| = \rank[\OM] - |\connectedComponents(\building)|$.
\end{proof}

\begin{corollary}
\label{coro:rank1}
For any maximal acyclic nested set~$\nested$ of~$(\building, \OM)$, and any~$B \in \nested$, the building set~$\restrBuilding$ on~$\restrGround$ is complete.
\end{corollary}

\begin{proof}
Since~$\restrOM$ has rank~$1$ by \cref{prop:rank1}, any two elements~$u,v \in \restrGround$ form a circuit of~$\restrOM$, hence a block of~$\restrBuilding$ by \cref{def:orientedBuildingSet}.
Since~$\restrBuilding$ is a building set, we conclude that~$\restrBuilding = 2^{\restrGround}$.
\end{proof}


\section{Facial nested complexes versus acyclic nested complexes}
\label{sec:facialNestedComplexesVsAcyclicNestedComplexes}

We now connect \cref{sec:facialBuildingSetsFacialNestedComplexes,sec:orientedBuildingSetsAcyclicNestedComplexes}.
More precisely, for any acyclic oriented matroid~$\OM$, we now show that 
\begin{itemize}
\item the facial building sets of \cref{def:facialBuildingSet} essentially coincide with the oriented building sets of \cref{def:orientedBuildingSet}, see \cref{thm:orientedBuildingSetsTofacialBuildingSets},
\item the facial nested complexes of \cref{def:facialNestedComplex} coincide with the acyclic nested complexes of~\cref{def:acyclicNestedComplex}, see \cref{thm:facialNestedComplexesVsAcyclicNestedComplexes}.
\end{itemize}
For this, we need to select the facial part of an oriented building set.

\begin{theorem}
\label{thm:orientedBuildingSetsTofacialBuildingSets}
The map
\(
\building \mapsto \building \cap \FL(\OM)
\)
is a surjection from oriented building sets for~$\OM$ to facial building sets for~$\OM$.
\end{theorem}

\begin{proof}
We first prove that the map is well-defined.
Consider an oriented building set~$(\building, \OM)$ and let~$\fbuilding \eqdef \building \cap \FL(\OM)$.
Fix a face~$F$ of~$\OM$, and denote by~$B_1, \dots, B_k$ the maximal facial blocks from~$\fbuilding_{\subseteq F}$, and by~$C_1, \dots, C_\ell$ the connected components of $\OM_{|F}$.
Then
\begin{equation}
\label{eq:proofFacialBuildingSetAreLasVergnasLatticeBuildingSets}
\FL(\OM_{|F}) \stackrel{(i)}{\cong} \prod_{j \in [\ell]} \FL(\OM_{|C_j}) \stackrel{(ii)}{=} \prod_{i \in [k]} \prod_{\substack{j \in [\ell] \\ B_i \cap C_j \ne \varnothing}} \FL(\OM_{|C_j}) \stackrel{(iii)}{\cong} \prod_{i \in [k]} \FL(\OM_{|B_i}),
\end{equation}
so that~$\fbuilding$ is indeed an $\FL(\OM)$-building set.
In \eqref{eq:proofFacialBuildingSetAreLasVergnasLatticeBuildingSets}, the equality~(i) holds by \cref{prop:directSumCartesianProductFL,prop:connectedComponentsDecomposition}.
To see equalities (ii) and (iii), observe that
\begin{itemize}
\item Each~$C_j$ belongs to~$\fbuilding$. Indeed, it belongs to~$\building$ as it is either a singleton or a union of supports of circuits with connected intersection graph, which are both elements of $\building$ by \cref{def:booleanBuildingSet,def:orientedBuildingSet}. Moreover, it is a face of~$\OM$ by \cref{cor:connectedComponentsOfFacesAreFaces}.
\item Each~$B_i$ coincides with the disjoint union~$\bigsqcup_{j \in [\ell], B_i \cap C_j \ne \varnothing} C_j$. Indeed, this union clearly contains~$B_i$ and is contained in~$F$. 
Since the connected components are themselves blocks, $B_i\cup \bigsqcup_{j \in [\ell], B_i \cap C_j \ne \varnothing} C_j$ belongs to~$\building$ as a union of intersecting blocks, hence equals to~$B_i$ by maximality of~$B_i$.
\end{itemize}

To prove the surjectivity, consider a facial building set~$\fbuilding$.
We define $\building$ to be the building closure (see \cref{def:booleanBuildingSetClosure}) of $\fbuilding \cup \set{\underline c}{c \in \circuits[\OM]}$, and we prove that $\fbuilding = \building \cap \FL(\OM)$.
The inclusion $\fbuilding \subseteq \building \cap \FL(\OM)$ is immediate, hence we just need to prove the inclusion~$\building \cap \FL(\OM) \subseteq \fbuilding$.
Consider thus~$B \in \building \cap \FL(\OM)$.
 
Combining the definition of facial building set with \cref{prop:directSumCartesianProductFL}, we know that
\[
\OM_{|B} = \OM_{|B_1} \oplus \cdots \oplus \OM_{|B_k},
\]
where $B_1,\dots,B_k$ are the maximal elements of $\fbuilding$ included in $B$.
Note that the direct sum structure implies that the support of every circuit of $\OM_{|B}$ belongs to one $B_i$.
Therefore, for each~$B_i$, the sets in~$\fbuilding \cup \set{\underline c}{c \in \circuits[\OM]}$ contained in~$B_i$ form a connected component of the intersection graph of $\fbuilding \cup \set{\underline c}{c \in \circuits[\OM]}$.
Hence, each $B_i$ is an inclusion maximal element of $\building_{\subseteq B}$ and, since they are faces, also of $(\building \cap \FL(\OM))_{\subseteq B}$. 
As $B \in \building \cap \FL(\OM)$, we must have $k=1$, and therefore $B=B_1\in \fbuilding$, as required.
\end{proof}

\begin{remark}
Note that the surjection of \cref{thm:orientedBuildingSetsTofacialBuildingSets} is not injective in general.
The inclusion minimal oriented building set mapped to a given facial building set~$\fbuilding$ is the building closure of~$\fbuilding \cup \set{\underline c}{c \in \circuits[\OM]}$ considered in the proof of \cref{thm:orientedBuildingSetsTofacialBuildingSets}.
Note that for the independent oriented matroid~$\OM_\triangle$ (\ie with no circuit), all subsets are faces, so that the map of \cref{thm:orientedBuildingSetsTofacialBuildingSets} is bijective.
\end{remark}

\begin{theorem}
\label{thm:facialNestedComplexesVsAcyclicNestedComplexes}
Let $\fbuilding \eqdef \building \cap \FL(\OM)$ be the facial building set of an oriented building set~$(\building,\OM)$. Then the facial nested complex $\nestedComplex[][\fbuilding, \OM]$ and the acyclic nested complex $\acyclicNestedComplex(\building, \OM)$ coincide.
\end{theorem}

To prove this statement, we need a couple of technical results.
We still denote by $\fbuilding \eqdef \building \cap \FL(\OM)$ the facial building set of the oriented building set~$(\building,\OM)$.

\begin{lemma}
\label{lem:smallestFaceContainingBlock}
If $B \in \building$ and $F$ is the inclusion minimal face of $\OM$ containing $B$, then $F \in \fbuilding$.
\end{lemma}

\begin{proof}
By definition, the smallest face containing $B$ is the smallest $F \supseteq B$ for which there is no $c \in \circuits$ with $c_- \subseteq F$ and $c_+ \not\subseteq F$, and can be obtained iteratively as follows.
Start with $F_0 \eqdef B$ and construct~$F_{i+1} \eqdef F_i \cup \underline c^i$ for some $c^i \in \circuits$ with $c^i_- \subseteq F_i$ but $c^i_+ \not\subseteq F_i$ as long as such a circuit~$c^i$ exists.
When no such circuit exists, we recover $F$.
Note that $F_0 \in \building$ and $\underline c^i \in \building$ for all~$i$, and therefore we also have $F_i \in \building$ for all~$i$, so that~$F \in \building$ and thus~$F \in \fbuilding$.
\end{proof}

\begin{corollary}
\label{cor:carrier}
If $B_1, B_2 \in \building$ with $B_1\cap B_2\neq\varnothing$ then~${B_1 \vee B_2 \in \fbuilding}$, where~$\vee$ is the join in~$\FL(\OM)$.
\end{corollary}

\begin{proof}
As~$B_1, B_2 \in \building$ and~$B_1 \cap B_2 \ne \varnothing$ we have~$B_1 \cup B_2 \in \building$.
As~$B_1 \vee B_2$ is the smallest face of~$\OM$ containing~$B_1 \cup B_2$, we obtain that~$B_1 \vee B_2 \in \fbuilding$ by \cref{lem:smallestFaceContainingBlock}.
\end{proof}

\begin{corollary}
\label{coro:connectedComponents}
We have~$\connectedComponents(\building)=\connectedComponents(\fbuilding)$.
\end{corollary}

\begin{proof}
We have clearly~$\connectedComponents(\building) \subseteq \connectedComponents(\fbuilding)$ since~$\building \supseteq \fbuilding$.
Conversely, $B \in \connectedComponents(\fbuilding)$ was not maximal in~$\building$, then there would be some $B\subsetneq B'\in \building$, and by \cref{lem:smallestFaceContainingBlock} there would be some $B''\in \fbuilding$ such that $B\subsetneq B'\subseteq B''\in \fbuilding$, contradicting the maximality of~$B \in \connectedComponents(\fbuilding)$.
\end{proof}

\begin{proof}[Proof of \cref{thm:facialNestedComplexesVsAcyclicNestedComplexes}]
Let $\nested$ be a facial nested set for~$(\fbuilding, \OM)$. 
Observe first that~$\nested$ is a nested set on~$\building$.
Indeed, we have
\begin{itemize}
\item $\connectedComponents(\building) = \connectedComponents(\fbuilding) \subseteq \nested$ by \cref{coro:connectedComponents},
\item any two elements of~$\nested$ are either nested or disjoint by \cref{cor:carrier},
\item for $k \ge 2$ pairwise disjoint blocks $B_1, \dots, B_k \in \nested$, the union $B_1\cup \dots \cup B_k$ is not in~$\building$, as otherwise we would have $B_1 \vee \dots \vee B_k \in \fbuilding$  by \cref{lem:smallestFaceContainingBlock}, contradicting the fact that $\nested$ is a facial nested set.
\end{itemize}
We now prove that~$\nested$ is acyclic.  
We prove that for all $\nested'\subseteq \nested$, we have that $\bigcup \nested'$ is a face.
By contradiction, consider an inclusion minimal $\nested'\subseteq \nested$ for which $\bigcup \nested'$ is not a face. Its elements $B_1,\dots,B_k$ have to be pairwise incomparable by minimality.
By the definition of the facial nested complex, $F=B_1\vee \cdots \vee B_k$ is not in $\fbuilding$.
By definition of facial building sets, we know that
\[
\OM_{|F}=\OM_{|C_1}\oplus\cdots\oplus \OM_{|C_r},
\]
where $C_1,\dots,C_r$ are the maximal elements of $\fbuilding_{\subseteq F}$.
In particular, $C_i\cap C_j=\varnothing$ for all $i\neq j$. Note also that if $B_i\cap C_j\neq \varnothing$, then $B_i\subseteq C_j$, as otherwise $B_i\vee C_j$ would be a larger block in $\fbuilding_{\subseteq F}$ by \cref{cor:carrier}.
Also, each $C_j$ contains at least some $B_i$, as otherwise $\bigcup_{i\neq j} C_i$ would be a smaller face containing $\nested'$, contradicting the minimality of~$F$.
  
Now, denote $\nested_i=\nested'_{\subseteq C_i}$.
By the minimality of $\nested'$, we know that $\bigcup \nested_i$ is a face of $C_i$ for each~$i$.
By the direct sum decomposition of $F$, any union of faces of the $C_i$'s must be a face of~$F$.
Therefore, we get that~$\bigcup_i \left(\bigcup \nested_i\right) = \bigcup \nested'$ must be a face, a contradiction.

Conversely, let $\nested$ be an acyclic nested set in $\acyclicNestedComplex(\building, \OM)$.
By \cref{lem:characterizationAcyclicNestedSets}\,(vi), every $B\in \nested$ is a face, and therefore belongs to $\fbuilding$.
Consider now some pairwise incomparable elements~$B_1, \dots, B_k$ of~$\nested$.
Since $\nested$ is acyclic for~$\OM$, we obtain that~$B_1 \cup \dots \cup B_k$ is a face of~$\OM$ by \cref{lem:characterizationAcyclicNestedSets}\,(vi), and therefore $B_1 \vee \dots \vee B_k = B_1 \cup \dots \cup B_k$.
Moreover, $B_1, \dots, B_k$ are pairwise incomparable in~$\FL(\OM)$, hence pairwise non nested, so that $B_1 \cup \dots \cup B_k$ is not in~$\building$ since~$\nested \in \acyclicNestedComplex(\building, \OM) \subseteq \nestedComplex[][\building]$.
Since~$\fbuilding \subseteq \building$, we conclude that~$B_1 \vee \dots \vee B_k = B_1 \cup \dots \cup B_k$ is not in~$\fbuilding$.
\end{proof}


\clearpage
\part{Geometric realizations}
\label{part:realizations}

In this part, we study geometric realizations of facial nested complexes (or equivalently acyclic nested complexes).
We first observe that the facial nested complexes of an oriented matroid~$\OM$ are face lattices of oriented matroids, which are realizable whenever~$\OM$ is (\cref{sec:orientedMatroidRealizations}).
We then provide an alternative polytopal realization of facial nested complexes of realizable oriented matroids as sections of nestohedra, with explicit coordinates (\cref{sec:polytopalRealizations}).
Finally, we connect facial nested complexes to the work of G.~Gaiffi~\cite{Gaiffi2003} on stratified compactifications of polyhedral cones  and to the work of S.~Brauner, C.~Eur, E.~Pratt, and R.~Vlad~\cite{BraunerEurPrattVlad} on wondertopes (\cref{sec:compactifications}).


\section{Oriented matroid realizations}
\label{sec:orientedMatroidRealizations}

In~\cite[Cor.~4.3]{FeichtnerMuller2005}, E.~M.~Feichtner and I.~M\"uller proved that all nested complexes over a finite atomic meet-semilattice~$\lattice$ are homeomorphic to the order complex of~$\lattice$.
Since the face lattices of oriented matroids encode face lattices of regular cell decompositions of spheres (see~\cref{thm:faceLatticeOrientedMatroidSphere}), their order complexes are the face lattices of the barycentric subdivisions of these spheres.
This shows that the facial nested complexes of an oriented matroid~$\OM$ are always spheres.

In this section, we show that these spheres are actually face lattices of oriented matroids, which are realizable whenever $\OM$ is. 
This will follow immediately from the interpretation of the combinatorial blow-ups of E.~M.~Feichtner and D.~Kozlov~\cite[Def.~3.1 \& Thm.~3.4]{FeichtnerKozlov2004} (\cref{subsec:combinatorialblowup}) on face lattices of oriented matroids as stellar subdivisions (\cref{subsec:stellarSubdivisions}). Note that in the realizable case, they can be obtained by stellar subdivisions of polytopes or, dually, as a sequence of face truncations.

We focus on the connected case, as the non-connected case reduces to it via \cref{coro:connectedFacialBuildingSet}.

\begin{definition}\label{def:sdbuilding}
Let~$(\fbuilding,\OM)$ be a connected facial building set on a ground set~$\ground$. 
We define $\sdOM$ as the oriented matroid on the ground set~$\fbuilding$ 
given by
\[
\sdOM=\sd(F_1,\sd(F_2, \dots\sd(F_m, \OM)\dots))_{|\fbuilding},
\]
where $F_1, \dots, F_m$ are the faces in~$\fbuilding \ssm \{\ground\}$ ordered in any way such that $F_i \subseteq F_j$ implies $i\geq j$, and where we label each of the vertices~$\asd[F]$ introduced in the successive stellar subdivisions as the associate faces~$F$.
\end{definition}

Note that if we perform a stellar subdivision on a vertex, this vertex ceases to be a face (and to belong to any face). Therefore, since all the vertices of $\OM$ are in~$\fbuilding$ by definition of facial building set, none of the original elements of $\OM$ belongs to a face of $\FL(\sdOM)$, which explains why we can restrict its ground set to be~$\fbuilding$.

Now, \cite[Thm.~3.4]{FeichtnerKozlov2004} yields the following result. We provide a direct proof for completeness.
 
\begin{theorem}
\label{thm:orientedMatroidRealizationConnected}
For any connected facial building set~$(\fbuilding, \OM)$, the facial nested complex $\nestedComplex[][\fbuilding, \OM]$ ordered by inclusion coincides with the face semilattice $\FL(\sdOM)_{<\topone}$.
\end{theorem}

\begin{proof}
By definition, the facial nested complex $\nestedComplex[][\fbuilding, \OM]$ is the $\FL(\OM)$-nested complex~$\nestedComplex[\FL(\OM)][\fbuilding]$. This is the same as the $\FL(\OM)_{<\topone}$-seminested complex~$\nestedComplex[\FL(\OM)_{<\topone}][\fbuilding_{<\topone}]$ by definition of seminested complex, see \cref{rmk:semilattice}.

We can now apply \cref{thm:blowup}, which says that this is isomorphic to $\Bl[\FL(\OM)_{<\topone}][\fbuilding_{<\topone}]$. 
By \cref{prop:stellarSubdivisionOrientedMatroid,def:sdbuilding}, this is exactly $\FL(\sdOM)_{<\topone}$. The labelling in the definition of $\sdOM$ has been chosen so that this isomorphism is the identity.
\end{proof}

This can be extended to disconnected building sets using \cref{coro:connectedFacialBuildingSet}.

\begin{theorem}
\label{thm:orientedMatroidRealizationArbitrary}
For any facial building set~$(\fbuilding, \OM)$ with~$\connectedComponents(\fbuilding) \eqdef \{F_1, \dots, F_k\}$, the facial nested complex~$\nestedComplex[][\fbuilding, \OM]$ is isomorphic to~$\FL(\sdOM[\OM_{|F_1} \boxplus \dots \boxplus \OM_{|F_k}][\fbuilding_{\le F_1} \boxplus \dots \boxplus \fbuilding_{\le F_k}])_{<\topone}$.
\end{theorem}

Finally, observe that if $\OM$ is realizable, then the face lattice of $\OM$ is the face lattice of a convex polytope. By \cref{prop:stellarSubdivisionOrientedMatroid}, so is $\sdOM$ for any facial building set~$\fbuilding$.

\begin{theorem}
\label{thm:stellarPolytopalRealization}
For any polytope~$\polytope$, and any facial building set~$\fbuilding$ for~$\OM(\b{A}_{\polytope})$, the facial nested complex $\nestedComplex[][\fbuilding, \OM(\b{A}_{\polytope})]$ is isomorphic to the boundary complex of a convex polytope.
When~$\fbuilding$ is connected, this polytope is obtained by a sequence of stellar subdivisions of~$\polytope$, or dually, by a sequence of face truncations of the polar of~$\polytope$.
\end{theorem}

\begin{example}
\label{exm:barycentricSubdivision2}
Following on \cref{exm:barycentricSubdivision1}, for the complete facial building set~$\hat\fbuilding$ for~$\OM(\b{A}_{\polytope})$, the facial nested complex~$\nestedComplex[][\hat\fbuilding, \OM(\b{A}_{\polytope})]$ is isomorphic to the complete barycentric subdivision of~$\polytope$, or dually, to the complete face truncation of the polar of~$\polytope$.
When~$\polytope$ is simplicial, \ie when the polar~$\polytope^\triangle$ is simple, all faces of the complete face truncation of~$\polytope^\triangle$ are products of permutahedra.
For instance, the complete face truncation of a simplex is a permutahedron, and the complete face truncation of a permutahedron is a permuto-permutahedron of~\cite{Gaiffi-permutonestohedra,CastilloLiu-permutoAssociahedron}.
See \cref{fig:permutopermutahedra}.
We note that for the same price, we can iterate the construction of~\cite{Gaiffi-permutonestohedra,CastilloLiu-permutoAssociahedron} to obtain iterated nested braid fans and iterated permutahedra.
We also obtain the simple permutoassociahedron of~\cite{BaralicIvanovicPetric,Ivanovic}.
In contrast, we note that all the polytopes resulting from our construction are simple (or dually simplicial), so our setting does not cover the permutoassociahedra~\cite{Kapranov,ReinerZiegler,CastilloLiu-permutoAssociahedron} nor their generalizations to permutonestohedra~\cite{Gaiffi-permutonestohedra} which are non-simple polytopes.
\begin{figure}
	\capstart
	\centerline{\includegraphics[scale=.18]{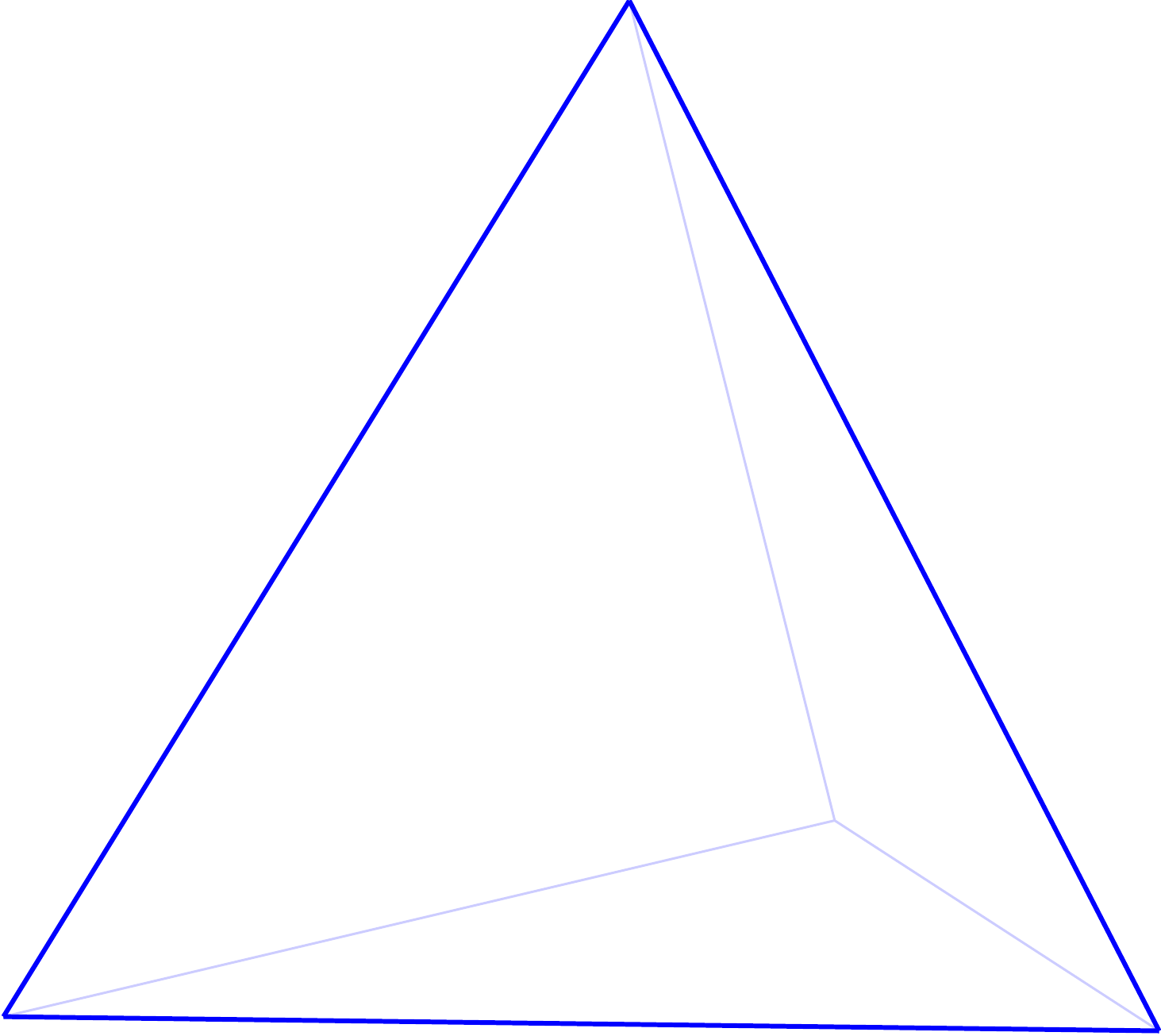}\qquad\includegraphics[scale=.15]{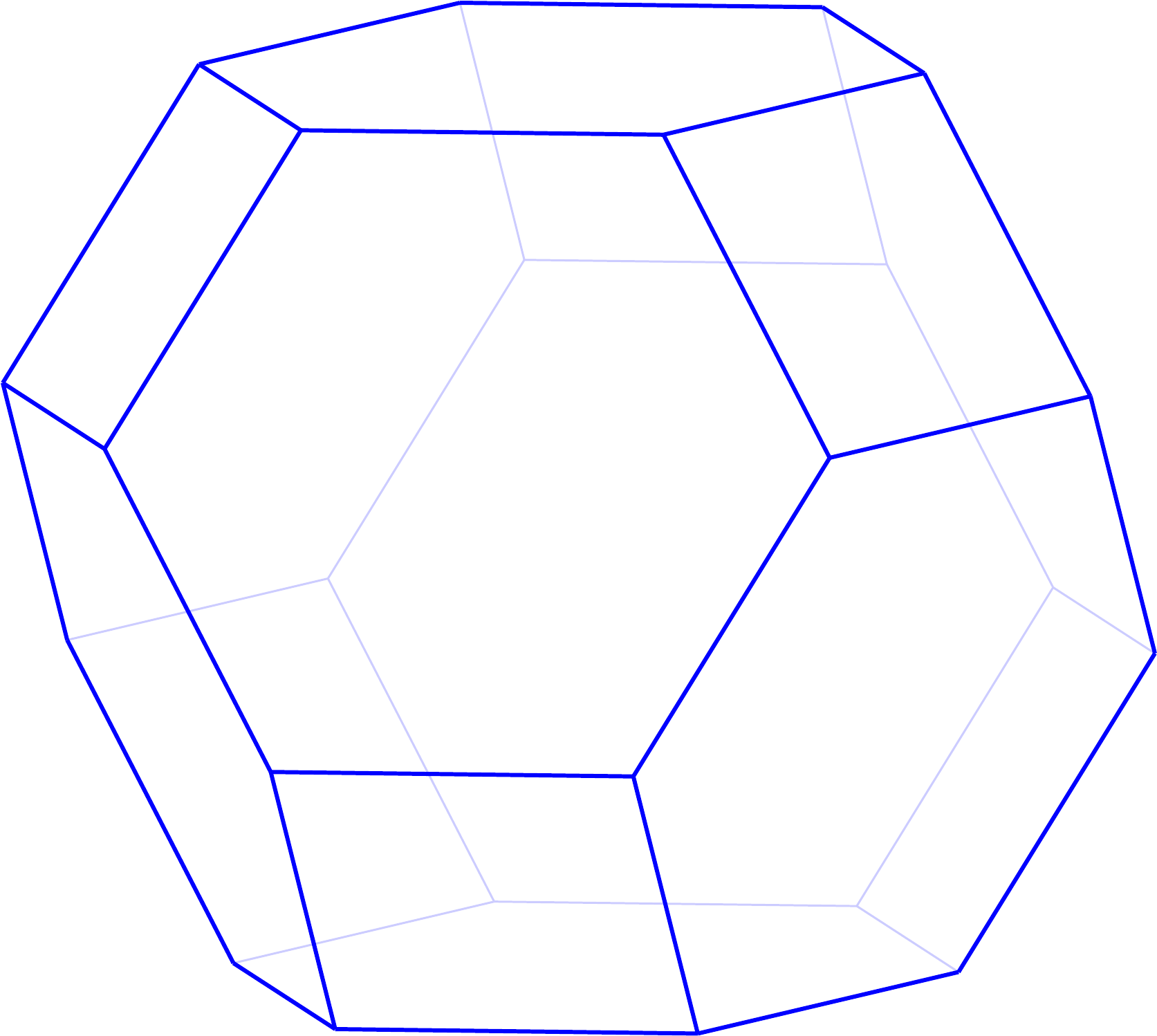}\qquad\includegraphics[scale=.15]{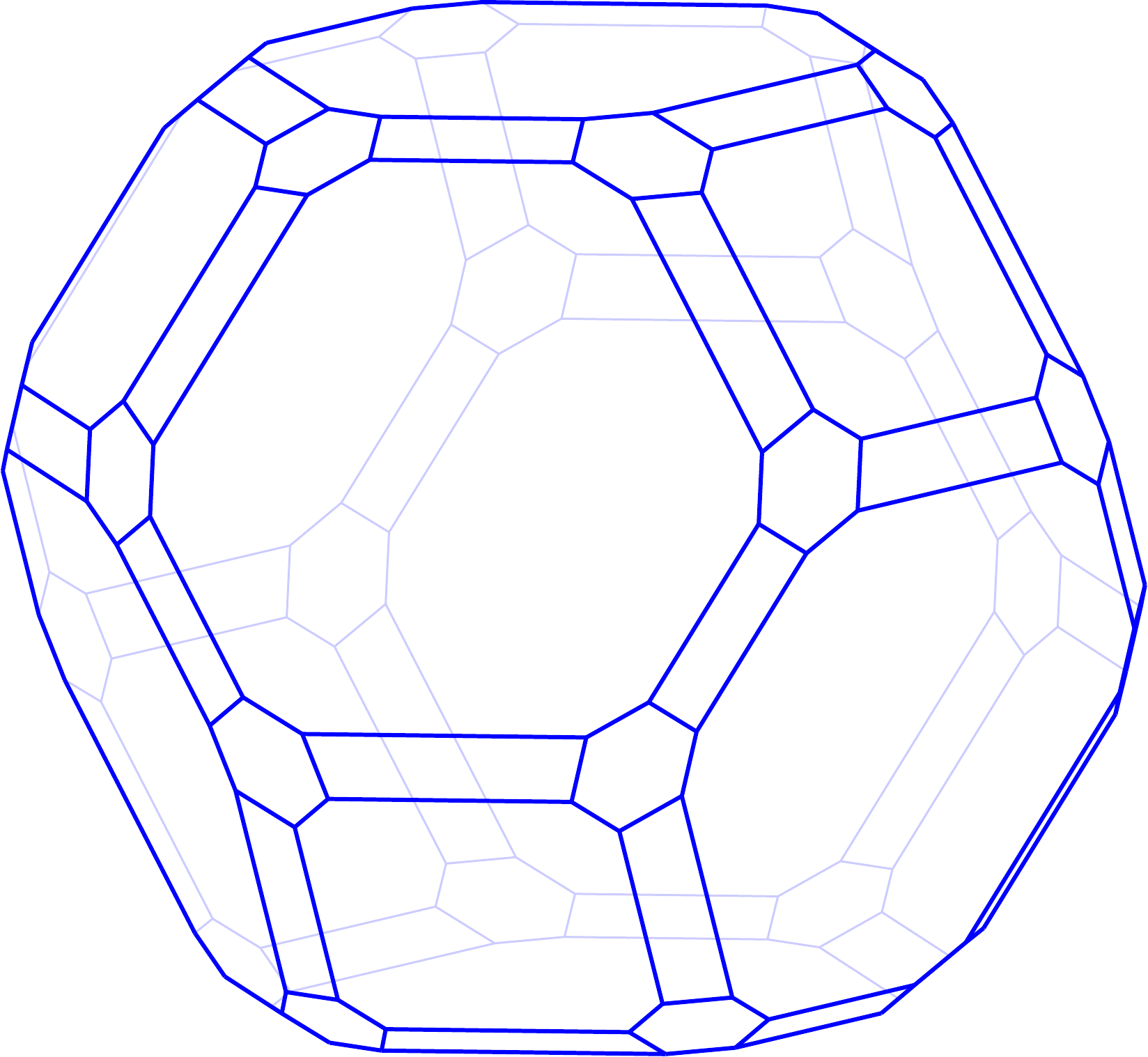}}
	\caption{Iterated permutahedra. Figure adapted from~\cite[Fig.~1]{CastilloLiu-permutoAssociahedron}.}
	\label{fig:permutopermutahedra}
\end{figure}
\end{example}

\begin{example}
The construction of~\cref{thm:stellarPolytopalRealization} was already applied in~\cite{Almeter} to realize the facial nested complexes of a simplicial polytope~$\polytope$ by face truncations of its polar~$\polytope^\triangle$.
For instance, the hyperoctahedral nested complexes discussed in~\cref{exm:facialBuildingSetCrossPolytope1,exm:facialNestedComplexCrossPolytope1} are realized by face truncations of a cube as in \cref{fig:facialNestedComplexesOctahedron}\,(right).
This recovers previous realizations of biassociahedra~\cite{BarnardReading} (see \cref{exm:facialNestedComplexCrossPolytope3}) and design graph associahedra~\cite{DevadossHeathVipismakul} (see \cref{exm:designNestedComplex}).
\end{example}


\section{Polytopal realizations}
\label{sec:polytopalRealizations}

The polytopal realization of \cref{thm:stellarPolytopalRealization} is not completely satisfactory, as it does not provide explicit coordinates.
Indeed, the stellar subdivisions (or dually the face truncations) rely on some parameters of the construction being sufficiently small, with no explicit bounds.
In this section, we exploit our embedding of the facial nested complex~$\nestedComplex[][\fbuilding, \OM(\b{A})]$ as an acyclic nested complex~$\acyclicNestedComplex(\building, \OM(\b{A}))$ to obtain explicit and combinatorially meaningful polytopal realizations.
Namely, we show that the acyclic faces of the nested complex~$\nestedComplex[][\building]$ can be selected by intersecting a suitable nestohedron~$\Nest(\building, \lambda)$ by the evaluation space of~$\b{A}$.
We first work in the space~$\R^\ground$ of the nestohedron of~$\building$ (\cref{subsec:acyclonestohedra1}) and then in the space~$\R\b{A}$ of a vector configuration~$\b{A}$ (\cref{subsec:acyclonestohedra2}).
We also discuss how this construction can be seen as an explicit version of \cref{thm:stellarPolytopalRealization} by commuting sections and truncations (\cref{subsec:commutingSectionsTruncations}).
Note that our realizations depend on the choice of~$\b{A}$: the same realizable oriented building set could be realized by different choices of vector configurations~$\b{A}$ which would produce distinct polytopal realizations.


\subsection{Acyclonestohedra in~$\R^\ground$}
\label{subsec:acyclonestohedra1}

In this section, we provide an explicit polytopal realization of the acyclic nested complex~$\acyclicNestedComplex(\building, \OM)$ as a section of a nestohedron for~$\building$.
We first fix some notations.

\begin{notation}
\label{not:vectorConfiguration}
Since each circuit~$c \in \circuits$ is the signature of a unique (up to rescaling) linear dependence~$\b{\delta} \in \dependences$, we can define
\begin{itemize}
\item $\b{H}_c^=$ as the hyperplane of~$\R^V$ satisfying the equation~$\dotprod{\b{\delta}}{\b{x}} = 0$,
\item $\b{H}_c^>$ as the open halfspace of~$\R^V$ satisfying the inequality~$\dotprod{\b{\delta}}{\b{x}} > 0$, and
\item~$r_c \eqdef \max \b{\delta}^{\ne 0} / \min \b{\delta}^{\ne 0}$ where~$\b{\delta}^{\ne 0} \eqdef \set{|\delta_s|}{s \in \ground} \ssm \{0\}$.
\end{itemize}
Since the dependence space~$\dependences$ is the orthogonal complement of the evaluation space~$\evaluations$, and is generated by the circuit dependences, we have
\[
\evaluations = \bigcap_{c \in \circuits} \b{H}_c^=.
\]
\end{notation}

\begin{definition}
\label{def:acyclonestohedron}
Consider~$\b{\rho} \eqdef (\rho_B)_{B \in \building} \in \R^\building_+$ given by~$\rho_B \eqdef 0$ if~$|B| = 1$ and~$\rho_B \eqdef R^{|B|}$ if~${|B| \ge 2}$, where~$R \eqdef |\building| \cdot \max_{c \in \circuits} r_c$.
The \defn{acyclonestohedron}~$\Acycl(\building, \b{A})$ is the intersection of the nestohedron~$\Nest(\building, \b{\rho})$ of \cref{def:nestohedron} with the evaluation space~$\evaluations$~of~$\b{A}$.
\end{definition}

\begin{theorem}
\label{thm:acyclonestohedron}
For any realizable oriented building set~$(\building, \OM(\b{A}))$, the acyclic nested complex $\acyclicNestedComplex(\building, \OM(\b{A}))$ is isomorphic to the boundary complex of the polar of the acyclonestohedron~$\Acycl(\building, \b{A})$.
\end{theorem}

\begin{remark}
Following \cref{exm:simplexAcyclicNestedComplex}, note that if~$\b{A}$ is linearly independent, then its evaluation space~$\evaluations$ is~$\R^S$, and the acyclonestohedra~$\Acycl(\building, \b{A})$ coincides with the classical nestohedron~$\Nest(\building, \b{\rho})$.
For instance, the acyclonestohedron of the graphical oriented building set of an oriented forest~$\digraph$ is the graph associahedron~of~$L(\digraph)$ (for instance, a permutahedron if~$\digraph$ is a star, and an associahedron if~$\digraph$~is~a~path).
\end{remark}

\begin{example}
Following on \cref{exm:orderComplexBoolean,exm:orderComplex,exm:barycentricSubdivision1,exm:barycentricSubdivision2}, we note that \cref{thm:acyclonestohedron} gives explicit coordinates to perform the barycentric subdivision of a polytope~$\polytope$, or dually, the omnitruncation of~$\polytope^\triangle$.
\end{example}

\begin{example}
\label{exm:acyclonestohedron}
Following \cref{exm:graphicalOrientedBuildingSet,exm:graphicalAcyclicNestedComplex}, \cref{fig:acyclonestohedra} illustrates the acyclonestohedra of the graphical acyclic nested complexes of \cref{fig:acyclicNestedComplexes}.
\begin{figure}
	\capstart
	\centerline{\includegraphics[scale=.45]{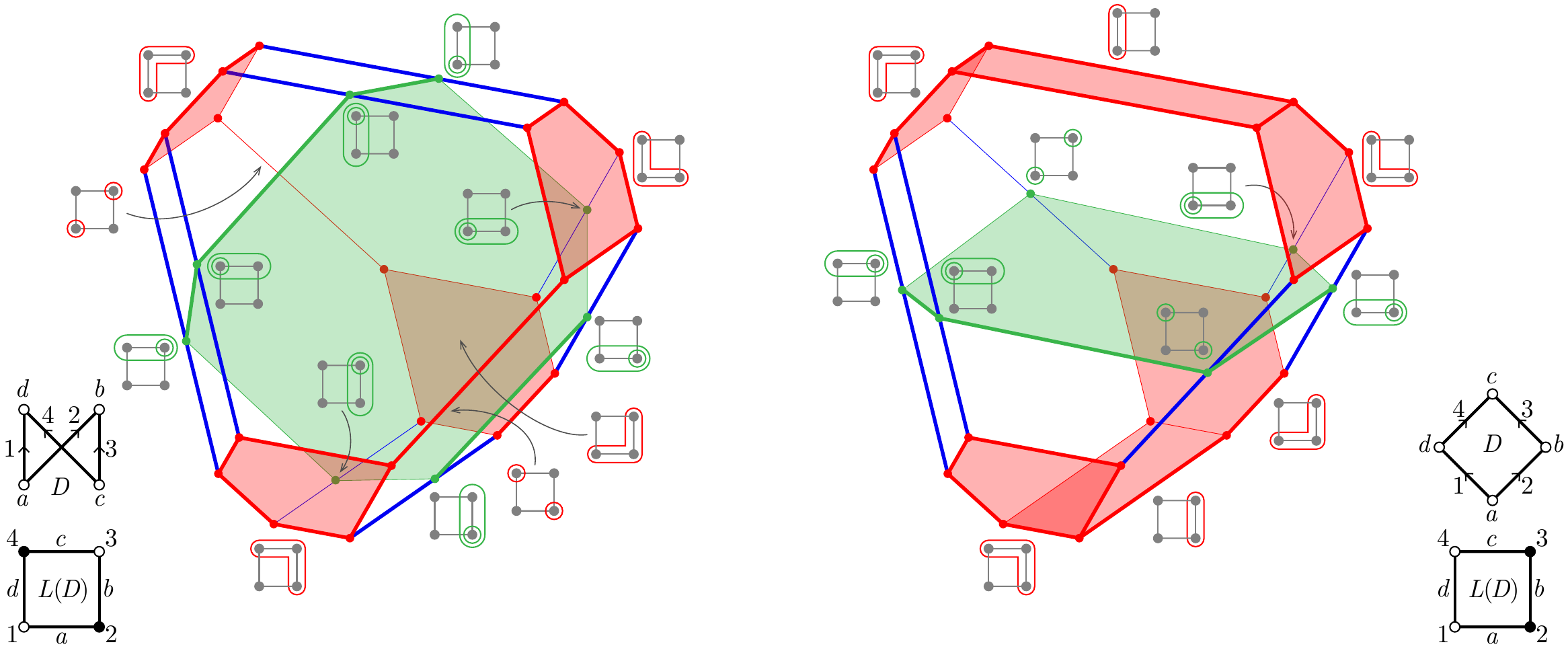}}
	\caption{The graphical acyclonestohedra (green polygons) realizing the graphical acyclic nested complexes of \cref{fig:acyclicNestedComplexes}, obtained as the section of the graph associahedron of the line graph~$L(\digraph)$ by the evaluation space of the graphical oriented matroid of~$\digraph$.}
	\label{fig:acyclonestohedra}
\end{figure}
\end{example}

We split the proof of \cref{thm:acyclonestohedron} into two statements.

\begin{lemma}
\label{lem:acyclonestohedron1}
If a nested set~$\nested$ on~$\building$ is not acyclic for~$\OM(\b{A})$, then the corresponding face of~$\Nest(\building, \b{\rho})$ does not intersect the evaluation space~$\evaluations$.
\end{lemma}

\begin{proof}
Assume that~$\nested$ is not acyclic for~$\OM(\b{A})$.
Then by \cref{lem:characterizationAcyclicNestedSets}\,(iii) there is~$c \in \circuits$ and~$\nested' \subseteq \nested$ with~${c_+ \not\subseteq \bigcup \nested'}$ while~${c_- \subseteq \bigcup \nested'}$.
Let~$\b{\delta} \in \dependences$ be such that~$c = \signature(\b{\delta})$.
Assume that~$\nested'$ is inclusion minimal, and define~${X \eqdef \underline{c} \cup \bigcup \nested'}$.
By minimality of~$\nested'$, any~$B \in \nested'$ intersects~$c_-$.
Hence, $X \in \building$ since all elements of~$\nested'$ are in~$\building$ and intersect~$\underline{c}$ which is also in~$\building$.
Consider a maximal nested set~$\nested''$ containing~$\nested$.
By \cref{prop:vertexDescriptionNestohedron}, the corresponding vertex~$\vertexNest[\nested''][\b{\rho}]$ of~$\Nest(\building, \b{\rho})$ satisfies
\begin{align*}
\dotprod{\b{\delta}}{\vertexNest[\nested''][\b{\rho}]}
& = \sum_{s \in c_+} \sum_{\substack{B \in \building \\ s \in B \subseteq \block{s}{\nested''}}} |\delta_s| \cdot \rho_B - \sum_{s \in c_-} \sum_{\substack{B \in \building \\ s \in B \subseteq \block{s}{\nested''}}} |\delta_s| \cdot \rho_B \\
& \ge \min \b{\delta}^{\ne 0} \rho_X - \sum_{\substack{B \in \building \\ B \subseteq \bigcup \nested'}} \max \b{\delta}^{\ne 0} \rho_B \\
& \ge \min \b{\delta}^{\ne 0} R^{|X|} - (|\building|-1) \cdot \max \b{\delta}^{\ne 0} \cdot R^{\max_{B \in \nested'} |B|}
> 0
\end{align*}
as~$R \ge |\building| \cdot r_c > (|\building|-1) \cdot \max \b{\delta}^{\ne 0} / \min \b{\delta}^{\ne 0}$ and~$|X| \ge \max_{B \in \nested'} |B| + 1$.
Hence, ${\vertexNest[\nested''][\b{\rho}] \in \b{H}_c^>}$ for every maximal nested set~$\nested''$ containing~$\nested$.
We conclude that the entire face of~$\Nest(\building, \b{\rho})$ corresponding to~$\nested$ is in~$\b{H}_c^>$, thus disjoint from~$\evaluations = \bigcap_{c \in \circuits} \b{H}_c^=$.
\end{proof}

\begin{lemma}
\label{lem:acyclonestohedron2}
If a nested set~$\nested$ on~$\building$ is acyclic for~$\OM(\b{A})$, then the corresponding face of~$\Nest(\building, \b{\rho})$ intersects the evaluation space~$\evaluations$.
\end{lemma}

\begin{proof}
Assume that~$\nested$ is a maximal acyclic nested set for~$(\building, \OM(\b{A}))$.
By \cref{coro:rank1}, the oriented matroid~$\restrOM$ has rank~$1$ for each~${B \in \nested}$, hence
\[
\rank = \rank[\OM(\b{A})] = \sum_{B \in \nested} \rank[\restrOM] = |\nested|.
\]
For each~$B \in \nested$, choose a spanning tree~$T_{B \in \nested}$ of the complete graph on~$\restrGround$.
Each edge~$(r,s)$ of~$T_{B \in \nested}$ forms a circuit of~$\restrOM$.
Thus, there exists a circuit~$c_{r,s}$ of~$\OM(\b{A})$ contained in~$B$ which contracts to the circuit~$(r, s)$ of~$\restrOM$.
Let~$\c{X}_{B \in \nested} \eqdef \set{c_{r,s}}{(r,s) \in T_{B \in \nested}}$ and consider $\c{X} = \bigcup_{B \in \nested} \c{X}_{B \in \nested}$.
As~$\c{X}$ is clearly linearly independent and has cardinality
\[
\sum_{B \in \nested} |T_{B \in \nested}| = \sum_{B \in \nested} |\restrGround|-1 = |\ground| - |\nested| = \corank = \dim(\dependences),
\]
the circuits of~$\c{X}$ form a basis of the dependence space~$\dependences$.
Hence, $\evaluations = \bigcap_{c \in \c{X}} \b{H}_c^=$.

To prove that the face of~$\Nest(\building, \b{\rho})$ corresponding to~$\nested$ intersects~$\evaluations{}$, it thus suffices to show that any region of the arrangement of hyperplanes~$\set{\b{H}_c^=}{c \in \c{X}}$ contains the vertex~$\vertexNest[\nested'][\b{\rho}]$ of some maximal nested set~$\nested'$ on~$\building$ containing~$\nested$.
Such a region corresponds to an orientation~$\c{O}$ of the forest formed by the trees~$T_{B \in \nested}$ for~$B \in \nested$.
Add to~$\c{O}$ all arcs~$i \to j$ for~$i \in B$ and~$j \in B'$ such that~$B,B' \in \nested$ with~$B \subsetneq B'$.
The resulting graph is acyclic, so we can consider a linear extension~$\pi$.
Then the vertex~$\vertexNest[\nested_\pi][\b{\rho}]$ of the nested set~$\nested_\pi \eqdef \bigcup_{i \in [p]} \connectedComponents(\building_{|\pi_1 \cup \dots \cup \pi_i})$ lies in the region defined by~$\c{O}$. 
\end{proof}

\begin{proof}[Proof of \cref{thm:acyclonestohedron}]
Immediate from \cref{lem:acyclonestohedron1,lem:acyclonestohedron2}.
\end{proof}


\subsection{Commuting sections and truncations}
\label{subsec:commutingSectionsTruncations}

We make a brief interlude to add some observations connecting the polytopal realizations of \cref{thm:stellarPolytopalRealization,thm:acyclonestohedron}.
Consider a polytope~$\polytope$ and its homogenized vector configuration~$\b{A}_{\polytope}$.
Let~$(\building, \OM(\b{A}_{\polytope}))$ be a connected oriented building set.
Then 
\begin{itemize}
\item the polar~$\polytope^\triangle$ is the section of the positive orthant~$\R_{\ge0}^d$ with the evaluation space~$\evaluations[\b{A}_{\polytope}]$,
\item the acyclonestohedron~$\Acycl(\building, \b{A}_{\polytope})$ is the section of the nestohedron~$\Nest(\building, \b{\rho})$ with the evaluation space~$\evaluations[\b{A}_{\polytope}]$.
\end{itemize}
In other words, intersecting with the evaluation space~$\evaluations[\b{A}_{\polytope}]$ geometrically embeds both the face lattice of~$\polytope^\triangle$ in the boolean lattice and the acyclic nested complex~$\acyclicNestedComplex(\building, \OM(\b{A}_{\polytope}))$ in the boolean nested complex~$\nestedComplex[][\building]$.

Building on this observation, we note that we can obtain the acyclonestohedron~$\Acycl(\building, \b{A}_{\polytope})$ from the simplex in two ways, by commuting sections and face truncations:
\begin{enumerate}
\item either we first truncate faces of the simplex to obtain the nestohedron~$\Nest(\building, \b{\rho})$, and \linebreak we then intersect with the evaluation space~$\evaluations[\b{A}_{\polytope}]$ to obtain the acyclonestohedron \linebreak $\Acycl(\building, \b{A}_{\polytope})$, as presented in \cref{thm:acyclonestohedron},
\item or we first intersect the simplex with evaluation space~$\evaluations[\b{A}_{\polytope}]$ to obtain the polar~$\polytope^\triangle$, and we then truncate faces of~$\polytope^\triangle$ (according to the facet supporting hyperplanes of the nestohedron~$\Nest(\building, \b{\rho})$ which intersect~$\evaluations[\b{A}_{\polytope}]$) to obtain the acyclonestohedron~$\Acycl(\building, \b{A}_{\polytope})$, thus obtaining explicit parameters to realize the construction mentioned in \cref{thm:stellarPolytopalRealization}.
\end{enumerate}
We can thus see \cref{thm:acyclonestohedron} as an effective version of \cref{thm:stellarPolytopalRealization}, with two major advantages.
First, it gives explicit and combinatorial meaningful coordinates for the truncations.
Second these truncations are compatible with the embedding of the acyclic nested complex~$\Acycl(\building, \b{A}_{\polytope})$ in the boolean nested complex~$\nestedComplex[][\building]$.

Finally, note that in the first way, the face truncations of the simplex corresponding to building blocks of~$\building \ssm \FL(\polytope)$ are superfluous as their supporting hyperplanes do not intersect~$\evaluations[\b{A}_{\polytope}]$.
Our construction of \cref{thm:acyclonestohedron} still performs these truncations as we want to work with a boolean building set~$\building$ containing~$\building \cap \FL(\polytope)$.

\begin{example}
Following \cref{exm:graphicalOrientedBuildingSet,exm:graphicalAcyclicNestedComplex,exm:acyclonestohedron}, \cref{fig:acyclonestohedraSectionsTruncations} illustrates the acyclonestohedra of \cref{fig:acyclonestohedra} obtained either by face truncations and section, or by section and face truncations.
\begin{figure}
	\capstart
	\centerline{\includegraphics[scale=.45]{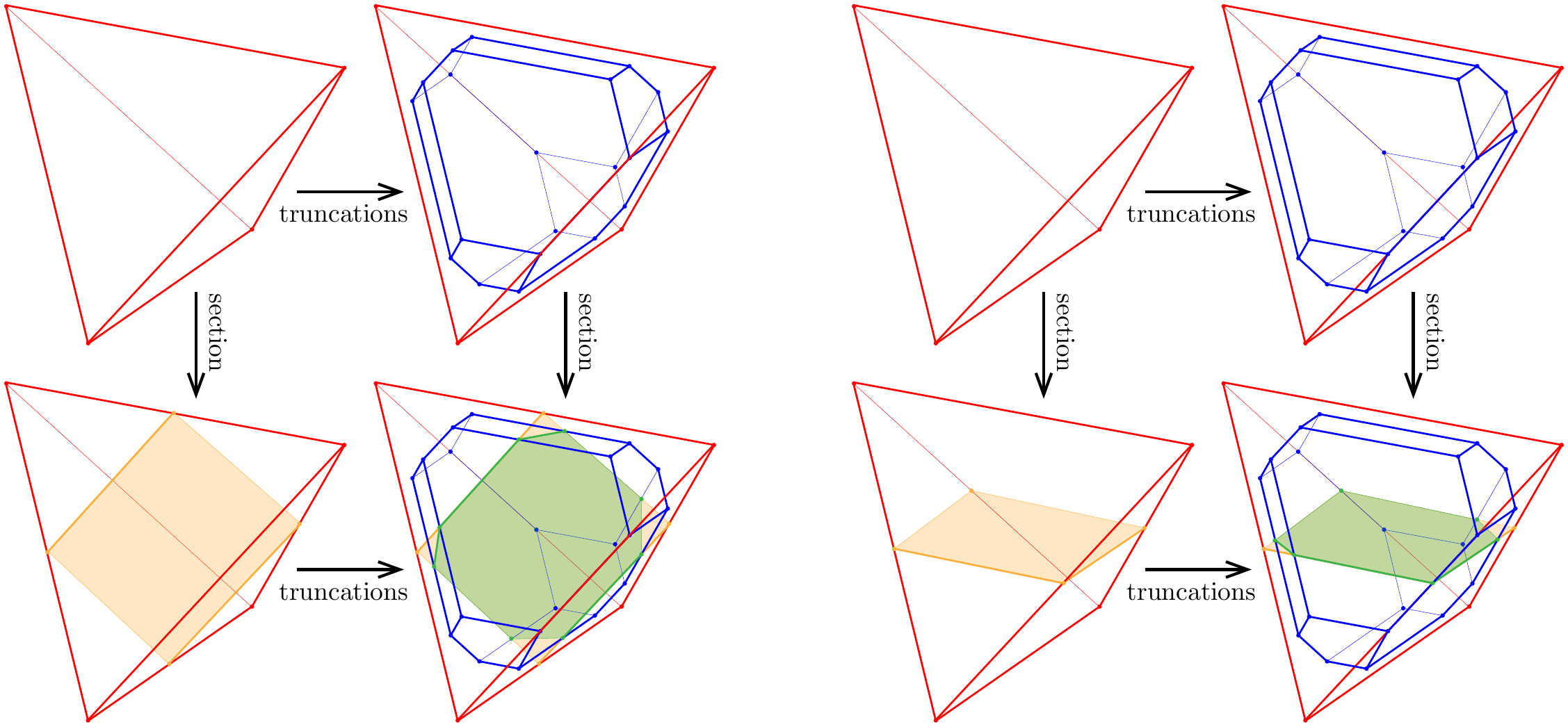}}
	\caption{The graphical acyclonestohedra (green) of \cref{fig:acyclonestohedra}, can be obtained from the simplex (red) in two equivalent ways, commuting the sequence of face truncations defining the nestohedron (blue) with the section defining the polar polytope (orange). See \cref{subsec:commutingSectionsTruncations} for details.}
	\label{fig:acyclonestohedraSectionsTruncations}
\end{figure}
\end{example}


\subsection{Acyclonestohedra in~$\R\b{A}$}
\label{subsec:acyclonestohedra2}

In this section, we give an alternative realization of the acyclic nested complex~$\acyclicNestedComplex(\building, \OM(\b{A}))$ by a polytope~$\AcyclProj(\building, \b{A})$ in the vector space~$\R\b{A}$ generated by our vector configuration~$\b{A}$. This polytope~$\AcyclProj(\building, \b{A})$ is obtained as an affine image of the acyclonestohedron~$\Acycl(\building, \b{A})$.
This was inspired from a construction of A.~Sack~\cite{Sack}.

\begin{definition}
\label{def:acyclonestohedronProj}
Consider the coefficients~$\b{\rho} \in \R^\building$ given in \cref{def:acyclonestohedron}.
The \defn{acyclonestohedron}~$\AcyclProj(\building, \b{A})$ is the polytope of~$\R\b{A}$ defined by the equalities~$\overline{g}_B(\b{y}) = 0$ for all~$B \in \connectedComponents(\building)$ and the inequalities~$\overline{g}_B(\b{y}) \ge 0$ for all~$B \in \building$, where
\[
\overline{g}_B(\b{y}) \eqdef \bigdotprod{\sum_{b \in B} \b{a}_b}{\b{y}} - \sum_{\substack{B' \in \building \\ B' \subseteq B}} \rho_{B'}.
\]
\end{definition}

\begin{proposition}
\label{prop:affineMapRealizations}
The acyclonestohedra~$\Acycl(\building, \b{A})$ of \cref{def:acyclonestohedron} and~$\AcyclProj(\building, \OM(\b{A}))$ of \cref{def:acyclonestohedronProj} are affinely equivalent.
\end{proposition}

\begin{proof}
We still denote by~$(\b{e}_s)_{s \in \ground}$ the standard basis of~$\R^\ground$.
Let~$\pi : \R^\ground \to \evaluations$ denote the projection on~$\evaluations$ parallel to~$\dependences$.
Consider the vectors~$\b{f}_s \eqdef \pi(\b{e}_s)$ of~$\evaluations$.

Observe that, for any dependence~$\b{\delta} \in \dependences$, we have
\[
\sum_{s \in \ground} \delta_s \b{f}_s = \sum_{s \in \ground} \delta_s \pi(\b{e}_s) = \pi\bigl(\sum_{s \in \ground} \delta_s \b{e}_s \bigr) = \pi(\b{\delta}) = \zero
\]
since the projection~$\pi$ is parallel to~$\dependences$.
As the dimension of~$\evaluations$ is the rank of~$\b{A}$, we obtain that the vectors~$(\b{f}_s)_{s \in \ground}$ have precisely the same dependences as the vectors~$(\b{a}_s)_{s \in \ground}$.
We can thus define an invertible linear map~$\Psi : \R\b{A} \to \evaluations$ such that~$\Psi(\b{a}_s) = \b{f}_s$ for all~$s \in \ground$.
We denote by~$\Psi^* :  \evaluations \to \R\b{A}$ the adjoint of~$\Psi$, \ie such that~$\dotprod{\Psi(\b{a})}{\b{x}} = \dotprod{\b{a}}{\Psi^*(\b{x})}$ for any~$\b{a} \in \R\b{A}$ and~$\b{x} \in \evaluations$.

We now observe that the map~$\Psi^*$ sends the acyclonestohedron~$\Acycl(\building, \b{A})$ to the acyclonestohedron~$\AcyclProj(\building, \OM(\b{A}))$.
Indeed, recall from \cref{prop:facetDescriptionNestohedron} that the inequality of the nestohedron~$\Nest(\building, \b{\rho})$ corresponding to a building block~$B \in \building$ is given by~$\bigdotprod{\sum_{b \in B} \b{e}_b}{\b{x}} \ge \Omega_B$ where~$\Omega_B \eqdef \sum_{B' \in \building, B' \subseteq B} \rho_{B'}$.
Observe now that~$\b{x} \in \evaluations$, we have
\[
\bigdotprod{\sum_{b \in B} \b{e}_b}{\b{x}} = \bigdotprod{\sum_{b \in B} \b{f}_b}{\b{x}} = \bigdotprod{\Psi\bigl(\sum_{b \in B} \b{a}_b\bigr)}{\b{x}} = \bigdotprod{\sum_{b \in B} \b{a}_b}{\Psi^*(\b{x})}.
\]
Hence, the inequality~$\bigdotprod{\sum_{b \in B} \b{e}_b}{\b{x}} \ge \Omega_B$ of the acyclonestohedron~$\Acycl(\building, \b{A})$ becomes the inequality~$\bigdotprod{\sum_{b \in B} \b{a}_b}{\b{y}} \ge \Omega_B$ of the acyclonestohedron~$\AcyclProj(\building, \OM(\b{A}))$.
The same observation holds replacing the inequalities corresponding to the blocks of~$\building$ by the equalities corresponding to the connected components of~$\building$.
\end{proof}

\begin{theorem}
\label{thm:alternativeRealization}
For any realizable oriented building set~$(\building, \OM(\b{A}))$, the acyclic nested complex $\acyclicNestedComplex(\building, \OM(\b{A}))$ is isomorphic to the boundary complex of the polar of the acyclonestohedron~$\AcyclProj(\building, \b{A})$.
\end{theorem}

\begin{proof}
Immediate consequence of \cref{thm:acyclonestohedron,prop:affineMapRealizations}.
\end{proof}

\begin{remark}\label{rmk:redundantinequalitiesacyclonestohedron}
Note that in \cref{def:acyclonestohedronProj}, the inequalities given by the acyclic blocks~$B$ are redundant, since the corresponding facets of~$\Nest(\building, \b{\rho})$ are not intersected by the evaluation space~$\evaluations$.
We could thus just delete them from the inequalities list.
In other words, the irredundant inequalities in \cref{def:acyclonestohedronProj} correspond to the blocks of~$\building \cap \FL(\OM)$.
\end{remark}


\section{Compactifications}
\label{sec:compactifications}

In the celebrated paper \cite{DeConciniProcesi1995}, C.~De Concini and C.~Procesi used nested sets on the lattice of flats of a subspace arrangement in a projective space to study \defn{wonderful compactifications} of its complement via sequences of projective blow-ups. 
Compactifications of interiors of polytopes have also been studied in different guises. 
In particular, the configuration space of $n$ ordered points in the real line is the interior of a simplex, and the (Stasheff) associahedron models its (real) W.~Fulton and R.~MacPherson's compactification \cite{FultonMacPherson1994, Devadoss2004}. 
P.~Galashin's main motivation for defining poset associahedra~\cite{Galashin} was that they model compactifications of the space of order preserving maps $P\to \R$, which can be identified with the interior of an order polytope, thus generalizing the aforementioned construction of the (Stasheff) associahedron.
G.~Gaiffi gave in~\cite{Gaiffi2003} general models for real compactifications of arrangements of subspaces and halfspaces, which includes in particular the case of polyhedral cones.
His work shows that all facial nested complexes model nice compactifications of interiors of polytopes, obtained via sequences of real blow-ups (\ie the ball, beams, and plates constructions of~\cite{KuperbergThurston}), which we now describe.

Let $\polytope\subseteq \R^d$ be a polytope. For simplicity, we will assume that it is full-dimensional and that $\b 0\in \inter{\polytope}$, where $\inter{\polytope}$ denotes the interior of~$P$.
We consider the vector configurations~$\b{A}_{\polytope}$ and~$\b{A}_{\polytope^\triangle}$ associated to~$\polytope$ and its polar~$\polytope^\triangle$ (as in \cref{exm:coneAcyclicRealizableOM}). 
Let $\{H_s\}_{s\in \ground}$ be the facet-defining hyperplanes of~$\cone(\b{A}_{\polytope})$ (which is the homogenization of $\polytope$). 

Let~$\fbuilding$ be  a connected facial building set of~$\OM(\b{A}_{\polytope^\triangle})$. For every $F\in \fbuilding$, we denote by $V_F \eqdef \bigcap_{s\in F}H_s$, and by $A_F \eqdef V_F^\perp$ its orthogonal subspace. Finally, let $\sph[{F}]\eqdef \sph[d]\cap A_F$, where $\sph[{d}]$ is the unit sphere. 

For a point $p\in \inter{\polytope}$, we denote by $p\mapsto \tilde p \eqdef\binom{p}{1}$ its homogenization in $\R^{d+1}$. For each $F\in \fbuilding$, let $\pi_F$ denote the orthogonal projection onto~$A_F$, and let $\tilde \pi_F: \inter{\polytope} \to \sph[F]$ be the map $p \mapsto \frac{\pi_F(\tilde p)}{\|\pi_F(\tilde p)\|}$, which is well-defined on the interior of $\polytope$. Finally, let $\rho \eqdef \prod_{F \in \fbuilding} \tilde \pi_F$ be the product map:
\[
\rho:
\begin{array}[t]{lcl}
\inter{\polytope} & \longrightarrow & \displaystyle \prod_{F \in \fbuilding} \sph[{F}] \\
p & \longmapsto & \bigl( \tilde \pi_F(p) \bigr)_{F \in \fbuilding}
\end{array}
\]
 
The closure of the image $\rho(\inter{\polytope})$ of $\inter{\polytope}$ under this map is a compactification $\polytope^{\fbuilding}$ of $\inter{\polytope}$ (note that the factor $\tilde \pi_\ground$ corresponding to $\ground\in \fbuilding$ is a diffeomorphism between $\inter{\polytope}$ and its image), that has many interesting properties analogous to the De Concini--Procesi wonderful models~\cite{DeConciniProcesi1995}. 

\begin{theorem}[{\cite[Sect.~6.1 \& Thm.~11.1]{Gaiffi2003}}]
\label{thm:compactification}
Consider a polytope~$\polytope$ with polar~$\polytope^\triangle$, and a connected facial building set~$\fbuilding$ of~$\FL(\polytope^\triangle)$.
Then there is a compactification $\polytope^{\fbuilding}$ of the interior of~$\polytope$ that is a stratified $C^\infty$ manifold with corners such that
\begin{enumerate}[(i)]
 \item except for the open dense stratum, all the strata lie in the boundary,
 \item the codimension~$1$ strata are in correspondence with the facial blocs of $\fbuilding$,
 \item the intersection of the closures of the strata indexed by a subset $ \nested\subseteq\fbuilding$ is non-empty if and only if $\nested$ is a $\FL(\polytope^\triangle)$-nested set,
 \item the strata of $\polytope^{\fbuilding}$ are indexed by the faces of the facial nested complex~$\nestedComplex[\FL(\polytope^\triangle)][\fbuilding]$.
\end{enumerate}
\end{theorem}

In fact, in the second part of~\cite{Gaiffi2003}, G.~Gaiffi generalizes his constructions to arbitrary conically stratified $C^\infty$ manifolds with corners, those that locally look like arrangements of linear subspaces and halfspaces. Therefore, an analogue of \cref{thm:compactification} also exists for a more general class of oriented matroids, those which are representable by $C^\infty$ pseudosphere arrangements (the Topological Representation Theorem \cite{FolkmanLawrence1978}\cite[Sect.~1.4]{BjornerLasVergnasSturmfelsWhiteZiegler} states that every oriented matroid can be represented by an arrangement of pseudospheres, although as far as we know, it is not known whether one can further require them to be $C^\infty$).

A recent alternative compactification of polytope interiors via projective blow-ups as been introduced in~\cite{BraunerEurPrattVlad} under the name \defn{wondertopes}.
They are defined as regions in De Concini--Procesi wonderful compactifications of hyperplane arrangement complements, and it is proven in~\cite{BraunerEurPrattVlad} that they have a positive geometry structure in the sense of~\cite{ArkaniHamedBaiLam-PositiveGeometries}.
The boundary structure of wondertopes is described in~\cite[Sect.~4.1]{BraunerEurPrattVlad} via the links of its elements.
As it can be verified by comparing~\cite[Lem.~4.3]{BraunerEurPrattVlad} with \cref{prop:singleLinkFacialBuildingSet}, the facial nested complex provides an explicit combinatorial description of the boundary structure of wondertopes.

Hence, facial nested complexes model the two compactifications of~\cite{Gaiffi2003} and~\cite{BraunerEurPrattVlad}, while the compactification techniques are slightly different.
This is not surprising as facial nested complexes record sequences of combinatorial blow-ups, which were designed to model geometric blow-ups.
In fact, the compactifications of~\cite{Gaiffi2003} and~\cite{BraunerEurPrattVlad} are closely related, as discussed in~\cite{Gaiffi2004}.


\clearpage
\part{Poset associahedra and affine poset cyclohedra}
\label{part:posetAssociahedra}

In this part, we show that the poset associahedra (\cref{sec:posetAssociahedra}) and affine poset cyclohedra (\cref{sec:affinePosetCyclohedra}) defined by P.~Galashin in~\cite{Galashin} can be interpreted as specific acyclic nested complexes.
Applying \cref{sec:orientedMatroidRealizations}, we thus recover the polytopal realizations of~\cite[Thm.~2.1]{Galashin} by stellar subdivisions of order polytopes.
Applying \cref{sec:polytopalRealizations}, we obtain nice explicit realizations as graphical acyclonestohedra.
Applying \cref{sec:compactifications}, we also recover stratified compactifications of the order polytopes akin to~\cite[Thm.~1.9\,\&\,1.12]{Galashin}.
We note that nice polytopal realizations for poset associahedra (with explicit integer inequality descriptions) were also independently obtained by A.~Sack~\cite{Sack}, and actually motivated our construction of acyclonestohedra in~$\R\b{A}$ (see \cref{rem:Sack} for more details).
It would be interesting to know if other results on the combinatorial structure of poset associahedra~\cite{NguyenSack1,NguyenSack2,Nguyen} can be extended in a meaningful way to all acyclonestohedra.


\section{Poset associahedra}
\label{sec:posetAssociahedra}

Here, we consider a classical finite poset~$(\poset, \preccurlyeq)$, whose Hasse diagram~$H(\poset)$ is connected.
(Note that the case of disconnected posets can be treated using \cref{coro:connectedLatticeBuildingSet,coro:connectedFacialBuildingSet}, we stick here with the convention of~\cite{Galashin} which slightly simplifies the exposition.)
We first recall the definition of poset associahedron from~\cite{Galashin}.

\begin{definition}[{\cite[Sect.~1.1]{Galashin}}]
\label{def:posetAssociahedron}
A \defn{pipe} of~$\poset$ is a subset~$Q$ of~$\poset$ with~$|Q| > 1$ which induces a connected subgraph of~$H(\poset)$ and which is order convex (\ie such that $p \preccurlyeq q \preccurlyeq r$ with $p,r\in Q$ implies $q\in Q$).
We denote by~$\c{D}(\poset)$ the directed graph with a vertex for each pipe of~$\poset$, with an arc joining a pipe~$Q$ to a pipe~$R$ if~$Q \cap R = \varnothing$ and there is~$q \in Q$ and~$r \in R$ such that~$q \prec r$.
A \defn{piping} of~$\poset$ is a set of pairwise nested or disjoint pipes of~$\poset$, which contains the pipe~$\poset$ itself, and which induces an acyclic subgraph of~$\c{D}(\poset)$ (see \cref{fig:exmPosetNested}).
The \defn{piping complex} of~$\poset$ is the simplicial complex~$\pipingComplex(\poset)$ whose faces are $\piping \ssm \{\poset\}$ for all pipings~$\piping$ of~$\poset$ (see \cref{fig:posetAssociahedraCombi}).
A \defn{poset associahedron} of~$\poset$ is a simple polytope whose dual has boundary complex isomorphic to the piping complex~$\pipingComplex(\poset)$ (see \cref{fig:posetAssociahedraCombi}).
\begin{figure}[b]
	\capstart
	\centerline{\includegraphics[scale=.5]{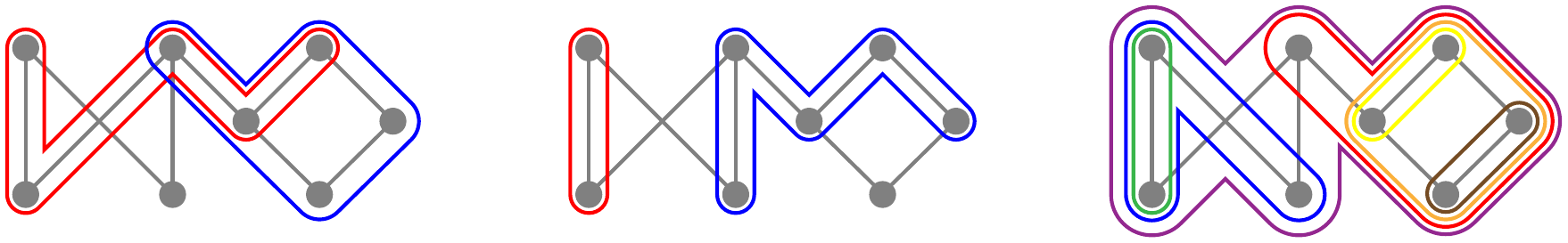}}
	\caption{Some incompatible pipes (left and middle), and a maximal piping (right) on a poset.}
	\label{fig:exmPosetNested}
\end{figure}
\end{definition}

We now interpret the piping complexes as acyclic nested complexes.

\begin{definition}
\label{def:posetOrientedBuildingSet}
The \defn{poset oriented building set}~$(\building(\poset), \OM(\poset))$ of~$\poset$ is the graphical oriented building set of \cref{exm:graphicalOrientedBuildingSet} for its Hasse diagram~$H(\poset)$.
That is,  $\building(\poset)$ is the graphical building set of the line graph~$L(\poset)$ of~$H(\poset)$ and~$\OM(\poset)$ is the graphical oriented matroid of~$H(\poset)$.
\end{definition}

\begin{proposition}
\label{prop:pipingComplexAsAcyclicNestedComplex}
The piping complex~$\pipingComplex(\poset)$ is isomorphic to the acyclic nested complex \linebreak $\acyclicNestedComplex(\building(\poset), \OM(\poset))$ of the poset oriented building set of~$\poset$.
\end{proposition}

\begin{proof}
For a subset~$Q \subseteq \poset$, denote by~$H(Q)$ the subgraph of the Hasse diagram~$H(\poset)$ induced by~$Q$, by~$\varepsilon(Q)$ the arcs of~$H(Q)$, and by~$L(Q)$ the subgraph of the line graph~$L(\poset)$ of~$H(\poset)$ induced by~$\varepsilon(Q)$.
Observe that
\begin{itemize}
\item $H(Q)$ is connected if and only if $L(Q)$ is connected,
\item if~$H(Q)$ is connected, then $|Q| > 1$ if and only if~$|\varepsilon(Q)| > 0$.
\end{itemize}

Hence, the map~$Q \mapsto \varepsilon(Q)$ is an injection from pipes of~$\poset$ to tubes of~$L(\poset)$. 
The tubes of~$L(\poset)$ that are not in the image of $\varepsilon$ correspond to subgraphs of~$H(\poset)$ that are either not induced or not order convex. 
If $\tube$ is one of these tubes, then there exist $i \prec j \in \poset$ and a directed path $\vec{\pi}$ from $i$ to $j$ such that $i$ and~$j$ belong to $H(\tube)$ but not all the edges of $\vec{\pi}$ do. Since $H(\tube)$ is connected, there is a (undirected) path $\pi'$ between $i$ and $j$ in $H(\tube)$. 
Combining $\vec{\pi}$ and $\pi'$, we get a circuit of~$\OM(\poset)$ such that $c_+ \subseteq \pi' \subset \tube$ while $\vec{\pi} \subseteq c_- \not\subseteq \tube$. 
Conversely, note that all the tubes of $L(\poset)$ that are not acyclic correspond to non-convex or non-induced subgraphs of $H(\poset)$.
Hence,~$Q \mapsto \varepsilon(Q)$ is a bijection between pipes of~$\poset$ and tubes of~$L(\poset)$ which are acyclic for~$\OM(\poset)$.

Moreover, observe that
\begin{itemize}
\item $Q \subseteq R$ if and only if~$\varepsilon(Q) \subseteq \varepsilon(R)$,
\item $Q$ and $R$ are disjoint if and only if $\varepsilon(Q)$ and $\varepsilon(R)$ are disjoint and non-adjacent,
\item $Q_1, \dots, Q_k$ form a directed cycle of~$\c{D}(\poset)$ if and only if there is a circuit~$c$ of~$\OM(\poset)$ such that $c_+ \subseteq \varepsilon(Q_1) \cup \dots \cup \varepsilon(Q_k)$ but $c_- \not\subseteq \varepsilon(Q_1) \cup \dots \cup \varepsilon(Q_k)$.
\end{itemize}
Hence, the map~$\piping \mapsto \set{\varepsilon(\pipe)}{\pipe \in \piping}$ is a bijection from pipings of~$\poset$ to tubings of~$L(\poset)$ which are acyclic for~$\OM(\poset)$.
This bijection is illustrated in \cref{fig:posetAssociahedraCombi}.
\end{proof}

\begin{example}
For instance, if the Hasse diagram~$H(\poset)$ is a tree, then the poset associahedron of~$\poset$ is just the graph associahedron of the line graph~$L(\poset)$, see \cref{exm:simplexAcyclicNestedComplex,exm:simplexAcyclicNestedComplexGraphical}.
We thus obtain the permutahedron when~$H(\poset)$ is a star, and the classical associahedron when~$H(\poset)$ is a line (these examples were already observed in~\cite[Sect.~1.1]{Galashin}).
In contrast, when~$H(\poset)$ is a cycle, $L(\poset)$ is also a cycle, but the poset associahedron of~$\poset$ is a section of the cyclohedron with the hyperplane normal to the unique circuit of~$H(\poset)$.
The two possible poset associahedra when~$H(\poset)$ is a $4$-cycle are represented in \cref{fig:posetAssociahedraCombi}, and obtained as sections in~\cref{fig:acyclonestohedra}.
\cref{fig:posetAssociahedron3,fig:posetAssociahedron4} illustrate two $3$-dimensional poset associahedra.
\begin{figure}
	\capstart
	\centerline{\includegraphics[scale=.45]{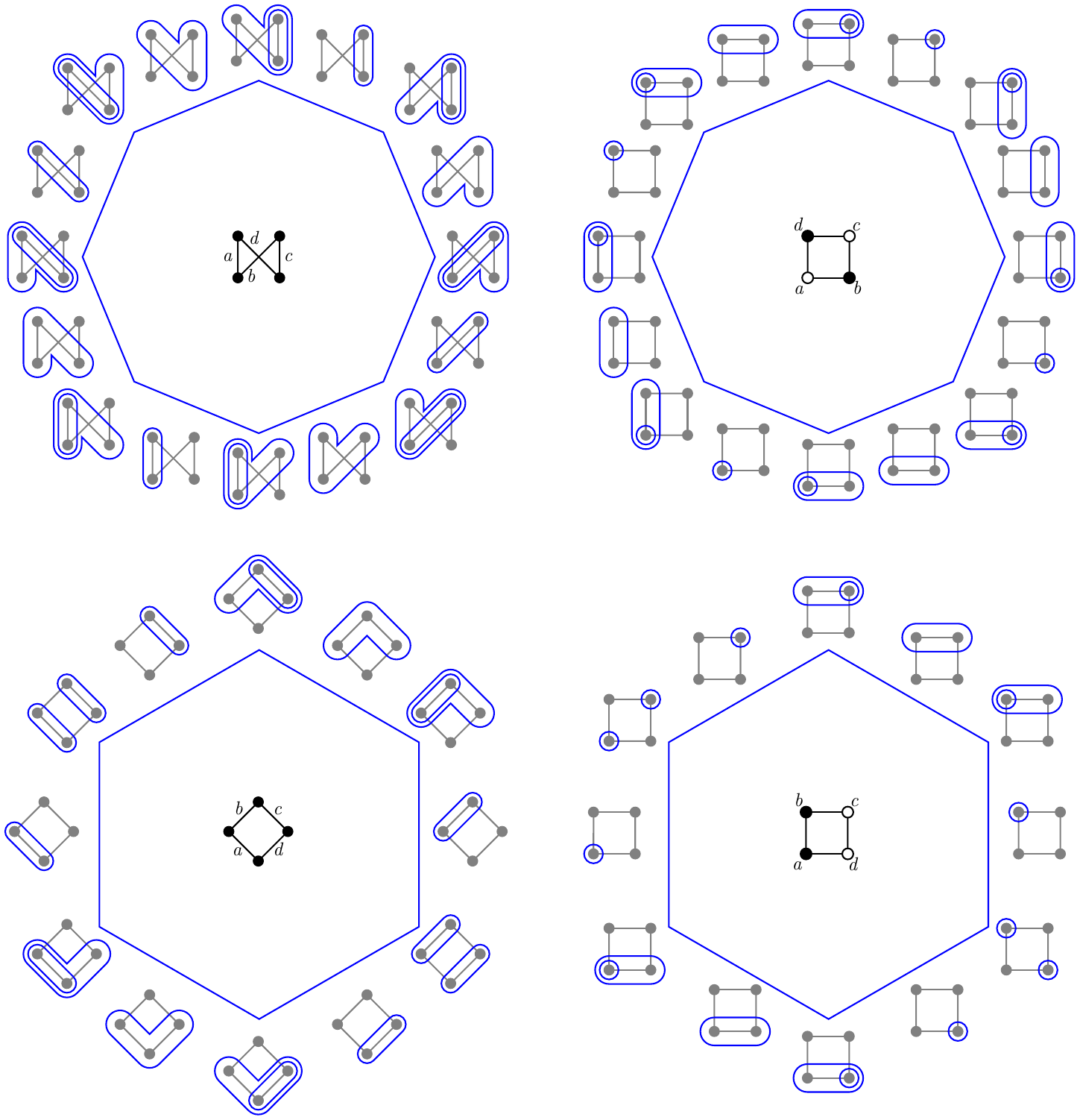}}
	\caption{The isomorphism between the piping complex of~$\poset$ (left) and the poset acyclic nested complex of~$(\building(\poset), \OM(\poset))$ (right) for two different posets~$\poset$. For both posets, the line graph~$L(\poset)$ is the $4$-cycle, and the signs of the unique circuit of~$\OM(\poset)$ are indicated by the black and white vertices of~$L(\poset)$. The pipes of~$\poset$ (left) correspond to the tubes of~$L(\poset)$ (right), and the pipings of~$\poset$ (left) correspond to the tubings of~$L(\poset)$ which are acyclic for~$\OM(\poset)$ (right). The maximal pipe/tube is omitted in all pipings/tubings.}
	\label{fig:posetAssociahedraCombi}
\end{figure}
\begin{figure}
	\capstart
	\centerline{\includegraphics[scale=.45]{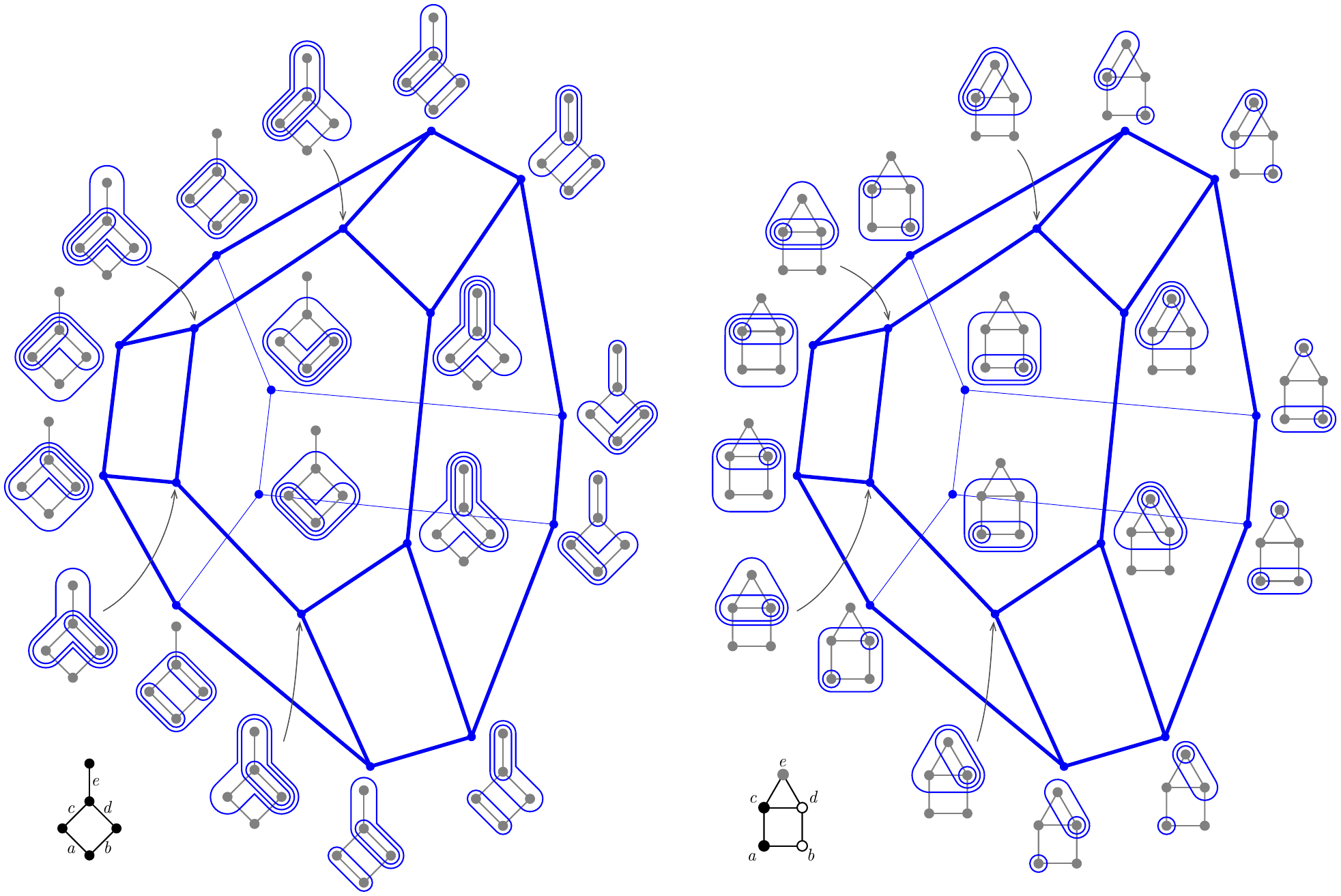}}
	\caption{A $3$-dimensional poset associahedron already considered in \cite[Fig.~1]{Galashin}. It illustrates the isomorphism between the piping complex~$\pipingComplex(\poset)$ (left) and the acyclic nested complex of~$\acyclicNestedComplex(\building(\poset), \OM(\poset))$ (right).}
	\label{fig:posetAssociahedron3}
\end{figure}
\begin{figure}
	\capstart
	\centerline{\includegraphics[scale=.45]{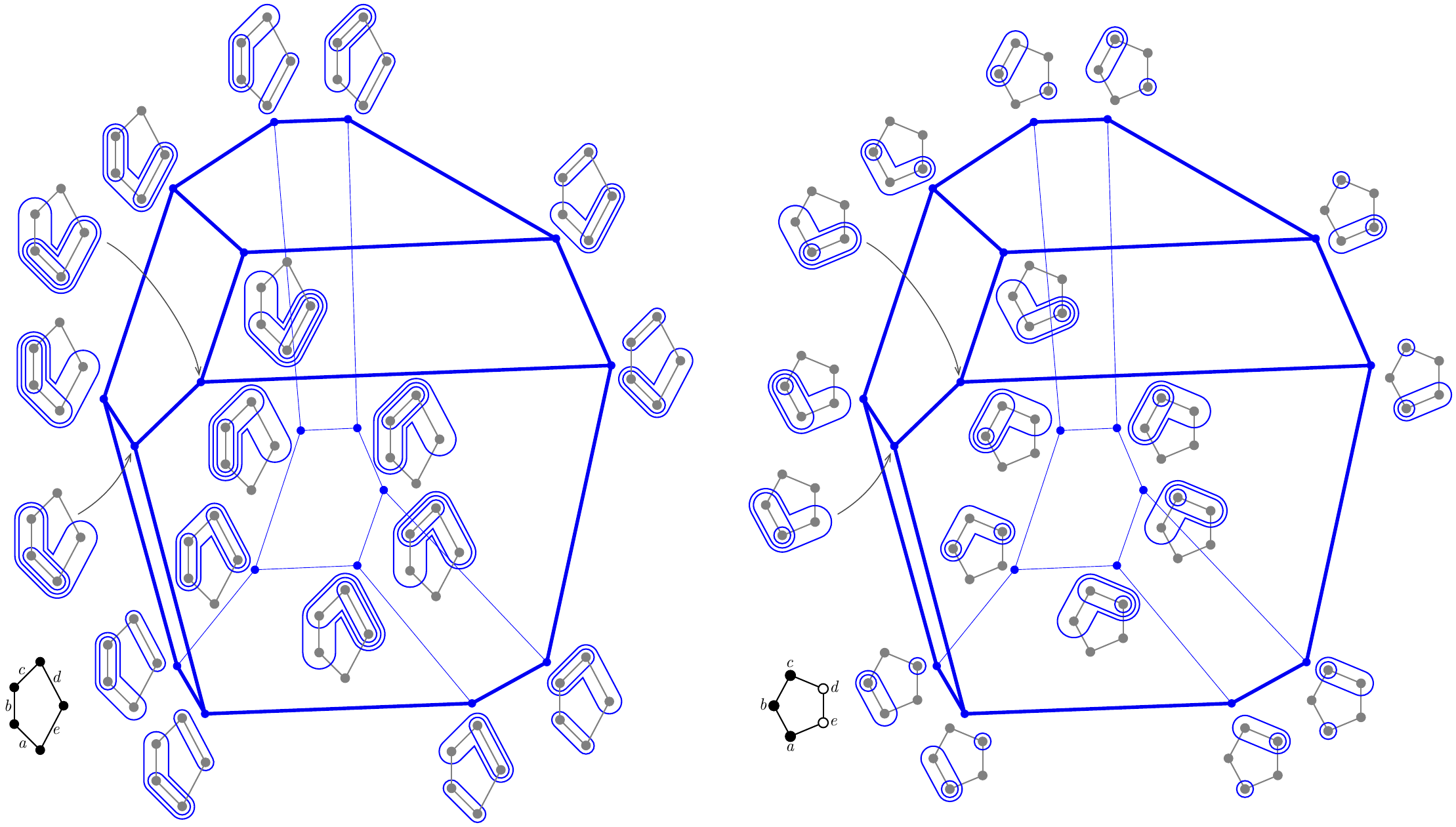}}
	\caption{Another $3$-dimensional poset associahedron. It illustrates the isomorphism between the piping complex~$\pipingComplex(\poset)$ (left) and the acyclic nested complex of~$\acyclicNestedComplex(\building(\poset), \OM(\poset))$ (right).}
	\label{fig:posetAssociahedron4}
\end{figure}
\end{example}

We derive from~\cref{prop:pipingComplexAsAcyclicNestedComplex,prop:linksAcyclicNestedComplex} that the family of piping complexes is closed by links, a result already obtained in \cite[Coro.~2.7\,(vi)]{Galashin}.
We need the following analogue of \cref{def:restrictionContractionOrientedBuildingSet} for posets.

\begin{definition}
\label{def:linksPipingComplex}
For a pipe~$\pipe$ in a piping~$\piping$ of~$\poset$, we define~$\poset_{\pipe \in \piping}$ as the transitive closure of the directed graph~$(H(\poset)_{|\pipe})_{\!/\bigcup_{\pipe' \in \piping, \, \pipe' \subsetneq \pipe} \pipe'}$.
\end{definition}

\begin{corollary}
\label{coro:linksPipingComplex}
The link of a piping~$\piping$ in the piping complex~$\pipingComplex(\poset)$ is the join of the piping complexes~$\pipingComplex(\poset_{\pipe \in \piping})$ for all pipes~$\pipe$ of~$\piping$.
\end{corollary}

Remember from \cref{exm:orderPolytope} that the Las Vergnas face lattice of the graphical oriented matroid~$\OM(\poset)$ is isomorphic to the face lattice of the \defn{order polytope} of~$\poset$, defined by
\[
\OrderPolytope(\poset) \eqdef \bigset{\b{x} \in \R^V}{\sum_{v \in V} x_v = 0, \sum_{(u,v) \in \digraph} x_v - x_u = 1 \text{ and } x_u \le x_v \text{ for each } u \preccurlyeq v \text{ in } \poset}.
\]

Combining~\cref{prop:pipingComplexAsAcyclicNestedComplex,thm:facialNestedComplexesVsAcyclicNestedComplexes,thm:compactification}, we obtain the following compactification, homeomorphic as a stratified space to that of~\cite[Thm.~1.9]{Galashin}.

\begin{corollary} 
There exists a compactification of~$\OrderPolytope(\poset)$ which is a stratified $C^\infty$ manifold with corners and whose combinatorics is encoded by the piping complex~$\pipingComplex(\poset)$.
\end{corollary}

Combining~\cref{prop:pipingComplexAsAcyclicNestedComplex,thm:facialNestedComplexesVsAcyclicNestedComplexes,thm:stellarPolytopalRealization}, we also recover the following result of~\cite[Sect.~2.1]{Galashin}.

\begin{corollary}
The piping complex~$\pipingComplex(\poset)$ is the boundary complex of a polytope obtained by stellar subdivisions on the order polytope of~$\poset$.
\end{corollary}

As pointed out in~\cite[Rem.~1.5]{Galashin}, this approach proves that the piping complex is polytopal (in particular that it is a shellable sphere), but does not provide explicit coordinates.
We now combine~\cref{prop:pipingComplexAsAcyclicNestedComplex} with~\cref{thm:acyclonestohedron,thm:alternativeRealization} to obtain explicit coordinates for poset associahedra.
See \cref{fig:acyclonestohedra,fig:posetAssociahedron3,fig:posetAssociahedron4} for illustrations.
We define~$\b{\rho} \in \R^{\building(\poset)}$ by~$\rho_{\tube} \eqdef |\building(\poset)|^{|\tube|}$.

\begin{corollary}
\label{coro:posetAssociahedronAsAcyclonestohedron}
The piping complex~$\pipingComplex(\poset)$ is the boundary complex of the polar of the acyclonestohedron~$\Acycl(\building(\poset), \b{A}(\poset))$, obtained as the intersection of the graph associahedron~$\Nest(\building(\poset), \b{\rho})$ of the line graph~$L(\poset)$ with the linear hyperplanes normal to~$\one_{c_+} - \one_{c_-}$ for all circuits~${c = (c_+, c_-)}$ of~$H(\poset)$.
\end{corollary}

\begin{corollary}
\label{coro:posetAssociahedronAsAcyclonestohedronProj}
The piping complex~$\pipingComplex(\poset)$ is the boundary complex of the polar of the acyclonestohedron~$\AcyclProj(\building(\poset), \b{A}(\poset))$, which is the polytope in~$\R^{\poset}$ defined by the equality~$\overline{g}_{\poset}(\b{y}) = 0$ and the inequalities~$\overline{g}_{\tube}(\b{y}) \ge 0$ for all~$\tube \in \building(\poset$), where
\[
\overline{g}_{\tube}(\b{y}) \eqdef \bigdotprod{\sum_{\substack{p,q \in \tube \\ p \precdot q}} \b{b}_p - \b{b}_q}{\b{y}} - \sum_{\substack{\tube' \in \building(\poset) \\ \tube' \subseteq \tube}} |\building(\poset)|^{|\tube|}.
\]
\end{corollary}

\begin{remark}
\label{rem:Sack}
The description of \cref{coro:posetAssociahedronAsAcyclonestohedronProj} was actually motivated by the construction of A.~Sack~\cite{Sack}.
He independently obtained polytopal realizations of poset associahedra directly in the space~$\R^{\poset}$.
His description is essentially the same as that of \cref{coro:posetAssociahedronAsAcyclonestohedronProj}, except that~$\smash{\sum_{\tube' \in \building(\poset), \tube' \subseteq \tube} |\building(\poset)|^{|\tube|}}$ is replaced by~$|\poset|^{2|\tube|}$.
His proof however is not using the description of \cref{coro:posetAssociahedronAsAcyclonestohedron}.
We note that his construction motivated us to provide the explicit coordinates of the description of acyclonestohedra~$\AcyclProj(\building, \b{A})$ in~$\R\b{A}$ given in \cref{subsec:acyclonestohedra2}, which are obtained as a projection of our construction of~$\Acycl(\building, \b{A})$ in~$\R^\ground$ from \cref{subsec:acyclonestohedra1}.
\end{remark}

\begin{remark}
In fact, in the case of poset associahedra, both \cref{coro:posetAssociahedronAsAcyclonestohedron,coro:posetAssociahedronAsAcyclonestohedronProj} even hold when replacing~$\b{\rho}_{\tube} = |\building(\poset)|^{|\tube|}$ by~${\b{\rho}_{\tube} = 4^{|\tube|}}$.
However, this requires a much more careful proof of \cref{lem:acyclonestohedron1}, where the contribution of the tubes counted negatively is compensated by the contribution of more than one tube counted positively.
Details can be found in the Master thesis of the first author~\cite{Mantovani}.
\end{remark}

\begin{remark}
In~\cite[Rem~3.3]{Galashin}, P.~Galashin asks for an algebraic variety reflecting the combinatorics of the piping complex of~$\poset$.
As a last application of \cref{prop:pipingComplexAsAcyclicNestedComplex,thm:facialNestedComplexesVsAcyclicNestedComplexes}, we note that such a variety is given by the toric variety associated to the facial nested complex as in~\cite[Sect.~5]{FeichtnerYuzvinsky2004}. 
\end{remark}


\section{Affine poset cyclohedra}
\label{sec:affinePosetCyclohedra}

Here, we show that our constructions apply to affine posets as well.
We first recall the definition of these combinatorial objects.

\begin{definition}[{\cite[Def.~1.10]{Galashin}}]
\label{def:affinePoset}
An \defn{affine poset} (of order~$n \ge 1$) is a poset~$\affinePoset = (\Z, \preccurlyeq)$ such~that
\begin{itemize}
\item for all~$i \in \Z$, $i \prec i+n$,
\item for all~$i,j \in \Z$, $i \preccurlyeq j$ if and only if~$i + n \preccurlyeq j + n$, 
\item for all~$i,j \in \Z$, there is~$k \in \N$ such that~$i \preccurlyeq j + kn$.
\end{itemize}
The order~$n$ of~$\affinePoset$ is denoted by~$|\affinePoset|$. The Hasse diagram of~$\affinePoset$ is denoted~$H(\affinePoset)$.
\end{definition}

\begin{definition}
For~$i \in \Z$, we denote by~$\quo{i}{n}$ and~$\mod{i}{n}$ the quotient and the remainder of the division of~$i$ by~$n$ (so that~$i = n \cdot \quo{i}{n} + \mod{i}{n}$).
We call
\begin{itemize}
\item \defn{element class} the residue class~$\cl{i}{n} \eqdef \set{i+kn}{k \in \Z}$ of an element~$i \in \Z$,
\item \defn{subset class} the residue class~$\cl{J}{n} \eqdef \set{\set{j+kn}{j \in J}}{k \in \Z}$ of a subset~$J \subseteq \Z$,
\item \defn{relation class} the residue class~$\cl{i \prec j}{n} \eqdef \set{i+kn \prec j+kn}{k \in \Z}$ of a relation~${i \prec j}$~in~$\affinePoset$.
\end{itemize}
\end{definition}

\begin{definition}[{\cite[Sect.~1.3]{Galashin}}]
\label{def:affinePosetAssociahedron}
A subset~$J$ of~$\Z$ is \defn{convex} if~$i \preccurlyeq j \preccurlyeq k$ and~$i,k \in J$ implies~${j \in J}$, \defn{connected} if the subgraph~$H(J)$ of~$H(\affinePoset)$ induced by~$J$ is connected, and \defn{thin} if it is either~$\Z$ or contains at most one element of each element class~$\cl{i}{n}$.
A \defn{pipe} of~$\affinePoset$ is a thin, convex and connected subset~$Q$ of~$\Z$ with~$|Q| > 1$.
We denote by~$\c{D}(\affinePoset)$ the directed graph with a vertex for each subset of~$\Z$, with an arc joining a subset~$Q$ to a subset~$R$ if~$Q \cap R = \varnothing$ and there is~$q \in Q$ and~$r \in R$ such that~$q \prec r$.
A \defn{pipe class} of~$\affinePoset$ is the subset class $\cl{\pipe}{n}$ of a pipe $\pipe$.
Two pipe classes  $\cl{\pipe}{n}$ and  $\cl{\pipe'}{n}$ are \defn{nested} if $\pipe + kn$ and $\pipe'$ are nested for some~$k \in \Z$, and \defn{disjoint} if~$\pipe + kn$ and~$\pipe'$ are disjoint for all~$k \in \Z$.
A \defn{piping} of~$\affinePoset$ is a collection of pairwise nested or disjoint pipe classes of~$\affinePoset$ that contains the pipe class~$\Z$ and 
induces an acyclic subgraph of~$\c{D}(\affinePoset)$ (see \cref{fig:exmAffinePosetNested}).
The \defn{affine piping complex} of~$\affinePoset$ is the simplicial complex~$\affinePipingComplex(\affinePoset)$ whose faces are~$\piping \ssm \{\Z\}$ for all pipings~$\piping$ of~$\affinePoset$ (see \cref{fig:affinePosetCyclohedraCombi}).
An \defn{affine poset cyclohedron} of~$\affinePoset$ is a simple polytope whose dual has boundary complex isomorphic to the affine piping complex~$\affinePipingComplex(\affinePoset)$ (see \cref{fig:affinePosetCyclohedron}).
\begin{figure}
	\capstart
	\centerline{\includegraphics[scale=.5]{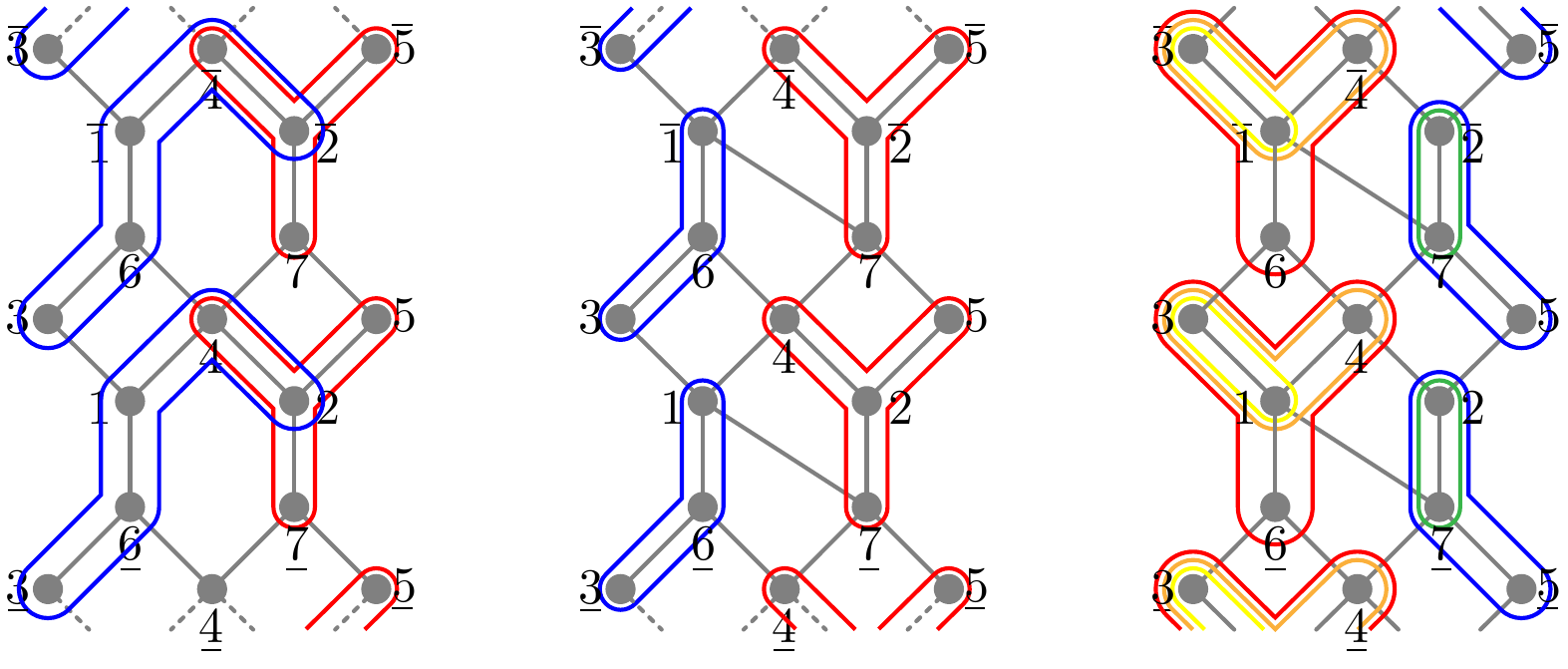}}
	\caption{Some incompatible pipes (left and middle), and a maximal piping (right) on an affine poset.}
	\label{fig:exmAffinePosetNested}
\end{figure}
\end{definition}

We now interpret the affine piping complexes in our setting of acyclic nested complexes.

\begin{definition}
\label{def:lineGraphAffinePoset}
The \defn{line graph}~$L(\affinePoset)$ is the graph with a vertex for each cover relation class~${\cl{i \precdot j}{n}}$ of~$\affinePoset$ and with an arc joining~$\cl{i \precdot j}{n}$ to~$\cl{i' \precdot j'}{n}$ whenever~$\{\cl{i}{n}, \cl{j}{n}\} \cap \{\cl{i'}{n}, \cl{j'}{n}\} \ne \varnothing$.
\end{definition}

\begin{definition}
\label{def:incidenceConfigurationAffinePoset}
Let~$(\b{b}_i)_{i \in [n+1]}$ be the standard basis of~$\R^{n+1}$.
For~$j \in \Z$, define the vector~$\tb{b}_j \eqdef \b{b}_{\mod{j}{n}} +  \quo{j}{n}\cdot\b{b}_{n+1}$.
The \defn{incidence configuration}~$\b{A}(\affinePoset)$ is the vector configuration with a vector~$\b{a}_{\cl{i \precdot j}{n}} \eqdef \b{a}_{i \precdot j} \eqdef \tb{b}_i - \tb{b}_j \in \R^{n+1}$ for each cover relation class~$\cl{i \precdot j}{n}$ of~$\affinePoset$ (observe here that~$\b{a}_{i \precdot j}$ indeed only depends on the relation class~$\cl{i \precdot j}{n}$).
\end{definition}

\begin{definition}
\label{def:affinePosetOrientedBuildingSet}
We consider the \defn{affine poset oriented building set}~$(\building(\affinePoset), \OM(\affinePoset))$ of~$\affinePoset$ where
\begin{itemize}
\item $\OM(\affinePoset)$ is the oriented matroid of the incidence configuration~$\b{A}(\affinePoset)$,
\item $\building(\affinePoset)$ is the building closure of $\hat\building(\affinePoset)\cup\circuitSupports[\OM(\affinePoset)]$, where $\hat\building(\affinePoset)$ is the graphical building set of the line graph~$L(\affinePoset)$, and $\circuitSupports[\OM(\affinePoset)]=\set{\underline{c}}{c\in\circuits[\OM(\affinePoset)]}$ is the set of supports of circuits of $\OM(\affinePoset)$.
\end{itemize}
\end{definition}

\begin{proposition}
\label{prop:affinePipingComplexAsAcyclicNestedComplex}
The affine piping complex~$\affinePipingComplex(\affinePoset)$ is isomorphic to the acyclic nested complex~$\acyclicNestedComplex(\building(\affinePoset), \OM(\affinePoset))$ of the affine poset oriented building set of~$\affinePoset$.
\end{proposition}

The proof of \cref{prop:affinePipingComplexAsAcyclicNestedComplex} follows the same lines as that of \cref{prop:pipingComplexAsAcyclicNestedComplex}, but the details are much more involved.
The problem here is that we had to include~$\circuitSupports[\OM(\affinePoset)]$ in~$\building(\affinePoset)$ in order to make~$(\building(\affinePoset), \OM(\affinePoset))$ an oriented building set.
We thus need to make sure that these additional blocks do not perturb too much the situation.
We delay this proof to the end of the section, and immediately derive some corollaries.

\begin{example}
For instance, if the Hasse diagram~$H(\affinePoset)$ is a chain, then the affine poset cyclohedron of~$\affinePoset$ is just the cyclohedron.
An affine poset cyclohedra is represented in \cref{fig:affinePosetCyclohedraCombi,fig:affinePosetCyclohedron}.
\begin{figure}
	\capstart
	\centerline{\includegraphics[scale=.45]{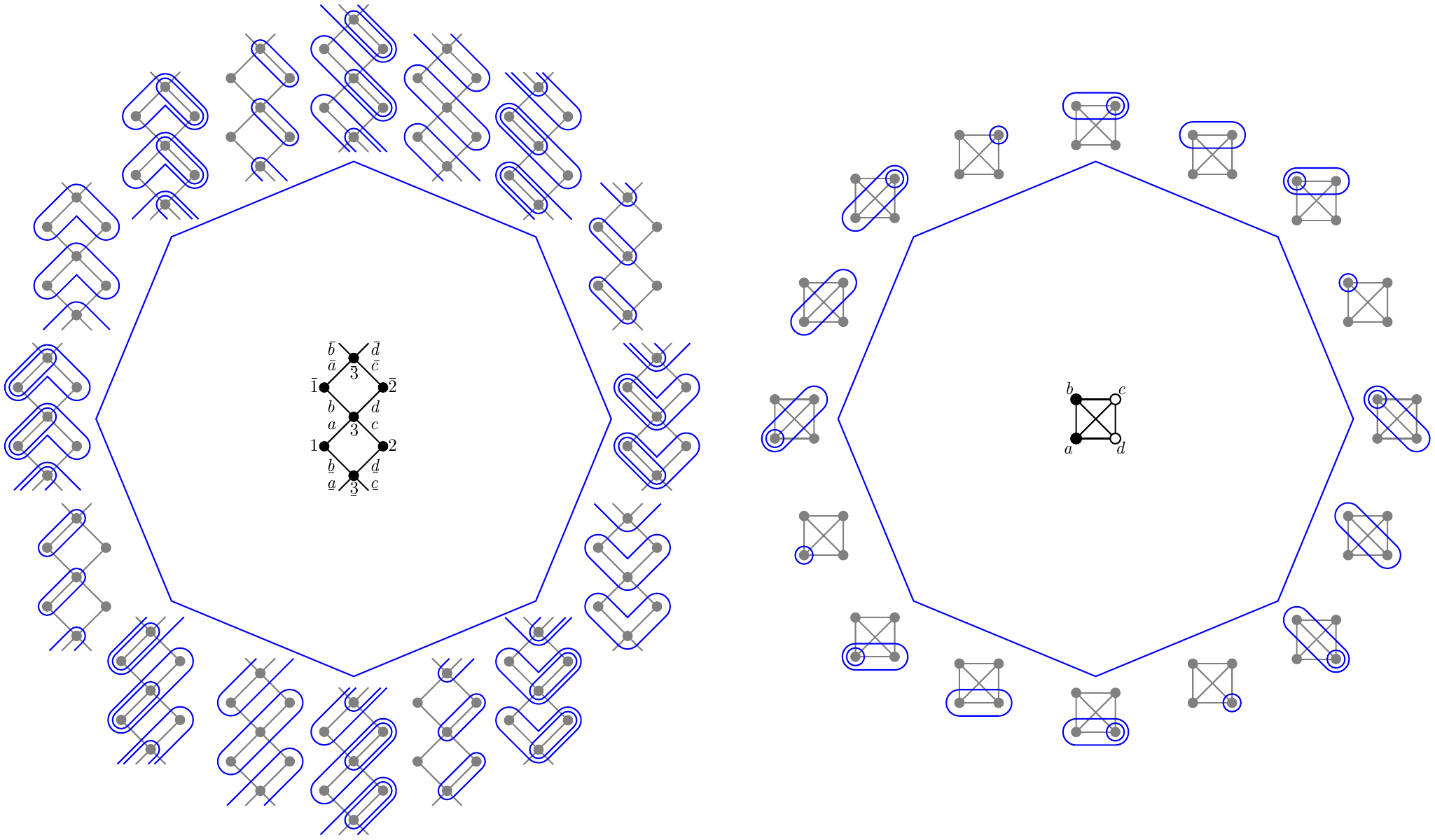}}
	\caption{The isomorphism between the affine piping complex of~$\affinePoset$ (left) and the affine graphical acyclic nested complex of~$(\building(\affinePoset), \OM(\affinePoset))$ (right). The line graph~$L(\affinePoset)$ is the $4$-clique, and the signs of the unique circuit of~$\OM(\affinePoset)$ are indicated by the black and white vertices of~$L(\affinePoset)$. The affine poset tubes of~$\affinePoset$ (left) correspond to the graph tubes of~$L(\affinePoset)$ (right), and the affine poset tubings of~$\affinePoset$ (left) correspond to the graph tubings of~$L(\affinePoset)$ which are acyclic for~$\OM(\affinePoset)$ (right). The maximal pipe/tube is omitted in all pipings/tubings.}
	\label{fig:affinePosetCyclohedraCombi}
\end{figure}
\begin{figure}
	\capstart
	\centerline{\includegraphics[scale=.45]{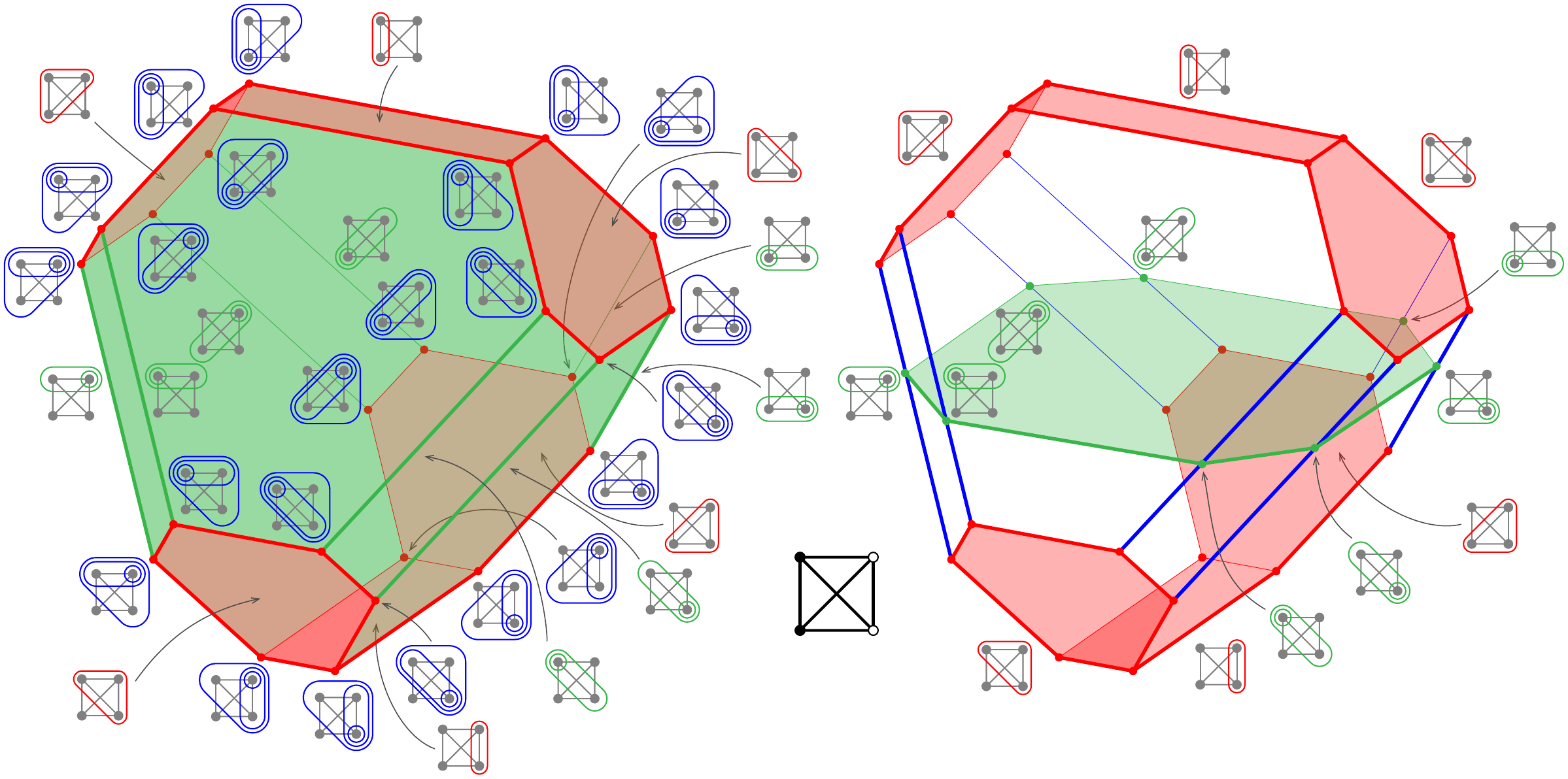}}
	\caption{The affine poset cyclohedron of \cref{fig:affinePosetCyclohedraCombi} obtained as a section of a graph associahedron. The line graph~$L(\affinePoset)$ is the $4$-clique, and the unique circuit of~$\OM(\affinePoset)$ is represented by the black and white vertices of~$L(\affinePoset)$. The left picture represents the permutahedron, with all maximal tubings of~$L(\affinePoset)$ in blue, the minimal cyclic tubings of~$L(\affinePoset)$ in red, and the maximal acyclic tubings of~$L(\affinePoset)$ in green. The right picture represents the affine poset cyclohedron (the green polygon) obtained as a section of the permutahedron by the evaluation space. The maximal tube is omitted in all tubings.}
	\label{fig:affinePosetCyclohedron}
\end{figure}
\end{example}

We derive from~\cref{prop:affinePipingComplexAsAcyclicNestedComplex,prop:linksAcyclicNestedComplex} that links of affine piping complexes are joins of piping complexes and affine piping complexes, a result already obtained in \cite[Coro.~4.10\,(vi)]{Galashin}.
We need the following analogue of \cref{def:restrictionContractionOrientedBuildingSet} for affine posets.

\begin{definition}
For a pipe class~$\cl{\pipe}{n}\neq\Z$ in a piping~$\piping$ of~$\affinePoset$, we define~$\poset_{\pipe \in \piping}$ as the finite poset obtained as the transitive closure of the directed graph~$(H(\affinePoset)_{|\pipe})_{\!/\bigcup_{\cl{\pipe'}{n} \in \piping, \, \pipe' \subsetneq \pipe} \pipe'}$ (the choice of the representative~$\pipe$ in~$\cl{\pipe}{n}$ gives isomorphic posets).
We define the affine poset~$\affinePoset_{\Z \in \piping}$ as the transitive closure of the directed graph~$(H(\affinePoset))_{\!/\bigcup_{\pipe' \in \piping, \, \pipe' \neq \Z} \pipe'}$.
\end{definition}

\begin{corollary}
\label{coro:linksAffinePipingComplex}
The link of an affine piping~$\piping$ in the affine piping complex~$\affinePipingComplex(\affinePoset)$ is the join of the affine piping complex $\affinePipingComplex(\affinePoset_{\Z \in \piping})$
and the (ordinary) piping complexes~$\pipingComplex(\poset_{\pipe \in \piping})$ for all~$\cl{\pipe}{n}\neq\Z$ of~$\piping$.
\end{corollary}

As~$\OM(\affinePoset)$ is defined by a vector configuration, its Las Vergnas face lattice is isomorphic to the face lattice of the polytope given by its positive tope, which is called \defn{affine order polytope}~$\OrderPolytope(\affinePoset)$ by P.~Galashin~\cite{Galashin}.

Combining~\cref{prop:pipingComplexAsAcyclicNestedComplex,thm:facialNestedComplexesVsAcyclicNestedComplexes,thm:compactification}, we obtain the following compactification, homeomorphic as a stratified space to that of~\cite[Thm.~1.12]{Galashin}.

\begin{corollary} 
There exists a compactification of~$\OrderPolytope(\affinePoset)$ which is a stratified $C^\infty$ manifold with corners and whose combinatorics is encoded by the affine piping complex~$\affinePipingComplex(\affinePoset)$.
\end{corollary}

Combining~\cref{prop:affinePipingComplexAsAcyclicNestedComplex,thm:facialNestedComplexesVsAcyclicNestedComplexes,thm:stellarPolytopalRealization}, we also recover the following result of~\cite[Sect.~4]{Galashin}.

\begin{corollary}
The affine piping complex~$\affinePipingComplex(\affinePoset)$ is the boundary complex of a polytope obtained by stellar subdivisions on the affine order polytope of~$\affinePoset$.
\end{corollary}

Again, this approach does not provide explicit coordinates.
We now combine~\cref{prop:affinePipingComplexAsAcyclicNestedComplex} with \cref{thm:acyclonestohedron,thm:alternativeRealization} to obtain explicit coordinates for affine poset cyclohedra.
We define~$\b{\rho} \in \R^{\building(\affinePoset)}$ by~$\rho_{\tube} \eqdef |\building(\affinePoset)|^{|\tube|}$.
For a circuit~$c = (c_+, c_-)$ of~$H(\affinePoset)$, we define
\[
\quo{c}{n} \eqdef \sum_{i\precdot j\in c_+} \big( \quo{i}{n}-\quo{j}{n} \big) - \sum_{i\precdot j\in c_-} \big( \quo{i}{n}-\quo{j}{n} \big) ,
\]
where~$\quo{i}{n}$ is the quotient of the division of~$i$ by~$n$.

\begin{corollary}
\label{coro:affinePosetCyclohedronAsAcyclonestohedron}
The affine piping complex~$\affinePipingComplex(\affinePoset)$ is the boundary complex of the polar of the acyclonestohedron~$\Acycl(\building(\affinePoset), \b{A}(\affinePoset))$, obtained as the intersection of the nestohedron~$\Nest(\building(\affinePoset), \b{\rho})$ with the linear hyperplanes normal to~$\one_{c_+} - \one_{c_-} - \quo{c}{n} \one_{\pi}$ for all circuits~$c = (c_+, c_-)$~of~$H(\affinePoset)$, where $\pi$ is any fixed path from~$1$ to~$1+n$.
\end{corollary}

\begin{corollary}
\label{coro:affinePosetCyclohedronAsAcyclonestohedronProj}
The affine piping complex~$\affinePipingComplex(\affinePoset)$ is the boundary complex of the polar of the acyclonestohedron~$\AcyclProj(\building(\affinePoset), \b{A}(\affinePoset))$, which is the polytope in~$\R^n$ defined by the equality~${\overline{g}_{[n]}(\b{y}) = 0}$ and the inequalities~$\overline{g}_{B}(\b{y}) \ge 0$ for all~$B \in \building(\affinePoset$), where
\[
\overline{g}_{B}(\b{y}) \eqdef \bigdotprod{\sum_{\substack{p,q \in B \\ p \precdot q}} \tb{b}_p - \tb{b}_q}{\b{y}} - \sum_{\substack{B' \in \building(\affinePoset) \\ B' \subseteq B}} |\building(\affinePoset)|^{|B|}.
\]
\end{corollary}

We finally come back to the proof of \cref{prop:affinePipingComplexAsAcyclicNestedComplex}.
We start with a few definitions and lemmas in preparation.

\begin{definition}
\label{def:quotientGraphAffinePoset}
The \defn{quotient graph}~$\quotientGraph$ is the directed graph with one vertex for each element class of $\affinePoset$ and an arc from~$\cl{i}{n}$ to~$\cl{j}{n}$ if~$i \precdot j + kn$ for some~$k \in \Z$.
\end{definition}

\begin{lemma}
\label{lem:circuitsAffinePoset}
Let $\sum_{i\precdot j \in \underline{c} } \delta_{i \precdot j} \b{a}_{i \precdot j} = \zero$ be a linear dependence corresponding to a circuit~$c$~of~$\OM(\affinePoset)$.
Then there exist oriented cycles $c_1, \dots, c_r$ of $\quotientGraph$ and $\delta_1,\dots,\delta_r\in\R_{>0}$ such that $\bigcup_{k \in [r]} \underline{c_k} = \underline{c}$ and $\sum_{k \in [r]} \delta_k\quo{c_k}{n}=0$.
Conversely, for any such cycles $c_1, \dots, c_r$ and ${\delta_1,\dots,\delta_r\in\R_{>0}}$, there is a circuit~$c$ of $\OM(\affinePoset)$ whose support is the set 
$\bigset{i\precdot j}{\sum_{i\precdot j\in (c_k)_+}\delta_k - \sum_{i\precdot j\in (c_k)_-}\delta_k \neq 0}$.
\end{lemma}

\begin{proof}
If we restrict the vectors $\b{a}_{i \precdot j}$ to their first $n$ coordinates, we recover a realization of the graphical oriented matroid of~$\quotientGraph$, whose circuits are given by oriented cycles of~$\quotientGraph$ (see~\cite[Prop.~1.1.7 \& Chap.~5]{Oxley} and \cite[Sect.~1.1]{BjornerLasVergnasSturmfelsWhiteZiegler}). Therefore, any linear dependence with support $\underline{c}$ decomposes as a linear combination of linear dependences the union of whose supports is $\underline{c}$. The additional condition $\sum_{k \in [r]}\delta_k\quo{c_k}{n}=0$ arises from imposing that the last coordinate adds up to zero as well.

For the converse, note that any linear combination of cycles gives a linear dependence on the first $n$ coordinates, and the condition on the $\delta_k$'s assures that it is a linear dependence on the last coordinate as well. The support of the circuit it induces is by definition the set of edges that have a non-zero coefficient.
\end{proof}

\begin{definition}
Any block $B$ of $\building(\affinePoset)$ induces an infinite $n$-periodic subgraph of the Hasse diagram~$H(\affinePoset)$, whose arcs correspond to the cover relations $i\precdot j$ such that $\cl{i\precdot j}{n}\in B$. We denote this subgraph by $H(B)$, and its connected components by $\connectedComponents(B)$.
\end{definition}

\begin{lemma}
\label{lem:affinePosetFiniteCc}
Let $B$ be a block of $\building(\affinePoset)$ such that $H(B)\neq H(\affinePoset)$. If $B$ is acyclic, then every connected component $K\in\connectedComponents(B)$ is finite.
\end{lemma}

\begin{proof}
Suppose $K\in\connectedComponents(B)$ is infinite. Since $\affinePoset$ has a finite number of element classes and $K$ is infinite, there exist $i,k\in\Z$ such that $\{i,i+kn\} \subset K$.
Since~$K$ is connected, it contains a (undirected) path $\pi$ between~$i$ and~$i+kn$. Since $H(B)\neq H(\affinePoset)$, there exists $j\notin H(B)$, which implies that $B$ contains no cover relation classes that involve $\cl{j}{n}$. Since $j\prec j+n$, there exists a directed path $\vec{\pi}'$ in $H(\affinePoset)$ from $j$ to $j+kn$ which is not all contained in $H(B)$. With a good choice of coefficients, the arcs of $\pi \cup \vec{\pi}'$ correspond to a circuit $c$ of $\OM(\affinePoset)$ such that $c_+ \subseteq \pi' \subset B$ while $\vec{\pi}' \subseteq c_- \not\subseteq B$, which implies that~$B$ is cyclic.
More precisely, if we reorient the arcs of $\pi$ to obtain a directed path $\vec{\pi}$ from $i$ to $i+kn$, we can set $c_+ \eqdef \pi \cap \vec{\pi}$ and $c_- \eqdef \vec{\pi}' \cup (\pi \ssm \vec{\pi})$. Note that the support of~$\vec{\pi}$ is a cycle $c_1$ of $\quotientGraph$ with $\quo{c_1}{n} = k$ and the support of~$\vec{\pi}'$ is a cycle $c_2$ of $\quotientGraph$ with $\quo{c_2}{n} = k$.
Therefore, thanks to \cref{lem:circuitsAffinePoset}, $c=(c_+,c_-)$ is a circuit. 
\end{proof}

\begin{lemma}
\label{lem:affinePosetFiniteCc1}
If $B\in\building(\affinePoset)\ssm\hat\building(\affinePoset)$, then $\connectedComponents(B)$ contains an infinite connected component.
\end{lemma}

\begin{proof}
Since $\building(\affinePoset)$ is the building closure of $\hat\building(\affinePoset)\cup\circuitSupports[\OM(\affinePoset)]$ and ~$B\notin\hat\building(\affinePoset)$, there is a circuit ~$c$ of ~$\OM(\affinePoset)$ such that $\underline{c}$ is contained in $B$ and is not a block of $\hat\building(\affinePoset)$. 
Let $c_1,\dots, c_r$ be cycles of~$\quotientGraph$ as in \cref{lem:circuitsAffinePoset}. Note that if $\quo{c_k}{n}=0$ for all $k\in[r]$, then $c$ corresponds to a connected collection of cycles of the line graph of~$\affinePoset$, therefore it belongs to $\hat\building(\affinePoset)$. 
Let $k\in [r]$ be such that $\quo{c_k}{n}\neq 0$. This implies that $c_k$ is the quotient of a path from $j$ to $j+hn$ in $H(\affinePoset)$ for some $j,h\in\Z$.
Considering that $H(B)$ is $n$-periodic (since it is induced by cover relation classes), the connected component $K\in\connectedComponents(B)$ that contains this path also contains paths from $j$ to $j+\ell hn$ for all $\ell\in\Z$, hence it is infinite. 
\end{proof}

\begin{proof}[Proof of \cref{prop:affinePipingComplexAsAcyclicNestedComplex}]
We start by defining a bijection between building blocks of~$\building(\affinePoset)$ which are acyclic for~$\OM(\affinePoset)$ and pipe classes of $\affinePoset$.
Let $B$ be an building block of $\building(\affinePoset)$ acyclic for~$\OM(\affinePoset)$, and let us define $\piping_B \eqdef \set{v(K)}{K\in\connectedComponents(B)}$, where $v(K)$ is the set of vertices of $K$. 
Since $B$ is acyclic, $B\in\hat\building(\affinePoset)$ by \cref{lem:affinePosetFiniteCc,lem:affinePosetFiniteCc1}, hence it corresponds to a connected subgraph of the line graph of $\affinePoset$. 
This implies that $\piping_B=\cl{\pipe}{n}$, where $\pipe$ is connected, thin by \cref{lem:affinePosetFiniteCc}, and convex by the acyclicity of $B$. Hence, $\piping_B$ is a pipe class of $\affinePoset$.
 
Conversely, let $\pipe$ be a pipe of $\affinePoset$ and let $B_{\pipe}$ be the set of cover relation classes induced by~$\pipe$.
We show that $B_{\pipe}$ is a building block of~$\building(\affinePoset)$ acyclic for~$\OM(\affinePoset)$.
Observe that if $B_{\pipe}$ is not acyclic, then $\pipe$ is not convex. In fact, if $B$ is not acyclic there exists a circuit~$c \in \circuits[\OM(\affinePoset)]$ with~$c_+ \subseteq B_{\pipe}$ while~$c_- \not\subseteq B_{\pipe}$. 
We decompose $c=c_1\cup\cdots \cup c_r$ into cycles of the quotient graph as in \cref{lem:circuitsAffinePoset}. 
If $\quo{c_k}{n} = 0$ for all $k \in [r]$, then $\underline{c}$ corresponds to a circuit of $H(\affinePoset)$, and we can deduce that $\pipe$ is not convex as in the case of non-affine posets (\cref{prop:pipingComplexAsAcyclicNestedComplex}). 
Otherwise, assume without loss of generality that $\quo{c_1}{n} > 0$. There are $i,k\in \Z$, $k>0$ such that $c_1$ induces a path $\pi$ from~$i$ to~$i+kn$ in $H(\affinePoset)$. Since $\pipe$ is thin, at least some edge class of this path is not in $B_{\pipe}$, and therefore~$(c_1)_- \not\subseteq B_{\pipe}$ (while still $(c_1)_+ \subseteq B_{\pipe}$). This means that~$\pi$ starts with a down path from $i$ to some element $i_1\in \pipe$, and ends with a down path from some element $i_2\in \pipe$ to $i+kn$. One of these two paths can be empty, but not both because of the aforementioned thinness. We also know that there is some increasing path from $i$ to $i+kn$ by the definition of affine poset. Combining these paths $i_1\to i$, $i\to i+kn$ and $i+kn\to i_2$, we see that $\pipe$ is not convex.

Since the maps $B\mapsto\piping_{B}$ and $\cl{\pipe}{n}\mapsto B_{\pipe}$ are clearly inverse to each other, we get a bijection between building blocks of~$\building(\affinePoset)$ acyclic for~$\OM(\affinePoset)$ and pipe classes of $\affinePoset$.
We now prove that these maps induce bijections between pipings of $\affinePoset$ and acyclic nested sets of~$(\building(\affinePoset), \OM(\affinePoset))$.

Note that if $B$ and~$B'$ are acyclic building blocks different from~$[n]$, then they both belong to the graphical building set $\hat\building(\affinePoset)$ by \cref{lem:affinePosetFiniteCc,lem:affinePosetFiniteCc1}, and we directly have that  $B$ and~$B'$ are nested if and only if $\piping_B$ and $\piping_{B'}$ are nested, and that $B$ and~$B'$ are disjoint and non-adjacent if and only if $\piping_B$ and $\piping_{B'}$ are disjoint.
It thus only remains to show that if $B_1,\dots,B_r$ are acyclic building blocks forming a nested set, then there exists a circuit $c$ of $\OM(\affinePoset)$ such that~$c_+ \subseteq B_1\cup\dots\cup B_r$ but~$c_- \not\subseteq B_1\cup\dots\cup B_r$ if and only if $\piping_{B_1}, \dots, \piping_{B_r}$ form a directed cycle of~$\c{D}(\affinePoset)$.
The ``if'' direction is straightforward.
To prove the ``only if'' direction, let $c$ be such a circuit. We decompose it $c=c_1\cup\cdots \cup c_r$ into cycles of the quotient graph following \cref{lem:circuitsAffinePoset}.
Again, if $\quo{c_i}{n} = 0$ for all $i$, then $\underline{c}$ corresponds to a circuit of $H(\affinePoset)$, and we can directly deduce that~$\piping_{B_1}, \dots, \piping_{B_r}$ form a directed cycle of $\c{D}(\affinePoset)$ as in the case of non-affine posets (\cref{prop:pipingComplexAsAcyclicNestedComplex}). Otherwise, as above, we can assume without loss of generality that $\quo{c_1}{n} > 0$ and find a path $\pi$ from $i$ to~$i+kn$.
Some edge class of this path is not in $\piping_{B_1}\cup\dots\cup \piping_{B_r}$, as otherwise there would be a pipe containing both $i$ and $i+kn$, contradicting the thinness.
Therefore, $(c_1)_- \not\subseteq B_1\cup\dots\cup B_r$ (while still $(c_1)_+ \subseteq B_1\cup\dots\cup B_r$).
This implies that $\pi \ssm (\piping_{B_1}\cup\dots\cup \piping_{B_r})$ is a non-empty collection of down paths whose endpoints belong to some pipes $\pipe_1,\dots, \pipe_\ell$ in~$\piping_{B_1}\cup\dots\cup \piping_{B_r}$. We also know that there is some increasing path from $i$ to $i+kn$ by the definition of affine poset.
The combination of these paths induces a directed cycle on the vertices $\pipe_1,\dots,\pipe_\ell$ of $\c{D}(\affinePoset)$.
\end{proof}


\clearpage
\part{Nested complex embeddings and the Bergman embedding}
\label{part:embeddings}

Embedding facial nested complexes as acyclic nested complexes (\cref{part:facialAcyclicNestedComplexes}) enabled us to produce polyhedral realizations (\cref{part:realizations}).
In this part, we discuss further embeddings of nested complexes and the resulting realizations.
We first investigate conditions on a map~$\lattice \to \lattice'$ between two lattices which guarantee that the $\lattice$-nested complexes embed as subcomplexes of $\lattice'$-nested complexes (\cref{sec:latticeEmbedding}).
We then apply these conditions to embed atomic nested complexes as subcomplexes of boolean nested complexes (\cref{sec:atomic}), and derive realizations of atomic nested complexes as subfans of boolean nested fans and as subcomplexes of boundary complexes of nestohedra, recovering the fan realizations of E.~M.~Feichtner and S.~Yuzvinsky~\cite{FeichtnerYuzvinsky2004}.
We finally apply our conditions to extend the connection between nested complexes over the face lattice and the flat lattice of an oriented matroid (\cref{sec:Bergman}).


\section{Nested complex embeddings}
\label{sec:latticeEmbedding}

In this section, we consider two lattices~$\lattice,\lattice'$ and discuss conditions for the $\lattice$-nested complexes to embed as subcomplexes of $\lattice'$-nested complexes.


\subsection{Compatible building sets}
\label{subsec:compatible}

Consider an order embedding~$\phi : \lattice \to \lattice'$ between two lattices~$\lattice,\lattice'$ (\ie $X \le_\lattice Y \iff \phi(X) \le_{\lattice'} \phi(Y)$ for all~$X,Y \in \lattice$).
We study the relation between the $\lattice$-building sets and their $\lattice$-nested complexes and the $\lattice'$-building sets and their $\lattice'$-nested complexes.

\begin{definition}
\label{def:phiCompatible}
Consider an order embedding~$\phi : \lattice \to \lattice'$ between two lattices~$\lattice,\lattice'$.
Let $\building$ be an $\lattice$-building set, and~$\building'$ be an $\lattice'$-building set.
We say that~$(\building, \building')$ is \defn{$\phi$-compatible} if~$\phi(\building)$ is contained in~$\building'$ and~$\phi$ embeds the $\lattice$-nested complex~$\nestedComplex[\lattice][\building]$ to a subcomplex of the $\lattice'$-nested complex~$\nestedComplex[\lattice'][\building']$.
\end{definition}

\begin{example}
\label{exm:maximalBuildingSetsCompatible}
$(\lattice_{>\botzero}, {\lattice'}_{\!\!>\botzero'})$ is always $\phi$-compatible.
Indeed, $\phi(\lattice_{>\botzero}) \subseteq \lattice'_{>\botzero'}$ since~$\phi$ is an order embedding, and any~$\nested \in \nestedComplex$ forms a chain, hence~$\phi(\nested)$ also forms a chain since~$\building$ is order preserving.
\end{example}

\begin{example}
\label{exm:phiCompatible}
Following on the examples of \cref{fig:acyclicNestedComplexes}, consider the three lattices
\[
\lattice_1 = 
\begin{tikzpicture}[scale=.8, baseline=1.1cm, inner sep=2]
	\node (0) at (0,0) {$\varnothing$};
	\node[draw] (1) at (-1.5,1) {$1$};
	\node[draw] (2) at (-.5,1) {$2$};
	\node[draw] (3) at (.5,1) {$3$};
	\node[draw] (4) at (1.5,1) {$4$};
	\node[draw] (12) at (-1.5,2) {$12$};
	\node[draw] (23) at (-.5,2) {$23$};
	\node[draw] (34) at (.5,2) {$34$};
	\node[draw] (14) at (1.5,2) {$14$};
	\node[draw] (1234) at (0,3) {$1234$};
	\foreach \a/\b in {0/1, 0/2, 0/3, 0/4,
                   1/12, 1/14,
                   2/12, 2/23,
                   3/23, 3/34,
                   4/14, 4/34,
                   12/1234, 23/1234, 34/1234, 14/1234}
		\draw (\a) -- (\b);
\end{tikzpicture}
\qquad
\lattice_2 = 
\begin{tikzpicture}[scale=.8, baseline=1.1cm, inner sep=2]
	\node (0) at (0,0) {$\varnothing$};
	\node[draw] (1) at (-1.5,1) {$1$};
	\node[draw] (2) at (-.5,1) {$2$};
	\node[draw] (3) at (.5,1) {$3$};
	\node[draw] (4) at (1.5,1) {$4$};
	\node[draw] (12) at (-1.5,2) {$12$};
	\node (13) at (-.5,2) {$13$};
	\node (24) at (.5,2) {$24$};
	\node[draw] (34) at (1.5,2) {$34$};
	\node[draw] (1234) at (0,3) {$1234$};
	\foreach \a/\b in {0/1, 0/2, 0/3, 0/4,
                   1/12, 1/13,
                   2/12, 2/24,
                   3/13, 3/34,
                   4/24, 4/34,
                   12/1234, 13/1234, 24/1234, 34/1234}
		\draw (\a) -- (\b);
\end{tikzpicture}
\quad\text{and}\quad
\lattice' = 
\begin{tikzpicture}[scale=.8, baseline=1.5cm, inner sep=2]
	\node (0) at (0,0) {$\varnothing$};
	\node[draw] (1) at (-1.5,1) {$1$};
	\node[draw] (2) at (-.5,1) {$2$};
	\node[draw] (3) at (.5,1) {$3$};
	\node[draw] (4) at (1.5,1) {$4$};
	\node[draw] (12) at (-2.5,2) {$12$};
	\node (13) at (-1.5,2) {$13$};
	\node[draw] (14) at (-.5,2) {$14$};
	\node[draw] (23) at (.5,2) {$23$};
	\node (24) at (1.5,2) {$24$};
	\node[draw] (34) at (2.5,2) {$34$};
	\node[draw] (123) at (-1.5,3) {$123$};
	\node[draw] (124) at (-.5,3) {$124$};
	\node[draw] (134) at (.5,3) {$134$};
	\node[draw] (234) at (1.5,3) {$234$};
	\node[draw] (1234) at (0,4) {$1234$};
	\foreach \a/\b in {0/1, 0/2, 0/3, 0/4,
                   1/12, 1/13, 1/14,
                   2/12, 2/23, 2/24,
                   3/13, 3/23, 3/34,
                   4/14, 4/24, 4/34,
                   12/123, 12/124,
                   13/123, 13/134,
                   14/124, 14/134,
                   23/123, 23/234,
                   24/124, 24/234,
                   34/134, 34/234,
                   123/1234, 124/1234, 134/1234, 234/1234}
		\draw (\a) -- (\b);
\end{tikzpicture}
\]
and their building sets~$\building_1 \eqdef \lattice_1 \ssm \{\botzero_{\lattice_1}\}$, $\building_2 \eqdef \lattice_2 \ssm \{\botzero_{\lattice_2}, 13, 24\}$, and~$\building' \eqdef \lattice' \ssm \{\botzero_{\lattice'}, 13, 24\}$ (we have boxed the blocks of these building sets).
Then the pairs~$(\building_1, \building')$ and~$(\building_2, \building')$ are $\iota$-compatible (where~$\iota$ is the inclusion map).
Note that~$\building'$ is the inclusion minimal $\lattice'$-building set $\iota$-compatible with~$\building_1$.
In contrast, the inclusion mininal $\lattice'$-building set $\iota$-compatible with~$\building_2$ is~$\building_2$ itself, not~$\building'$.
\end{example}

\begin{example}
More generally, for an oriented matroid~$\OM$ on a ground set~$\ground$, \cref{thm:orientedBuildingSetsTofacialBuildingSets,thm:facialNestedComplexesVsAcyclicNestedComplexes} show that if~$\building$ is an acyclic building set for~$\OM$, then $\fbuilding \eqdef \building \cap \FL(\OM)$ is a facial building set for~$\OM$, and $(\fbuilding, \building)$ is $\phi$-compatible, where~$\phi : \FL(\OM) \to 2^\ground$ is defined by~$\phi(F) \eqdef \ground_{\le F}$.
This example will be largely extended in \cref{subsec:atomic}.
\end{example}

As already illustrated in \cref{part:realizations}, one can use the $\phi$-compatibility of a pair~$(\building, \building')$ to construct geometric realizations of~$\nestedComplex[\lattice][\building]$ from geometric realizations of~$\nestedComplex[\lattice'][\building']$ as follows.
We note however that the resulting realizations are in general not polytopal, see \cref{rem:badFan}.

\begin{proposition}
Consider an order embedding~$\phi : \lattice \to \lattice'$ and a $\phi$-compatible pair~$(\building, \building')$.
Then any polyhedral complex realizing~$\nestedComplex[\lattice'][\building']$ contains a subcomplex realizing~$\nestedComplex$.
\end{proposition}

\begin{proof}
As $(\building,\building')$ is $\phi$-compatible, $\nestedComplex$ embeds as a subcomplex of~$\nestedComplex[\lattice'][\building']$, so we just need to select the faces of the polyhedral complex realizing~$\nestedComplex[\lattice'][\building']$ to obtain a subcomplex realizing~$\nestedComplex$.
\end{proof}

\begin{example}
\label{exm:subrealizations}
Following on \cref{exm:phiCompatible}, \cref{fig:acyclicNestedComplexes} illustrates polyhedral realizations of the nested complexes~$\nestedComplex[\lattice_1][\building_1]$ and~$\nestedComplex[\lattice_2][\building_2]$ as subrealizations of the boolean nested complex~$\nestedComplex[][\building']$.
The fans of \cref{fig:acyclicNestedComplexes}\,(top) are subfans of the nested fan~$\nestedFan[][\building']$, and the polytopal complexes of \cref{fig:acyclicNestedComplexes}\,(bottom) are subcomplexes of the boundary complex of the nestohedron~$\Nest(\building')$.
The subfans of~$\nestedFan[][\building']$ are normal to the subcomplexes of~$\Nest(\building')$.
This example will be largely extended in \cref{subsec:subrealizations}.
\end{example}


\subsection{Compatibility guarantees}
\label{subsec:tame}

We now try to find sufficient conditions for a pair~$(\building, \building')$ to be $\phi$-compatible.
For this, we consider the following properties of~$\phi$.

\begin{definition}
\label{def:tame}
We say that a map~$\phi : \lattice \to \lattice'$ is
\begin{itemize}
\item \defn{atom exhaustive} if~$\phi(\lattice)$ contains all atoms~$A$ of~$\lattice'$ such that~$A \le \phi(\topone_\lattice)$, 
\item \defn{join preserving} if~$\phi(X \vee_\lattice Y) = \phi(X) \vee_{\lattice'} \phi(Y)$ for any~$X,Y \in \lattice$,
\item \defn{cover preserving} if it sends cover relations of~$\lattice$ to cover relations of~$\lattice'$. (Note that this holds in particular if~$\lattice$ and~$\lattice'$ are ranked and~$\phi$ is \defn{rank shifting}, meaning that~$\rank[\phi(x)] - \rank[x]$ is the same for all~$x \in \lattice$.)
\end{itemize}
We say that~$\phi : \lattice \to \lattice'$ is \defn{tame} if $\phi$ is an order embedding and satisfies at least one of these three properties.
\end{definition}

\begin{remark}
Note that the three conditions of \cref{def:tame} are independent.
We have represented in \cref{fig:VennDiagram} an example in each region of the Venn diagram of these conditions.
\begin{figure}
	\capstart
	\centerline{\includegraphics[scale=.9]{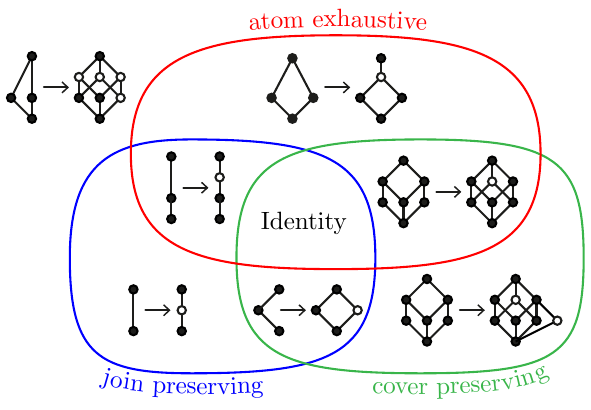}}
	\caption{The conditions of \cref{def:tame} are independent.}
	\label{fig:VennDiagram}
\end{figure}
\end{remark}

We first observe that if~$\lattice$ has a product structure, then a cover preserving order embedding also preserves joins compatible with the product structure.

\begin{lemma}
\label{lem:coverPreserving}
If~$\phi : \lattice \to \lattice'$ is a cover preserving order embedding and~$B_1, \dots, B_k \in \lattice$ are such that~$\lattice_{\le B_1 \vee \dots \vee B_k} \cong \lattice_{\le B_1} \times \dots \times \lattice_{\le B_k}$, then~$\phi(B_1 \vee_\lattice \dots \vee_\lattice B_k) = \phi(B_1) \vee_{\lattice'} \dots \vee_{\lattice'} \phi(B_k)$.
\end{lemma}

\begin{proof}
By an immediate induction, we can clearly assume that~$k = 2$.
Consider two saturated chains~$\botzero = X_0 \lessdot X_1 \lessdot \dots \lessdot X_p = B_1$ and~$\botzero = Y_0 \lessdot Y_1 \lessdot \dots \lessdot Y_q = B_2$.
We prove by induction on~$i + j$ that~$\phi(X_i \vee_\lattice Y_j) = \phi(X_i) \vee_{\lattice'} \phi(Y_j)$ for all~$0 \le i \le p$ and~$0 \le j \le q$.
It holds for~$i = 0$ and~$j = 0$ since~$\phi(X_0 \vee_\lattice Y_0) = \phi(\botzero) = \phi(\botzero) \vee_{\lattice'} \phi(\botzero) = \phi(X_0) \vee_{\lattice'} \phi(Y_0)$.
Assume now that~$i > 0$ (the case~$j > 0$ is symmetric).
We have
\[
\phi(X_{i-1} \vee_\lattice Y_j) = \phi(X_{i-1}) \vee_{\lattice'} \phi(Y_j) \le \phi(X_i) \vee_{\lattice'} \phi(Y_j) \le \phi(X_i \vee_\lattice Y_j),
\]
where the equality holds by induction hypothesis and the inequalities hold since~$\phi$ is order preserving.
Moreover,~$X_{i-1} \vee_\lattice Y_j \lessdot X_i \vee_\lattice Y_j$ is a cover relation of~$\lattice$ (since~$X_{i-1} \lessdot X_i$ is a cover relation and~$\lattice_{\le B_1 \vee B_2}$ has a product structure).
Since~$\phi$ is cover preserving, we obtain that~$\phi(X_{i-1} \vee_\lattice Y_j) \lessdot \phi(X_i \vee_\lattice Y_j)$ is a cover relation.
As~$X_i \not\le X_{i-1} \vee_\lattice Y_j$ and~$\phi$ is order reflecting, we have~$\phi(X_i) \not\le \phi(X_{i-1} \vee_\lattice Y_j)$, hence~$\phi(X_{i-1} \vee_\lattice Y_j) \ne \phi(X_i) \vee_{\lattice'} \phi(Y_j)$.
We obtain that~$\phi(X_i \vee_\lattice Y_j) = \phi(X_i) \vee_{\lattice'} \phi(Y_j)$, which proves our induction.
Finally, for~$i = p$ and~$j = q$, we obtain that~$\phi(B_1 \vee_\lattice B_2) = \phi(B_1) \vee_{\lattice'} \phi(B_2)$ as desired.
\end{proof}

\begin{proposition}
\label{prop:phiCompatible}
If~$\phi$ is tame and~$\building = \phi^{-1}(\building')$ (or equivalently~$\building' \cap \phi(\lattice) = \phi(\building)$), then $(\building, \building')$ is $\phi$-compatible.
\end{proposition}

\begin{proof}
Consider $\nested \in \nestedComplex[\lattice][\building]$ and~$\nested' \eqdef \phi(\nested)$.
We want to prove that~$\nested' \in \nestedComplex[\lattice'][\building']$ that is,
\begin{itemize}
\item $\nested' \cap \connectedComponents(\building') = \varnothing$, and
\item for any~$B_1', \dots, B_k'$ pairwise incomparable in~$\nested'$, the join~$B_1' \vee_{\lattice'} \dots \vee_{\lattice'} B_k'$ is not in~$\building'$.
\end{itemize}

For the first point, we have~$\building' \cap \phi(\lattice) = \phi(\building)$, thus $\connectedComponents(\building') \cap \phi(\lattice) \subseteq \connectedComponents(\phi(\building))= \phi(\connectedComponents(\building))$ since~$\phi$ is an order embedding. Hence,~$\nested \cap \connectedComponents(\building) = \varnothing$ implies $\nested' \cap \connectedComponents(\building') = \varnothing$.

For the second point, let~$B_1', \dots, B_k'$ be pairwise incomparable in~$\nested'$, and~${Y' \!\eqdef\! B_1' \vee_{\lattice'} \dots \vee_{\lattice'} B_k'}$.
Let~$B_1, \dots, B_k \in \nested$ such that~$B_i' = \phi(B_i)$, let~${Z \eqdef B_1 \vee_\lattice \dots \vee_\lattice B_k}$, and~$Z' \eqdef \phi(Z)$.
Note that~$Y' \le Z'$ since~$\phi$ is order preserving.
We now distinguish three cases, corresponding to the different cases in \cref{def:tame}.

\para{$\phi$ is join preserving}
As~$\nested \in \nestedComplex$, we have~$Z \notin \building$.
As~$\phi$ is lattice preserving, we obtain that~$Y' = Z' \notin \phi(\building) = \building'$.

\para{$\phi$ is cover preserving}
As~$\nested \in \nestedComplex$, we have~$Z \notin \building$ and moreover~$\lattice_{\le Z} \cong \prod_{i \in [k]} \lattice_{\le B_i}$ by \cref{prop:nestedFactorization}.
By \cref{lem:coverPreserving}, we obtain that~$Y' = Z' \notin \phi(\building) = \building'$.

\para{$\phi$ is atom exhaustive}
Assume by contradiction that~$Y' \in \building'$.
Since~$\building'$ is an $\lattice'$-building set, we have~$\lattice'_{\le Z'} \cong \prod_{j \in [\ell]} \lattice'_{\le C_j'}$ where~$\{C'_1, \dots, C'_\ell\} = \max(\building'_{\le Z'})$.
As~$Y' \in \building'$ and~$Y' \le Z'$, there is~$j \in [\ell]$ such that~$Y' \le C_j'$ by maximality.
If~$\ell > 1$, then there is~$i \in [\ell]$ distinct from~$j$, and an atom~${A' \le C_i'}$.
As~$\phi$ is atom exhaustive, there is~$A \in \lattice$ such that~${A' = \phi(A)}$.
Note that~$A' \le Z'$ so that~$A \le Z$ since~$\phi$ is order reflecting.
Moreover, the product structure of~$\lattice'_{\le Z'}$ ensures that~$A' \not\le C_j'$, therefore for any~$m \in [k]$, we have~$A' \not\le B_m'$, hence~${A \not\le B_m}$ since~$\phi$ is order preserving.
Additionally, $A$ is an atom of~$\lattice$ since~$\phi$ is order embedding.
Hence,~${A \in \building}$, and we obtain that~${\max(\building_{\le Z}) \!\ne\! \{B_1, \dots, B_k\}}$, which by \cref{prop:nestedFactorization} contradicts that~${\{B_1, \dots, B_k\} \in \nestedComplex}$, hence that~${\nested \in \nestedComplex}$.
\end{proof}

\begin{example}
Note that none of the conditions of \cref{prop:phiCompatible} is superfluous to guarantee $\phi$-compatibility.
Indeed, consider the lattices
\[
\lattice = 
\begin{tikzpicture}[scale=.7, baseline=.6cm]
	\node (a) at (0,0) {$a$};
	\node (b) at (-1,1) {$b$};
	\node (c) at (1,1) {$c$};
	\node (d) at (0,2) {$d$};
	\draw (a) -- (b);
	\draw (a) -- (c);
	\draw (b) -- (d);
	\draw (c) -- (d);
\end{tikzpicture}
\qquad\text{and}\qquad
\lattice' = 
\begin{tikzpicture}[scale=.7, baseline=.95cm]
	\node (1) at (0,0) {$1$};
	\node (2) at (-1,1) {$2$};
	\node (3) at (0,1) {$3$};
	\node (4) at (1,1) {$4$};
	\node (5) at (-1,2) {$5$};
	\node (6) at (0,2) {$6$};
	\node (7) at (1,2) {$7$};
	\node (8) at (0,3) {$8$};
	\draw (1) -- (2);
	\draw (1) -- (3);
	\draw (1) -- (4);
	\draw (2) -- (5);
	\draw (2) -- (6);
	\draw (3) -- (5);
	\draw (3) -- (7);
	\draw (4) -- (6);
	\draw (4) -- (7);
	\draw (5) -- (8);
	\draw (6) -- (8);
	\draw (7) -- (8);
\end{tikzpicture}
.
\]
Then
\begin{itemize}
\item for the identity map~$\iota : \lattice \to \lattice$, the $\lattice$-building sets~$\building_1 \eqdef \{b,c\}$ and~$\building_2 \eqdef \{b,c,d\}$ form a $\iota$-incompatible pair~$(\building_1, \building_2)$ (since~$\{b,c\} \in \nestedComplex \ssm \nestedComplex[\lattice][\building']$ as~$b \vee_\lattice c = d \in \building' \ssm \building$), although $\iota$ satisfies all three conditions of \cref{def:tame},
\item for the (non-tame) order embedding~$\phi : \lattice \to \lattice'$ given by~$a \mapsto 1$, $b \mapsto 2$, $c \mapsto 3$, and~${d \mapsto 8}$, the $\lattice$-building set~$\building \eqdef \{b,c\}$ and the $\lattice'$-building set~$\building' \eqdef \{2,3,4,5\}$ form a \mbox{$\phi$-incom}\-patible pair~$(\building, \building')$ (since~$\{b,c\} \in \nestedComplex$ $b \vee_\lattice c = d \notin \building$ but~$\phi(\nested) = \{2,3\} \notin \nestedComplex[\lattice'][\building']$ as~${2 \vee_{\lattice'} 3 = 5 \in \building'}$), although~$\building = \phi^{-1}(\building')$.
\end{itemize}
\end{example}

\begin{example}
Note that none of the conditions of \cref{prop:phiCompatible} is necessary to have \mbox{$\phi$-compa}\-tibility.
Indeed, consider the lattices
\[
\lattice = 
\begin{tikzpicture}[scale=.7, baseline=.6cm]
	\node (a) at (0,0) {$a$};
	\node (b) at (-1,1) {$b$};
	\node (c) at (1,1) {$c$};
	\node (d) at (0,2) {$d$};
	\draw (a) -- (b);
	\draw (a) -- (c);
	\draw (b) -- (d);
	\draw (c) -- (d);
\end{tikzpicture}
\qquad\text{and}\qquad
\lattice' = 
\begin{tikzpicture}[scale=.7, baseline=.95cm]
	\node (1) at (0,0) {$1$};
	\node (2) at (-1,1) {$2$};
	\node (3) at (0,1) {$3$};
	\node (4) at (1,1) {$4$};
	\node (5) at (-1,2) {$5$};
	\node (6) at (0,3) {$6$};
	\draw (1) -- (2);
	\draw (1) -- (3);
	\draw (1) -- (4);
	\draw (2) -- (5);
	\draw (3) -- (5);
	\draw (4) -- (6);
	\draw (5) -- (6);
\end{tikzpicture}
\]
and the order embedding~$\phi : \lattice \to \lattice'$ given by~$a \mapsto 1$, $b \mapsto 2$, $c \mapsto 3$, and $d \mapsto 6$, the $\lattice$-building set~$\building \eqdef \{b,c\}$, and the $\lattice'$-building set~$\building' \eqdef \{2,3,6\}$.
Then $(\building, \building')$ is~$\phi$-compatible while~$\phi$ is not tame and~$\building \ne \phi^{-1}(\building')$.
\end{example}

\begin{remark}
\label{rem:BackmanDanner}
Following an early presentation of our work, where \cref{prop:phiCompatible} was not presented, S.~Backman and R.~Danner independently investigated in~\cite{BackmanDanner} conditions on a map~$\phi : \lattice \to \lattice'$ to guarantee that the $\lattice$-nested complexes embed as subcomplexes of $\lattice'$-nested complexes. 
They prove in~\cite[Thm.~4.5]{BackmanDanner} that if~$\phi$ is a meet preserving order embedding such that, below any element of~$\phi(\lattice)$, any irreducible of~$\lattice'$  is the image of an irreducible of~$\lattice$, then any $\lattice$-building set~$\building$ is the preimage~$\phi^{-1}(\building')$ of an $\lattice'$-building set~$\building'$, and the $\lattice$-nested complex~$\nestedComplex$ embeds as a subcomplex of the $\lattice'$-nested complex~$\nestedComplex[\lattice'][\building']$ (here, irreducible means that the corresponding lower interval does not exhibit a product structure).
We note that this result can a posteriori be derived directly from \cref{prop:buildingClosure,prop:phiCompatible}.
Namely, the condition that any irreducible of~$\lattice'$  is the image of an irreducible of~$\lattice$ ensures that
\begin{itemize}
\item the map~$\phi$ is atom exhaustive (since any atom is irrecucible),
\item for any $\lattice$-building set~$\building$, the $\lattice'$-building closure~$\building'$ of~$\phi(\building)$ satisfies~$\building = \phi^{-1}(\building')$ (this can be seen from the second sentence of~\cref{prop:buildingClosure}).
\end{itemize}
In view of \cref{prop:phiCompatible}, the conditions in~\cite[Thm.~4.5]{BackmanDanner} can be weakened: we just assume that $\phi$ is an order embedding (not necessarily that it preserves meets), and atom exhaustive (not necessarily that it is consistent).
Finally, we note that, besides the atom exhaustive criterion, \cref{prop:buildingClosure,prop:phiCompatible} provide alternative criteria to guarantee nested complex embeddings, which will be particularly useful in \cref{sec:Bergman}.
\end{remark}


\subsection{Pushable and pullable building sets}
\label{subsec:pushablePullable}

In practice, we will be given only one of~$\building$ and~$\building'$, and we will try to find the other one to create a $\phi$-compatible pair~$(\building, \building')$.
For that, we will try to apply \cref{prop:phiCompatible}, and thus to find a pair with~$\building = \phi^{-1}(\building')$.
The following definition thus considers the natural candidates.

\begin{definition}
\label{def:phiPushablePullable}
Consider an order embedding~$\phi : \lattice \to \lattice'$ between two lattices~$\lattice,\lattice'$.
We say that
\begin{itemize}
\item an $\lattice'$-building set~$\building'$ is \defn{$\phi$-pullable} if $\phi^{-1}(\building')$ is an $\lattice$-building set,
\item an $\lattice$-building set~$\building$ is \defn{$\phi$-pushable} if there exists an $\lattice'$-building set~$\building'$ such that~${\building \!=\! \phi^{-1}(\building')}$.
\end{itemize}
\end{definition}

\begin{lemma}
An $\lattice$-building set~$\building$ is pushable if and only if~$\building = \phi^{-1}(\building')$ where~$\building'$ is the \mbox{$\lattice'$-building} closure of~$\phi(\building)$.
\end{lemma}

\begin{proof}
Assume that there is an $\lattice'$-building set~$\building''$ such that~${\building = \phi^{-1}(\building'')}$, and let~$\building'$ denote the $\lattice'$-building closure of~$\phi(\building)$.
Then~$\phi(\building) \subseteq \building' \subseteq \building''$, hence~$\phi(\building) \subseteq \building' \cap \phi(\lattice) \subseteq \building'' \cap \phi(\lattice) = \phi(\building)$, which proves that~$\building'$ suits.
The reverse implication is immediate.
\end{proof}

\begin{corollary}
If~$\phi : \lattice \to \lattice'$ is tame, then
\begin{itemize}
\item if an $\lattice'$-building set~$\building'$ is \defn{$\phi$-pullable}, then~$(\building, \building')$ is $\phi$-compatible, where~$\building = \phi^{-1}(\building')$,
\item if an $\lattice$-building set~$\building$ is \defn{$\phi$-pushable}, then~$(\building, \building')$ is $\phi$-compatible, where~$\building'$ is the \mbox{$\lattice'$-building} closure of~$\phi(\building)$.
\end{itemize}
\end{corollary}

\begin{example}
Consider the inclusion~$\iota : \lattice \to \lattice'$ where~$\lattice$ is a bottom interval of~$\lattice'$.
It is clearly tame (in fact, it satisfies all three conditions of \cref{def:tame}).
Moreover, any $\lattice'$-building set~$\building'$ is $\iota$-pullable (as the decomposition of any~$Y \in \lattice$ coincides in~$\lattice$ and~$\lattice'$) and any $\lattice$-building~$\building$ set is $\iota$-pushable (as the $\lattice'$-building closure of~$\building$ only adds elements in~$\lattice' \ssm \lattice$ by \cref{prop:buildingClosure}).
Note in particular that this example implies that, for any face~$F$ of an oriented matroid~$\OM$, facial nested complexes of~$\OM_{|F}$ embed in facial nested complexes of~$\OM$.
\end{example}

\begin{remark}
\label{rem:notAllPushablePullable}
Note that not all $\lattice$-building sets~$\building$ are $\phi$-pushable, nor all $\lattice'$-building sets~$\building'$ are $\phi$-pullable, even if~$\phi$ satisfies all three conditions of \cref{def:tame}.
For instance, consider the two order embeddings
\[
\lattice = 
\begin{tikzpicture}[scale=.5, baseline=.7cm]
	\node (a) at (0,0) {$a$};
	\node (b) at (-1,1) {$b$};
	\node (c) at (1,1) {$c$};
	\node (d) at (0,2) {$d$};
	\node (e) at (1,3) {$e$};
	\draw (a) -- (b);
	\draw (a) -- (c);
	\draw (b) -- (d);
	\draw (c) -- (d);
	\draw (d) -- (e);
\end{tikzpicture}
\qquad\xrightarrow{\phi}\qquad
\lattice' = 
\begin{tikzpicture}[scale=.5, baseline=.7cm]
	\node (1) at (0,0) {$1$};
	\node (2) at (-1,1) {$2$};
	\node (3) at (1,1) {$3$};
	\node (4) at (0,2) {$4$};
	\node (5) at (2,2) {$5$};
	\node (6) at (1,3) {$6$};
	\draw (1) -- (2);
	\draw (1) -- (3);
	\draw (2) -- (4);
	\draw (3) -- (4);
	\draw (3) -- (5);
	\draw (4) -- (6);
	\draw (5) -- (6);
\end{tikzpicture}
\qquad\xrightarrow{\phi'}\qquad
\lattice' = 
\begin{tikzpicture}[scale=.5, baseline=.7cm]
	\node (1) at (0,0) {$\alpha$};
	\node (2) at (-1,1) {$\beta$};
	\node (3) at (1,1) {$\gamma$};
	\node (4) at (0,2) {$\delta$};
	\node (5) at (1,2) {$\epsilon$};
	\node (6) at (2,2) {$\zeta$};
	\node (7) at (1,3) {$\eta$};
	\draw (1) -- (2);
	\draw (1) -- (3);
	\draw (2) -- (4);
	\draw (3) -- (4);
	\draw (3) -- (5);
	\draw (3) -- (6);
	\draw (4) -- (7);
	\draw (5) -- (7);
	\draw (6) -- (7);
\end{tikzpicture}
\]
where~$\phi$ is given by~$a \mapsto 1$, $b \mapsto 2$, $c \mapsto 3$, $d \mapsto 4$, and $e \mapsto 6$ and~$\phi'$ is given by~$1 \mapsto \alpha$, $2 \mapsto \beta$, $3 \mapsto \gamma$, $4 \mapsto \delta$, $5 \mapsto \zeta$, and $6 \mapsto \eta$.
Then the $\lattice'$-building set~$\{2,3,5\}$ is neither $\phi$-pullable nor $\phi'$-pushable.
\end{remark}


\section{Atomic nested complex embeddings}
\label{sec:atomic}

In this section, we show that when~$\lattice$ is atomic, the $\lattice$-nested complex of an $\lattice$-building set can always be embedded as a subcomplex of the boolean nested complex of a boolean building set.
We then use it to realize the $\nestedComplex$ as a subfan of a boolean nested fan (recovering the fan realization of $\nestedComplex$ in~\cite{FeichtnerYuzvinsky2004}), and as a subcomplex of the boundary complex of a nestohedron.


\subsection{Atomic nested complexes}
\label{subsec:atomic}

We first recall the definition of atomic lattices and embed their nested complexes into boolean nested complexes.

\begin{definition}
A lattice~$\lattice$ is \defn{atomic} if every element is a join of atoms.
\end{definition}

\begin{proposition}
\label{prop:atomicTame}
For any finite atomic lattice~$\lattice$ with atoms $\ground$, the map~$\phi : \lattice \mapsto 2^\ground$ defined by~$\phi(X) \eqdef \ground_{\le X}$ is an atom exhaustive order embedding, hence is tame.
\end{proposition}

\begin{proof}
The map~$\phi$ is
\begin{itemize}
\item atom exhaustive since~$\phi(s) = \{s\}$ for all~$s \in \ground$,
\item order preserving since~$X \le Y$ implies~$S_{\le X} \subseteq S_{\le Y}$ for all~$X,Y \in \lattice$,
\item order reflecting since~$X = \bigvee_{s \in S_{\le X}} s$ for all~$X \in \lattice$ since~$\lattice$ is atomic.
\qedhere
\end{itemize}
\end{proof}

\begin{proposition}
\label{prop:atomicPushable}
Consider a finite atomic lattice~$\lattice$ with atoms $\ground$ and the map~$\phi : \lattice \mapsto 2^\ground$ defined by~$\phi(X) \eqdef \ground_{\le X}$.
For any $\lattice$-building set~$\building$, the (boolean) building closure~$\building'$ of $\phi(\building)$ satisfies~$\building = \phi^{-1}(\building')$.
In other words, any $\lattice$-building set~$\building$ is $\phi$-pushable.
\end{proposition}

\begin{proof}
Let~$X \in \lattice$ such that~$\phi(X) \in \building'$ and $\{B_1, \dots, B_k\} = \max(\building_{\le X})$.
As $\building$ is an $\lattice$-building set, we have~$\lattice_{\le X} \cong \prod_{i \in [k]} \lattice_{\le B_i}$, so that~$\phi(B_1), \dots, \phi(B_k)$ are pairwise disjoint.
As~$\phi(X)$ is in the (boolean) building closure~$\building'$ of~$\phi(\building)$, \cref{def:booleanBuildingSetClosure} ensures that there are~${C_1, \dots, C_\ell \in \building}$ such that~$\phi(X) = \bigcup_{i \in [\ell]} \phi(C_i)$ and the intersection graph of~$\phi(C_1), \dots, \phi(C_k)$ is connected.
As~$\phi(C_i) \subseteq \phi(X)$ and~$\phi$ is order reflecting, we obtain that~$C_1, \dots, C_\ell \in \building_{\le X}$.
As the intersection graph of~$\phi(C_1), \dots, \phi(C_k)$ is connected while~$\phi(B_1), \dots, \phi(B_k)$ are pairwise disjoint, we conclude that~$C_1, \dots, C_\ell$ all belong to the same~$B_j$.
We thus obtain that~$\bigcup_{i \in [\ell]} \phi(C_i) \subseteq \phi(B_j) \subseteq \phi(X)$.
The equality~$\phi(X) = \bigcup_{i \in [\ell]} \phi(C_i)$ thus shows that~$\phi(B_j) = \phi(X)$.
As~$\phi$ is injective, we obtain that~$X = B_j \in \building$.
\end{proof}

\begin{corollary}
\label{coro:atomicInBoolean}
The $\lattice$-nested complex $\nestedComplex$ of any $\lattice$-building set~$\building$ on a finite atomic lattice~$\lattice$ with atoms $\ground$ is isomorphic to a subcomplex of the (boolean) nested complex $\nestedComplex[][\building']$ on the (boolean) building set~$\building'$ obtained as the (boolean) building closure of $\set{\ground_{\le B}}{B \in \building}$.
\end{corollary}

\begin{proof}
This follows from \cref{prop:atomicTame,prop:atomicPushable,prop:phiCompatible}.
\end{proof}

\begin{remark}
\label{rem:atomicInBoolean}
\cref{prop:phiCompatible} actually tells that~$\nestedComplex$ is isomorphic to a subcomplex of~$\nestedComplex[][\building']$ for any boolean building set~$\building'$ such that~$\building = \set{X \in \lattice}{\ground_{\le X} \in \building'}$.
The existence of such a building set~$\building'$ is not clear in general (as discussed in \cref{rem:notAllPushablePullable}), but is guarantied here by \cref{prop:atomicPushable}, which says that the building closure of $\set{\ground_{\le B}}{B\in\building}$ suits.
\end{remark}

\begin{example}
Following on \cref{exm:phiCompatible}, \cref{coro:atomicInBoolean} embeds the nested complexes~$\nestedComplex[\lattice_1][\building_1]$ and~$\nestedComplex[\lattice_2][\building_2]$ into the boolean nested complex~$\nestedComplex[][\building']$.
\end{example}


\subsection{Polyhedral realizations}
\label{subsec:subrealizations}

We now use the embedding of \cref{coro:atomicInBoolean} to obtain polyhedral fan realizations of the $\lattice$-nested complexes of $\lattice$-building sets for any finite atomic lattice~$\lattice$.

\begin{corollary}
\label{coro:subrealizations}
The $\lattice$-nested complex $\nestedComplex$ of any $\lattice$-building set~$\building$ on a finite atomic lattice~$\lattice$ with atoms $\ground$ is realized by 
\begin{itemize}
\item a subfan of the (boolean) nested fan $\nestedFan[][\building']$,
\item a subcomplex of the boundary complex of the polar of the (boolean) nestohedron~$\Nest(\building')$,
\end{itemize}
where~$\building'$ is the (boolean) building closure of $\set{\ground_{\le B}}{B \in \building}$.
\end{corollary}

\begin{remark}
Following \cref{rem:atomicInBoolean}, \cref{coro:subrealizations} holds for any boolean building set~$\building'$ such that~$\building = \set{X \in \lattice}{\ground_{\le X} \in \building'}$.
\end{remark}

\begin{example}
Following on \cref{exm:phiCompatible}, \cref{coro:subrealizations} realizes the nested complexes~$\nestedComplex[\lattice_1][\building_1]$ and~$\nestedComplex[\lattice_2][\building_2]$ as in \cref{fig:acyclicNestedComplexes}.
\end{example}

\begin{remark}
\label{rem:badFan}
As illustrated in \cref{fig:acyclicNestedComplexes}, note that
\begin{enumerate}
\item the dimension of the affine span of the subfan of \cref{coro:subrealizations} may exceed the dimension of $\nestedComplex$,
\item the subcomplex of the boundary complex of the nestohedron~$\Nest(\building')$ does not form the boundary complex of a polytope.
\end{enumerate}
As established in \cref{sec:polytopalRealizations}, when $\lattice$ is the face lattice of a realizable oriented matroid, there are alternative realizations as a complete simplicial fan in an affine subspace of the right dimension, and as the boundary complex of a polytope.
\end{remark}

We now give an explicit description of the subfan of \cref{coro:subrealizations} to recover the explicit fan realizations of E.~M.~Feichtner and S.~Yuzvinsky~\cite{FeichtnerYuzvinsky2004}.

\begin{definition}[\cite{FeichtnerYuzvinsky2004}]
\label{def:latticeNestedFan}
Let~$\lattice$ be a finite atomic lattice with atoms~$\ground$.
The \defn{$\lattice$-nested fan} of an $\lattice$-building set~$\building$ is the polyhedral fan
\[
\nestedFan \eqdef \set{\con[\nested]}{\nested \in \nestedComplex }
\]
where~$\con[\nested] \eqdef \cone\set{\b{e}_B}{B \in \nested}$ is the polyhedral cone generated by the characteristic vectors $\b{e}_X \eqdef \sum_{s \in \ground_{\le X}} \b{e}_s \in \R^\ground$ of~$\ground_{\le X}$ for all~$X \in \nested$.
\end{definition}

\begin{proposition}[\cite{FeichtnerYuzvinsky2004}]
\label{prop:subfanFY}
For any finite atomic lattice~$\lattice$ and any $\lattice$-building set~$\building$, the $\lattice$-nested fan~$\nestedFan$ is a simplicial fan realization of the $\lattice$-nested complex~$\nestedComplex$.
\end{proposition}

\begin{proof}
This result is proved in~\cite{FeichtnerYuzvinsky2004}.
It can alternatively be seen as a consequence of \cref{def:nestedFan,thm:nestedFan,coro:subrealizations}.
\end{proof}

\begin{remark}
Following up on \cref{rem:BackmanDanner}, we note that \cite[Coro.~4.6]{BackmanDanner} is a direct consequence of \cref{coro:subrealizations}.
\end{remark}


\section{The Bergman embedding}
\label{sec:Bergman}

Generalizing tropicalizations of linear varieties to non-necessarily realizable (unoriented) matroids, the Bergman fan~$\Bergfan[\UOM]$ of a matroid~$\UOM$ was first introduced in \cite[Sect.~9.3]{Sturmfels2002} by B.~Sturmfels, who suggested studying its combinatorial structure.
This was first done by F.~Ardila and C.~Klivans in~\cite{ArdilaKlivans2006}, who showed that the Bergman fan~$\Bergfan[\UOM]$ could be triangulated using the order complex of its flat lattice~$\Flats(\UOM)$ (which is the nested complex of the maximal building set, by \cref{exm:orderComplex}).
E.~M.~Feichtner and B.~Sturmfels then extended this result in \cite{FeichtnerSturmfels} by showing that, for any $\Flats(\UOM)$-building set~$\fbuilding[G]$, the nested fan $\nestedFan[{\Flats(\UOM)}][{\fbuilding[G]}]$ of \cref{prop:subfanFY} triangulates the Bergman fan. The Chow ring of the matroid~$\UOM$ is the toric variety associated to its fan~$\Bergfan[\UOM]$ \cite[Sect.~5]{FeichtnerYuzvinsky2004}, and it would later become a central ingredient for the celebrated proof of the Heron--Rota--Welsh conjecture by K.~Adiprasito, J.~Huh, and E.~Katz~\cite{AdiprasitoHuhKatz}.

Motivated by the theory of total positivity, F.~Ardila, C.~Klivans, and L.~Williams studied the positive Bergman fan of an oriented matroid $\OM$, generalizing the notion of positive tropical linear variety \cite{ArdilaKlivansWilliams2006}.
They proved that the order complex of $\FL(\OM)$ triangulates the positive Bergman fan. Our first result is the extension of this result to arbitrary facial building sets.

Let $\OM$ be an oriented matroid on $\ground$, and consider a weight function $\omega\in\R^\ground$. For any circuit $C\in \circuits[\OM]$, define $\ini_\omega(C)$ to be the elements of $C$ which have the largest weight. Let $\ini_\omega(\circuits[\OM]) \eqdef \set{\ini_\omega(C)}{C\in \circuits[\OM]}$ and define $\OM_\omega$ to be the oriented matroid on $\ground$ whose collection of circuits is in $\ini_\omega(\circuits[\OM])$ (see \cite{ArdilaKlivansWilliams2006}) for details).

The \defn{Bergman fan} $\Bergfan$ and the \defn{positive Bergman fan} $\posBergfan$ are respectively defined as
\[
\Bergfan \eqdef \set{\omega\in\R^\ground}{\OM_\omega\text{ has no loops}}
\qquad\text{and}\qquad
\posBergfan \eqdef \set{\omega\in\R^\ground}{\OM_\omega\text{ is acyclic}}.
\]

\begin{proposition}\label{prop:PosBergmanFacial}
Let $\OM$ be an acyclic oriented matroid and~$\fbuilding$ be a facial building set for~$\OM$.
Then the $\FL(\OM)$-nested fan~$\nestedFan[{\FL(\OM)}][\fbuilding]$ is a triangulation of the positive Bergman fan~$\posBergfan$.
\end{proposition}

\begin{proof}
This is proved in \cite[Cor.~3.5]{ArdilaKlivansWilliams2006} for the maximal building set (whose nested complex is the order complex of $\FL(\OM)$, see \cref{exm:orderComplex}). 
In \cite[Thm.~4.2]{FeichtnerMuller2005}, it is proved that if~$\fbuilding\subseteq \fbuilding'$ are building sets, then~$\nestedFan[{\FL(\OM)}][\fbuilding]$ and $\nestedFan[{\FL(\OM)}][\fbuilding']$ have the same support, which concludes the proof.
\end{proof}

The  construction of~\cite{ArdilaKlivansWilliams2006} sees the positive Bergman fan~$\posBergfan$ as a subfan of the Bergman fan~$\Bergfan$, which implies in particular that any subdivision of the latter induces a subdivision of the former.
They focus specifically on two subdivisions:
\begin{itemize}
\item The `fine' subdivision is associated to the maximal building sets of the lattices of faces and flats of $\OM$, and the construction provides an embedding of the corresponding nested complexes. The `fine' subdivision is a special case of \cref{exm:maximalBuildingSetsCompatible} for the inclusion of the face lattice inside the flat lattice.
\item The `coarse' subdivision is used to study the positive Bergman fan of the braid arrangement. For arbitrary matroids, the `coarse' subdivision is not necessarily associated to a nested complex set, and it might not even be a simplicial complex \cite[Ex.~5.9]{FeichtnerSturmfels}. E.~M.~Feichtner and B.~Sturfmels characterize in \cite[Thm.~5.3]{FeichtnerSturmfels} when it is a nested complex, in which case it is the nested complex of the minimal building set of \cref{exm:minMaxBuildingSet}. In particular, the braid arrangement is an example where the `coarse' subdivision is a nested complex by \cite[Rmk.~5.4]{FeichtnerSturmfels}, thus providing another example of a nested complex embedding from the face to the flat lattices of an oriented matroid. The `coarse' subdivision was used again in \cite{ArdilaReinerWilliams2006} with the positive Bergman fans of Coxeter arrangements, where we have again that the coarse subdivision is a nested complex \cite[Thm 1.2]{ArdilaReinerWilliams2006}.
\end{itemize}
Here, we apply the results of \cref{sec:latticeEmbedding} to study more generally compatible pairs of facial and flatial building sets, beyond these particular examples.

\begin{definition}
A \defn{flat} of an oriented matroid $\OM$ is the zero set of a covector, that is $\ground \ssm \underline{v^*}$ for $v^* \in \covectors$.
The \defn{flat lattice}~$\Flats(\OM)$ is the set of flats of $\OM$ ordered by inclusion.
It is ranked by dimension.
\end{definition}

\begin{remark}
Note that:
\begin{enumerate}
\item The flat lattice~$\Flats(\OM)$ is actually an invariant of the underlying unoriented matroid~$\UOM$.
\item By definition every face in~$\FL(\OM)$ is also a flat. 
\end{enumerate}
\end{remark}

\begin{definition}
A \defn{flatial building set} is a pair~$(\fbuilding[G],\OM)$, where~$\OM$ is an oriented matroid and~$\fbuilding[G]$ is a $\Flats(\OM)$-building set over the flat lattice~$\Flats(\OM)$.
We also often say that~$\fbuilding[G]$ is a flatial building set for~$\OM$.
The \defn{flatial nested complex} of a flatial building set~$(\fbuilding[G],\OM)$ is the~$\Flats(\OM)$-nested complex~$\nestedComplex[\Flats(\OM)][\building]$ of~$\fbuilding[G]$.
\end{definition}

\begin{remark}
Specializing \cref{prop:linksLatticeNestedComplex} to flatial building sets, we recover the link description independently established~\cite[Thm.~1.6]{BraunerEurPrattVlad}.
\end{remark}

The following proposition is immediate.

\begin{proposition}
\label{prop:BergmanTame}
The inclusion map~$\iota : \FL(\OM)\xhookrightarrow{}\Flats(\OM)$ is a ranked order embedding, hence it is tame.
\end{proposition}

\begin{lemma}
\label{lem:directSumCartesianProductFlats}
The flat lattice of the direct sum is the Cartesian product of flat lattices.
Namely, if~$\OM$ and~$\OM'$ are oriented matroids on disjoint ground sets $\ground$ and $\ground'$, then
\[
\Flats(\OM \oplus \OM') \cong \Flats(\OM) \times \Flats(\OM'),
\]
where the isomorphism is given by $F \mapsto (F \cap \ground, F \cap \ground')$.
Conversely, for an oriented matroid $\OM$ with ground set $\ground \sqcup \ground'$, if~$\Flats(\OM) \cong \Flats(\OM_{|\ground}) \times \Flats(\OM_{|\ground'})$, then $\OM = \OM_{|\ground} \oplus \OM_{|\ground'}$.
\end{lemma}

\begin{proof}
The proof is analogous to that of \cref{prop:directSumCartesianProductFL}.
\end{proof}

\begin{proposition}
\label{prop:BergmanPullable}
For any flatial building set~$\fbuilding[G]$ for~$\OM$, the intersection~$\fbuilding[F] \eqdef \fbuilding[G] \cap \FL(\OM)$ is a facial building set for~$\OM$.
In other words, any $\Flats(\OM)$-building set~$\fbuilding[G]$ is $\iota$-pullable.
\end{proposition}

\begin{proof}
Let $F \in \FL(\OM)$. 
As~$\fbuilding[G]$ is an $\Flats(\OM)$-building set, $\Flats(\OM)_{\leq F} \cong \prod_{i \in [k]} \Flats(\OM)_{\leq F_i}$, where $\{F_1, \dots, F_k\} \eqdef \max(\fbuilding[G]_{\le F})$.
By \cref{lem:directSumCartesianProductFlats}, $\OM_{|F}=\OM_{|F_1} \oplus \dots \oplus \OM_{|F_k}$.		
Hence, by \cref{prop:directSumCartesianProductFL}, we have~$\FL(\OM_{|F}) = \FL(\OM_{|F_1}) \times \dots \times \FL(\OM_{|F_k})$.
Thus, we have~$F_1, \dots, F_k \in \FL(\OM)$ and~$\fbuilding[F] \eqdef \fbuilding[G] \cap \FL(\OM)$ is a facial building set.
\end{proof}

\begin{proposition}
\label{prop:BergmanPushable}
If all the elements of $\OM$ are in~$\FL(\OM)$, then for any facial building set~$\fbuilding$, the $\Flats(\OM)$-building closure~$\fbuilding[G]$ of~$\fbuilding$ satisfies~$\fbuilding[G] \cap \FL(\OM) = \fbuilding$.
In other words, if~$S \subseteq \FL(\OM)$, any facial building set~$\fbuilding$ is $\iota$-pushable.
\end{proposition}

\begin{proof}
The inclusion~$\fbuilding[G] \cap \FL(\OM) \supseteq \fbuilding$ holds by definition of building closure.
To prove the inclusion~$\fbuilding[G] \cap \FL(\OM) \subseteq \fbuilding$, consider an element~$Z \in \FL(\OM) \ssm \fbuilding$.
Since~$Z$ does not belong to~$\fbuilding$, and using that all the elements of~$\OM$ are in~$\FL(\OM)$ in \cref{prop:directSumCartesianProductFL}, we know that
\[
\OM_{|Z} = \OM_{|F_1} \oplus \cdots \oplus \OM_{|F_k},
\]
where $\{F_1, \dots, F_k\} \eqdef \max(\fbuilding_{\leq Z})$. 
Following \cref{prop:buildingClosure}, we iteratively construct~$\fbuilding[G]$ by adding from small to large all elements of~$\Flats(\OM)$ whose lower interval does not decompose into a Cartesian product. However, note that for all $Y\subseteq Z$ we have 
\[
\OM_{|Y} = \OM_{|F_1 \cap Y} \oplus \cdots \oplus \OM_{|F_k \cap Y}
\]
and by \cref{lem:directSumCartesianProductFlats}
\[
\Flats(\OM_{|Y})= \Flats(\OM_{|F_1 \cap Y})\oplus \cdots \oplus  \Flats(\OM_{|F_k \cap Y}).
\]
Therefore, if~$Y$ is not contained in any~$F_i$ it has a Cartesian product decomposition and will not be added to~$\building[G]$.
In particular, $Z \notin \building[G]$ as desired.
\end{proof}

\begin{corollary}
Let $\OM$ be an oriented matroid with face lattice~$\FL(\OM)$ and flat lattice~$\Flats(\OM)$.
Then
\begin{enumerate}
\item any flatial building set~$\fbuilding[G]$ for~$\OM$ yields a facial building set~$\fbuilding \eqdef \fbuilding[G] \cap \FL(\OM)$ for $\OM$ and the facial nested complex $\nestedComplex[\FL(\OM)][\fbuilding]$ is a subcomplex of the flatial nested complex~$\nestedComplex[\Flats(\OM)][\fbuilding[G]]$,
\item if all the elements of $\OM$ are in~$\FL(\OM)$, then the facial nested complex $\nestedComplex[\FL(\OM)][\fbuilding]$ of any facial building set~$\fbuilding$ for~$\OM$ is a subcomplex of the flatial nested complex~$\nestedComplex[\Flats(\OM)][\fbuilding[G]]$ of the flatial building set for~$\OM$ obtained as the $\Flats(\OM)$-building closure of~$\fbuilding$.
\end{enumerate}
\end{corollary}

\begin{proof}
This follows from \cref{prop:BergmanTame,prop:BergmanPullable,prop:BergmanPushable,prop:phiCompatible}.
\end{proof}


\newpage
\addtocontents{toc}{\vspace{.1cm}}
\section*{Acknowledgments}

We thank Federico Ardila, Spencer Backman, Sarah Brauner, Giovanni Gaiffi, Pavel Galashin, Georg Loho, Luca Moci, and Andrew Sack for various suggestions.
We are grateful to the organizers and participants of the Simons Center Workshop on Combinatorics and Geometry of Convex Polyhedra (Mar.~2023), the Oberwolfach workshop on Geometric, Algebraic, and Topological Combinatorics (Dec.~2023), and of the 36th International Conference on Formal Power Series and Algebraic Combinatorics (FPSAC 24, Bochum, Jul.~2024) for inputs and discussions on presentations of this work.


\bibliographystyle{alpha}
\bibliography{posetAssociahedra}
\label{sec:biblio}

\end{document}